\input ./style/arxiv-general.cfg
\documentclass[aop,MSNbibl,seceqn,dvips]{arximspdf}
\makeatletter
   \@ifpackageloaded{graphicx}{}{\usepackage{graphicx}}
\makeatother
\usepackage{mathbh}
\usepackage{mathrsfs}

\doi{10.1214/15-AOP1038}
\volume{44}
\issue{4}
\pubyear{2016}
\firstpage{2889}
\lastpage{2979}
\docsubty{FLA}

\makeatletter
\newcommand{\fontas}{\fontsize{8.36}{10.36}\selectfont}
\def\sfrac#1#2{#1/#2}
\def\vfrac#1#2{(#1)/#2}
\def\afrac#1#2{#1/(#2)}

\def\sklfrac#1#2{(#1/#2)}
\def\sklvfrac#1#2{((#1)/#2)}

\def\sklvafrac#1#2{((#1)/(#2))}
\newcommand{\bbbv}{\mathbh{v}}

\newcommand{\rrvert}{\vert}

\newcommand{\rrVert}{\Vert}
\newcommand{\llvert}{\vert}
\newcommand{\llVert}{\Vert}
\renewcommand{\mid}{|}

\newcommand{\eqref}[1]{(\ref{#1})}
\newcommand{\overset}{\stackrel}
\newtheorem{teo}{Theorem}[section]
\newproclaim{defn}[teo]{Definition}
\newproclaim{rem}[teo]{Remark}

\newtheorem{cor}[teo]{Corollary}
\newtheorem{lem}[teo]{Lemma}
\newtheorem{claim}[teo]{Claim}
\newtheorem{prop}[teo]{Proposition}

\newcommand{\BN}{{\mathbb{N}}}
\newcommand{\BP}{{\mathbb{P}}}
\newcommand{\BZ}{{\mathbb{Z}}}
\newcommand{\CF}{{\mathcal{F}}}
\newcommand{\CG}{{\mathcal{G}}}
\newcommand{\CP}{{\mathcal{P}}}

\newcommand{\ind}{{\mathbh{1}}}
\renewcommand{\epsilon}{{\varepsilon}}
\newcommand{\Om}{{\Omega}}
\newcommand{\si}{{\sigma}}
\newcommand{\ep}{{\varepsilon}}
\newcommand{\esup}{\operatorname{esssup}}
\makeatother

\begin{document}
\begin{frontmatter}

\title{Local limit theorem and equivalence of dynamic and static
points of view for certain ballistic random walks in i.i.d. environments}
\runtitle{Local limit for certain ballistic RWRE}

\begin{aug}
\author[A]{\fnms{Noam}~\snm{Berger}\thanksref{M1,M2,T1}\ead[label=e1]{berger@math.huji.ac.il}},
\author[A]{\fnms{Moran}~\snm{Cohen}\corref{}\thanksref{M1,T1}\ead[label=e2]{moranski@math.huji.ac.il}}
\and
\author[B]{\fnms{Ron}~\snm{Rosenthal}\thanksref{M3,T2}\ead[label=e3]{ron.rosenthal@math.ethz.ch}}
\runauthor{N. Berger, M. Cohen and R. Rosenthal}
\affiliation{Hebrew University of Jerusalem\thanksmark{M1}, TU Munich\thanksmark{M2} and ETH Z\"urich\thanksmark{M3}}
\address[A]{N. Berger\\
M. Cohen\\
Einstein Institute of Mathematics\\
Edmond J. Safra Campus\\
The Hebrew University of Jerusalem \\
Givat Ram. Jerusalem 9190401\\
Israel\\
\printead{e1}\\
\phantom{E-mail:\ }\printead*{e2}}
%
\address[B]{R. Rosenthal\\
Department Mathematik\\
Eidgen\"{o}ssische Technische Hochschule\\
Gruppe 3, HG E 66.2\\
R\"{a}mistrasse 101, 8092\\
Z\"{u}rich\\
Switzerland\\
\printead{e3}}
\end{aug}
\thankstext{T1}{Supported in part by ERC StG Grant 239990.}
\thankstext{T2}{Supported in part by an ETH fellowship and by ERC StG
Grant 239990.}

%
\received{\smonth{8} \syear{2014}}
%
\revised{\smonth{6} \syear{2015}}

%
\begin{abstract}
In this work, we discuss certain ballistic random walks in random
environments
on $\mathbb{Z}^{d}$, and prove the equivalence between the static
and dynamic points of view in dimension $d\geq4$. Using this equivalence,
we also prove a version of a local limit theorem which relates the
local behavior of the quenched and annealed measures of the random
walk by a prefactor.
\end{abstract}

%
\end{frontmatter}

\setcounter{footnote}{2}

\section{Introduction}\label{sec1}

\subsection{Background}\label{sec1.1}

Let $d\geq1$. A random walk in a random environment (RWRE) on $\mathbb{Z}^{d}$
is defined by the following procedure: let $\mathcal{M}_{d}$ denote
the space of all probability measures on $\mathcal{E}_{d}= \{ \pm
e_{i} \} _{i=1}^{d}$
(the standard unit coordinate vectors), and define $\Omega=
(\mathcal{M}_{d} )^{\mathbb{Z}^{d}}$.
An \textit{environment} is an element $\omega\in\Omega$. For $x\in
\mathbb{Z}^{d}$
and $e\in\mathcal{E}_{d}$, we denote by $\omega(x,e )$
the probability that the measure $\omega(x )$ gives to
$e$. Let $P$ be an i.i.d. measure on $\Omega$, in the sense that
$P=\nu^{\mathbb{Z}^{d}}$ for some probability measure $\nu$ on
$\mathcal{M}_{d}$.
Throughout this paper, we assume that $P$ is uniformly elliptic, that is,
there exists some $\eta>0$ such that for every $e\in\mathcal{E}_{d}$
%
\begin{equation}
P \bigl( \bigl\{ \omega\in\Omega: \omega(0,e )\geq\eta\bigr\} \bigr
)=\nu
\bigl( \bigl\{ \omega\in\mathcal{M}_{d}: \omega(e )\geq\eta\bigr\}
\bigr)=1.\label{eqellipticconstant}
\end{equation}
For a given, fixed environment $\omega\in\Omega$ and $x\in\mathbb{Z}^{d}$,
the \textit{quenched} random walk on it (or the \emph{quenched law})
is a time homogeneous Markov chain on $\mathbb{Z}^{d}$ with transition
probabilities
\[
P_{\omega}^{x} (X_{n+1}=y+e\mid X_{n}=y )=
\omega(y,e)\qquad\forall y\in\mathbb{Z}^{d}, e\in\mathcal{E}_{d}
\]
and initial distribution $P_{\omega}^{x} (X_{0}=x )=1$.
We let $\mathbf{P}^{x}=P\otimes P_{\omega}^{x}$ be the joint law
of the environment and the walk, and define the \textit{annealed} (or \textit{averaged}) law as its marginal on the space of trajectories
\[
\mathbb{P}^{x} (\cdot)=\int_{\Omega}P_{\omega
}^{x}
(\cdot)\,dP (\omega).
\]
We use the notation $E$, $E_{\omega}^{x}$ and $\mathbb{E}^{x}$
for the expectations of the measures $P$, $P_{\omega}^{x}$ and~$\mathbb{P}^{x}$,
respectively.

In \cite{SZ99,Ze02}, Sznitman and Zerner proved that the limiting
velocity of the random walk
\[
\bbbv=\lim_{n\to\infty}\frac{X_{n}}{n}
\]
exists for $P$-almost every environment and $P_{\omega}^{0}$-almost
every trajectory of the random walk on it. A question to remain open,
which in fact is one the most important open questions in the field,
is whether the limiting velocity is an almost sure constant.

An important family of measures $P$ for the model is given by the
following definition.

\begin{defn}
The RWRE is said to be ballistic if the limiting velocity is a nonzero
almost sure constant.
\end{defn}

\subsection{Conditions for ballisticity}\label{sec1.2}

In \cite{Sz01,Sz02}, Sznitman introduced two criteria for ballisticity
of RWRE, called conditions $(T)$ and $(T')$. In order to give a
formal definition of these conditions, some preliminary definitions
are needed.

\begin{defn}
Let $\ell\in S^{d-1}:= \{ x\in\mathbb{R}^{d}: \llVert
x\rrVert _{2}=1 \} $
be a direction in $\mathbb{R}^{d}$.
\begin{longlist}[(2)]
\item[(1)] For $L>0$ and a sequence $ \{ X_{n} \} $ (in $\mathbb
{Z}^{d}$),
define
\[
T_{L}=T_{L}^{ (\ell)} \bigl( \{ X_{n} \}
\bigr)=\inf\bigl\{ n\geq0: \langle X_{n},\ell\rangle\geq L \bigr\},
\]
where $ \langle\cdot,\cdot\rangle$ denotes the standard
inner product in $\mathbb{R}^{d}$.
\item[(2)] Similarly, for a set $A\subset\mathbb{Z}^{d}$ and a sequence
$ \{ X_{n} \} $
(in $\mathbb{Z}^{d}$), denote
\[
T_{A}=T_{A} \bigl( \{ X_{n} \} \bigr)=\inf\{
n\geq0: X_{n}\in A \}.
\]
\end{longlist}
\end{defn}

We can now state the definition of Sznitman's ballisticity conditions.

\begin{defn}
(1) Given $0<\gamma\leq1$, we say that $P$ satisfies condition
$ (T_{\gamma} )$
in direction $\ell_{0}\in S^{d-1}$ if for every $\ell\in S^{d-1}$
in some neighborhood of $\ell_{0}$ there exists a finite constant
$C$ such that
\[
\mathbb{P}^{0} \bigl(T_{L}^{ (-\ell)}<T_{L}^{ (\ell
)}
\bigr)<Ce^{-L^{\gamma}}.
\]

(2) $P$ is said to satisfy condition $ (T )$ if it satisfies
condition $ (T_{1} )$.

(3) $P$ is said to satisfy condition $ (T' )$ if it satisfies
condition $ (T_{\gamma} )$ for some $\frac{1}{2}<\gamma<1$.
\end{defn}

%
\begin{rem}
It was shown in \cite{Sz02} that all the conditions $ (T_{\gamma
} )$
for $\frac{1}{2}<\gamma<1$ are equivalent.
\end{rem}

The relation between ballisticity and the above definition is given
by the following theorem and conjecture.\vadjust{\goodbreak}

\begin{teo}[(Sznitman \cite{Sz02})]\label
{teoSznitmanballisticitycriterion}
If condition $ (T' )$
holds for some direction $\ell_{0}\in S^{d-1}$, then the RWRE is
ballistic, and the limiting velocity $\bbbv$ satisfies \mbox{$\langle\bbbv,\ell_{0} \rangle>0$}.
In addition, under this assumption, condition $ (T' )$ holds
for all $\ell\in S^{d-1}$ satisfying $ \langle\bbbv,\ell
\rangle>0$.
\end{teo}

\begin{conjecture*}[(Sznitman)]
Condition $ (T' )$ is equivalent to ballisticity.
\end{conjecture*}
In recent years, several improvements of Theorem~\ref
{teoSznitmanballisticitycriterion}
have been proved: in \cite{DR11}, Drewitz and Ram\'{\i}rez showed that
for some constant $\gamma_{d}\in(0.366,0.388 )$ which is
dimension dependent $ (T_{\gamma} )$ for $\gamma\in
(\gamma_{d},1 )$
are all equivalent. In \cite{berger2008slowdown}, Theorem 1.4, Berger
showed that in dimension $d\geq4$ the condition $ (T_{\gamma
} )$
for $\gamma\in(0,1 )$ implies ballisticity. In an additional
work \cite{DR12}, Drewitz and Ram\'{\i}rez showed that in dimension
$d\geq4$
all the conditions $ (T_{\gamma} )$ for $\gamma\in
(0,1 )$
are equivalent. In \cite{BDR12}, Berger, Drewitz and Ram\'{\i}rez showed
that in fact (fast enough) polynomial decay (see Definition~\ref{defPcondition}
below) is equivalent to condition $ (T_{\gamma} )$ for any
$0<\gamma<1$. Finally, in \cite{CR13} Campos and Ram\'{\i}rez proved
ballisticity for some nonuniformly elliptic environments satisfying
(fast enough) polynomial decay.

\begin{defn}[{[Condition $ (\mathscr{P} )$]}]
\label{defPcondition} Let $N_{0}$ be an even integer. For a coordinate
direction $\ell=\ell_{1}$, let $\ell_{2},\ldots,\ell_{d}$ be any
fixed completion of $\ell_{1}$ to an orthonormal basis of $\mathbb{R}^{d}$
and define
\begin{eqnarray*}
\operatorname{Box}_{x} &=& \biggl\{ y\in\mathbb{Z}^{d}: -
\frac
{N_{0}}{2}< \langle y-x,\ell\rangle<N_{0}, \langle y-x,
\ell_{j} \rangle<25N_{0}^{3}~\forall2\leq j\leq d
\biggr\},
\\
\widetilde{\operatorname{Box}}_{x}&=& \bigl\{ y\in\mathbb{Z}^{d}:
\tfrac
{1}{3}N_{0}\leq\langle y-x,\ell\rangle<N_{0},
\langle y-x,\ell_{j} \rangle<N_{0}^{3}~\forall2
\leq j\leq d \bigr\},
\\
\partial\operatorname{Box}_{x}&=& \bigl\{ y\in\mathbb{Z}^{d}\setminus
\operatorname{Box}_{x}: \exists z\in\operatorname{Box}_{x}\mbox{ such that }
\llVert y-z\rrVert_{1}=1 \bigr\}
\end{eqnarray*}
and
\[
\partial_{+}\operatorname{Box}_{x}= \bigl\{ y\in\partial
\operatorname{Box}_{x}: \langle y-x,\ell\rangle\geq N_{0}, \bigl
\llvert\langle y-x,\ell_{j} \rangle\bigr\rrvert<25N_{0}^{3}~\forall2\leq j\leq d \bigr\}.
\]
Fix $M>0$ and $\ell\in S^{d-1}$. We say that condition $\mathscr
{P}_{M}\mid\ell$
is fulfilled if
\[
\sup_{x\in\widetilde{\operatorname{Box}}_{0}}\mathbb{P}^{x} (T_{\partial\operatorname{Box}_{0}}\neq
T_{\partial_{+}\operatorname{Box}_{0}} )<\frac{1}{N_{0}^{M}}
\]
holds for some $N_{0}\geq\exp(100+4d (\ln\eta
)^{2} )$,
where $\eta$ is the ellipticity constant defined in (\ref
{eqellipticconstant}).
We say that condition $ (\mathscr{P} )$ holds in direction
$\ell$ if condition $\mathscr{P}_{M}\mid\ell$ holds for some $M>15d+5$.
\end{defn}

\begin{defn}
Throughout this paper, we denote by $ (\mathscr{P} )$ the
following equivalent conditions:
\begin{itemize}
\item[(1)]$ (T' )$.
\item[(2)]$ (T_{\gamma} )$ for some $\gamma\in
(0,1 )$.
\item[(3)]$ (T_{\gamma} )$ for all $\gamma\in(0,1 )$.
\item[(4)]$ (\mathscr{P} )$.
\end{itemize}
\end{defn}

\subsection{The environment viewed from the particle}\label{sec1.3}

Let $ \{ X_{n} \} $ be a RWRE. The \textit{environment viewed
from the particle} is the discrete time process $ \{ \overline
{\omega}_{n} \} $
defined on~$\Omega$ by
\[
\overline{\omega}_{n}=\sigma_{X_{n}}\omega,
\]
where for $x\in\mathbb{Z}^{d}$ we denote by $\sigma_{x}$ the shift
in direction $x$ of $\omega$, that is, $\sigma_{x}\omega
(y,\cdot)=\omega(x+y,\cdot)$
for every $y\in\mathbb{Z}^{d}$.

Beside the fact that the environment viewed from the particle process
takes values in a compact space, it has the advantage of being Markovian,
cf. \cite{bolthausen2002satic}, with respect to the transition kernel
%
\begin{equation}
\mathfrak{R}g (\omega)=\sum_{e\in\mathcal
{E}_{d}}\omega(0,e )g (
\sigma_{e}\omega),\label{eqTransitionkernel}
\end{equation}
defined for every bounded measurable function $g:\Omega\to\mathbb{R}$.

It is natural to ask what are the invariant measures of the Markov
chain $ \{ \overline{\omega}_{n} \} $.

\begin{defn}
A probability measure $Q$ on $\Omega$ is said to be invariant (or
invariant with respect to the point of view of the particle), if for
every bounded continuous function $g:\Omega\to\mathbb{R}$
%
\begin{equation}
\int_{\Omega}\mathfrak{R}g (\omega)\,dQ (\omega)=\int
_{\Omega}g (\omega)\,dQ (\omega).\label{eqinvariantprobabilitymeasure}
\end{equation}
\end{defn}

One can find many examples for invariant measures with respect to
the process $ \{ \overline{\omega}_{n} \} $. For example,
every Dirac probability measure of any translation invariant environment
provides such an example. One additional method to obtain invariant
measures is by taking any sub-sequential limit of the C\'{e}asro means
$ \{ \frac{1}{n}\sum_{k=0}^{n-1}\mathfrak{R}^{k}\nu\} $,
where $\nu$ is any probability measure on $\Omega$ and $\mathfrak
{R}\nu$
is the measure defined by the identity $\int_{\Omega}f (\omega
)\,d (\mathfrak{R}\nu) (\omega)=\int_{\Omega}\mathfrak{R}f (\omega
)\,d\nu(\omega)$
for every bounded measurable function $f:\Omega\to\mathbb{R}$.

As it turns out, an invariant measure $Q$ is particularly useful
when it is also equivalent to the original measure $P$. In this case,
we say that the static point of view (the one related to $P$) is
equivalent to the dynamic point of view (the one related to $Q$).
If such a measure exists, it can be used to prove law of large numbers
and central limit theorem type results; see, for example, \cite
{Ko85,KV86,bolthausen2002ten,RA03,SS04,Ze04,BB07,MP07,DR13}
and the references therein.

The existence of an equivalent invariant measure was proved in several
cases. In the one-dimensional case, the existence of an equivalent
measure was proved by Alili \cite{Al99}. In the reversible case,
also known as random conductance model, the existence of an invariant
equivalent measure is a well-known fact for most cases. For balanced
RWRE, the existence of such a measure was proved by Lawler in \cite{La87}.
Later on, this was strengthened to the case of balanced elliptic RWRE
by Guo and Zeitouni in \cite{GZ12} and even further to the nonelliptic
case (for genuinely $d$-dimen\-sional measures) by Berger and Deuschel
in \cite{BD14}. For Dirichlet random walks, a classification for
the cases where such a measure exists was proved by Sabot in~\cite{Sa13}.
Finally, partial results in the ballistic case are also known; see
Section~\ref{subKnown-results} below.

The following result was proved by Kozlov in \cite{Ko85} (for the
proof see also \mbox{\cite{bolthausen2002ten,DR13}}).

\begin{teo}[(Kozlov \cite{Ko85})]
\label{teoKozlovstheorem} Assume $P$ is elliptic\footnote{$P$ is
called elliptic if $P (\min_{e\in\mathcal{E}_{d}}\omega
(x,e )>0 )=1$ $\forall x\in\mathbb{Z}^{d}$.}
and ergodic with respect to $ \{ \sigma_{x} \} _{x\in
\mathbb{Z}^{d}}$.
Assume there exists an invariant probability measure $Q$ for the
environment seen from the point of view of the particle which is absolutely
continuous with respect to $P$. Then the following hold:
\begin{longlist}[(4)]
\item[(1)] $Q$ is equivalent to $P$.

\item[(2)]  The environment viewed from the particle with initial law $Q$ is
ergodic.

\item[(3)] $Q$ is the unique invariant probability measure for the point of
view of the particle which is absolutely continuous with respect to
$P$.

\item[(4)]  The C\'{e}asro means $ \{ \frac{1}{N+1}\sum_{k=0}^{N}\mathfrak
{R}^{k}P \} $
converge weakly to $Q$.
\end{longlist}
\end{teo}
\subsection{Main goal}\label{sec1.4}

This paper has two purposes. The first is to prove the equivalence
of the dynamic and static point of views under condition $
(\mathscr{P} )$,
uniform ellipticity and the additional assumption that $d\geq4$.
The second purpose of this paper is to prove a certain type of local
limit theorem relating the quenched and annealed laws by a prefactor.

\subsection{Known results in the strongly ballistic case}\label{sec1.5}
\label{subKnown-results}

Let $d\geq2$. In \cite{Sz01}, Sznitman proved an annealed CLT under
condition $ (T' )$. The ideas he presented may also be used
to prove an annealed local CLT. For completeness, we present a proof
of the annealed local CLT in the \hyperref[appendi]{Appendix}. In \cite{bolthausen2002satic},
Bolthausen and Sznitman proved the equivalence of the static and dynamic
point of views for certain (nonnestling) ballistic random walks in
random environment, when $d\geq4$ and the disorder is low. In \cite{RA03},
Rassoul-Agha proved the existence of an equivalent invariant measure
on half spaces under Kalikow's condition, mixing and uniform ellipticity.
In \cite{BZ08}, Berger and Zeitouni and in \cite{RAS05,RAS07,RAS09}
Rassoul-Agha and Sepp\"al\"ainen proved quenched invariance principle
under moments assumptions for the first regeneration time. In particular,
a quenched CLT holds under condition $ (\mathscr{P} )$.

\subsection{Main results}\label{sec1.6}

Our two main results are the following.

\begin{teo}
\label{teoAbsolutelycontinuousinvariantmeasure} Let $d\geq4$
and assume $P$ is uniformly elliptic, i.i.d. and satisfies condition
$ (\mathscr{P} )$. Then there exists a unique probability
measure $Q$ on the space of environments which is invariant with
respect to the point of view of the particle and is equivalent to
the original measure $P$. In addition, $E [ (\frac
{dQ}{dP} )^{k} ]<\infty$
for every $k\in\mathbb{N}$.
\end{teo}

\begin{teo}
\label{teoPrefactorthm} Let $d\geq4$ and assume $P$ is uniformly
elliptic, i.i.d. and satisfies condition $ (\mathscr{P} )$.
Then there exists a unique measurable, nonnegative function $f\in
L^{1} (\Omega,P )$
such that for $P$-almost every $\omega\in\Omega$
%
\begin{equation}
\lim_{n\to\infty}\sum_{x\in\mathbb{Z}^{d}}\bigl\llvert
P_{\omega
}^{0} (X_{n}=x )-\mathbb{P}^{0}
(X_{n}=x )f (\sigma_{x}\omega)\bigr\rrvert
=0.\label{eqPrefactorthm}
\end{equation}
This unique function $f$ is the Radon--Nikodym derivative $\frac{dQ}{dP}$
of the probability measure $Q$ constructed in Theorem~\ref
{teoAbsolutelycontinuousinvariantmeasure}.
\end{teo}

\subsection{Remarks about lower dimensions}\label{sec1.7}

In this paper, we only prove Theorem~\ref
{teoAbsolutelycontinuousinvariantmeasure}
and Theorem~\ref{teoPrefactorthm} in dimension $4$ or higher.
Here, we wish to remark about the situation in lower dimensions.

For $d=1$, the existence of an equivalent measure
which is invariant with respect to the point of view of the particle
was proved by Alili; see \cite{Al99}.

We conjecture that similar results should hold in dimension 3. In
fact, the only place in the proof where we directly use the condition
$d\geq4$ is in \cite{berger2008slowdown}, Lemma 4.10; see also Lemma
\ref{lemIntersectionlemma} below. On the other hand, we believe
that in dimension $2$ an equivalent probability measure which is
invariant with respect to the point of view of the particle does not
exist.

\subsection{Structure of the paper and general remarks}\label{sec1.8}

In Section~\ref{sec2}, we recall some of the notation from \cite{berger2008slowdown}
as well as some of the result obtained there. In addition, we prove
a slightly different version of \cite{berger2008slowdown}, Lemma 4.2,
thus giving annealed estimations for a fixed time. In Section~\ref{sec3}, we
generalize \cite{berger2008slowdown}, Proposition 4.5, which gives
an upper bound on the difference between the annealed and quenched
distribution, to include estimations on the exit time of the box.
Section~\ref{sec4} is devoted to converting the estimation obtained in Section~\ref{sec3}
for $ (d-1 )$-dimensional cubes in a time interval into
a result about $d$-dimensional cubes in a fixed time. In Section~\ref{sec5}, we
bootstrap the result for large $d$-dimensional cubes obtained
in Section~\ref{sec4} all the way to boxes of finite size. Section~\ref{sec6} is devoted
to the proof of the first main result, the existence of an equivalent
probability measure on the space of environments which is invariant
with respect to the point of view of the particle. Finally, in
Section~\ref{sec7} we prove the second main result regarding the existence of a
prefactor.

Throughout this paper, the value of constants $c$ and $C$ may change
from one line to the next. Numbered constants, such as
$c_{1},c_{2},\ldots$
are fixed according to their first appearance in the text. Expectation
with respect to a measure $\mu$ which is not $P, P_{\omega}$ or
$\mathbb{P}$
is denoted by $E_{\mu}$. Finally, some of the inequalities may only
hold for large enough values of $N,n$ and $M$.

\section{Notation and other preliminary results}\label{sec2}

We start by recalling some of the notation and results from \cite
{berger2008slowdown}
to be used throughout the paper. In addition, we cite an inequality
by McDiarmid for future use and state analogous result to \cite
{berger2008slowdown}, Lemma~4.2,
for the annealed measure in a fixed time.

For $k,N\geq0$, define $R_{k} (N )= \lfloor e^{
(\log N )^{(k+2)/(k+3)} } \rfloor$
and denote $R (N )=R_{1} (N )$. Note that
$R_{0} (N )= \lfloor\log N \rfloor$
and that for every $k,n\geq0$, and every large enough $N$,
\[
R_{k}^{n} (N ):= \bigl(R_{k} (N )
\bigr)^{n}\leq R_{k+1} (N )<N.
\]

Let
\[
\vartheta=\lim_{n\to\infty}\frac{X_{n}}{\llVert X_{n}\rrVert _{2}}
\]
be the direction of the speed. We assume without loss of generality
that $ \langle\vartheta,e_{1} \rangle>0$ and note that
due to the results of \cite{Sz01,Sz02}, this implies that $
(\mathscr{P} )$
holds both in direction $\vartheta$ and in direction $e_{1}$.

\begin{defn}
For $k\in\mathbb{N}$, define $H_{k}$ to be the hyperplane
$H_{k}= \{ x\in\mathbb{Z}^{d}: \langle x,e_{1}
\rangle=k \} $.
\end{defn}

\begin{defn}
By the term $N^{-\xi(1 )}$, we mean a nonnegative function
of $N\in\mathbb{N}$ which decays faster than any polynomial, that is,
if $f (N )=N^{-\xi(1 )}$, then for every $k\in
\mathbb{N}$
\[
\lim_{N\to\infty}N^{k}f (N )=0.
\]
Note that $N^{-\xi(1 )}$ is independent of the environment
unless otherwise stated.
\end{defn}

\begin{defn}
For two nonempty sets $A,B\subset\mathbb{Z}^{d}$, we define $\operatorname{dist} (A,B )=\min\{ \llVert x-y\rrVert _{1}:
x\in A, y\in B\} $.
If $A= \{ x \} $ we write $\operatorname{dist} (x,B )$
instead of $\operatorname{dist} ( \{ x \},B )$.
\end{defn}

\begin{defn}
For\vspace*{1pt} $x= (x_{1},\ldots,x_{d} )\in\mathbb{Z}^{d}$ and $n\in
\mathbb{N}$,
we denote $x\leftrightarrow n$ if $x$ and $n$ have the same parity,
that is, $\sum_{i=1}^{d}x_{i}+n$ is even. In a similar way for $x,y\in
\mathbb{Z}^{d}$,
we denote $x\leftrightarrow y$ if $\sum_{i=1}^{d}
(x_{i}+y_{i} )$
is even.
\end{defn}

\begin{defn}
Recall that for $x\in\mathbb{Z}^{d}$ we denote by $\sigma_{x}$ the
shift in direction $x$ in $\omega$, that is, $\sigma_{x}\omega
(y,\cdot)=\omega(x+y,\cdot)$
for every $y\in\mathbb{Z}^{d}$.
\end{defn}

\begin{defn}
For $z\in\mathbb{Z}^{d}$ and $N\in\mathbb{N}$, we define (see also
Figure~\ref{figThe-basic-block}):
\begin{longlist}[(2)]
\item[(1)] the parallelogram of size $N$ and center $z$ to be
\begin{eqnarray*}
\mathcal{P} (z,N ) &=& \biggl\{ x\in\mathbb{Z}^{d}: \bigl\llvert\langle
x-z,e_{1} \rangle\bigr\rrvert<N^{2},
\\
&&{}  \biggl\llVert x-z-
\vartheta\cdot\frac{ \langle x-z,e_{1} \rangle
}{ \langle\vartheta,e_{1} \rangle}\biggr\rrVert_{\infty
}<NR_{5}
(N ) \biggr\}.
\end{eqnarray*}

\item[(2)] The middle third of $\mathcal{P} (z,N )$
\begin{eqnarray*}
\tilde{\mathcal{P}} (z,N ) &=& \biggl\{ x\in\mathbb{Z}^{d}: \bigl\llvert
\langle x-z,e_{1} \rangle\bigr\rrvert<\frac
{1}{3}N^{2},
\\
&&{}  \biggl\llVert x-z-\vartheta\cdot\frac{ \langle
x-z,e_{1} \rangle}{ \langle\vartheta,e_{1} \rangle
}\biggr\rrVert
_{\infty} <\frac{1}{3}NR_{5} (N ) \biggr\}.
\end{eqnarray*}

\item[(3)] The boundary of $\mathcal{P} (z,N )$
\[
\partial\mathcal{P} (z,N )= \bigl\{ x\in\mathbb{Z}^{d}\setminus
\mathcal{P} (z,N ): \exists y\in\mathcal{P} (z,N )\mbox{ s.t. }\llVert
x-y\rrVert
_{1}=1 \bigr\}.
\]

%
\begin{figure}

\includegraphics{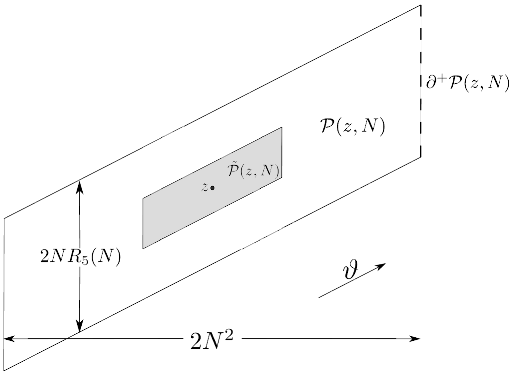}

\caption{The basic block $\mathcal{P} (z,N )$: the box
$\tilde{\mathcal{P}} (z,N )$
is in gray and the right boundary $\partial^{+}\mathcal{P}
(z,N )$
is the dashed line.}
\label{figThe-basic-block}
\end{figure}

\item[(4)] The right boundary of $\mathcal{P} (z,N )$
\[
\partial^{+}\mathcal{P} (z,N )= \bigl\{ x\in\partial\mathcal{P} (z,N
): \langle x-z,e_{1} \rangle=N^{2} \bigr\}.
\]
\end{longlist}
\end{defn}

\subsection{Regeneration times}\label{sec2.1}
\label{subRegeneration-times}

\begin{defn}
Let $ \{ X_{n} \} $ be a nearest-neighbor sequence in
$\mathbb{Z}^{d}$,
and let $\ell\in S^{d-1}$ be a direction. We say that $t$ is a regeneration
time for $ \{ X_{n} \} $ in direction $\ell$ if the following
holds:
\begin{longlist}[(2)]
\item[(1)]$ \langle X_{s},\ell\rangle< \langle
X_{t},\ell\rangle$
for every $s<t$.
\item[(2)]$ \langle X_{t+1},\ell\rangle> \langle
X_{t},\ell\rangle$.
\item[(3)]$ \langle X_{s},\ell\rangle> \langle
X_{t+1},\ell\rangle$
for every $s>t+1$.
\end{longlist}
\end{defn}
The following theorem summarize the results on the regeneration time
structure.

\begin{teo}[{(\cite{SZ99,Sz02})}]\label{teoregenerationtimesthm}Assume that $P$
satisfies $ (T_{\gamma} )$ in direction $\ell_{0}$ for
some $\gamma>0$. Then:
\begin{longlist}[(2)]
\item[(1)] With probability one, there exist infinitely many regeneration times,
which we denote by $\tau_{1}<\tau_{2}<\cdots.$
\item[(2)] The ensemble
\[
\bigl\{ (\tau_{n+1}-\tau_{n},X_{\tau_{n+1}}-X_{\tau
_{n}}
) \bigr\} _{n\geq1}
\]
is an i.i.d. ensemble under the annealed measure.
\item[(3)] There exists $C>0$ such that for every $n\in\mathbb{N}$
\[
\mathbb{P} (\tau_{2}-\tau_{1}=n )\leq C\mathbb{P} (
\tau_{1}=n ),
\]
and for every $y\in\mathbb{Z}^{d}$
\[
\mathbb{P} (X_{\tau_{2}}-X_{\tau_{1}}=y )\leq C\mathbb{P}
(X_{\tau_{1}}=y ).
\]

\item[(4)] There exists $c>0$ such that for every $n$,
\[
\mathbb{P} \bigl(\exists k\leq\tau_{1}: \llVert X_{k}
\rrVert_{\infty}>n \bigr)\leq e^{-cn^{\gamma}}.
\]
\end{longlist}
\end{teo}

The following is the main technical statement from \cite{berger2008slowdown}.

\begin{teo}[(\cite{berger2008slowdown}, Proposition~2.2)]
\label{teoRegenerationtimestimeestimate}
If $d\geq4$, and $P$ satisfies condition $ (\mathscr{P} )$
in one of the $2d$-principal directions, then for every $\alpha<d$
\[
\mathbb{P} (\tau_{1}>k )\leq\exp\bigl(- (\log t )^{\alpha}
\bigr).
\]
\end{teo}

\begin{cor}
\label{corlengthofregenerations}For $N\in\mathbb{N}$ denote by
$B_{N}=B_{N} ( \{ X_{n} \} )$, the event
\[
B_{N} \bigl( \{ X_{n} \} \bigr)= \bigl\{ \forall1\leq k
\leq N^{2}: \tau_{k}-\tau_{k-1}\leq R (N ) \bigr
\},
\]
where $\tau_{0}=0$. Then $\mathbb{P} (B_{N} )\geq
1-N^{-\xi(1 )}$.
\end{cor}

\begin{rem}
\label{RemTheeventAN}Note that the event $B_{N} ( \{
X_{n} \} )$
implies the event that the distance traveled between two regeneration
times is bounded by $R (N )$ as well, that is,
\[
A_{N} \bigl( \{ X_{n} \} \bigr)= \bigl\{ \forall1\leq k
\leq N^{2}: \max\bigl\{ \llVert X_{t}-X_{\tau_{k-1}}
\rrVert_{\infty}: \tau_{k-1}\leq t\leq\tau_{k} \bigr
\} \leq R (N ) \bigr\}
\]
satisfies $B_{N} ( \{ X_{n} \} )\subset
A_{N} ( \{ X_{n} \} )$
and in particular $\mathbb{P} (A_{N} )\geq1-N^{-\xi
(1 )}$.
\end{rem}

\subsection{Intersections of paths of random walks}\label{sec2.2}
\label{subIntersections-of-paths}

The following lemma estimates the number of intersections of two independent
random walks in dimension $d\geq4$. This is in fact the only place
in the proof where the assumption $d\geq4$ is used explicitly. Denote
by $P_{\omega,\omega}^{z,z}$, $E_{\omega,\omega}^{z,z}$ the law
(resp., expectation) of two random walks on the same environment
$\omega$, which conditioned on $\omega$ evolve independently according
to the quenched law of $\omega$ starting from $z$.

\begin{lem}[(\cite{berger2008slowdown}, Lemma 4.10)]\label{lemIntersectionlemma}
Let $d\geq4$ and assume $P$ is uniformly elliptic, i.i.d. and satisfies
$ (\mathscr{P} )$. Let $X^{ (1 )}= \{
X_{n}^{ (1 )} \} $
and $X^{ (2 )}= \{ X_{n}^{ (2 )} \} $
be two\vspace*{1pt} independent random walks running in the same environment $\omega$.
For $i\in\{ 1,2 \}$, let $ [X^{ (i
)} ]$
be the set of points visited by $X^{ (i )}$. Then
\begin{eqnarray*}
&& P \bigl( \bigl\{ \omega\in\Omega: E_{\omega,\omega}^{0,0} \bigl[\bigl
\llvert\bigl[X^{ (1 )} \bigr]\cap\bigl[X^{
(2 )} \bigr]\cap
\mathcal{P} (0,N )\bigr\rrvert\ind_{A_{N} ( \{ X_{n}^{ (1 )} \} )\cap
A_{N} ( \{ X_{n}^{ (2 )} \} )} \bigr]\geq R_{2}
(N ) \bigr\} \bigr)
\\
&&\qquad =N^{-\xi(1 )}.
\end{eqnarray*}
\end{lem}

For future use, we denote
%
\begin{eqnarray}\label{eqintersectionevent}
J (N )&=& \bigl\{ \omega\in\Omega:
E_{\omega,\omega
}^{z,z} \bigl[\bigl\llvert\bigl[X^{ (1 )} \bigr]
\cap\bigl[X^{ (2 )} \bigr]\cap\mathcal{P} (0,N )\bigr\rrvert
\nonumber\\[-8pt]\\[-8pt]\nonumber
&&{}\times
\ind_{A_{N} ( \{ X_{n}^{ (1 )} \}
)\cap A_{N} ( \{ X_{n}^{ (2 )} \}
)} \bigr]\leq R_{2} (N ),~\forall z\in\tilde{\mathcal{P}} (0,N )\bigr\}.
\end{eqnarray}
Therefore, due to the last lemma, we have $P (J (N
) )=1-N^{-\xi(1 )}$.

\subsection{McDiarmid's inequality}\label{sec2.3}

The following Azuma type inequality, proved by McDiarmid in \cite{McD98},
is used in Section~\ref{secAdding-time-estimation}.

\begin{teo}[(\cite{McD98}, Theorem 3.14)]\label{teoAzumatypeinequality} Let
$ \{ M_{k} \} _{k=0}^{n}$ be a martingale with respect to
a probability measure $\mathtt{P}$, given by $M_{k}=E_{\mathtt
{P}} [X\mid\mathscr{F}_{k} ]$,
with $M_{0}=E_{\mathtt{P}} [X ]$. For $1\leq k\leq n$ let
$U_{k}=\esup(\llvert M_{k}-M_{k-1}\rrvert \mid\mathscr
{F}_{k-1} )$
and define $U=\sum_{k=1}^{n}U_{k}^{2}$. Then
\[
\mathtt{P} \bigl(\llvert M_{n}-M_{0}\rrvert>\alpha, U\leq
c \bigr)\leq2e^{-\afrac{\alpha^{2}}{2c}}.
\]
\end{teo}

\subsection{Annealed estimation for a fixed time}\label{sec2.4}

In this subsection, we state some standard estimations on the annealed
measure of the random walk. The proof is a standard and straightforward
use of Fourier transform techniques applied to the regeneration structure
described in Section~\ref{subRegeneration-times}. The first three
claims are proved in a very similar way to the proof of \cite
{berger2008slowdown}, Lemma~4.2
(see also Lemma~\ref{lemAnnealedderivativeestimationsd-1+time}
for another version). The formal statement is the following.

\begin{lem}
\label{lemgeneralannealedestimations} Assume that $P$ is uniformly
elliptic, i.i.d. and satisfies $ (\mathscr{P} )$. Then for
large enough $n\in\mathbb{N}$ and $x,y,z,w\in\mathbb{Z}^{d}$ such
that $\llVert x-y\rrVert _{1}=1$, $\llVert z-w\rrVert _{1}=1$
%
\begin{eqnarray}
\mathbb{P}^{z} (X_{n}=x )&\leq& Cn^{-\sfrac{d}{2}},\label{eqgeneralannealedestimation1}
\\
%
\bigl\llvert\mathbb{P}^{z} (X_{n}=x )-
\mathbb{P}^{z} (X_{n+1}=y )\bigr\rrvert&\leq&
Cn^{-\vfrac{d+1}{2}},\label{eqgeneralannealedestimation2}
\\
%
\bigl\llvert\mathbb{P}^{z} (X_{n}=x )-
\mathbb{P}^{w} (X_{n+1}=x )\bigr\rrvert&\leq&
Cn^{-\vfrac{d+1}{2}}.\label{eqgeneralannealedestimation3}
\end{eqnarray}
In addition, for every $\varepsilon>0$ and every partition $\Pi
_{n}^{ (\varepsilon)}$
of $\mathbb{Z}^{d}$ into boxes of side length~$n^{\varepsilon}$.
%
\begin{equation}
\sum_{\Delta\in\Pi_{n}^{ (\varepsilon)}} \mathop{\sum
_{x\in\Delta}}_{x\leftrightarrow n} \Bigl[\max_{y\in\Delta}
\mathbb{P}^{0} (X_{n}=y )-\mathbb{P}^{0}
(X_{n}=x ) \Bigr]\leq Cn^{-\sfrac{1}{2}+3d\varepsilon}.\label{eqgeneralannealedestimation4}
\end{equation}
\end{lem}

The proof of Lemma~\ref{lemgeneralannealedestimations} can be
found in Appendix~\ref{appendixsubAnnealed-derivative-estimations}.

Before turning to the last estimation of this subsection, we state
here a very simple claim to be used in several places.

\begin{claim}
\label{Clmsimpleclaimannealedtoquenched}Let $A$ be an event
in the $\sigma$-algebra of $ (\mathbb{Z}^{d} )^{\mathbb{N}}$ and
assume that $\mathbb{P} (A )\leq\varepsilon$, then
$P ( \{ \omega\in\Omega: P_{\omega} (A
)\geq\sqrt{\varepsilon} \} )\leq\sqrt{\varepsilon}$.
In particular, if a sequence of events $ \{ A_{N} \} $ satisfies
$\mathbb{P} (A_{N} )=1-N^{-\xi(1 )}$, then
$P ( \{ \omega\in\Omega: P_{\omega} (A_{N}
)=1-N^{-\xi(1 )} \} )=1-N^{-\xi
(1 )}$.
\end{claim}
\begin{pf}
Define the random variable $X:\Omega\to[0,1 ]$ by $X
(\omega)=P_{\omega} (A )$.
By the Markov inequality, $P (X (\omega)\geq\sqrt
{\varepsilon} )\leq\frac{E [X (\omega)
]}{\sqrt{\varepsilon}}=\frac{\mathbb{P} (A )}{\sqrt
{\varepsilon}}\leq\frac{\varepsilon}{\sqrt{\varepsilon}}=\sqrt
{\varepsilon}$.
\end{pf}
Next, we show that the location of the walk at time $n$ is concentrated
in a box which is a bit larger than $\sqrt{n}$. More formally, we
have the following.

\begin{lem}
\label{lemgoodestimationforthelocation} Assume that $P$ is
uniformly elliptic, i.i.d. and satisfies $ (\mathscr{P} )$.
Then:
\begin{longlist}[(2)]
\item[(1)]$\mathbb{P}^{z} (\llVert X_{n}-\mathbb{E}^{z}
[X_{n} ]\rrVert _{\infty}>\sqrt{n}R_{5} (n
) )\leq e^{-R_{5} (n )}=n^{-\xi(1 )}$,
\item[(2)] \vspace*{1pt}$P ( \{ \omega\in\Omega: P_{\omega}^{z}
(\llVert X_{n}-\mathbb{E}^{z} [X_{n} ]\rrVert
_{\infty}>\sqrt{n}R_{5} (n ) )\leq e^{-\sklfrac {1}{2}R_{5} (n )} \} )=1-n^{-\xi(1 )}$,
\item[(3)] For every $\delta>0$ there exists $C>0$ such that $\mathbb
{P}^{z} (\llVert X_{n}-\mathbb{E}^{z} [X_{n}
]\rrVert _{\infty}>C\sqrt{n} )<\delta$.
\end{longlist}
\end{lem}

The proof of Lemma~\ref{lemgoodestimationforthelocation} can
be found in Appendix~\ref{appendixsubAnnealed-derivative-estimations}.

\section{Adding time estimation}\label{sec3}
\label{secAdding-time-estimation}

The goal of this section is to prove a generalized version of \cite
{berger2008slowdown}, Proposition 4.5.
The original lemma gives a bound on the difference between the probability
measures $\mathbb{P}^{z} (X_{T_{\partial\CP(0,N
)}}\in\cdot)$
and $P_{\omega}^{z} (X_{T_{\partial\CP(0,N )}}\in
\cdot)$
to hit any cube in a partition of $\partial^{+}\CP(0,N )$
into cubes of side length $N^{\theta}$, for any $0<\theta\leq1$.
This estimation immediately implies that the total variation of the
two measures goes to zero as $N$ goes to infinity. Here, we show
that if an estimation on the hitting time $T_{\partial\CP
(0,N )}$
is added, then a similar estimation can be derived. More formally,
we have the following.

\begin{prop}
\label{propTimed-1boxdifference}Let $d\geq4$ and assume $P$
is uniformly elliptic, i.i.d. and satisfies $ (\mathscr{P} )$.
For\vspace*{1pt} every $0<\theta\leq1$, let $F (N )=F (N,\theta
)$
be the event that for every $z\in\tilde{\CP} (0,N )$, every
cube $\Delta$ of side length $N^{\theta}$ which is contained in
$\partial^{+}\CP(0,N )$ and every interval $I$ of length
$N^{\theta}$
\begin{eqnarray*}
&& \bigl\llvert P_{{\omega}}^{z} (X_{T_{\partial\CP(0,N
)}}\in\Delta,
T_{\partial\CP(0,N )}\in I )-\BP^{z} (X_{T_{\partial\CP(0,N )}}\in
\Delta,
T_{\partial\CP(0,N )}\in I )\bigr\rrvert
\\
&&\qquad \leq CN^{-d (1-\theta)-\sklvafrac
{d-2}{d+2}\theta}.
\end{eqnarray*}
Then $P (F (N ) )=1-N^{-\xi(1 )}$.
\end{prop}

The proof of Proposition~\ref{propTimed-1boxdifference} follows
the one of \cite{berger2008slowdown}, Proposition 4.5 (see also \cite
{berger2008slowdown}, Section~4,
and in particular Lemma 4.15). Here are the main steps of the proof:
the proof starts with another version for annealed derivatives bounds
(see Lemma~\ref{lemAnnealedderivativeestimationsd-1+time}).
Next, in Lemma~\ref{lemDistanceoftimefromexpectation} we prove
an annealed concentration inequality for the hitting time $T_{\partial
\mathcal{P} (0,N )}$.
Lemma~\ref{lemFirstdifferenceestimation} provides a first weak
estimation for the difference between the quenched and annealed hitting
probabilities for large enough boxes, that is, $\theta>\frac{d}{d+1}$.
Using induction and the estimation from the last lemma, we prove an
upper bound on the probability to hit a given box of side length
$N^{\theta}$
in a time interval of length $N^{\theta}$ for every $0<\theta\leq1$
(see Lemma~\ref{lemUpperestimationonthequenchedexit}). In Lemma
\ref{lemEstimationforfarhyperplane}, we use the upper quenched
estimations in order to show that the difference between the quenched
and annealed hitting probabilities, in a slightly further hyperplane
are as required. Finally, in the proof of Proposition~\ref
{propTimed-1boxdifference},
we show how to translate the estimations from the further hyperplane
back to the original hyperplane. The first main tool used in the proof
is an environment exposure procedure, which in the context of ballistic
RWRE already appeared in the work of Bolthausen and Sznitman \cite
{bolthausen2002satic}.
This exposure procedure defines a martingale and allows the use of
Azuma's and McDiarmid's inequalities. The second main tool is the
intersection estimate for two independent random walks from Lemma
\ref{lemIntersectionlemma}.

\begin{rem}
In Section~\ref{secLorentztrans}, we use Proposition~\ref
{propTimed-1boxdifference}
for boxes whose side length is only asymptotic to $N^{\theta}$ (for
some $0<\theta<1$), that is, the side length is $N^{\theta}+o
(N^{\theta} )$.
One can verify that the same proof holds for such boxes as well.
\end{rem}

We start by stating another version for the estimation on the annealed
measure (see Lemma~\ref{lemgeneralannealedestimations} and \cite
{berger2008slowdown}, Lemma 4.2).

\begin{lem}[(Annealed derivative estimations)]
\label{lemAnnealedderivativeestimationsd-1+time} Assume $P$
is uniformly elliptic, i.i.d. and satisfies $ (\mathscr{P} )$.
Fix $z_{1}\in\BZ^{d}$, $N\in\mathbb{N}$ and let $z\in\tilde{\CP
} (z_{1},N )$.
Let $ \{ X_{n} \} $ be an RWRE starting at $z$. Then for
large enough $N$:
\begin{longlist}[(2)]
\item[(1)] For every $m\in\mathbb{N}$ and every $x\in\partial^{+}\CP
(z_{1},N )$
%
\begin{equation}
\mathbb{P}^{z} ({T_{\partial\mathcal{P}(0,N)}}=m, X_{{T_{\partial
\mathcal{P}(0,N)}}}=x
)<CN^{-d}.\label
{eqannealedderivativeestimation1}
\end{equation}

\item[(2)] For every $m\in\mathbb{N}$ and every $x,y\in\partial^{+}\CP
(z_{1},N )$
such that $\llVert x-y\rrVert _{1}=1$
%
\begin{eqnarray}\label{eqannealedderivativeestimation2}
&& \bigl\llvert\mathbb{P}^{z} ({T_{\partial\mathcal{P}(0,N)}}=m,
X_{{T_{\partial\mathcal{P}(0,N)}}}=x )-\mathbb{P}^{z} ({T_{\partial
\mathcal{P}(0,N)}}=m+1,
X_{{T_{\partial
\mathcal{P}(0,N)}}}=y )\bigr\rrvert\hspace*{-20pt}
\nonumber\\[-8pt]\\[-8pt]\nonumber
&&\qquad <CN^{-d-1}.
\end{eqnarray}

\item[(3)] For every $m\in\mathbb{N}$, every $x\in\partial^{+}\CP
(z_{1},N )$
and every $1\leq j\leq d$
%
\begin{eqnarray}\label
{eqannealedderivativeestimation3}
&& \bigl\llvert\mathbb{P}^{z} ({T_{\partial\mathcal{P}(0,N)}}=m,
X_{{T_{\partial\mathcal{P}(0,N)}}}=x )-\mathbb{P}^{z+e_{j}}
({T_{\partial\mathcal{P}(0,N)}}=m+1,
X_{{T_{\partial\mathcal{P}(0,N)}}}=x )\bigr\rrvert\hspace*{-17pt}
\nonumber\\[-8pt]\\[-8pt]\nonumber
&&\qquad <CN^{-d-1}.
\end{eqnarray}
\end{longlist}
\end{lem}

The proof of Lemma~\ref{lemAnnealedderivativeestimationsd-1+time}
can be found in Appendix~\ref{appendixsubAnnealed-derivative-estimations}.

Next, we prove an annealed concentration inequality for the hitting
time $T_{\partial\CP}$.

\begin{lem}
\label{lemDistanceoftimefromexpectation} Let $d\geq4$ and assume
$P$ is uniformly elliptic, i.i.d. and satisfies $ (\mathscr
{P} )$.
Then
%
\begin{equation}
\mathbb{P}^{z} (T_{\partial\CP(0,N )}\neq T_{\partial^{+}\CP(0,N )}
)=N^{-\xi(1
)}\label{eqannealedtimeestimations1}
\end{equation}
and for every $z\in\tilde{\CP} (0,N )$
%
\begin{equation}
\mathbb{P}^{z} \bigl(\bigl\llvert T_{\partial\CP(0,N
)}-
\mathbb{E}^{z} [T_{\partial\CP(0,N )} ]\bigr\rrvert>NR_{2} (N )
\bigr)=N^{-\xi(1
)}.\label{eqannealedtimeestimations2}
\end{equation}
\end{lem}

\begin{pf}
The fact that (\ref{eqannealedtimeestimations1}) holds was proved
in \cite{berger2008slowdown}, Lemma 4.2(1). For (\ref
{eqannealedtimeestimations2}),
we first show that $\llvert \mathbb{E}^{z} [\tau_{k}
]-\mathbb{E}^{z} [\tau_{k}\mid B_{N} ]\rrvert =N^{-\xi
(1 )}$
for every $1\leq k\leq N^{2}$, where $B_{N}$ is as defined in Corollary
\ref{corlengthofregenerations}. Indeed, using the notation $\tau_{0}=0$,
for every $1\leq k\leq N^{2}$
\begin{eqnarray*}
&& \mathbb{E}^{z} \bigl[\llvert\tau_{k}-\tau_r{k-1}
\rrvert\cdot\ind_{B_{N}^{c}} \bigr]
\\
&&\qquad   \leq\mathbb{E}^{z} \bigl[
\llvert\tau_{k}-\tau_{k-1}\rrvert\cdot\ind_{\exists j\neq k \llvert
\tau_{j}-\tau
_{j-1}\rrvert \geq R (N )}
\bigr]+\mathbb{E}^{z} \bigl[\llvert\tau_{k}-
\tau_{k-1}\rrvert\cdot\ind_{\llvert \tau_{k}-\tau
_{k-1}\rrvert \geq R (N )} \bigr]
\\
&&\qquad  \leq\mathbb{E}^{z} \bigl[\llvert\tau_{k}-
\tau_{k-1}\rrvert\bigr]\mathbb{P}^{z} \bigl(B_{N}^{c}
\bigr)+\sum_{t>R (N
)}t\cdot\mathbb{P}^{z}
\bigl(\llvert\tau_{k}-\tau_{k-1}\rrvert=t \bigr)
\\
&&\qquad  \leq\mathbb{E}^{z} \bigl[\llvert\tau_{k}-
\tau_{k-1}\rrvert\bigr]\mathbb{P}^{z} \bigl(B_{N}^{c}
\bigr)+\sum_{t>R (N
)}t\cdot\exp\bigl(- (\log t
)^{\alpha} \bigr)=N^{-\xi
(1 )},
\end{eqnarray*}
where for the last inequality we used Theorem~\ref{teoregenerationtimesthm}
and for the last equality we used Corollary~\ref{corlengthofregenerations}.
Therefore, for every $1\leq k\leq N^{2}$
\begin{eqnarray*}
&& \bigl\llvert\mathbb{E}^{z} [\tau_{k}-\tau_{k-1}
\mid B_{N} ]-\mathbb{E}^{z} [\tau_{k}-
\tau_{k-1} ]\bigr\rrvert
\\
&&\qquad  \leq\bigl\llvert\mathbb{E}^{z} [
\tau_{k}-\tau_{k-1}\mid B_{N} ]-
\mathbb{E}^{z} \bigl[ (\tau_{k}-\tau_{k-1} )\ind
_{B_{N}} \bigr]\bigr\rrvert
\\
&&\quad\qquad{} +\bigl\llvert\mathbb{E}^{z} \bigl[ (\tau_{k}-
\tau_{k-1} )\ind_{B_{N}} \bigr]-\mathbb{E}^{z} [
\tau_{k}-\tau_{k-1} ]\bigr\rrvert
\\
&&\qquad  =\bigl\llvert\bigl(1-\mathbb{P}^{z} (B_{N} ) \bigr)
\mathbb{E}^{z} \bigl[ (\tau_{k}-\tau_{k-1} )\mid B_{N} \bigr]\bigr\rrvert +\bigl\llvert\mathbb{E}^{z} \bigl[
(\tau_{k}-\tau_{k-1} )\ind_{B_{N}^{c}} \bigr]\bigr
\rrvert
\\
&&\qquad  =\mathbb{P}^{z} \bigl(B_{N}^{c} \bigr)
\mathbb{E}^{z} \bigl[ (\tau_{k}-\tau_{k-1} )\mid
B_{N} \bigr]+\mathbb{E}^{z} \bigl[ (\tau_{k}-
\tau_{k-1} )\ind_{B_{N}^{c}} \bigr]
\\
&&\qquad  \leq R (N )\mathbb{P}^{z} \bigl(B_{N}^{c}
\bigr)+N^{-\xi
(1 )}=N^{-\xi(1 )}.
\end{eqnarray*}
Summing the differences $ \{ \mathbb{E}^{z} [\tau_{j}-\tau
_{j-1}\mid B_{N} ]-\mathbb{E}^{z} [\tau_{j}-\tau
_{j-1} ] \} _{j=1}^{k}$
gives
\[
\bigl\llvert\mathbb{E}^{z} [\tau_{k}\mid B_{N}
]-\mathbb{E}^{z} [\tau_{k} ]\bigr\rrvert
=N^{-\xi(1
)},\qquad1\leq k\leq N^{2}.
\]

Since we know that $\mathbb{P}^{z} (B_{N} )=1-N^{-\xi
(1 )}$
(see Corollary~\ref{corlengthofregenerations}), it is enough to
show that
\[
\mathbb{P}^{z} \bigl(\bigl\llvert T_{\partial\CP(0,N
)}-
\mathbb{E}^{z} [T_{\partial\CP(0,N )} ]\bigr\rrvert>NR_{2} (N )
\mid B_{N} \bigr)=N^{-\xi
(1 )}.
\]
Under the event $B_{N}$, there exist some $1\leq k\leq N^{2}$ such
that $\tau_{k}\leq T_{\partial P}\leq\tau_{k}+R (N )$ and
thus (using the first estimation)
\begin{eqnarray*}
&& \mathbb{P}^{z} \bigl(\bigl\llvert T_{\partial\CP(0,N )}-
\mathbb{E}^{z} [T_{\partial\CP(0,N )} ]\bigr\rrvert>NR_{2} (N )
\mid B_{N} \bigr)
\\
&&\qquad  \leq\sum_{k=1}^{N^{2}}
\mathbb{P}^{z} \biggl(\bigl\llvert\tau_{k}-\mathbb
{E}^{z} [\tau_{k} ]\bigr\rrvert>\frac{1}{2}NR_{2}
(N ) \Big|  B_{N} \biggr)
\\
&&\qquad  \leq\sum_{k=1}^{N^{2}}\mathbb{P}^{z}
\biggl(\bigl\llvert\tau_{k}-\mathbb{E}^{z} \bigl[
\tau_{k}\mid B_{N} \bigr]\big\rrvert>\frac
{1}{4}NR_{2}
(N )\Big| B_{N} \biggr)+N^{-\xi(1 )}.
\end{eqnarray*}
Note that conditioned on $B_{N}$ the first $N$ regenerations are
still i.i.d., so by Azuma's inequality this can be bounded by
\[
\sum_{k=1}^{N^{2}}2\exp\biggl(-
\frac{N^{2}R_{2}^{2} (N
)}{32kR^{2} (N )} \biggr)+N^{-\xi(1 )}\leq e^{-R_{2} (N )}+N^{-\xi(1
)}=N^{-\xi
(1 )}.
\]\upqed
\end{pf}

\begin{lem}
\label{lemFirstdifferenceestimation} Let $d\geq4$ and assume
$P$ is uniformly elliptic, i.i.d. and satisfies $ (\mathscr
{P} )$.
Fix $0<\theta\leq1$. Let $L (N )=L (\theta,N )$
be the event that for every $\frac{2}{5}N^{2}\leq M\leq N^{2}$, every
$z\in\tilde{\mathcal{P}} (0,N )$, every $ (d-1
)$-dimensional
cube $\Delta$ of size $N^{\theta}$ which is contained in $H_{M}$
and every interval $I\subset\mathbb{N}$ of length $N^{\theta}$
\[
\bigl\llvert P_{{\omega}}^{z} (X_{T_{M}}\in\Delta,
T_{M}\in I, B_{N} )-\mathbb{P}^{z}
(X_{T_{M}}\in\Delta, T_{M}\in I, B_{N} )\bigr
\rrvert\leq N^{d (\theta-1 )}.
\]
Then for $\theta>\frac{d}{d+1}$, $P (L (\theta,N
) )=1-N^{-\xi(1 )}$.
\end{lem}

\begin{pf}
Fix\vspace*{1pt} $\theta$, and let $\frac{d}{d+1}<\theta'<\theta$. Let $V=
[N^{2\theta'} ]$.
Fix $\frac{2}{5}N^{2}\leq M\leq N^{2}$, $v\in H_{M+V}$ and $m\in
\mathbb{N}$.
Finally denote by $\CG$ the $\si$-algebra determined by the configuration
on
\[
\CP^{M} (0,N ):=\CP(0,N )\cap\bigl\{ x: \langle x,e_{1}
\rangle\leq M \bigr\}.
\]
We are interested in the quantity (see also Figure~\ref{fig2})
\[
J^{ (M )} (v,m )=E \bigl[P_{{\omega}}^{z}
(X_{T_{M+V}}=v,T_{M+V}=m, B_{N} )\mid\CG\bigr].
\]

\begin{figure}

\includegraphics{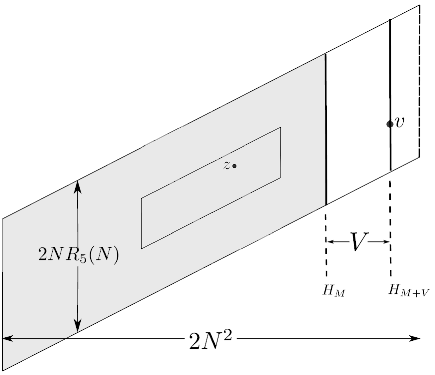}

\caption{The quantity $J^{ (M )} (v,m )$ is the
probability
of hitting the point $v$ at time $m$, conditioned on the environments
in the gray\vspace*{1pt} area, and averaged (annealed) elsewhere. The small parallelogram
indicates the middle third $\tilde{\mathcal{P}} (0,N )$
in which the random walk starts.}
\label{fig2}
\end{figure}
In order to estimate $J^{ (M )} (v,m )$ we order
the vertices of $\mathcal{P}^{M} (0,N )$ lexicographically,
$x_{1},x_{2},\ldots$ with the first coordinate being the most significant
and let $ \{ \CF_{k} \} $ be the $\sigma$-algebra determined
by ${\omega} (x_{1},\cdot),\ldots,\omega
(x_{k},\cdot)$,
so in particular for every $-N^{2}+1\leq l\leq M$ the vertices in
$H_{l}\cap\mathcal{P}^{M} (0,N )$ are exposed after those
in $H_{l-1}\cap\mathcal{P}^{M} (0,N )$.

Consider the martingale $M_{k}=E [P_{{\omega}}^{z}
(X_{T_{M+V}}=v, T_{M+V}=m\llvert B_{N} )\rrvert \mathcal
{F}_{k} ]$.
In order to use McDiarmid's inequality, we first bound $U_{k}:=\esup
(\llvert M_{k}-M_{k-1}\rrvert \mid \CF_{k-1} )$.
We claim that for an event of environments with $P$ probability $\geq
1-N^{-\xi(1 )}$
\[
U_{k}\leq CR (N )E \bigl[P_{{\omega}}^{z}
\bigl(x_{k} \mbox{ is visited }\llvert B_{N} \bigr)\rrvert
\CF_{k-1} \bigr]V^{-\vfrac{d+1}{2}}.
\]
Indeed, let ${\omega}'$ be an environment that agrees with ${\omega}$
everywhere
except possibly in~$x_{k}$. Let $\mathtt{P}$ be the probability
measure under which the random walk has quenched transition probabilities
given by $\omega$ in $ \{ x_{i}: i\leq k \} $ and averaged
(annealed) transition probabilities in $\mathbb{Z}^{d}\setminus
\{ x_{i}: i\leq k \} $
conditioned the event $B_{N}$. Similarly, let $\mathtt{P}'$ be the
probability measure defined like $\mathtt{P}$ with $\omega'$ instead
of $\omega$. More\vspace*{1pt} formally for an event $A\subset(\BZ
^{d} )^{\BN}$,
we have $\mathtt{P} (A )=E [P_{\omega}^{z}
(A\mid B_{N} )\mid\CF_{k} ]$
and equivalently for $\mathtt{P}'$. Then
\begin{eqnarray*}
U_{k} & \leq&\sup_{{\omega}'}\bigl\llvert
\mathtt{P}' (X_{T_{M+V}}=v,T_{M+V}=m )-\mathtt{P}
(X_{T_{M+V}}=v,T_{M+V}=m )\bigr\rrvert
\\[-3pt]
& \leq&\sup_{{\omega}'}\bigl\llvert\mathtt{P}'
\bigl(X_{T_{M+V}}=v,T_{M+V}=m, \{ x_{k} \mbox{ is not
visited} \} \bigr)
\\[-3pt]
&&{} -\mathtt{P} \bigl(X_{T_{M+V}}=v,T_{M+V}=m, \{
x_{k} \mbox{ is not visited} \} \bigr)\bigr\rrvert
\\[-3pt]
&&{} +\sup_{{\omega}'}\bigl\llvert\mathtt{P}'
\bigl(X_{T_{M+V}}=v,T_{M+V}=m, \{ x_{k} \mbox{ is
visited} \} \bigr)
\\[-3pt]
&&{}-\mathtt{P} \bigl(X_{T_{M+V}}=v,T_{M+V}=m, \{
x_{k} \mbox{ is visited} \} \bigr)\bigr\rrvert,
\end{eqnarray*}
where the supremum is taken over all environments ${\omega}'$ that agree
with ${\omega}$ on $\BZ^{d}\setminus\{ x_{k} \} $. Note that
on the event $ \{ x_{k} \mbox{ is not visited} \} $, the
distributions $\mathtt{P}$ and $\mathtt{P}'$ are the same and, therefore,
the difference of the probabilities equals zero. On the other hand,
on the event $ \{ x_{k} \mbox{ is visited} \} $, we can
couple both walks together until the first hitting time of $x_{k}$
(which in particular implies that the hitting time of $x_{k}$ is
the same). Since we conditioned on the event $B_{N}$, the next regeneration
time after hitting $x_{k}$ is at most $R (N )$ steps later.
Therefore, from Lemma~\ref{lemAnnealedderivativeestimationsd-1+time}
it follows that
\begin{eqnarray*}
&& \bigl\llvert\mathtt{P}^{z} (X_{T_{M+V}}=v,
T_{M+V}=m\mid x_{k} \mbox{ is visited} )-
\mathbb{P}^{x_{k}} (X_{T_{M+V}}=v, T_{M+V}=m )\bigr\rrvert
\\[-2pt]
&&\qquad
<CR (N )V^{-\vfrac{d+1}{2}}
\end{eqnarray*}
and
\begin{eqnarray*}
&& \bigl\llvert\mathtt{P}'^{z} (X_{T_{M+V}}=v,
T_{M+V}=m\mid x_{k} \mbox{ is visited} )-
\mathbb{P}^{x_{k}} (X_{T_{M+V}}=v, T_{M+V}=m )\bigr\rrvert
\\
&&\qquad <CR (N )V^{-\vfrac{d+1}{2}}.
\end{eqnarray*}
Consequently, we get
\[
U_{k}\leq CR (N )V^{-\vfrac{d+1}{2}}\mathtt{P}^{z}
(x_{k} \mbox{ is visited} )
\]
as required.

Next, we show that $U:=\sum_{k}U_{k}^{2}$ is bounded by
$CR_{2}^{2} (N )V^{-d-1}$
provided $\omega\in J (N )$. Indeed, noting that if $x$
is visited and $B_{N}$ holds, then the first visit to the layer
$H_{ \langle x,e_{1} \rangle-1}$
is in the box $B (x )= \{ y: y\in H_{ \langle
x,e_{1} \rangle-1}, \llVert y-x\rrVert _{\infty}\leq
R (N ) \} $
it follows that
\begin{eqnarray*}
U_{k} & \leq& CR (N )V^{-\vfrac{d+1}{2}}\mathbf{P}^{z}
\bigl(x_{k}\in[X ]\mid\CF_{k-1},B_{N} \bigr)
\\
& \leq& CR (N )V^{-\vfrac{d+1}{2}}\sum_{y\in B
(x_{k} )}
\mathbf{P}^{z} (T_{ \langle x,e_{1}
\rangle-1}=y\mid\mathcal{F}_{k-1} )
\\
& =&CR (N )V^{-\vfrac{d+1}{2}}\sum_{y\in B (x_{k}
)}P_{\omega}^{z}
(T_{ \langle x,e_{1} \rangle
-1}=y )
\\
& \leq& CR (N )V^{-\vfrac{d+1}{2}}\sum_{y\in B
(x_{k} )}P_{\omega}^{z}
\bigl(y\in[X ] \bigr).
\end{eqnarray*}
Since\vspace*{1pt} $\llvert B (x_{k} )\rrvert \leq C\cdot2^{d}\cdot
R^{d} (N )$
and every $y\in\mathbb{Z}^{d}$ is in $B (x )$ for at most
$2^{d}R^{d} (N )$ points $x\in\mathbb{Z}^{d}$, we get
\begin{eqnarray*}
U&:=&\sum_{k=1}^{n}U_{k}^{2}
 \leq C\sum_{k=1}^{n}R^{2} (N
)V^{-d-1}\cdot\biggl[\sum_{y\in B (x_{k} )}P_{{\omega
}}^{z}
\bigl(y\in[X ] \bigr) \biggr]^{2}
\\
& \leq& R^{2} (N )\cdot2^{d}R^{d} (N
)V^{-d-1}\sum_{k=1}^{n}\sum
_{y\in B (x_{k} )}P_{{\omega
}}^{z} \bigl(y\in[X
] \bigr)^{2}
\\
& \leq& C\cdot2^{2d}\cdot R^{2d+2} (N )V^{-d-1}\sum
_{y\in
\mathcal{P} (0,N )}P_{{\omega}}^{z} \bigl(y\in[X
], B_{N} \bigl( \{ X_{n} \} \bigr) \bigr)^{2}
\\
&\overset{ (1 )} {\leq}& C\cdot2^{2d}\cdot R^{2d+2} (N )\cdot
R_{2} (N )V^{-d-1}\leq C\cdot R_{2}^{2}
(N )V^{-d-1},
\end{eqnarray*}
where for $ (1 )$ we used the assumption $\omega\in J
(N )$.
Thus, by McDiarmid's inequality (see Theorem~\ref{teoAzumatypeinequality})
for every $\delta>0$
\begin{eqnarray*}
&&  P \bigl(\bigl\llvert E \bigl[P_{{\omega}}^{z}
(X_{T_{M+V}}=v, T_{M+V}=m, B_{N}\mid\CG) \bigr]
\\
&&\quad{} -
\mathbb{P}^{z} (X_{T_{M+V}}=v,
 T_{M+V}=m,
B_{N} )\bigr\rrvert>\delta\bigr)
\\
&&\qquad \leq P \bigl(J (N )^{c} \bigr)+2\exp\biggl(-\frac
{\delta^{2}}{2CR_{2}^{2} (N )V^{-d-1}}
\biggr).
\end{eqnarray*}
In particular, for $\delta=\frac{1}{4}N^{-d}=\frac{1}{4}V^{-\vfrac{d+1}{2}}V^{\eta}$,
with $\eta=\frac{ (d+1 )\theta'-d}{2\theta'}>0$ we get
\begin{eqnarray*}
&& P \biggl(\bigl\llvert E \bigl[P_{{\omega}}^{z}
(X_{T_{M+V}}=v, T_{M+V}=m, B_{N}\mid\CG) \bigr]
\\
&&\quad{}-
\mathbb{P}^{z} (X_{T_{M+V}}=v, T_{M+V}=m,
B_{N} )\bigr\rrvert>\frac
{1}{4}N^{-d} \biggr)
\\
&&\qquad \leq N^{-\xi(1 )}+2\exp\biggl(-\frac{ [N^{2\theta
'} ]^{2\eta}}{32CR_{2}^{2} (N )} \biggr)=N^{-\xi
(1 )}.
\end{eqnarray*}
Using Corollary~\ref{corlengthofregenerations}, this also gives
\begin{eqnarray*}
\hspace*{-5pt}&& P \bigl(\bigl\llvert E \bigl[P_{{\omega}}^{z}
(X_{T_{M+V}}=v, T_{M+V}=m\mid\CG) \bigr]-\mathbb{P}^{z}
(X_{T_{M+V}}=v, T_{M+V}=m )\bigr\rrvert>\tfrac{1}{2}N^{-d}
\bigr)
\\
\hspace*{-5pt}&&\qquad =N^{-\xi(1 )}.
\end{eqnarray*}

Let $W_{1} (N )\subset\Om$ be the event that
\begin{eqnarray*}
&& \bigl\llvert E \bigl[P_{{\omega}}^{z} (X_{T_{M+V}}=v,
T_{M+V}=m, B_{N}\mid\CG) \bigr]-\mathbb{P}^{z}
(X_{T_{M+V}}=v, T_{M+V}=m, B_{N} )\bigr\rrvert
\\
&&\qquad \leq
\tfrac{1}{2}N^{-d}
\end{eqnarray*}
for every $\frac{2}{5}N^{2}\leq M\leq N^{2}$, every $v\in H_{M+V}\cap
\CP(0,2N )$,
every $z\in\tilde{\mathcal{P}} (0,N )$ and every $m\in
[0,N^{3} ]$.
Then by the above argument $P (W_{1} (N )
)=1-N^{-\xi(1 )}$.
Consider now ${\omega}\in W_{1} (N )$, $\frac
{2}{5}N^{2}\leq M\leq N^{2}$
a cube $\Delta$ of side length $N^{\theta}$ which is contained in
$H_{M}$ and an interval $I$ of length $N^{\theta}$. We wish to
estimate
\[
\bigl\llvert P_{{\omega}}^{z} (X_{T_{M}}\in\Delta,
T_{M}\in I, B_{N} )-\mathbb{P}^{z_{1}}
(X_{T_{M}}\in\Delta, T_{M}\in I, B_{N} )\bigr
\rrvert.
\]
Let $c (\Delta)$ and $c (I )$ be the centers
of the cube $\Delta$ and the interval $I$, respectively. Let $c'
(\Delta)=c (\Delta)+V\frac{\vartheta}{
\langle\vartheta,e_{1} \rangle}$,
$c' (I )=c (I )+V\frac{1}{ \langle\bbbv,e_{1} \rangle}$
and define (see also Figure~\ref{figBixsmallbox})
\begin{eqnarray*}
\Delta^{ (1 )} &=& \bigl\{ v\in H_{M+V}: \bigl\llVert
v-c' (\Delta)\bigr\rrVert_{\infty}<\tfrac{1}{2}\cdot
\tfrac{9}{10}\cdot N^{\theta} \bigr\},
\\
\Delta^{ (2 )} &=& \bigl\{ v\in H_{M+V}: \bigl\llVert
v-c' (\Delta)\bigr\rrVert_{\infty}<\tfrac{1}{2}\cdot
\tfrac{11}{10}\cdot N^{\theta} \bigr\},
\\
I^{ (1 )}&=& \bigl\{ t\in\mathbb{N}: \bigl\llvert t-c' (I )
\bigr\rrvert<\tfrac{1}{2}\cdot\tfrac{9}{10}\cdot N^{\theta
}
\bigr\},
\end{eqnarray*}
and
\[
I^{ (2 )}= \bigl\{ t\in\mathbb{N}: \bigl\llvert t-c' (I )
\bigr\rrvert<\tfrac{1}{2}\cdot\tfrac{11}{10}\cdot N^{\theta
}
\bigr\}.
\]

\begin{figure}

\includegraphics{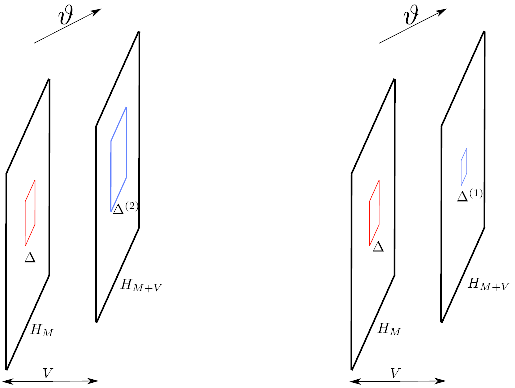}

\caption{Using\vspace*{1pt} the annealed walk in the area between
the hyperplane $H_{M}$ and $H_{M+V}$, we can turn the estimates on
the hitting probabilities of the small box $\Delta^{ (1 )}$
and big box $\Delta^{ (2 )}$ in the hyperplane $H_{M+V}$
into both quenched and annealed estimations for the hitting probabilities
of the box $\Delta$ in the hyperplane $H_{M}$. The\vspace*{1pt} probability to
hit $\Delta$ and not to hit $\Delta^{ (2 )}$ as well as
the probability not to hit $\Delta$ but to hit $\Delta^{
(1 )}$
are of order $N^{-\xi(1 )}$.}
\label{figBixsmallbox}
\end{figure}
Annealed estimations (for the proof see Appendix~\ref
{AppendixsubMore-annealed-estimations})
yields
%
\begin{eqnarray}
\qquad \mathbb{P}^{z} \bigl(X_{T_{M+V}}\in\Delta^{ (1 )},
T_{M+V}\in I^{ (1 )} \bigr)&<&\mathbb{P}^{z}
(X_{T_{M}}\in\Delta, T_{M}\in I )+N^{-\xi(1
)},\label{eqFirstdifferenceestimationannealedestimation1}
\\
%
\mathbb{P}^{z} \bigl(X_{T_{M+V}}\in\Delta^{ (2 )},
T_{M+V}\in I^{ (2 )} \bigr)&>&\mathbb{P}^{z}
(X_{T_{M}}\in\Delta, T_{M}\in I )-N^{-\xi(1
)}\label{eqFirstdifferenceestimationannealedestimation2}
\end{eqnarray}
and also, due to Claim~\ref{Clmsimpleclaimannealedtoquenched},
for an event $W_{2} (N )$ of $P$ probability $\geq1-N^{-\xi
(1 )}$
%
\begin{eqnarray}
&& E \bigl[P_{{\omega}}^{z} \bigl(X_{T_{M+V}}\in
\Delta^{ (1
)}, T_{M+V}\in I^{ (1 )} \bigr)\mid
\mathcal{G} \bigr]
\nonumber\\[-8pt]\label{eqFirstdifferenceestimationannealedestimation3} \\[-8pt]\nonumber
&&\qquad <P_{\omega}^{z} (X_{T_{M}}\in\Delta,
T_{M}\in I )+N^{-\xi(1 )},
%
\\
&& E \bigl[P_{{\omega}}^{z} \bigl(X_{T_{M+V}}\in
\Delta^{ (2
)}, T_{M+V}\in I^{ (2 )} \bigr)\mid
\mathcal{G} \bigr]
\nonumber\\[-8pt]\label{eqFirstdifferenceestimationannealedestimation4} \\[-8pt]\nonumber
&&\qquad >P_{\omega}^{z} (X_{T_{M}}\in\Delta,
T_{M}\in I )-N^{-\xi(1 )}.
\end{eqnarray}
Thus, from the definition of $W_{1} (N )$, $W_{2}
(N )$
and the last $4$ estimations, it follows that $W_{1} (N
)\cap W_{2} (N )\subset L (\theta,N )$
and the proof is complete.
\end{pf}

\begin{lem}
\label{lemUpperestimationonthequenchedexit} Let\vspace*{1pt} $d\geq4$ and
assume $P$ is uniformly elliptic, i.i.d. and satisfies $ (\mathscr
{P} )$.
For every $0<\theta\leq1$ and $h\in\mathbb{N}$ let $D^{
(\theta,h )} (N )$
be the event that for every $z\in\tilde{\CP} (0,N )$, every
$\frac{1}{2}N^{2}\leq M\leq N^{2}$, every $ (d-1 )$-dimensional
cube of side length $N^{\theta}$ which is contained in $H_{M}$ and
every interval $I\subset\mathbb{N}$ of length $N^{\theta}$
%
\begin{equation}
P_{{\omega}}^{z} (X_{T_{M}}\in\Delta, T_{M}
\in I )\leq R_{h} (N )N^{ (\theta-1 )d}\label
{eqUpperestimationonthequenchedexit1}
\end{equation}
and
%
\begin{equation}
P_{\omega}^{z} (X_{T_{M}}\in\Delta)\leq
R_{h} (N )N^{ (\theta-1 ) (d-1 )}.\label
{eqUpperestimationonthequenchedexit2}
\end{equation}
Then for every $0<\theta\leq1$ there exists $h=h (\theta)$
such that $P (D^{ (\theta,h )} (N )
)=1-N^{-\xi(1 )}$.
\end{lem}

\begin{pf}
The proof of (\ref{eqUpperestimationonthequenchedexit2}) is
the content of \cite{berger2008slowdown}, Lemma 4.13, and, therefore,
we restrict attention to the proof of (\ref
{eqUpperestimationonthequenchedexit1}).
We prove the lemma by a descending induction on $\theta$. From Lemma
\ref{lemFirstdifferenceestimation} together with Lemma~\ref
{lemAnnealedderivativeestimationsd-1+time}(1),
$P (D^{ (\theta,1 )} )\geq P (L
(\theta,N ) )=1-N^{-\xi(1 )}$
for every $\frac{d}{d+1}<\theta\leq1$. For the induction step, fix
$\theta$ and assume that the statement of the lemma holds for some
$\theta'$ such that $\theta>\frac{d}{d+1}\theta'$. Define
$h'=h (\theta' )$
and $\rho=\frac{\theta}{\theta'}>\frac{d}{d+1}$. Let
\[
S (N )=D^{ (\rho,1 )} (N )\cap\mathop{\bigcap_{z\in\CP(0,2N )}}
_{s\in[-2NR_{5} (N ),2NR_{5} (N ) ]} \sigma_{z}\varrho_{s}
\bigl(D^{ (\theta
',h' )} \bigl( \bigl[N^{\rho} \bigr] \bigr) \bigr)\cap T
(N,\rho),
\]
where $\varrho$ is the time shift for the random walk, defined by
$\varrho(X_{1},X_{2},\ldots)= (X_{2},X_{3},\ldots
)$
and
\begin{eqnarray*}
T (N,\rho) &=& \bigl\{ {\omega}\in\Om:
\forall
v\in\CP(0,N ),  P_{{\omega}}^{v} \bigl(X_{T_{\partial\CP(v,
[N^{\rho} ] )}}
\notin\partial^{+}\CP\bigl(v, \bigl[N^{\rho} \bigr] \bigr)
\bigr)<e^{-R (N )}
\\
&&{} \mbox{and }P_{{\omega}}^{v} \bigl(\bigl\llvert
T_{\partial\CP
(v, [N^{\rho} ] )}-\mathbb{E}^{v} [T_{\partial
\CP(v, [N^{\rho} ] )} ]\bigr\rrvert
>NR_{2} (N ) \bigr)=N^{-\xi(1 )}
\bigr\}.
\end{eqnarray*}
From the definition of $S (N )$, Lemma~\ref
{lemDistanceoftimefromexpectation}
and the induction assumption, we know that $P (S (N
) )=1-N^{-\xi(1 )}$.
Therefore, we need to show that for some $h$ and all $N$ large enough,
we have $S (N )\subset D^{ (\theta,h )}
(N )$.
To this end, fix ${\omega}\in S (N )$, $z\in\tilde{\CP
} (0,N )$,
$\frac{1}{2}N^{2}\leq M\leq N^{2}$, a $ (d-1 )$-dimensional
cube $\Delta$ of size length $N^{\theta}$ in $\CP(0,N
)\cap H_{M}$
and an interval $I\subset[\mathbb{E}^{z} [T_{M}
]-NR_{2} (N ),\mathbb{E}^{z} [T_{M}
]+NR_{2} (N ) ]$
of length $N^{\theta}$. As before, we denote by $c (\Delta)$
and $c (I )$ the centers of $\Delta$ and $I$, respectively.
Let $V= [N^{\rho} ]^{2}$, $c (\Delta)'=c
(\Delta)-V\frac{\vartheta}{ \langle\vartheta,e_{1} \rangle}$
and $c (I )'=c (I )-V\frac{1}{ \langle
\bbbv,e_{1} \rangle}$.
Since ${\omega}\in\bigcap_{z\in\CP(0,2N )}\bigcap_{s\in[-2NR_{5} (N
),2NR_{5} (N )
]}\sigma_{z}\varrho_{s} (D^{ (\theta',h' )}
( [N^{\rho} ] ) )$
it follows that for every $v\in H_{M-V}$ and every $t\in\mathbb{N}$
\[
P_{{\omega}}^{v} (X_{T_{M}}\in\Delta, T_{M}
\in I-t )<R_{h'} (N )N^{\rho(\theta'-1 )d}=R_{h'} (N
)N^{ (\theta-\rho)d}.
\]
In addition, due to the Markov property of the quenched law
\begin{eqnarray*}
&& P_{{\omega}}^{z} (X_{T_{M}}\in\Delta, T_{M}
\in I )
\\
&&\qquad  =\mathop{\sum_{v\in H_{M-V}\cap\CP(x', [N^{\rho} ] )}}_{
\llvert t-c' (I )\rrvert \leq N^{\rho}R_{5} (N^{\rho
} )}
P_{{\omega}}^{z} (X_{T_{M-V}}=v, T_{M-V}=t
)P_{{\omega}}^{v} (X_{T_{M}}\in\Delta, T_{M}
\in I-t )
\\
&&\quad\qquad{}+N^{-\xi(1 )}
\\
&&\qquad  \leq \mathop{\sum_{v\in H_{M-V}\cap\CP(x', [N^{\rho} ] )}}_
{\llvert t-c' (I )\rrvert \leq N^{\rho}R_{5} (N^{\rho
} )}P_{{\omega}}^{z}
(X_{T_{M-V}}=v, T_{M-V}=t )R_{h'} (N
)N^{ (\theta-\rho
)d}+N^{-\xi(1 )}.
\end{eqnarray*}
Now, the last sum can be separated into the sum over $2^{d-1}R_{5}
(N^{\rho} )^{d-1}$
$ (d-1 )$-dimensional cubes of side length $N^{\rho}$ and
$2R_{2} (N^{\rho} )$ intervals of length $N^{\rho}$. Since
${\omega}\in D^{ (\rho,1 )} (N )$ the
probability to
hit each of these cubes in any of these time intervals is bounded
by $R_{1} (N )N^{ (\rho-1 )d}$. Thus,
\begin{eqnarray*}
P_{{\omega}}^{z} (X_{T_{M}}\in\Delta, T_{M}
\in I ) & <&2^{d}R_{5} \bigl(N^{\rho}
\bigr)^{d}R_{1} (N )N^{
(\rho-1 )d}R_{h'} (N
)N^{ (\theta-\rho
)d}+N^{-\xi(1 )}
\\
& \leq& R_{\max\{ 6,h' \} +1} (N )N^{
(\theta-1 )d},
\end{eqnarray*}
and the proof is complete by taking $h=\max\{ 6,h' \} +1$.
\end{pf}
\begin{lem}
\label{lemEstimationforfarhyperplane} Let $d\geq4$ and assume
$P$ is uniformly elliptic, i.i.d. and satisfies $ (\mathscr
{P} )$.
Let $\CG$ be the $\si$-algebra generated by $ \{ {\omega}
(z ): \langle z,e_{1} \rangle\leq N^{2} \} $.
Let $\eta>0$, $V= [N^{\eta} ]$ and define $R (N,\eta
)$
to be the event that for every $z\in\tilde{\CP} (0,N )$,
every $v\in H_{N^{2}+V}$ and every $m\in\mathbb{N}$
\begin{eqnarray*}
&& \bigl\llvert E \bigl[P_{{\omega}}^{z} (X_{T_{N^{2}+V}}=v,
T_{N^{2}+V}=m )\mid\CG\bigr]-\mathbb{P}^{z}
(X_{T_{N^{2}+V}}=v, T_{N^{2}+V}=m )\bigr\rrvert
\\
&&\qquad \leq N^{-d}V^{\vfrac{1-d}{6}}.
\end{eqnarray*}
Then $P (R (N,\eta) )=1-N^{-\xi(1
)}$.
\end{lem}

\begin{pf}
Let $v\in H_{N^{2}+V}$, $m\in\mathbb{N}$ and let $\theta>0$ be
such that $\theta<\frac{1}{20}\eta$. Let $K$ be an integer such
that $2^{-K-1}N^{2}\leq V<2^{-K}N^{2}$, and for $0\leq k\leq K$
define (see also Figure~\ref{figThebullet})
\begin{eqnarray*}
\CP^{ (k )} &=&\CP(0,N )\cap\bigl\{ x\in\mathbb{Z}^{d}: 0\leq
N^{2}- \langle x,e_{1} \rangle\leq2^{-k}N^{2}
\bigr\} \qquad\forall1\leq k\leq K,
\\
\CP^{ (0 )}&=&\CP(0,N )\cap\biggl\{ x\in\mathbb{Z}^{d}:
\frac{N^{2}}{2}\leq N^{2}- \langle x,e_{1} \rangle\biggr
\},
\\
F (v ) &=& \biggl\{ x\in\mathcal{P} (0,N ): \biggl\llVert x-v-\vartheta
\frac{ \langle x-v,e_{1} \rangle
}{ \langle\vartheta,e_{1} \rangle}\biggr\rrVert_{\infty
}\leq\bigl\llvert\langle
v-x,e_{1} \rangle\bigr\rrvert^{\sfrac {1}{2}}R_{2} (N )
\biggr\},
\\
\CP^{ (k )} (v ) &=&\mathcal{P}^{ (k
)}\cap F (v ),
\end{eqnarray*}
and
\[
\widehat{\CP}^{ (k )} (v )= \bigl\{ y\in\mathbb{Z}^{d}:
\exists x\in\mathcal{P}^{ (k )} (v ) \mbox{ such that }\llVert x-y\rrVert
_{\infty
}<R_{2} (N ) \bigr\}.
\]

\begin{figure}

\includegraphics{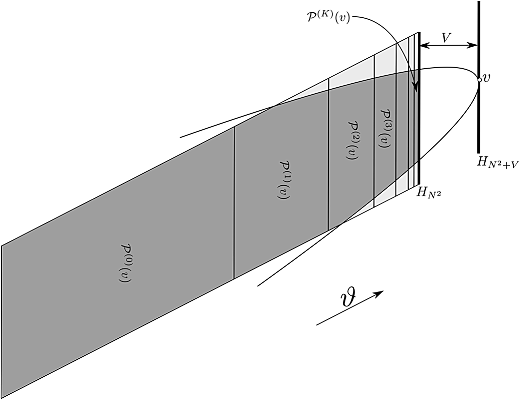}

\caption{The dark gray areas are $\mathcal{P}^{ (k )}
(v )$
for different values of $k$. The environment in the light gray area
has negligible influence on the probability of hitting $v$. (The
picture is not to scale.)}
\label{figThebullet}
\end{figure}

Condition on the event $D^{ (\theta,h )}$ from Lemma~\ref
{lemUpperestimationonthequenchedexit},
with $h$ such that\break $P (D^{ (\theta,h )} (N
) )=1-N^{-\xi(1 )}$.
As in \cite{berger2008slowdown}, Lemma 4.14, for $0\leq k\leq K$
and ${\omega}\in D^{ (\theta,h )}$, we have the estimation
\begin{eqnarray*}
V (k ) & =&E_{{\omega},{\omega}}^{z,z} \bigl[
\bigl\llvert\bigl[X^{ (1 )} \bigr]\cap\bigl[X^{ (2 )} \bigr]\cap
\CP^{ (k )} (v )\bigr\rrvert\bigr]
\\
& \leq&\cases{ R_{2} (N ), &\quad$k=0$,
\cr
R_{h+1} (N
)N^{2 (\vfrac{d+1}{2}+ (1-\theta
) (1-d ) )}2^{-k \lfloor\vfrac{d+1}{2}
\rfloor}, &\quad$1\leq k\leq K$,}
\end{eqnarray*}
where $E_{\omega,\omega}^{z,z}$ is as defined in Section~\ref
{subIntersections-of-paths}.

Indeed, for $k=0$ this follows from Lemma~\ref{lemIntersectionlemma}
while for $k>0$
\begin{eqnarray*}
V (k ) & =&\sum_{x\in\mathcal{P}^{ (k )}
(v )} \bigl[P_{\omega}^{z}
(x \mbox{ is visited} ) \bigr]^{2}
\\
& \leq&\sum_{x\in\mathcal{P}^{ (k )} (v
)} \biggl[\sum
_{y: \llVert y-x\rrVert _{\infty}<R
(N )}P_{\omega}^{z} (X_{T_{ \langle y,e_{1}
\rangle}}=y )
\biggr]^{2}+N^{-\xi(1 )}
\\
& \overset{ (1 )} {\leq}&\sum_{x\in\mathcal{P}^{
(k )} (v )}C\cdot
R^{d} (N )\cdot\sum_{y: \llVert y-x\rrVert _{\infty}<R (N )}
\bigl[P_{\omega}^{z} (X_{T_{ \langle y,e_{1} \rangle
}}=y )
\bigr]^{2}+N^{-\xi(1 )}
\\
& \overset{ (2 )} {\leq}&C\cdot R^{2d} (N )\sum
_{y\in\widehat{\mathcal{P}}^{ (k )} (v )} \bigl[P_{\omega}^{z}
(X_{T_{ \langle y,e_{1} \rangle
}}=y ) \bigr]^{2}+N^{-\xi(1 )}
\\
& \overset{ (3 )} {\leq}&R_{2} (N )\sum_{y\in
\widehat{\mathcal{P}}^{ (k )} (v )}R_{h}^{2}
(N )N^{2 (1-\theta) (1-d )}
\\
& \leq& R_{h+1} (N )N^{2 (\vfrac{d+1}{2}+ (1-\theta
) (1-d ) )}2^{-k \lfloor\vfrac {d+1}{2} \rfloor},
\end{eqnarray*}
where for $ (1 )$ we used Cauchy--Schwarz inequality, for
$ (2 )$ we used the fact that each point is counted at most
$R^{d} (N )$ times and for $ (3 )$ the assumption
$\omega\in D^{ (\theta,h )} (N )$.

We now use again the filtration $ \{ \CF_{i} \} $ from Lemma
\ref{lemFirstdifferenceestimation}, and consider the martingale
\[
M_{i}=E \bigl[P_{{\omega}}^{z} (X_{T_{N^{2}+V}}=v, T_{N^{2}+V}=m\mid
B_{N} )\mid\mathcal{F}_{i}
\bigr].
\]
In order to use McDiarmid's inequality, we need to bound $U_{i}=\esup
(\llvert M_{i}-M_{i-1}\rrvert \mid\mathcal{F}_{i-1} )$
under the assumption $\omega\in D^{ (\theta,h )}
(N )$.
Let $x$ be such that $\omega_{x}$ is measurable with respect to
$\CF_{i}$ but not with respect to $\CF_{i-1}$. By a similar argument
as in the proof of Lemma~\ref{lemFirstdifferenceestimation}, we
have $U_{i}=N^{-\xi(1 )}$ if $x\notin F (v )$,
while for $x\in F (v )$
\[
U_{i}\leq R (N )E \bigl[P_{{\omega}}^{z} (x \mbox{ is
hit}\mid B_{N} )\mid\mathcal{F}_{i-1} \bigr]\operatorname{Der}
\bigl(N^{2}+V- \langle x,e_{1} \rangle\bigr),
\]
where $\operatorname{Der} (N^{2}+V- \langle x,e_{1} \rangle)$
is the maximal derivative of the annealed distribution with respect
to both place and time with distance $N^{2}+V- \langle
x,e_{1} \rangle$
to the hitting hyperplane. By Lemma~\ref
{lemAnnealedderivativeestimationsd-1+time},
this derivative is bounded by $CN^{-d-1}2^{k\sklfrac{d}{2}}$ and, therefore,
whenever $\omega\in D^{ (\theta,h )}$
\begin{eqnarray*}
U & =&\sum_{i}U_{i}^{2}\leq
C\sum_{k=0}^{K}V (k )N^{-2
(d+1 )}2^{kd}+N^{-\xi(1 )}
\\
& \leq& CR_{h+1} (N )N^{-2 (d+1 )}
\\
&&{} +CR_{h+1} (N )N^{2 (\vfrac{d+1}{2}+ (1-\theta
) (1-d ) )-2 (d+1 )}\sum_{k=1}^{K}2^{kd-k\sklvfrac{d+1}{2}}+N^{-\xi(1 )}
\\
& \leq& CR_{h+1} (N ) \bigl[N^{-2 (d+1
)}+N^{2 (\vfrac{d+1}{2}+ (1-\theta) (1-d
) )-2 (d+1 )}c2^{\sklvfrac{d-1}{2}K}
\bigr]
\\
&& {} +N^{-\xi
(1 )}
\\
& \leq& CR_{h+1} (N ) \bigl[N^{-2 (d+1
)}+N^{2 (\vfrac{d+1}{2}+ (1-\theta) (1-d
) )-2 (d+1 )}2^{\sklvfrac{d-1}{2}K}
\bigr]
\\
&& {} +N^{-\xi
(1 )}.
\end{eqnarray*}
Recalling that $K$ was chosen so that $2^{K}<N^{2}V^{-1}$ we can
bound the last sum term by
\begin{eqnarray*}
&& CR_{h+1} (N ) \bigl[N^{-2 (d+1 )}+N^{2 (\vfrac{d+1}{2}+ (1-\theta) (1-d )
)-2 (d+1 )}N^{ (d-1 )}V^{-\vfrac{d-1}{2}}
\bigr]
\\
&&\quad{} +N^{-\xi(1 )}
\\
&&\qquad =  CR_{h+1} (N ) \bigl[N^{-2 (d+1
)}+N^{-2d-2\theta(1-d )}V^{-\vfrac{d-1}{2}}
\bigr]+N^{-\xi(1 )}
\\
&&\qquad \leq CN^{-2d}N^{-\sklvfrac{d-1}{6}+\epsilon}
\end{eqnarray*}
for some small enough $\ep>0$.

Using McDiarmid's inequality (see Theorem~\ref{teoAzumatypeinequality}),
\begin{eqnarray*}
&& P \bigl(\bigl\llvert E \bigl[P_{{\omega}}^{z}
(X_{T_{N^{2}+V}}=v, T_{N^{2}+V}=m, B_{N} )\mid\CG\bigr]
\\
&&\quad{} -
\mathbb{P}^{z} (X_{T_{N^{2}+V}}=v, T_{N^{2}+V}=m,
B_{N} )\bigr\rrvert>N^{-d}V^{\vfrac{1-d}{6}} \bigr)
\\
&&\qquad \leq P \bigl(\bigl\llvert E \bigl[P_{{\omega}}^{z}
(X_{T_{N^{2}+V}}=v, T_{N^{2}+V}=m, B_{N} )\mid\CG\bigr]
\\
&&\quad\qquad{}-
\mathbb{P}^{z} (X_{T_{N^{2}+V}}=v, T_{N^{2}+V}=m,
B_{N} )\bigr\rrvert>N^{-d}V^{\vfrac{1-d}{6}},
D^{ (\theta,h )} (N ) \bigr)
\\
&&\quad\qquad{} +N^{-\xi(1 )}
\\
&&\qquad \leq P \bigl(\bigl\llvert E \bigl[P_{{\omega}}^{z}
(X_{T_{N^{2}+V}}=v, T_{N^{2}+V}=m, B_{N} )\mid\CG\bigr]
\\
&&\quad\qquad{}-
\mathbb{P}^{z} (X_{T_{N^{2}+V}}=v, T_{N^{2}+V}=m,
B_{N} )\bigr\rrvert>N^{-d}V^{\vfrac{1-d}{6}},
\\
&&\hspace*{3pt}\quad\qquad U\leq
CN^{-2d-\sklvfrac{d-1}{6}+\varepsilon} \bigr)+N^{-\xi(1 )}
\\
&&\qquad \leq C\exp\bigl(-cN^{-\sklvfrac{d+1}{6} (1-2\eta
)-\varepsilon} \bigr)+N^{-\xi(1 )}=N^{-\xi
(1 )}.
\end{eqnarray*}
Since $P (B_{N} )=1-N^{-\xi(1 )}$ by Corollary
\ref{corlengthofregenerations}, this completes the proof.
\end{pf}
We are finally ready to prove Proposition~\ref{propTimed-1boxdifference}.
\begin{pf*}{Proof of Proposition~\ref{propTimed-1boxdifference}}
Let $\eta>0$ and define $V= [N^{\eta} ]$. By Lem\-ma~\ref
{lemEstimationforfarhyperplane},
we know that $P (R (N,\eta) )=1-N^{-\xi
(1 )}$.
As before, all we need to show is that $R (N,\eta)\cap
S (N,\eta)\subseteq F (N,\theta)$
for an\vspace*{1pt} appropriate choice of $\eta>0$ and some event $S (N,\eta
)$
satisfying $P (S (N,\eta) )=1-N^{-\xi
(1 )}$.
This is done identically as in the last step of the proof of Lemma
\ref{lemFirstdifferenceestimation}. Let ${\omega}\in R (N,\eta
)$,
let $\Delta$ be a cube of side length $N^{\theta}$ which is contained
in $\partial^{+}\CP(0,N )$ and let $I$ be an interval
of length $N^{\theta}$ in $\mathbb{N}$. As in Lemma~\ref
{lemFirstdifferenceestimation}
(see also Figure~\ref{figBixsmallbox}), we denote by $c (\Delta
)$
and $c (I )$ the center of $\Delta$ and $I$, respectively,
and let $c' (\Delta)=c (\Delta)+V\frac
{\vartheta}{ \langle\vartheta,e_{1} \rangle}$,
$c' (I )=c (I )+\mathbb{E}^{0} [T_{V} ]$.

Let $\Delta^{ (1 )}$ and $\Delta^{ (2 )}$ be
$ (d-1 )$-dimensional
cubes that are contained in $H_{N^{2}+V}$, centered at $c'
(\Delta)$
and are of side lengths $N^{\theta}-R_{3} (N )\sqrt{V}$
and $N^{\theta}+R_{3} (N )\sqrt{V}$, respectively. In a similar
fashion, let $I^{ (1 )}$ and $I^{ (2 )}$ be intervals
centered at $c' (I )$ which are of lengths $N^{\theta
}-R_{3} (N )\sqrt{V}$
and $N^{\theta}+R_{3} (N )\sqrt{V}$, respectively.

As in the proof of Lemma~\ref{lemFirstdifferenceestimation} (the
proof can be found in Appendix~\ref{AppendixsubMore-annealed-estimations}),
we know that
%
\begin{eqnarray}
&& \mathbb{P}^{z} \bigl(X_{T_{N^{2}+V}}\in\Delta^{ (1 )},
T_{N^{2}V}\in I^{ (1 )} \bigr)
\nonumber\\[-8pt]\label{eqTimed-1boxdifferenceannealedestimation1} \\[-8pt]\nonumber
&&\qquad < \mathbb{P}^{z}
(X_{T_{N^{2}}}\in\Delta, T_{N^{2}}\in I )+N^{-\xi
(1 )},
\\
%
&& \mathbb{P}^{z} \bigl(X_{T_{N^{2}+V}}\in\Delta^{ (2 )},
T_{N^{2}+V}\in I^{ (2 )} \bigr)
\nonumber\\[-8pt]\label{eqTimed-1boxdifferenceannealedestimation2} \\[-8pt]\nonumber
&&\qquad > \mathbb{P}^{z}
(X_{T_{N^{2}}}\in\Delta, T_{N^{2}}\in I )-N^{-\xi
(1 )}
\end{eqnarray}
and using Claim~\ref{Clmsimpleclaimannealedtoquenched} for an
event $S (N,\eta)$ such that $P (S (N,\eta
) )=1-N^{-\xi(1 )}$
we have
%
\begin{eqnarray}
&& E \bigl[P_{{\omega}}^{z} \bigl(X_{T_{N^{2}+V}}\in
\Delta^{
(1 )}, T_{N^{2}+V}\in I^{ (1 )} \bigr)\mid
\mathcal{G} \bigr]
\nonumber\\[-8pt]\label{eqTimed-1boxdifferenceannealedestimation3} \\[-8pt]\nonumber
&&\qquad <P_{\omega}^{z} (X_{T_{N^{2}}}\in\Delta,
T_{N^{2}}\in I )+N^{-\xi(1 )},
\\
%
&& E \bigl[P_{{\omega}}^{z} \bigl(X_{T_{N^{2}+V}}\in
\Delta^{
(2 )}, T_{N^{2}+V}\in I^{ (2 )} \bigr)\mid
\mathcal{G} \bigr]
\nonumber\\[-8pt]\label{eqTimed-1boxdifferenceannealedestimation4} \\[-8pt]\nonumber
&&\qquad >P_{\omega}^{z} (X_{T_{N^{2}}}\in\Delta,
T_{N^{2}}\in I )-N^{-\xi(1 )}.
\end{eqnarray}
In addition, on the event $R (N,\eta)$, for $i=1,2$
\begin{eqnarray*}
&& \bigl\llvert E \bigl[P_{{\omega}}^{z} \bigl(X_{T_{N^{2}+V}}\in
\Delta^{ (i )}, T_{N^{2}+V}\in I^{ (i )} \bigr)\mid\CG
\bigr]-\mathbb{P}^{z} \bigl(X_{T_{N^{2}+V}}\in\Delta^{ (i )},
T_{N^{2}+V}\in I^{ (i )} \bigr)\bigr\rrvert
\\
&&\qquad \leq\bigl\llvert
\Delta^{ (i )}\bigr\rrvert\cdot\bigl\llvert I^{ (i )}\bigr\rrvert
\cdot N^{-d}V^{\vfrac{1-d}{6}}.
\end{eqnarray*}
Therefore, for ${\omega}\in R (N,\eta)\cap S (N,\eta
)$
we have
\begin{eqnarray*}
&& \bigl\llvert P_{{\omega}}^{z} (X_{T_{\partial\CP(0,N
)}}\in\Delta,
T_{\partial\CP(0,N )}\in I )-\mathbb{P}^{z} (X_{T_{\partial\CP(0,N
)}}\in\Delta,
T_{\partial\CP(0,N )}\in I )\bigr\rrvert
\\
&&\qquad \leq \bigl(\bigl\llvert\Delta^{ (1 )}\bigr\rrvert\bigl\llvert
I^{
(1 )}\bigr\rrvert+\bigl\llvert\Delta^{ (2 )}\bigr\rrvert
\bigl\llvert I^{ (2 )}\bigr\rrvert\bigr)N^{-d}V^{\vfrac{1-d}{6}}
\\
&&\quad\qquad{} +
\bigl(\bigl\llvert\Delta^{ (2 )}\bigr\rrvert\bigl\llvert
I^{ (2
)}\bigr\rrvert-\bigl\llvert\Delta^{ (1 )}\bigr\rrvert
\bigl\llvert I^{
(1 )}\bigr\rrvert\bigr)CN^{-d}+N^{-\xi(1 )}
\\
&&\qquad \leq C \bigl[ \bigl(N^{\theta}+R_{3} (N )\sqrt{V}
\bigr)^{d}N^{-d}V^{\vfrac{1-d}{6}}+R_{3} (N )
\sqrt{V}N^{\theta
(d-1 )-d} \bigr].
\end{eqnarray*}
Taking $\eta<2\theta$ we can bound the last term by
\[
C \bigl[N^{\theta d-d+\eta\sklvfrac{1-d}{6}}+R_{3} (N )N^{\theta(d-1
)-d+\sfrac{\eta}{2}} \bigr].
\]
Notice that the exponents of the powers of $N$ are the same when
$\eta=\frac{6\theta}{d+2}<2\theta$, in which case the last bound
equals $C (1+R_{3} (N ) )\cdot N^{-d
(1-\theta)-\sklvafrac{d-1}{d+2}\theta}\leq CN^{-d (1-\theta
)-\sklvafrac{d-2}{d+2}\theta}$. Thus, the proof is complete.
\end{pf*}

\section{From $(d-1)$-dimensional boxes and time intervals to $d$-dimensional boxes in a fixed time}\label{sec4}
\label{secLorentztrans}

The goal of this section is to use the estimation proved in
Section~\ref{secAdding-time-estimation}, for the difference between the
quenched and annealed probabilities to hit boxes in a hyperplane within
a time interval, in order to achieve similar estimation for the difference
between the quenched and annealed probabilities to hit a $d$-dimensional
box in a specific time. Formally, we have the following.

\begin{prop}\label{PropLorentztransformation-1} Let $d\geq4$ and assume $P$
is uniformly elliptic, i.i.d. and satisfies $ (\mathscr{P} )$.
For every $0<\theta\leq\frac{1}{2}$, let $H (N )=H
(N,\theta)$
be the event that for every $z\in\tilde{\mathcal{P}} (0,N )$
and every $d$-dimensional cube $\Delta$ of side length $N^{\theta}$
\[
\bigl\llvert P_{\omega}^{z} (X_{N}\in\Delta)-
\mathbb{P}^{z} (X_{N}\in\Delta)\bigr\rrvert\leq
CN^{-d
(1-\theta)-\sklfrac{1}{3}\theta}.
\]
Then $P (H (N ) )=1-N^{-\xi(1 )}$.
\end{prop}
\begin{rem}
The constant $\frac{1}{3}$ can in fact be replaced by any number
which is smaller than $\min\{ \frac{1}{2},\frac
{d-2}{d+2} \} $.
\end{rem}

The idea of the proof is to exploit the estimation of Proposition
\ref{propTimed-1boxdifference} and the fact that regeneration
times occur quite often. More precisely, we show that the event of
hitting a box $\Delta$ at time $N$ is bounded both from below and
from above by the event of hitting a certain hyperplane in a specific
$ (d-1 )$-dimensional box within a specific time interval.
This implies that the difference between the probabilities is roughly
the same as in Proposition~\ref{propTimed-1boxdifference}, and
thus gives the required result.

\begin{pf*}{Proof of Proposition \ref{PropLorentztransformation-1}}
Due to Lemma~\ref{lemgoodestimationforthelocation}, we may restrict
ourselves to boxes $\Delta$ whose center $c (\Delta)$
satisfies $\llVert c (\Delta)-\mathbb{E}^{z}
[X_{N} ]\rrVert _{\infty}<\sqrt{N}R_{5} (N )$.
Given a cube $\Delta$ of side length $N^{\theta}$ such that $c
(\Delta)$
satisfies $\llVert c (\Delta)-\mathbb{E}^{z}
[X_{N} ]\rrVert _{\infty}<\sqrt{N}R_{5} (N )$
let $\Delta^{ (1 )}$ and $\Delta^{ (2 )}$ be the
$ (d-1 )$-dimensional cubes in the hyperplane $H_{
\langle c (\Delta),e_{1} \rangle-N^{\theta}}$
with center $c (\Delta)-\frac{N^{\theta}}{ \langle
\vartheta,e_{1} \rangle}\vartheta$
and side length $N^{\theta}-R_{5} (N )N^{\theta/2}$
and $N^{\theta}+R_{5} (N )N^{\theta/2}$, respectively.
Noting that
\begin{eqnarray*}
L &:=& \bigl\langle c (\Delta),e_{1} \bigr\rangle-N^{\theta
}\geq
\bigl\langle\mathbb{E}^{z} [X_{N} ],e_{1} \bigr
\rangle-2\sqrt{N}R_{5} (N )
\\
&\geq& cN-2\sqrt{N}R_{5} (N ),
\end{eqnarray*}
it follows from Proposition~\ref{propTimed-1boxdifference}
[for every $\omega\in F (N,\theta)$ and every $z\in
\tilde{\mathcal{P}} (0,N )$]
that
%
\begin{eqnarray}\label{eqLorentztrans1}
&& \bigl\llvert P_{\omega}^{z} \bigl(X_{T_{\partial\mathcal{P} (0,\sqrt
{L} )}}\in
\Delta^{ (1 )}, T_{\partial\mathcal
{P} (0,\sqrt{L} )}\in I^{ (1 )} \bigr)\nonumber
\\
&&\quad{}  -\mathbb
{P}^{z} \bigl(X_{T_{\partial\mathcal{P} (0,\sqrt{L}
)}}\in\Delta^{ (1 )},
T_{\partial\mathcal{P}
(0,\sqrt{L} )}\in I^{ (1 )} \bigr)\bigr\rrvert
\\
&&\qquad \leq
CN^{-d (1-\theta)-\sklvafrac{d-2}{d+2}\theta},\nonumber
\end{eqnarray}
with $I^{ (1 )}=N-\frac{N^{\theta}}{ \langle\bbbv,e_{1} \rangle}+ [-\frac{1}{2} (N^{\theta
}-R_{5} (N )N^{\theta/2} ),\frac
{1}{2} (N^{\theta}-R_{5} (N )N^{\theta/2} ) ]$,
and
%
\begin{eqnarray}\label{eqLorentztrans2}
&& \bigl\llvert P_{\omega}^{z} \bigl(X_{T_{\partial\mathcal{P} (0,\sqrt
{L} )}}\in
\Delta^{ (2 )}, T_{\partial\mathcal
{P} (0,\sqrt{L} )}\in I^{ (2 )} \bigr)\nonumber
\\
&&\quad{} -\mathbb
{P}^{z} \bigl(X_{T_{\partial\mathcal{P} (0,\sqrt{L}
)}}\in\Delta^{ (2 )},
T_{\partial\mathcal{P}
(0,\sqrt{L} )}\in I^{ (2 )} \bigr)\bigr\rrvert
\\
&&\qquad \leq
CN^{-d (1-\theta)-\sklvafrac{d-2}{d+2}\theta},\nonumber
\end{eqnarray}
with $I^{ (2 )}=N-\frac{N^{\theta}}{ \langle\bbbv,e_{1} \rangle}+ [-\frac{1}{2} (N^{\theta
}+R_{5} (N )N^{\theta/2} ),\frac
{1}{2} (N^{\theta}+R_{5} (N )N^{\theta/2} ) ]$.

In addition, by a standard CLT type arguments, as the one in Lemma
\ref{lemFirstdifferenceestimation} (see Appendix~\ref
{AppendixsubMore-annealed-estimations}
for the proof), we have the following annealed estimations:
%
\begin{eqnarray}
\mathbb{P}^{z} \bigl(X_{T_{\partial\mathcal{P} (0,\sqrt
{L} )}}\in\Delta^{ (1 )},
T_{\partial\mathcal
{P} (0,\sqrt{L} )}\in I^{ (1 )}, X_{N}\notin\Delta\bigr)&\leq&
N^{-\xi(1 )},\label{eqLorentztrans3}
\\
%
\mathbb{P}^{z} \bigl( \bigl(X_{T_{\partial\mathcal{P} (0,\sqrt
{L} )}}\in\Delta^{ (2 )}, T_{\partial\mathcal
{P} (0,\sqrt{L} )}\in I^{ (2 )} \bigr)^{c}, X_{N}
\in\Delta\bigr)&\leq& N^{-\xi(1 )}.\label{eqLorentztrans4}
\end{eqnarray}
Using Claim~\ref{Clmsimpleclaimannealedtoquenched} again, this
also implies that
%
\begin{eqnarray}\label{eqLorentztrans5}
&& P \bigl( \bigl\{ \omega\in\Omega: P_{\omega}^{z}
\bigl(X_{T_{\partial\mathcal{P} (0,\sqrt{L} )}}\in\Delta^{ (1 )},
T_{\partial\mathcal{P} (0,\sqrt
{L} )}\in
I^{ (1 )}, X_{N}\notin\Delta\bigr)\leq N^{-\xi(1 )} \bigr
\} \bigr)
\nonumber\\[-8pt]\\[-8pt]\nonumber
&&\qquad \geq1-N^{-\xi(1 )}
\end{eqnarray}
and
%
\begin{eqnarray}\label{eqLorentztrans6}
&& P \bigl( \bigl\{ \omega\in\Omega: P_{\omega}^{z} \bigl(
\bigl(X_{T_{\partial\mathcal{P} (0,\sqrt{L} )}}\in\Delta^{ (2 )},
T_{\partial\mathcal{P} (0,\sqrt
{L} )}\in
I^{ (2 )} \bigr)^{c}, X_{N}\in\Delta\bigr)\leq
N^{-\xi(1 )} \bigr\} \bigr)\hspace*{-20pt}
\nonumber\\[-8pt]\\[-8pt]\nonumber
&&\qquad \geq1-N^{-\xi
(1 )}.
\end{eqnarray}
Combining all of the above, we get for an event with $P$ probability
$\geq1-N^{-\xi(1 )}$ that
\begin{eqnarray*}
&& P_{\omega}^{z} (X_{N}\in\Delta)
\\
&&\hspace*{-3.8pt}\qquad \overset{\fontas\eqref{eqLorentztrans6}} {\leq}  P_{\omega}^{z}
\bigl(X_{T_{\partial\mathcal{P} (0,\sqrt{L} )}}\in\Delta^{ (2 )},
T_{\partial\mathcal{P} (0,\sqrt
{L} )}\in
I^{ (2 )}, X_{N}\in\Delta\bigr)+N^{-\xi(1 )}
\\
&&\qquad\leq P_{\omega}^{z} \bigl(X_{T_{\partial\mathcal{P} (0,\sqrt
{L} )}}\in
\Delta^{ (2 )}, T_{\partial\mathcal
{P} (0,\sqrt{L} )}\in I^{ (2 )}
\bigr)+N^{-\xi
(1 )}
\\
&&\hspace*{-3.8pt}\qquad \overset{\fontas\eqref{eqLorentztrans2}} {\leq}  \mathbb{P}^{z}
\bigl(X_{T_{\partial\mathcal{P} (0,\sqrt{L} )}}\in\Delta^{ (2 )},
T_{\partial\mathcal{P} (0,\sqrt
{L} )}\in
I^{ (2 )} \bigr)+CN^{-d (1-\theta )-\sklvafrac{d-2}{d+2}\theta}
\\
&&\quad\qquad{} +N^{-\xi(1 )}
\\
&&\qquad \leq \mathbb{P}^{z} \bigl(X_{T_{\partial\mathcal{P} (0,\sqrt
{L} )}}\in\Delta^{ (1 )}, T_{\partial\mathcal
{P} (0,\sqrt{L} )}\in I^{ (1 )} \bigr)
\\
&&\quad\qquad{} +\mathbb{P}^{z}
\bigl(X_{T_{\partial\mathcal{P} (0,\sqrt{L}
)}}\in\Delta^{ (2 )}\setminus\Delta^{ (1 )},
T_{\partial\mathcal{P} (0,\sqrt{L} )}\in I^{
(2 )} \bigr)
\\
&&\quad\qquad{}+  \mathbb{P}^{z} \bigl(X_{T_{\partial\mathcal{P} (0,\sqrt
{L} )}}\in\Delta^{ (2 )},
T_{\partial\mathcal
{P} (0,\sqrt{L} )}\in I^{ (2 )}\setminus I^{ (1 )}
\bigr)
\\
&&\quad\qquad{}+CN^{-d (1-\theta)-\sklvafrac{d-2}{d+2}\theta}
\\
&&\hspace*{-3.8pt}\qquad \overset{\fontas\eqref{eqLorentztrans3}} {\leq}  \mathbb{P}^{z}
\bigl(X_{T_{\partial\mathcal{P} (0,\sqrt{L} )}}\in\Delta^{ (1 )},
T_{\partial\mathcal{P} (0,\sqrt
{L} )}\in
I^{ (1 )}, X_{N}\in\Delta\bigr)
\\
&&\quad\qquad{}+\mathbb{P}^{z}
\bigl(X_{T_{\partial\mathcal{P} (0,\sqrt
{L} )}}\in\Delta^{ (2 )}\setminus\Delta^{
(1 )},
T_{\partial\mathcal{P} (0,\sqrt{L} )}\in I^{ (2 )} \bigr)
\\
&&\quad\qquad{}+  \mathbb{P}^{z} \bigl(X_{T_{\partial\mathcal{P} (0,\sqrt
{L} )}}\in\Delta^{ (2 )},
T_{\partial\mathcal
{P} (0,\sqrt{L} )}\in I^{ (2 )}\setminus I^{ (1 )}
\bigr)
\\
&&\quad\qquad{}+CN^{-d (1-\theta)-\sklvafrac{d-2}{d+2}\theta}
\\
&&\qquad\leq \mathbb{P}^{z} (X_{N}\in\Delta)+\mathbb
{P}^{z} \bigl(X_{T_{\partial\mathcal{P} (0,\sqrt{L}
)}}\in\Delta^{ (2 )}\setminus
\Delta^{ (1 )}, T_{\partial\mathcal{P} (0,\sqrt{L} )}\in I^{
(2 )} \bigr)
\\
&&\quad\qquad{}+  \mathbb{P}^{z} \bigl(X_{T_{\partial\mathcal{P} (0,\sqrt
{L} )}}\in\Delta^{ (2 )},
T_{\partial\mathcal
{P} (0,\sqrt{L} )}\in I^{ (2 )}\setminus I^{ (1 )}
\bigr)
\\
&&\quad\qquad{}+CN^{-d (1-\theta)-\sklvafrac{d-2}{d+2}\theta}
\\
&&\hspace*{-1.8pt}\qquad\overset{(1)} {\leq}  \mathbb{P}^{z} (X_{N}\in\Delta
)+CR_{5} (N )\cdot N^{\theta(d-\sfrac{1}{2}
)}\cdot N^{-d}+CN^{-d (1-\theta)-\sklvafrac{d-2}{d+2}\theta}
\\
&&\qquad=  \mathbb{P}^{z} (X_{N}\in\Delta)+CN^{-d
(1-\theta)-\sklfrac{1}{3}\theta},
\end{eqnarray*}
whereas for $ (1 )$, we used the annealed derivative estimation
proved in Lem\-ma~\ref{lemAnnealedderivativeestimationsd-1+time}
as well as the fact that the number of pairs $ (x,t )$ such
that $ (x,t )\in(\Delta^{ (2 )}\setminus
\Delta^{ (1 )},I^{ (2 )} )$
or $ (x,t )\in(\Delta^{ (2 )},I^{
(2 )}\setminus I^{ (1 )} )$
is bounded by $CR_{5} (N )N^{\theta(d-\sfrac
{1}{2} )}$.
The other direction
\[
P_{\omega}^{z} (X_{N}\in\Delta)\geq
\mathbb{P}^{z} (X_{N}\in\Delta)-CN^{-d (1-\theta)-\sklfrac
{1}{3}\theta}
\]
follows via the same argument except we use (\ref{eqLorentztrans1}),
(\ref{eqLorentztrans4})
and (\ref{eqLorentztrans5}) instead of (\ref{eqLorentztrans2}),
(\ref{eqLorentztrans3}) and (\ref{eqLorentztrans6}).
\end{pf*}

\section{Total variation bound for finite boxes}\label{sec5}

In the previous section, it was shown that for every $0<\theta\leq1$
the difference between the quenched and annealed probabilities to
hit a $d$-dimensional box of side length $N^{\theta}$ at time $N$
is bounded by $CN^{-d (1-\theta)-\sklvafrac{d-2}{d+2}\theta}$.
Since a must be inside a the box of side length $2N+1$ around its
starting point at time $N$, it in particular implies that the total
variation between the quenched and annealed distribution over any
partition of $\mathbb{Z}^{d}$ into $d$-dimensional boxes of side
length $N^{\theta}$ goes to zero as $N$ goes to infinity. The goal
of this section is to strengthen this result and prove that the same
result hold for partitions of $\mathbb{Z}^{d}$ into boxes whose side
length is of constant size, independent of~$N$. More formally, we
have the following.

\begin{teo}
\label{teoTotalvariationonsmallboxes} Let $d\geq4$ and assume
$P$ is uniformly elliptic, i.i.d. and satisfies $ (\mathscr
{P} )$.
For $N,M\in\mathbb{N}$ denote by $G (N )=G (N,M )$
the set of environments $\omega\in\Omega$ such that for every $z\in
\mathbb{Z}^{d}$
satisfying $\llVert z\rrVert _{\infty}\leq N$
\[
\sum_{\Delta\in\Pi}\bigl\llvert P_{\omega}^{z}
(X_{N}\in\Delta)-\mathbb{P}^{z} (X_{N}\in\Delta
)\bigr\rrvert\leq\frac{C_{2}}{M^{c_{1}}}+\frac{C_{2}}{N^{c_{1}}},
\]
where $\Pi$ is any partition of $\mathbb{Z}^{d}$ into boxes of side
length $M$. Then for an appropriate $0<c_{1},C_{2}<\infty$, $P
(G (N ) )=1-N^{-\xi(1 )}$.
\end{teo}

The idea of the proof is to shrink the size of the boxes repeatedly,
each time by a constant factor from the previous step. This is done
as follows: first, we fix some factor, say $\theta=\frac{1}{200}$.
Then, in the $k$th step of the process, we let the random walk
run for $N^{1/2^{k}}$ steps and ask for the difference between the
annealed and quenched measures hit the same box of side $N^{\theta/2^{k}}$.
Repeating the last procedure roughly $\log\log N$ times, we get boxes
of constant side length $M$. The idea is to bound the total variation
of the $ (k+1 )$th step of this process by the one of
the $k$th step. Denoting by $\lambda_{k}$ the total variation
of the $k$th step, we show that $\lambda_{k}\leq\lambda
_{k-1}+CN_{k}^{-\alpha}$
for some $C,\alpha>0$. An additional short calculation then yields
the result.

\begin{pf*}{Proof of Theorem \ref{teoTotalvariationonsmallboxes}}
We start by introducing some notation to be used throughout the proof.
Let $\theta=\frac{1}{200}$. For $j\in\mathbb{N}$ denote
$N_{j}= \lfloor N^{\sfrac{1}{2^{j}}} \rfloor$
and let $r (N )= \lceil\log_{2} (\frac{\log
N}{\theta\log M} ) \rceil$
(which is the minimal natural number such that $N_{r (N
)}^{\theta}\leq M$).
Moreover, denote $n_{0}=n-\sum_{j=1}^{r (N )}N_{j}$ and
$n_{k}=\sum_{j=1}^{k}N_{j}$, $\forall 1\leq
k\leq r (N )$.
For $0\leq k\leq r (n )$, let $\Pi_{k}$ be a partition of
$\mathbb{Z}^{d}$ into boxes of side length $ \lfloor
N_{k}^{\theta} \rfloor$.
Finally, for $0\leq k\leq r (N )$ let
\[
\lambda_{k}=\sum_{\Delta\in\Pi_{k}}\bigl\llvert
P_{\omega}^{z} (X_{n_{k}}\in\Delta)-
\mathbb{P}^{z} (X_{n_{k}}\in\Delta)\bigr\rrvert.
\]
Note that in particular $\lambda_{r (N )}$ is the total
variation between the quenched and annealed measures on cubes of side
length $\leq M$ which is the term we wish to bound from above. If
one wish to be slightly more precise, one should replace $N_{r
(N )}$
by~$M$, and thus obtaining total variation for boxes of side length
$M$, this however does not influence the estimates to follow.

As stated before the main idea of the proof is to prove an inequality
of the form
\[
\lambda_{k}\leq\lambda_{k-1}+CN_{k}^{-\alpha}
\qquad\forall1\leq k\leq r (N )
\]
for some $\alpha>0$, which immediately implies $\lambda_{r
(N )}\leq\lambda_{1}+C\sum_{k=1}^{r (N
)}N_{k}^{-\alpha}$.
As it turns out the last term is bounded by $C_{2}M^{-c_{1}}$ for
some constants $0<c_{1},C_{2}<\infty$, while the first term, that is,
$\lambda_{1}$, is bounded (due to Proposition~\ref
{PropLorentztransformation-1})
by $CN^{-\sklvafrac{d-2}{d+2}\theta}$, and the result follows.

We now turn to the estimation of $\lambda_{k}$. By the triangle inequality
and the Markov property of $P_{\omega}$, we have
%
\begin{eqnarray}
\lambda_{k} & =&\sum_{\Delta\in\Pi_{k}}\bigl\llvert
P_{\omega}^{z} (X_{n_{k}}\in\Delta)-
\mathbb{P}^{z} (X_{n_{k}}\in\Delta)\bigr\rrvert
\nonumber
\\
& =&\sum_{\Delta\in\Pi_{k}}\biggl\llvert\sum
_{\Delta'\in\Pi
_{k-1}} \bigl[P_{\omega}^{z}
\bigl(X_{n_{k}}\in\Delta, X_{n_{k-1}}\in\Delta' \bigr)-
\mathbb{P}^{z} \bigl(X_{n_{k}}\in\Delta, X_{n_{k-1}}\in
\Delta' \bigr) \bigr]\biggr\rrvert
\nonumber
\\
& \leq&\sum_{\Delta\in\Pi_{k}}\sum_{\Delta'\in\Pi_{k-1}}
\bigl\llvert P_{\omega}^{z} \bigl(X_{n_{k}}\in\Delta,
X_{n_{k-1}}\in\Delta' \bigr)-\mathbb{P}^{z}
\bigl(X_{n_{k}}\in\Delta, X_{n_{k-1}}\in\Delta' \bigr)
\bigr\rrvert
\nonumber
\\
& \leq&\sum_{\Delta\in\Pi_{k}}\sum_{\Delta'\in\Pi_{k-1}}
\biggl\llvert\sum_{u\in\Delta'}P_{\omega}^{u}
(X_{n_{k}-n_{k-1}}\in\Delta)
\nonumber\\[-8pt]\label{eqlamkestimationpart1} \\[-8pt]\nonumber
&& {} \times \bigl[P_{\omega}^{z}
(X_{n_{k-1}}=u )-\mathbb{P}^{z} \bigl(X_{n_{k-1}}\in
\Delta' \bigr)P_{\omega}^{z} \bigl(X_{n_{k-1}}=u
\mid X_{n_{k-1}}\in\Delta' \bigr) \bigr]\biggr\rrvert
\\
&&{} +\sum_{\Delta\in\Pi_{k}}\sum_{\Delta'\in\Pi_{k-1}}
\biggl\llvert\sum_{u\in\Delta'}\mathbb{P}^{z}
\bigl(X_{n_{k-1}}\in\Delta' \bigr)P_{\omega}^{z}
\bigl(X_{n_{k-1}}=u\mid X_{n_{k-1}}\in\Delta' \bigr)
\nonumber\\[-8pt]\label{eqlamk+estimationpart2}\\[-8pt]\nonumber
&&{}\times
\bigl[P_{\omega}^{u} (X_{n_{k}-n_{k-1}}\in\Delta)-
\mathbb{P}^{u} (X_{n_{k}-n_{k-1}}\in\Delta) \bigr]\biggr\rrvert
\\
&&{} +\sum_{\Delta\in\Pi_{k}}\sum_{\Delta'\in\Pi_{k-1}}
\biggl\llvert\sum_{u\in\Delta'}\mathbb{P}^{u}
(X_{n_{k}-n_{k-1}}\in\Delta)
\nonumber\\[-8pt]\label{eqlamk+estimationpart3}\\[-8pt]\nonumber
&& {}\times \bigl[\mathbb{P}^{z}
\bigl(X_{n_{k-1}}\in\Delta' \bigr)P_{\omega}^{z}
\bigl(X_{n_{k-1}}=u\mid X_{n_{k-1}}\in\Delta' \bigr)-
\mathbb{P}^{z} (X_{n_{k-1}}=u ) \bigr]\biggr\rrvert
\\
&&{} +\sum_{\Delta\in\Pi_{k}}\sum_{\Delta'\in\Pi_{k-1}}
\biggl\llvert\sum_{u\in\Delta'}\mathbb{P}^{u}
(X_{n_{k}-n_{k-1}}\in\Delta)\mathbb{P}^{z} (X_{n_{k-1}}=u )
\nonumber\\[-8pt]\label{eqlamk+estimationpart4}\\[-8pt]\nonumber
&&{} -
\mathbb{P}^{z} \bigl(X_{n_{k}}\in\Delta, X_{n_{k-1}}\in
\Delta' \bigr)\biggr\rrvert.
\end{eqnarray}
We turn to estimate each of the terms (\ref
{eqlamkestimationpart1})--(\ref{eqlamk+estimationpart4})
separately.

For the first term (\ref{eqlamkestimationpart1}), we have
\begin{eqnarray*}
\eqref{eqlamkestimationpart1}
&=& \sum_{\Delta\in\Pi_{k}}\sum_{\Delta'\in\Pi_{k-1}}
\biggl\llvert\sum_{u\in\Delta'}P_{\omega
}^{u}
(X_{n_{k}-n_{k-1}}\in\Delta)
\\
&&{}\times \bigl[P_{\omega
}^{z}
(X_{n_{k-1}}=u )-\mathbb{P}^{z} \bigl(X_{n_{k-1}}\in
\Delta' \bigr)P_{\omega}^{z} \bigl(X_{n_{k-1}}=u
\mid X_{n_{k-1}}\in\Delta' \bigr) \bigr]\biggr\rrvert
\\
&\leq& \sum_{\Delta\in\Pi_{k}}\sum_{\Delta'\in\Pi_{k-1}}
\sum_{u\in\Delta'}P_{\omega}^{u}
(X_{n_{k}-n_{k-1}}\in\Delta)
\\
&&{}\times \bigl\llvert P_{\omega}^{z}
(X_{n_{k-1}}=u )-\mathbb{P}^{z} \bigl(X_{n_{k-1}}\in
\Delta' \bigr)P_{\omega}^{z} \bigl(X_{n_{k-1}}=u
\mid X_{n_{k-1}}\in\Delta' \bigr)\bigr\rrvert
\\
&=& \sum_{\Delta'\in\Pi_{k-1}}\sum_{u\in\Delta'}
\bigl\llvert P_{\omega
}^{z} (X_{n_{k-1}}=u )
\\
&&{}-
\mathbb{P}^{z} \bigl(X_{n_{k-1}}\in\Delta'
\bigr)P_{\omega}^{z} \bigl(X_{n_{k-1}}=u\mid
X_{n_{k-1}}\in\Delta' \bigr)\bigr\rrvert
\\
&=& \sum_{\Delta'\in\Pi_{k-1}}\sum_{u\in\Delta'}P_{\omega
}^{z}
\bigl(X_{n_{k-1}}=u\mid X_{n_{k-1}}\in\Delta' \bigr)
\\
&&{}\times
\bigl\llvert P_{\omega}^{z} \bigl(X_{n_{k-1}}\in
\Delta' \bigr)-\mathbb{P}^{z} \bigl(X_{n_{k-1}}\in
\Delta' \bigr)\bigr\rrvert
\\
&=&\sum_{\Delta'\in\Pi_{k-1}}\bigl\llvert P_{\omega}^{z}
\bigl(X_{n_{k-1}}\in\Delta' \bigr)-\mathbb{P}^{z}
\bigl(X_{n_{k-1}}\in\Delta' \bigr)\bigr\rrvert=
\lambda_{k-1}.
\end{eqnarray*}
For the second term (\ref{eqlamk+estimationpart2}), the triangle
inequality yields
\begin{eqnarray*}
\eqref{eqlamk+estimationpart2} & =&\sum_{\Delta\in\Pi_{k}}\sum
_{\Delta'\in\Pi_{k-1}}\biggl\llvert\sum
_{u\in\Delta'}\mathbb{P}^{z} \bigl(X_{n_{k-1}}\in
\Delta' \bigr)P_{\omega}^{z} \bigl(X_{n_{k-1}}=u
\mid X_{n_{k-1}}\in\Delta' \bigr)
\\
&&{}\times  \bigl[P_{\omega
}^{u}
(X_{n_{k}-n_{k-1}}\in\Delta)-\mathbb{P}^{u} (X_{n_{k}-n_{k-1}}\in
\Delta) \bigr]\biggr\rrvert
\\
& \leq&\sum_{\Delta'\in\Pi_{k-1}}\sum_{u\in\Delta'}
\mathbb{P}^{z} \bigl(X_{n_{k-1}}\in\Delta'
\bigr)P_{\omega}^{z} \bigl(X_{n_{k-1}}=u\mid
X_{n_{k-1}}\in\Delta' \bigr)
\\
&&{}\times \sum_{\Delta\in
\Pi_{k}}
\bigl\llvert P_{\omega}^{u} (X_{n_{k}-n_{k-1}}\in\Delta)-
\mathbb{P}^{u} (X_{n_{k}-n_{k-1}}\in\Delta)\bigr\rrvert.
\end{eqnarray*}
By Lemma~\ref{lemgoodestimationforthelocation}, this can be bounded
by
%
\begin{eqnarray}\label{eqlamkestimationpart22}
&& \sum_{\Delta'\in\Pi_{k-1}}\sum_{u\in\Delta'}
\mathbb{P}^{z} \bigl(X_{n_{k-1}}\in\Delta'
\bigr)P_{\omega}^{z} \bigl(X_{n_{k-1}}=u\mid
X_{n_{k-1}}\in\Delta' \bigr)\nonumber
\\
&&\quad{}\times  \mathop{\sum
_{\Delta\in\Pi_{k}}}_{\operatorname{dist} (\Delta,u )\leq\sqrt
{n_{k}-n_{k-1}}R_{5} (n_{k}-n_{k-1} )}\bigl\llvert P_{\omega
}^{u}
(X_{n_{k}-n_{k-1}}\in\Delta)-\mathbb{P}^{u} (X_{n_{k}-n_{k-1}}\in
\Delta)\bigr\rrvert\hspace*{-20pt}
\\
&&\qquad\quad{} +  (n_{k}-n_{k-1} )^{-\xi(1 )}.\nonumber
\end{eqnarray}

We say that a cube $\Delta'\in\Pi_{k-1}$ is good if for every $u\in
\Delta'$
and every $\Delta\in\Pi_{k}$
\[
\bigl\llvert P_{\omega}^{u} (X_{n_{k}-n_{k-1}}\in\Delta)-
\mathbb{P}^{u} (X_{n_{k}-n_{k-1}}\in\Delta)\bigr\rrvert\leq
CN_{k}^{ (\theta-1 ) (d-1 )-\sklfrac
{1}{3}\theta},
\]
otherwise we say that $\Delta'$ is bad. Note that the condition holds
trivially for all $\Delta$ such that $\operatorname{dist} (u,\Delta
)>n_{k}-n_{k-1}$.
Noting that:
\begin{itemize}
\item  For every $u$ we only need to consider boxes $\Delta$ such that
$\operatorname{dist} (\Delta,u )\leq\sqrt
{n_{k}-n_{k-1}}R_{5} (n_{k}-n_{k-1} )=\sqrt
{N_{k}}R_{5} (N_{k} )$
whose number is bounded by $\frac{N_{k}^{\sfrac{d}{2}}R_{5}^{d}
(N_{k} )}{\llvert \Delta\rrvert }$.
\item We only need to consider boxes $\Delta'$ such that $\operatorname{dist} (z,\Delta' )\leq n_{k-1}$.
\item The event $G_{N}= \{ \mbox{all boxes }\Delta' \mbox{ such
that }\operatorname{dist} (\Delta',z )\leq n_{k-1}\mbox{ are
good} \} $
satisfies $P (G_{N}^{c} )\leq n_{k-1}^{d}\cdot N^{-\xi
(1 )}=N^{-\xi(1 )}$
due to Proposition~\ref{PropLorentztransformation-1}.
\end{itemize}

We conclude that (\ref{eqlamkestimationpart22}) is bounded
by
\begin{eqnarray*}
&& CN_{k}^{ (\theta-1 ) (d-1 )-\sklfrac{1}{3}\theta
}\cdot\frac{N_{k}^{\sfrac{d}{2}}R_{5} (N_{k} )}{
\lfloor N_{k}^{\theta} \rfloor^{d}}+P
\bigl(G_{N}^{c} \bigr)
\\[-1pt]
&&\qquad \leq CN_{k}^{1-\sklfrac{4}{3}\theta-\sfrac{d}{2}}R_{5}^{d}
(N_{k} )+N^{-\xi(1 )}\leq CN_{k}^{1-\sklfrac {4}{3}\theta-\sfrac{d}{2}}R_{5}^{d}
(N_{k} ).
\end{eqnarray*}

Turning to deal with (\ref{eqlamk+estimationpart3}), notice that
\begin{eqnarray*}
\eqref{eqlamk+estimationpart3} & =&\sum_{\Delta\in\Pi_{k}}\sum
_{\Delta'\in\Pi_{k-1}}\biggl\llvert\sum
_{u\in\Delta'}\mathbb{P}^{u} (X_{n_{k}-n_{k-1}}\in\Delta)
\\[-1pt]
&&{}\times  \bigl[\mathbb{P}^{z} \bigl(X_{n_{k-1}}\in\Delta'
\bigr)P_{\omega}^{z} \bigl(X_{n_{k-1}}=u\mid
X_{n_{k-1}}\in\Delta' \bigr)-\mathbb{P}^{z}
(X_{n_{k-1}}=u ) \bigr]\biggr\rrvert
\\[-1pt]
& \leq&\sum_{\Delta\in\Pi_{k}}\sum_{\Delta'\in\Pi_{k-1}}
\mathbb{P}^{z} \bigl(X_{n_{k-1}}\in\Delta' \bigr)
\\[-1pt]
&&{}\times
\Bigl\llvert\max_{u\in\Delta
'}\mathbb{P}^{u}
(X_{n_{k}-n_{k-1}}\in\Delta)-\min_{u\in\Delta'}\mathbb{P}^{u}
(X_{n_{k}-n_{k-1}}\in\Delta)\Bigr\rrvert
\\[-1pt]
& =& \sum_{\Delta'\in\Pi_{k-1}}\mathbb{P}^{z}
\bigl(X_{n_{k-1}}\in\Delta' \bigr)
\\[-1pt]
&&{}\times \sum
_{\Delta\in\Pi_{k}}\Bigl\llvert\max_{u\in\Delta'}
\mathbb{P}^{u} (X_{n_{k}-n_{k-1}}\in\Delta)-\min_{u\in\Delta'}
\mathbb{P}^{u} (X_{n_{k}-n_{k-1}}\in\Delta)\Bigr\rrvert
\\[-1pt]
&\overset{ (1 )} {\leq} & \sum_{\Delta'\in\Pi
_{k-1}}
\mathbb{P}^{z} \bigl(X_{n_{k-1}}\in\Delta' \bigr)
\\[-1pt]
&&{}\times
\mathop{\mathop{\sum_{\Delta\in\Pi_{k}~\mathrm{s.t.}~\exists u\in\Delta
'}}_{\operatorname{dist} (\Delta,\mathbb{E}^{u}
[X_{n_{k}-n_{k-1}} ] )}}_{\leq\sqrt{n_{k}-n_{k-1}}R_{5} (n_{k}-n_{k-1} )}
\Bigl\llvert\max_{u\in\Delta'}\mathbb{P}^{u}
(X_{n_{k}-n_{k-1}}\in\Delta)
\\[-1pt]
&&{}-\min_{u\in\Delta
'}\mathbb{P}^{u}
(X_{n_{k}-n_{k-1}}\in\Delta)\Bigr\rrvert+ (n_{k}-n_{k-1}
)^{-\xi(1 )},
\end{eqnarray*}
where for $ (1 )$ we used Lemma~\ref
{lemgoodestimationforthelocation}.
Due to the annealed derivative estimation from Lemma~\ref
{lemgeneralannealedestimations},
we can bound the last term by
\begin{eqnarray*}
&& \sum_{\Delta'\in\Pi_{k-1}} \mathbb{P}^{z}
\bigl(X_{n_{k-1}}\in\Delta' \bigr)
\underbrace{ \biggl(\frac{dN_{k-1}^{\theta}+\sqrt
{n_{k}-n_{k-1}}R_{5} (n_{k}-n_{k-1} )}{ \lfloor
N_{k}^{\theta} \rfloor} \biggr)^{d}}_{\mathrm{number~of~relevant~boxes}} \cdot
\underbrace{\vphantom{\frac
{C}{ (n_{k} )^{\vfrac{d+1}{2}}}} \bigl\lfloor
N_{k}^{\theta
} \bigr\rfloor^{d}}_{\mathrm{size~of~each~box}}
\\
&&\quad{}\times \underbrace{\frac{C}{
(n_{k}-n_{k-1} )^{\vfrac{d+1}{2}}}}_{\mathrm{derivative~estimation}}\,  +\, (n_{k}-n_{k-1}
)^{-\xi(1 )}
\\
&&\qquad =  \frac{C (dN_{k-1}^{\theta}+\sqrt{N_{k}}R_{5}
(N_{k} ) )^{d}}{N_{k}^{\vfrac{d+1}{2}}}+N_{k}^{-\xi
(1 )}\leq CR_{6}
(N_{k} )N_{k}^{-\sfrac{1}{2}}.
\end{eqnarray*}
Finally, for (\ref{eqlamk+estimationpart4}) we have
\begin{eqnarray*}
\eqref{eqlamk+estimationpart4} & =&\sum_{\Delta\in\Pi_{k}}\sum
_{\Delta'\in\Pi_{k-1}}\biggl\llvert\sum
_{u\in\Delta'}\mathbb{P}^{u} (X_{n_{k}-n_{k-1}}\in\Delta
)\mathbb{P}^{z} (X_{n_{k-1}}=u )
\\
&&{} -\mathbb{P}^{z}
\bigl(X_{n_{k}}\in\Delta, X_{n_{k-1}}\in\Delta' \bigr)
\biggr\rrvert
\\
& =&\sum_{\Delta\in\Pi_{k}}\sum_{\Delta'\in\Pi_{k-1}}
\biggl\llvert\sum_{u\in\Delta'}\mathbb{P}^{z}
(X_{n_{k-1}}=u )
\\
&&{}\times \bigl[\mathbb{P}^{u} (X_{n_{k}-n_{k-1}}\in
\Delta)-\mathbb{P}^{z} (X_{n_{k}}\in\Delta\mid
X_{n_{k-1}}=u ) \bigr]\biggr\rrvert
\\
& \leq&\sum_{\Delta'\in\Pi_{k-1}}\sum_{u\in\Delta'}
\mathbb{P}^{z} (X_{n_{k-1}}=u )
\\
&&{}\times \sum
_{\Delta\in\Pi_{k}}\bigl\llvert\mathbb{P}^{u}
(X_{n_{k}-n_{k-1}}\in\Delta)-\mathbb{P}^{z} (X_{n_{k}}\in
\Delta\mid X_{n_{k-1}}=u )\bigr\rrvert.
\end{eqnarray*}
Notice that under the event $B_{N}$, which by Corollary~\ref
{corlengthofregenerations}
satisfies $P (B_{N} )\geq1-N^{-\xi(1 )}$, the
first regeneration time after hitting $u$ is after no more than $R
(N )$
steps. Therefore, the distance between the regeneration times of both
annealed walks started in $u$ and started in $z$ conditioned to
hit $u$ is at most $2R (N )$ of one another. Using the
annealed derivative estimation from Lemma~\ref{lemgeneralannealedestimations}
for the annealed walks after the regeneration times, we get
\begin{eqnarray*}
\bigl\llvert\mathbb{P}^{u} (X_{n_{k}-n_{k-1}}\in\Delta)-
\mathbb{P}^{z} (X_{n_{k}}\in\Delta\mid X_{n_{k-1}}=u )
\bigr\rrvert&\leq& \frac{CR (N )\cdot N_{k}^{d\theta}}{
(n_{k}-n_{k-1}-R (N ) )^{\vfrac{d+1}{2}}}
\\
&\leq&\frac
{CR (N )\cdot N_{k}^{d\theta}}{N_{k}^{\vfrac{d+1}{2}}},
\end{eqnarray*}
recalling that due to Lemma~\ref{lemgoodestimationforthelocation}
we only need to consider boxes $\Delta$ at distance $\leq\sqrt
{n_{k}-n_{k-1}}R_{5} (n_{k}-n_{k-1} )$
from the annealed expectation $\mathbb{E}^{u}
[X_{n_{k}-n_{k-1}} ]$,
it follows that
\begin{eqnarray*}
\eqref{eqlamk+estimationpart4} & \leq&\sum_{\Delta'\in\Pi
_{k-1}}\sum
_{u\in\Delta'}\mathbb{P}^{z} (X_{n_{k-1}}=u
)\cdot\biggl(\frac{dN_{k-1}^{\theta}+\sqrt{n_{k}-n_{k-1}}R_{5}
(n_{k}-n_{k-1} )}{ \lfloor N_{k}^{\theta} \rfloor
} \biggr)^{d}
\\
&&{}\times \frac{CR (N )\cdot N_{k}^{d\theta
}}{N_{k}^{\vfrac{d+1}{2}}}+N_{k}^{-\xi(1 )}
\\
& =& \biggl(\frac{dN_{k-1}^{\theta}+\sqrt{N_{k}}R_{5}
(N_{k} )}{ \lfloor N_{k}^{\theta} \rfloor} \biggr)^{d}\cdot\frac{CR (N
)\cdot N_{k}^{d\theta
}}{N_{k}^{\vfrac{d+1}{2}}}+N_{k}^{-\xi(1 )}
\\
&\leq & CR_{6} (N )N_{k}^{-\sfrac{1}{2}}.
\end{eqnarray*}
Combining all of the above we conclude that under the event $G_{N}\cap B_{N}$
(whose probability is ${\geq}1-N^{-\xi(1 )}$) for every
$k\geq1$
\begin{eqnarray*}
\lambda_{k} &\leq& \lambda_{k-1}+CN_{k}^{1-\sklfrac{4}{3}\theta-\sfrac
{d}{2}}R_{5}^{d}
(N_{k} )+CN_{k}^{-\sfrac{1}{2}}+CR (N )N_{k}^{-\sfrac{1}{2}}
\\
&\leq& \lambda_{k-1}+CN_{k}^{-\sfrac{1}{2}}R_{5}
(N_{k} )\leq\lambda_{k-1}+CN_{k}^{-\sfrac{1}{3}}.
\end{eqnarray*}
Consequently,
\begin{eqnarray*}
\lambda_{r (N )} & \leq&\lambda_{1}+C\sum
_{k=1}^{r
(N )}N_{k}^{-\sfrac{1}{3}}=
\lambda_{1}+C\sum_{k=1}^{r
(N )}
\frac{1}{ \lfloor N^{\sfrac{1}{2^{k}}} \rfloor
^{\sfrac{1}{3}}}
\\
& \leq&\lambda_{1}+C\sum_{k=1}^{r (N )}N^{-\afrac
{1}{3\cdot2^{k}}}
\leq\lambda_{1}+C\int_{1}^{r (N
)+1}e^{-\afrac{1}{3\cdot2^{t}}\cdot\log N}\,dt
\\
&&\hspace*{-37pt} \overset{u=\afrac{1}{3\cdot2^{t}}\cdot\log N}{=}\lambda_{1}+C\int
_{\alpha_{N}}^{\beta_{N}}\frac{e^{-u}}{-\ln3\cdot u}\,du,
\end{eqnarray*}
where $\alpha_{N}=\frac{1}{6}\cdot\log N$ and $\beta_{N}=\frac
{1}{3\cdot2^{r (N )+1}}\cdot\log N$.
Since for large enough $N$, we have $\beta_{N}\geq\frac{\theta
}{\sqrt[6]{M}}\log N\geq1$
we get
\begin{eqnarray*}
\lambda_{k} &\leq& \lambda_{1}+C\int_{\alpha_{N}}^{\beta
_{N}}-e^{-u}\,du=
\lambda_{1}+C \bigl[e^{-u} \bigr]_{\alpha_{N}}^{\beta
_{N}}
\leq\lambda_{1}+Ce^{-\beta_{N}}
\\
&=& \lambda_{1}+
\frac{C}{N^{\afrac
{1}{3\cdot2^{r (n )+1}}}}
\leq \lambda_{1}+\frac
{C}{M^{\sfrac{1}{6}}}.
\end{eqnarray*}
Finally, recalling the definition $\lambda_{1}$ and the fact that
$n_{0}\geq cN$ it follows from Proposition~\ref{PropLorentztransformation-1},
that $\lambda_{1}\leq CN^{-\sklfrac{1}{3}\theta}$, which completes the
proof.
\end{pf*}

\section{Proof of Theorem \texorpdfstring{\protect\ref{teoAbsolutelycontinuousinvariantmeasure}}{1.10}}\label{sec6}

In this section, we prove our first main result, that is, the existence
of a probability measure on the space of environments, which is equivalent
to the original i.i.d. measure and is invariant with respect to the
point of view of the particle. The proof is divided into two parts.
In the first and main part of the proof the existence of an invariant
measure which is not singular with respect to the original i.i.d.
measure is proved. In the second part, we show that the existence of
such a measure guarantees the existence of an equivalent invariant
measure.

In order to prove the existence of a nonsingular invariant measure,
we exploit the result from the last section which allows us to construct
a coupling of the annealed and quenched law of the walk at time $N$
such that for most environments, that is, with $P$ probability $\geq
1-N^{-\xi(1 )}$,
will keep them at distance at most $M$ of one another with positive
probability independent of $N$. Using the uniform ellipticity, the
last coupling can be strengthen to guarantee the walks will coincide
at time $N$ with positive probability, which again is uniform in
$N$. Defining now, two random environments $\omega_{N},\omega
_{N}^{\prime}$
which are the original environment shifted according to the location
of the annealed and quenched random walks at time $N$, respectively,
we get a coupling of the two such that $\omega_{N}=\omega_{N}^{\prime}$
with positive probability. Taking a Cesaro partial limit of the laws
of $\omega_{N}$ and $\omega_{N}^{\prime}$, we get two probability measures
on environments which are the original i.i.d. measure and an invariant
measure with respect to the point of view of the particle, respectively.
By taking the above coupling to the limit, we can conclude that both
measures will give the same environment with positive probability,
and, therefore, in particular that they are not singular.

In the second part of the proof (see Lemma~\ref
{LemEitherSingularorabsolutelycontinuous}),
we use general properties of probability measures which are invariant
with respect to the point of view of the particle in order to show
that the existence of a nonsingular invariant probability measure
guarantees the existence of an equivalent invariant one. Recently,
we learned that the method of obtaining an absolutely continuous probability
measure from a nonsingular one already appeared in \cite{RA03}, Lemma 5.
For the readers convenience and in order to keep the section self-contained,
we include a proof below.

In Section~\ref{subPropertiesoftheRadonNikodymderivative},
we discuss several properties of the Radon--Nikodym derivative of the
invariant measure with respect to the i.i.d. measure. This includes
estimation on its average on a box as well as the existence of all
of its moments.

\subsection{Existence of an equivalent
measure}\label{sec6.1}
\label{subExistenceofanequivalentmeasure}

\begin{lem}
\label{Lemmathereexistsnonsingular} Let $d\geq4$ and assume
$P$ is uniformly elliptic, i.i.d. and satisfies $ (\mathscr
{P} )$.
Then there exists a measure $Q$ on the space of environments
which is invariant with respect to the point of view of the particle
and is not singular with respect to the original i.i.d. measure $P$.
\end{lem}

\begin{pf}
Fix $0<\ep<1$, a large $M\in\mathbb{N}$ and denote by $K
(N )=K (N,M,\epsilon)$
the set of environments $\omega\in\Omega$ such that
%
\begin{equation}
\sum_{\Delta\in\Pi(M )}\bigl\llvert P_{{\omega}}^{0}
(X_{N}\in\Delta)-\mathbb{P}^{0} (X_{N}\in\Delta
)\bigr\rrvert<\epsilon,\label{eqTotalvariationbox}
\end{equation}
where $\Pi(M )$ is a partition of $\mathbb{Z}^{d}$ into
$d$-dimensional boxes of side length $M$. By Theorem~\ref
{teoTotalvariationonsmallboxes},
for every $\varepsilon>0$ there exists $M\in\mathbb{N}$ (independent
of $N$) such that $P (K (N ) )\geq1-N^{-\xi
(1 )}$.
Equation (\ref{eqTotalvariationbox}) tells us that the total variation
distance of the respective distributions $\mathbb{P}^{0}
(X_{N}\in\cdot)$
and $P_{{\omega}}^{0} (X_{N}\in\cdot)$ on $\Pi
(M )$
is less than $\varepsilon$ and that therefore there exists a coupling
$\tilde{\Theta}_{\omega,N,M}$ on $\Pi(M )\times\Pi
(M )$
of both measures such that $\tilde{\Theta}_{\omega,N,M}
(\Lambda_{\Pi} )>1-\varepsilon$,
where $\Lambda_{\Pi}= \{ (\Delta,\Delta):
\Delta\in\Pi(M ) \} $.

Next, using the last coupling, we show how to construct a new coupling
of $\mathbb{P}^{0} (X_{N}=\cdot)$ and $P_{\omega
}^{0} (X_{N}=\cdot)$
on $\mathbb{Z}^{d}\times\mathbb{Z}^{d}$ which gives a positive (independent
of $N$) probability to the event $\Lambda= \{ (x,x ): x\in\mathbb
{Z}^{d} \} $.
Define $\Theta_{\omega,N}$ on $\mathbb{Z}^{d}\times\mathbb{Z}^{d}$
by
\begin{eqnarray*}
\Theta_{\omega,N} (x,y ) &:=& \sum_{\Delta,\Delta'\in\Pi
(M )}\tilde{
\Theta}_{\omega,N-dM,M} \bigl(\Delta,\Delta' \bigr)
\mathbb{P}^{0} (X_{N}=x\mid X_{N-dM}\in\Delta
)P_{\omega}^{0}
\\
&&{}\times  \bigl(X_{N}=y\mid X_{N-dM}
\in\Delta' \bigr).
\end{eqnarray*}
Note that due to the law of total probability $\Theta_{\omega,N}$
is indeed a coupling of $\mathbb{P}^{0} (X_{N}=\cdot)$
and $P_{\omega}^{0} (X_{N}=\cdot)$.

For $x\in\mathbb{Z}^{d}$, let $\Delta_{x}$ be the unique cube that
contains $x$ in the partition $\Pi(M )$. Since the side
length of each box in the partition $\Pi(M )$ is $M$ it
follows that the random walk can reach from each point in the
box $\Delta_{x}$ to $x$ in less than $dM$ steps. Recalling also
that the law of $P$ is uniformly elliptic with elliptic constant
$\eta$ [see (\ref{eqellipticconstant})] we conclude that
\begin{eqnarray*}
\Theta_{\omega,N} (x,x ) & \geq& \tilde{\Theta}_{\omega,N-dM,M} (
\Delta_{x},\Delta_{x} )\mathbb{P}^{0}
(X_{N}=x\mid X_{N-dM}\in\Delta_{x}
)P_{\omega}^{0}
\\
&&{}\times  (X_{N}=x\mid X_{N-dM}\in
\Delta_{x} )
\\
& \geq& \tilde{\Theta}_{\omega,N-dM,M} (\Delta_{x},\Delta_{x}
)\eta^{2dM}.
\end{eqnarray*}
Summing over $x$, we get
\begin{eqnarray*}
\Theta_{\omega,N} (\Lambda) &=&\sum_{x\in\mathbb
{Z}^{d}}
\Theta_{\omega,N} (x,x )\geq\sum_{x\in\mathbb
{Z}^{d}}\tilde{
\Theta}_{\omega,N,M} (\Delta_{x},\Delta_{x} )
\eta^{2dM}
\\
&=&\sum_{\Delta\in\Pi(M )}\tilde{
\Theta}_{\omega,N,M} (\Delta,\Delta)M^{d}\eta^{2dM}> (1-
\epsilon)M^{d}\eta^{2dM}.
\end{eqnarray*}

The last coupling allows us to construct for every $N$ two probability
measures on $\Omega$ that coincide with positive probability (independent
of $N$). Indeed, for $N\in\mathbb{N}$ let $Q_{N}$ and $P_{N}$
be defined by
\[
P_{N} (A )=E \biggl[\sum_{x\in\mathbb{Z}^{d}}\mathbb
{P}^{0} (X_{N}=x )\ind_{\sigma_{x}\omega\in A} \biggr]
\]
and
\[
Q_{N} (A )=E \biggl[\sum_{x\in
\mathbb{Z}^{d}}P_{\omega}^{0}
(X_{N}=x )\ind_{\sigma
_{x}\omega\in A} \biggr].
\]
Note that for every $N\in\mathbb{N}$ the measure $P_{N}$ is in fact
the i.i.d. measure $P$ since the annealed walk is independent of
the environment distribution. Indeed, for every measurable event
$A\subset\Omega$
\begin{eqnarray*}
P_{N} (A ) & =&E \biggl[\sum_{x\in\mathbb{Z}^{d}}\mathbb
{P}^{0} (X_{N}=x )\ind_{\sigma_{x}\omega\in A} \biggr]=\sum
_{x\in\mathbb{Z}^{d}}\mathbb{P}^{0} (X_{N}=x )E [
\ind_{\sigma_{x}\omega\in A} ]
\\
& =&\sum_{x\in\mathbb{Z}^{d}}\mathbb{P}^{0}
(X_{N}=x )P (\sigma_{-x}A )=\sum
_{x\in\mathbb{Z}^{d}}\mathbb{P}^{0} (X_{T_{N}}=x )P (A )=P
(A ).
\end{eqnarray*}
Also note that using the coupling $\Theta_{\omega,N}$ we have for
every measurable event $A$
\begin{eqnarray*}
\bigl\llvert Q_{N} (A )-P_{N} (A )\bigr\rrvert& =&\biggl
\llvert E \biggl[\sum_{x\in\mathbb{Z}^{d}} \bigl[
\mathbb{P}^{0} (X_{N}=x )-P_{\omega}^{0}
(X_{N}=x ) \bigr]\ind_{\sigma_{x}\omega\in A} \biggr]\biggr\rrvert
\\
& =&\biggl\llvert E \biggl[\sum_{x\in\mathbb{Z}^{d}} \biggl[\sum
_{y\in
\mathbb{Z}^{d}}\Theta_{\omega,N} (x,y )-\sum
_{z\in
\mathbb{Z}^{d}}\Theta_{\omega,N} (z,x ) \biggr]\ind
_{\sigma_{x}\omega\in A} \biggr]\biggr\rrvert
\\
& =&\biggl\llvert E \biggl[\sum_{x\in\mathbb{Z}^{d}} \biggl[\sum
_{y\neq
x}\Theta_{\omega,N} (x,y )-\sum
_{z\neq x}\Theta_{\omega,N} (z,x ) \biggr]
\ind_{\sigma_{x}\omega\in A} \biggr]\biggr\rrvert
\\
& \leq&\max\biggl\{ \sum_{x\in\mathbb{Z}^{d}}\sum
_{y\neq x}\Theta_{\omega,N} (x,y ),\sum
_{x\in\mathbb{Z}^{d}}\sum_{z\neq
x}
\Theta_{\omega,N} (z,x ) \biggr\}
\\
&<& 1- (1-\epsilon)M^{d}
\eta^{2dM}.
\end{eqnarray*}
Let $ \{ n_{k} \} $ be a subsequence such that the weak
limits of the Cesaro sequences $ \{ \frac{1}{n_{k}}\sum
_{N=0}^{n_{k}-1}Q_{N} \} _{k\geq1}$,
$ \{ \frac{1}{n_{k}}\sum_{N=0}^{n_{k}-1}P_{N} \} _{k\geq1}$
and $ \{ \frac{1}{n_{k}}\sum_{N=0}^{n_{k}-1}\Theta_{\omega,N} \} _{k\geq1}$
exists.\vspace*{2pt} Since for~every $N\in\mathbb{N}$ the measure $P_{N}$ equals
$P$ it follows that the limit of $ \{ \frac{1}{n_{k}}\sum
_{N=0}^{n_{k}-1}P_{N} \} _{k\geq1}$
is $P$ as well. Next, notice that the weak limit of $ \{ \frac
{1}{n_{k}}\sum_{N=0}^{n_{k}-1}Q_{N} \} _{k\geq1}$
which we\vspace*{1pt} denote by $Q$ is invariant with respect to the point of
view of the particle [see (\ref{eqTransitionkernel}) and (\ref
{eqinvariantprobabilitymeasure})
for the definition]. Indeed, for every bounded continuous function
$f:\Omega\to\mathbb{R}$
\begin{eqnarray*}
\int_{\Omega}\mathfrak{R}f (\omega)\,dQ (\omega) & =&\lim
_{k\to\infty}\frac{1}{n_{k}}\sum_{N=0}^{n_{k}-1}
\int_{\Omega}\mathfrak{R}f (\omega)\,dQ_{N} (\omega)
\\
&=&
\lim_{k\to\infty}\frac{1}{n_{k}}\sum_{N=0}^{n_{k}-1}
\int_{\Omega}\sum_{e\in\mathcal{E}_{d}}\omega(0,e
)f (\sigma_{e}\omega)\,dQ_{N} (\omega)
\\
& =&\lim_{k\to\infty}\frac{1}{n_{k}}\sum
_{N=0}^{n_{k}-1}\int_{\Omega}f (\omega
)\,dQ_{N+1} (\omega)
\\
&=&\lim_{k\to\infty}\frac{1}{n_{k}}\sum
_{N=1}^{n_{k}}\int_{\Omega}f
(\omega)\,dQ_{N} (\omega)
\\
&=& \int_{\Omega}f (\omega)\,dQ
(\omega),
\end{eqnarray*}
where $\mathfrak{R}$ is as in (\ref{eqTransitionkernel}). Finally,
we show that $Q$ and $P$ are not singular. Using the coupling of
$P_{N}$ and $Q_{N}$, for every event $A\subset\Omega$ we have
\begin{eqnarray*}
\bigl\llvert P (A )-Q (A )\bigr\rrvert& =&\lim_{k\to
\infty}
\frac{1}{n_{k}}\Biggl\llvert\sum_{N=0}^{n_{k}-1}
\bigl(P_{N} (A )-Q_{N} (A ) \bigr)\Biggr\rrvert
\\
& \leq&\lim_{k\to\infty}\frac{1}{n_{k}}\sum
_{N=0}^{n_{k}-1}\bigl\llvert P_{N} (A
)-Q_{N} (A )\bigr\rrvert
\\
&\leq& 1- (1-\epsilon)M^{d}
\eta^{2dM}.
\end{eqnarray*}
Since this holds for all events, it follows that $\llVert P-Q\rrVert
_{\mathrm{TV}}\leq1- (1-\epsilon)M^{d}\eta^{2dM}$,
and thus $P$ and $Q$ are not singular.
\end{pf}

\begin{lem}
\label{LemEitherSingularorabsolutelycontinuous} Assume $P$
is uniformly elliptic and i.i.d. If there exists a probability measure
$Q$ on the space of environments which is invariant with respect
to the point of view of the particle and is not singular with respect
to $P$, then there exists a probability measure $\tilde{Q}$ which
is invariant with respect to the point of view of the particle and
is also equivalent to $P$.
\end{lem}

\begin{pf}
Denote by $Q=Q_{c}+Q_{s}$ the Lebesgue decomposition of $Q$ to an
absolutely continuous part $Q_{c}$ (w.r.t. $P$) and a singular part
$Q_{s}$ (w.r.t. $P$). Let $f=\frac{dQ_{c}}{dP}$ and define $A= \{
\omega\in\Omega: f (\omega)=0 \} $.
From the invariance with respect to the point of view of the particle
and the uniform ellipticity, we have
\[
Q=\sum_{e\in\mathcal{E}_{d}}\omega(e )\sigma_{e}\circ
Q\geq\eta\sum_{e\in\mathcal{E}_{d}}\sigma_{e}\circ Q
\]
and,\vspace*{1pt} therefore, $\sigma_{e}\circ Q\ll Q$ for every $e\in\mathcal{E}_{d}$.
Since in addition we have $ (\sigma_{e}\circ Q )_{s}=\sigma
_{e}\circ Q_{s}$,
$ (\sigma_{e}\circ Q )_{c}=\sigma_{e}\circ Q_{c}$ and
$\frac{d (\sigma_{e}\circ Q )_{c}}{dP} (\cdot
)=f (\sigma_{e} (\cdot) )$
we get that
%
\begin{equation}
Q_{c}=\sum_{e\in\mathcal{E}_{d}}\omega(e )\sigma
_{e}\circ Q_{c}\geq\eta\sum_{e\in\mathcal{E}_{d}}
\sigma_{e}\circ Q_{c}\label{eqequivalentmeasures}
\end{equation}
and thus
\[
f (\omega)\geq\eta\sum_{e\in\mathcal{E}_{d}}f \bigl(
\sigma_{e} (\omega) \bigr).
\]
Consequently, $\omega\in A$ implies $\sigma_{e}\omega\in A$ for
every $e\in\mathcal{E}_{d}$, $P$-a.s.\vadjust{\goodbreak}

In particular, we get that $A$ is $\sigma_{e_{1}}$ invariant and,
therefore, by ergodicity that it is a 0--1 event. This immediately
implies that if $Q$ is not singular with respect to~$P$, that is,
$P (A )\neq1$,
then $P (A )=0$ and thus $P\ll Q_{c}$. Taking $\tilde
{Q}=\frac{Q_{c}}{Q_{c} (\Omega)}$,
we get that $\tilde{Q}$ is equivalent to the i.i.d. measure and is
invariant with respect to the point of view of the particle [by the
first equality in (\ref{eqequivalentmeasures})].
\end{pf}

\begin{rem}
\label{Remthemeasureisinfactequivalent} \hspace*{-1pt}Note that the sequence
of probability measures\break $ \{ \sum_{N=0}^{n-1}Q_{N} \}
_{n\geq1}$
equals to $ \{ \sum_{N=0}^{n-1}\mathfrak{R}^{N}P \}
_{n\geq1}$.
Recalling Theorem~\ref{teoKozlovstheorem} it follows that the
measure $\sum_{N=0}^{n-1}Q_{N}$ converges (without taking a subsequence)
to the equivalent measure $Q$ which is the \textit{unique} probability
measure invariant with respect to the point of view of the particle.
In particular, there is no need to restrict ourselves to the absolutely
continuous part as done in Lemma~\ref{LemEitherSingularorabsolutelycontinuous}.
\end{rem}
\subsection{Some properties
of the Radon--Nikodym derivative}\label{sec6.2}
\label{subPropertiesoftheRadonNikodymderivative}

In this subsection, we discuss some properties of the equivalent probability
measure $Q$ and its Radon--Nikodym derivative. The next definition
will be useful in the statement of the lemmas.

\begin{defn}
Given two environments $\omega,\omega'\in\Omega$ define their distance
by
\[
\operatorname{dist} \bigl(\omega,\omega' \bigr)=\inf\bigl\{ \llVert x
\rrVert_{1}: \omega'=\sigma_{x}\omega\bigr\},
\]
where the infimum over an empty set is defined to be infinity.
\end{defn}

For future use, we denote by $\Psi$ and $\Psi_{N}$ the couplings
of $P$ and $Q$ and of $P_{N}$ and $Q_{N}$, respectively, on $\Omega
\times\Omega$,
that is,
%
\begin{equation}
\Psi_{N} (A )=E \biggl[\sum_{x,y\in\mathbb{Z}^{d}}\Theta
_{\omega,N} (x,y )\ind_{ (\sigma_{x}\omega,\sigma
_{y}\omega)\in A} \biggr],\label{eqDefinitionofPsiN}
\end{equation}
and $\Psi$ is the weak limit of the Cesaro sequence $ \{ \frac
{1}{n}\sum_{N=0}^{n-1}\Psi_{N} \} _{n=1}^{\infty}$
along any converging sub-sequence which we denote from here on by
$ \{ n_{k} \} _{k\geq1}$.

Our main goal is to prove the following concentration inequality for
the average of the Radon--Nikodym derivative on a box.

\begin{lem}
\label{lemRadonNikodymderivativeestimate} Let $M\in\mathbb{N}$
and denote by $\Delta_{0}$ a $d$-dimensional cube of side length
$M$ in $\mathbb{Z}^{d}$. Then for every $\varepsilon>0$,
\[
P \biggl(\biggl\llvert\frac{1}{\llvert \Delta_{0}\rrvert }\sum_{x\in
\Delta
_{0}}
\frac{dQ}{dP} (\sigma_{x}\omega)-1\biggr\rrvert>\varepsilon
\biggr)\leq M^{-\xi(1 )}.
\]
\end{lem}

As a first step toward the proof of Lemma~\ref
{lemRadonNikodymderivativeestimate},
we prove the following.

\begin{lem}
\label{lemPerliminarylemmatoRadonNikodymderivativeestimate}For
$M\in\mathbb{N}$ let $D_{M}^{ (1 )}:\Omega\to
[0,\infty]$
and $D_{M}^{ (2 )}:\Omega\to[0,\infty]$ be defined
by
\[
D_{M}^{ (1 )} (\omega)=E_{\Psi} [\ind
_{\operatorname{dist} (\omega_{1},\omega_{2} )>dM}\mid\mathfrak{F}_{\omega
_{1}} ] (\omega)
\]
and
\[
D_{M}^{ (2 )} (\omega)=E_{\Psi
} [
\ind_{\operatorname{dist} (\omega_{1},\omega_{2}
)>dM}\mid\mathfrak{F}_{\omega_{2}} ] (\omega),
\]
where $\mathfrak{F}_{\omega_{1}}$, $\mathfrak{F}_{\omega_{2}}$ are
the $\sigma$-algebras generated by the first, respectively, second
coordinate in $\Omega\times\Omega$ and $\Psi$ is as defined below
(\ref{eqDefinitionofPsiN}). For every $M\in\mathbb{N}$, there
exists an event $\mathbf{F}_{M}$ with the following properties:
\begin{longlist}[(2)]
\item[(1)]$P (\mathbf{F}_{M} )=1-M^{-\xi(1 )}$.
\item[(2)] For every $\varepsilon>0$, if $M$ is large enough, then
$D_{M}^{ (1 )} (\omega)\leq\varepsilon\ind
_{\mathbf{F}_{M}} (\omega)+\ind_{\mathbf
{F}_{M}^{c}} (\omega)$
and $\frac{dQ}{dP} (\omega)D_{M}^{ (2 )}
(\omega)\leq\varepsilon\ind_{\mathbf{F}_{M}} (\omega
)+\ind_{\mathbf{F}_{M}^{c}} (\omega)$.
\end{longlist}
\end{lem}

\begin{pf}
Let
\begin{eqnarray*}
\mathbf{F}_{M}&=&\bigcap_{k=M}^{\infty}
\biggl\{ \omega\in\Omega: \forall x\in[-k,k ]^{d}\cap
\mathbb{Z}^{d},
\\
&&{} \sum_{\Delta\in\Pi_{M}}\bigl\llvert
\mathbb{P}^{x} (X_{k}\in\Delta)-P_{\omega}^{x}
(X_{k}\in\Delta)\bigr\rrvert\leq\frac{C_{2}}{M^{c_{1}}}+
\frac{C_{2}}{k^{c_{1}}} \biggr\},
\end{eqnarray*}
where $\Pi_{M}$ is a partition of $\mathbb{Z}^{d}$ into boxes of
side length $M$ and $0<c_{1},C_{2}<\infty$ are the constants from
Theorem~\ref{teoTotalvariationonsmallboxes}. Thus, by the same
theorem, we have $P (\mathbf{F}_{M} )=1-M^{-\xi
(1 )}$.
Fix some $\varepsilon>0$. The definition of $\mathbf{F}_{M}$ together
with the definition of the couplings $\tilde{\Theta}_{\omega,k,M}$
constructed in the proof of Lemma~\ref{Lemmathereexistsnonsingular}
implies\vspace*{1pt} that for every $\omega\in\mathbf{F}_{M}$, every $k\geq M$
and every $x\in[-k,k ]^{d}\cap\mathbb{Z}^{d}$ we have
$\tilde{\Theta}_{\sigma_{x}\omega,k,M} (\Lambda_{\Pi
_{M}} )>1-\frac{2C_{2}}{M^{c_{1}}}>1-\varepsilon$
for large enough $M$, where as before $\Lambda_{\Pi_{M}}= \{
(\Delta,\Delta): \Delta\in\Pi_{M} \} $.

Before turning to prove the estimates for $ \{ D_{M}^{
(i )} (\omega) \} _{i\in\{ 1,2 \} }$,
we prove a similar results for the conditional expectations of $\Psi_{N}$.
For\vspace*{2pt} $N,M\in\mathbb{N}$ and $i\in\{ 1,2 \}$, define
$D_{M,N}^{ (i )}:\Omega\to[0,\infty]$
by $D_{M,N}^{ (i )} (\omega)=E_{\Psi_{N}}
[\ind_{\operatorname{dist} (\omega_{1},\omega_{2} )>dM}\mid\mathfrak{F}_{\omega
_{i}} ] (\omega)$.
Note that for $P$-almost every environment $\omega\in\Omega$ we
have
%
\begin{equation}
D_{M,N}^{ (1 )} (\omega)=\sum_{x,y\in\mathbb
{Z}^{d}}
\Theta_{\sigma_{-x}\omega,N} (x,y )\ind_{\llVert x-y\rrVert
_{1}>dM}\label
{eqdefinitionofconditionedprobabilitydistance}
\end{equation}
and
%
\begin{equation}
D_{M,N}^{ (2 )} (\omega)= \biggl(\frac
{dQ_{N}}{dP} (\omega)
\biggr)^{-1}\sum_{x,y\in\mathbb
{Z}^{d}}
\Theta_{\sigma_{-y}\omega,N} (x,y )\ind_{\llVert x-y\rrVert
_{1}>dM}.\label
{eqdefinitionofconditionedprobabilitydistance2}
\end{equation}
Indeed, using (\ref{eqDefinitionofPsiN}) we have for every measurable
event $A\subset\Omega$
\begin{eqnarray*}
&& E_{\Psi_{N}} [\ind_{A\times\Omega}\ind_{\operatorname{dist}
(\omega_{1},\omega_{2} )>dM} ]
\\
&&\qquad  =
\Psi_{N} \bigl(A\times\Omega\cap\bigl\{ (\omega_{1},
\omega_{2} ): \operatorname{dist} (\omega_{1},
\omega_{2} )>dM \bigr\} \bigr)
\\
&&\qquad  = E \biggl[\sum_{x,y\in\mathbb{Z}^{d}}\Theta_{\omega,N} (x,y )
\ind_{ (\sigma_{x}\omega,\sigma_{y}\omega
)\in A\times\Omega}\ind_{\operatorname{dist} (\sigma_{x}\omega,\sigma_{y}\omega
)>dM} \biggr]
\\
&&\qquad  =\sum_{x,y\in\mathbb{Z}^{d}}E \bigl[\Theta_{\omega,N} (x,y )
\ind_{\sigma_{x}\omega\in A}\ind_{\llVert x-y\rrVert _{1}>dM} \bigr],
\end{eqnarray*}
which by translation invariance of $P$ equals
\begin{eqnarray*}
&& \sum_{x,y\in\mathbb{Z}^{d}}E \bigl[\Theta_{\sigma_{-x}\omega,N} (x,y )
\ind_{\omega\in A}\ind_{\llVert x-y\rrVert _{1}>dM} \bigr]
\\
&&\qquad =E \biggl
[\sum
_{x,y\in\mathbb{Z}^{d}}\Theta_{\sigma_{-x}\omega,N} (x,y )\ind_{\omega
\in A}\ind
_{\llVert x-y\rrVert _{1}>dM} \biggr].
\end{eqnarray*}
Due to the fact that the first marginal of $\Psi_{N}$ is $P$ the
last term equals
\[
E_{\Psi_{N}} \biggl[\ind_{ (\omega,\omega' )\in A\times
\Omega}\cdot\sum
_{x,y\in\mathbb{Z}^{d}}\Theta_{\sigma_{-x}\omega,N} (x,y )\ind
_{\llVert x-y\rrVert _{1}>dM}
\biggr],
\]
which by the definition of conditional expectation implies (\ref
{eqdefinitionofconditionedprobabilitydistance}).
A similar argument shows that
\begin{eqnarray*}
\hspace*{-2pt}&& E_{\Psi_{N}} [\ind_{\Omega\times A}\ind_{\operatorname{dist}
(\omega_{1},\omega_{2} )>dM} ]
\\
\hspace*{-2pt}&&\qquad  = E \biggl[\sum
_{x,y\in
\mathbb{Z}^{d}}\Theta_{\sigma_{-x}\omega',N} (x,y )\ind
_{\omega'\in A}\ind_{\llVert x-y\rrVert _{1}>dM} \biggr]
\\
\hspace*{-2pt}&&\qquad  = E_{Q_{N}} \biggl[ \biggl(\frac{dQ_{N}}{dP} \bigl(
\omega' \bigr) \biggr)^{-1}\sum
_{x,y\in\mathbb{Z}^{d}}\Theta_{\sigma_{-x}\omega
',N} (x,y )\ind_{\llVert x-y\rrVert _{1}>dM}\ind
_{\omega'\in A} \biggr]
\\
\hspace*{-2pt}&&\qquad  =E_{\Psi_{N}} \biggl[ \biggl(\frac{dQ_{N}}{dP} (\omega_{2} )
\biggr)^{-1}\sum_{x,y\in\mathbb{Z}^{d}}
\Theta_{\sigma
_{-x}\omega_{2},N} (x,y )\ind_{\llVert x-y\rrVert
_{1}>dM}\cdot\ind_{\Omega\times A} (
\omega_{1},\omega_{2} ) \biggr]
\end{eqnarray*}
and thus that (\ref{eqdefinitionofconditionedprobabilitydistance2})
holds as well.

Since $\Theta_{\sigma_{-x}\omega,N} (x,y )>0$ implies
$x\in[-N,N ]^{d}\cap\mathbb{Z}^{d}$,
it follows that for large enough $M$, every $\omega\in\mathbf{F}_{M}$
and every $N\geq M$
\begin{eqnarray*}
&& \sum_{x,y\in\mathbb{Z}^{d}}\Theta_{\sigma_{-x}\omega,N} (x,y )
\ind_{\llVert x-y\rrVert _{1}>dM}
\\
&&\qquad   = 1-\sum_{x,y\in\mathbb{Z}^{d}}
\Theta_{\sigma_{-x}\omega,N} (x,y )\ind_{\llVert x-y\rrVert _{1}\leq dM}
\\
&&\qquad  \leq 1-\min_{z\in[-N,N ]^{d}\cap\mathbb{Z}^{d}}\sum_{x,y\in\mathbb{Z}^{d}}
\Theta_{\sigma_{-z}\omega,N} (x,y )\ind_{\llVert x-y\rrVert _{1}\leq dM}
\\
&&\qquad  \leq 1-\min_{z\in[-N,N ]^{d}\cap\mathbb{Z}^{d}}\sum_{\Delta\in\Pi
_{M}}\sum
_{x,y\in\Delta}\Theta_{\sigma_{-z}\omega,N} (x,y )
\\
&&\qquad  = 1-\min_{z\in[-N,N ]^{d}\cap\mathbb{Z}^{d}} \biggl(\sum_{\Delta\in\Pi_{M}}
\tilde{\Theta}_{\sigma_{-z}\omega,N,M} (\Delta,\Delta) \biggr)
\\
&&\qquad  = 1-\min_{z\in[-N,N ]^{d}\cap\mathbb{Z}^{d}}\tilde{\Theta}_{\sigma
_{-z}\omega,N,M} (
\Lambda_{\Pi_{M}} )<\varepsilon.
\end{eqnarray*}
Thus,
\[
D_{M,N}^{ (1 )} (\omega)=\sum_{x,y\in\mathbb
{Z}^{d}}
\Theta_{\sigma_{-x}\omega,N} (x,y )\ind_{\llVert x-y\rrVert
_{1}>dM}\leq\varepsilon
\ind_{\mathbf
{F}_{M}} (\omega)+\ind_{\mathbf{F}_{M}^{c}} (\omega)
\]
and similarly,
\[
\frac{dQ_{N}}{dP} (\omega)D_{M,N}^{ (2 )} (\omega)=\sum
_{x,y\in\mathbb{Z}^{d}}\Theta_{\sigma
_{-y}\omega,N} (x,y )
\ind_{\llVert x-y\rrVert
_{1}>dM}\leq\varepsilon\ind_{\mathbf{F}_{M}} (\omega)+
\ind_{\mathbf{F}_{M}^{c}} (\omega).
\]

Next, we turn to prove the estimate for $ \{ D_{M}^{ (i
)} \} _{i\in\{ 1,2 \} }$.
It is enough to show that along some sub-sequence of $ \{
n_{k} \} _{k\geq1}$
(which for simplicity we still denote by $ \{ n_{k} \}
_{k\geq1}$)
%
\begin{eqnarray}\label{eqConvergenceoftheDM}
D_{M}^{ (1 )} (\omega) &=&\lim_{k\to\infty
}
\frac{1}{n_{k}}\sum_{N=0}^{n_{k}-1}D_{M,N}^{ (1 )}
(\omega)\quad\mbox{and}\quad
\nonumber\\[-8pt]\\[-8pt]\nonumber
D_{M}^{ (2 )} (\omega) &=&  \biggl(
\frac{dQ}{dP} (\omega) \biggr)^{-1}\lim_{k\to\infty}
\frac{1}{n_{k}}\sum_{N=0}^{n_{k}-1}
\frac
{dQ_{N}}{dP} (\omega)D_{M,N}^{ (2 )} (\omega),\qquad P\mbox{-a.s.}
\end{eqnarray}
Indeed, if (\ref{eqConvergenceoftheDM}) holds, then for $P$-almost
every $\omega$ we have
\begin{eqnarray*}
D_{M}^{ (1 )} (\omega) & =&\lim_{k\to\infty
}
\frac{1}{n_{k}}\sum_{N=0}^{n_{k}-1}D_{M,N}^{ (1 )}
(\omega)
\\
&=&\lim_{k\to\infty}\frac{1}{n_{k}} \Biggl[\sum
_{N=0}^{M-1}D_{M,N}^{ (1 )} (\omega
)+\sum_{N=M}^{n_{k}-1}D_{M,N}^{ (1 )}
(\omega) \Biggr]
\\
& \leq&\lim_{k\to\infty}\frac{1}{n_{k}} \Biggl[M+\sum
_{N=M}^{n_{k}-1}D_{M,N}^{ (1 )} (\omega
) \Biggr]
\\
&\leq&\lim_{k\to\infty}\frac{1}{n_{k}} \Biggl[M+\sum
_{N=M}^{n_{k}-1} \bigl(\varepsilon
\ind_{\mathbf{F}_{M}} (\omega)+\ind_{\mathbf{F}_{M}^{c}} (\omega)
\bigr) \Biggr]
\\
& =&\varepsilon\ind_{\mathbf{F}_{M}} (\omega)+\ind_{\mathbf
{F}_{M}^{c}} (\omega)
\end{eqnarray*}
and similarly
\begin{eqnarray*}
&& \frac{dQ}{dP} (\omega)D_{M}^{ (2 )} (\omega)
\\
&&\qquad =  \lim_{k\to\infty}\frac{1}{n_{k}}\sum
_{N=0}^{n_{k}-1}\frac
{dQ_{N}}{dP} (\omega
)D_{M,N}^{ (2 )} (\omega)
\\
&&\qquad =\lim_{k\to\infty}
\frac{1}{n_{k}} \Biggl[\sum_{N=0}^{M-1}
\frac{dQ_{N}}{dP} (\omega)D_{M,N}^{
(2 )} (\omega)+\sum
_{N=M}^{n_{k}-1}\frac
{dQ_{N}}{dP} (\omega
)D_{M,N}^{ (2 )} (\omega) \Biggr]
\\
&&\qquad \leq \lim_{k\to\infty}\frac{1}{n_{k}} \Biggl[\sum
_{N=0}^{M-1}\frac{dQ_{N}}{dP} (\omega)+\sum
_{N=M}^{n_{k}-1}\frac{dQ_{N}}{dP}D_{M,N}^{ (1 )}
(\omega) \Biggr]
\\
&&\qquad \leq\lim_{k\to\infty}\frac{1}{n_{k}} \Biggl[\sum
_{N=0}^{M-1}\frac{dQ_{N}}{dP} (\omega)+
\sum_{N=M}^{n_{k}-1} \bigl(\varepsilon
\ind_{\mathbf{F}_{M}} (\omega)+\ind_{\mathbf{F}_{M}^{c}} (\omega)
\bigr) \Biggr]
\\
&&\qquad =  \varepsilon\ind_{\mathbf{F}_{M}} (\omega)+\ind_{\mathbf
{F}_{M}^{c}} (\omega).
\end{eqnarray*}
Turning to prove (\ref{eqConvergenceoftheDM}), for every measurable
event $A\subset\Omega$ we have
\begin{eqnarray*}
&& E \bigl[D_{M}^{ (1 )} (\omega)\ind_{A} (\omega
) \bigr]
\\
&&\qquad  = E_{\Psi} \bigl[\ind_{\operatorname{dist}
(\omega_{1},\omega_{2} )>dM}\cdot\ind_{A\times\Omega} (
\omega_{1},\omega_{2} ) \bigr]
\\
&&\qquad =  \Psi\bigl( \bigl\{ (\omega_{1},\omega_{2} ):
\operatorname{dist} (\omega_{1},\omega_{2} )>dM \bigr\} \cap A\times
\Omega\bigr)
\\
&&\qquad \overset{ (1 )} {=}\lim_{k\to\infty
}\frac{1}{n_{k}}
\sum_{N=0}^{n_{k}-1}\Psi_{N} \bigl(
\bigl\{ (\omega_{1},\omega_{2} ): \operatorname{dist} (\omega
_{1},\omega_{2} )>dM \bigr\} \cap A\times\Omega\bigr)
\\
&&\qquad =  \lim_{k\to\infty}\frac{1}{n_{k}}\sum
_{N=0}^{n_{k}-1}E_{\Psi
_{N}} \bigl[
\ind_{\operatorname{dist} (\omega_{1},\omega_{2}
)>dM}\cdot\ind_{A\times\Omega} (\omega_{1},
\omega_{2} ) \bigr]
\\
&&\qquad \overset{ (2 )} {=} \lim_{k\to\infty}
\frac
{1}{n_{k}}\sum_{N=0}^{n_{k}-1}E_{\Psi_{N}}
\bigl[D_{M,N}^{
(1 )} (\omega_{1} )\cdot
\ind_{A\times\Omega} (\omega_{1},\omega_{2} ) \bigr]
\\
&&\qquad =  \lim_{k\to\infty}\frac{1}{n_{k}}\sum
_{N=0}^{n_{k}-1}E \bigl[D_{M,N}^{ (1 )}
(\omega_{1} )\cdot\ind_{A} (\omega_{1} )
\bigr]
\\
&&\qquad =\lim_{k\to\infty}E \Biggl[\frac{1}{n_{k}}\sum
_{N=0}^{n_{k}-1}D_{M,N}^{ (1 )} (
\omega_{1} )\cdot\ind_{A} (\omega_{1} ) \Biggr],
\end{eqnarray*}
where $ (1 )$ is due to the definition of $\Psi$ below
(\ref{eqDefinitionofPsiN}) and $ (2 )$ uses the definition
of $D_{M,N}^{ (1 )}$ as the conditional expectation. This
implies that $\frac{1}{n_{k}}\sum_{N=0}^{n_{k}-1}D_{M,N}^{
(1 )}$
converges in $L^{1} (P )$ to $D_{M}^{ (1 )}$ and
thus by standard arguments contains a subsequence that converges $P$-almost
surely. Similarly, for $D_{M}^{ (2 )}$
\begin{eqnarray*}
&& E_{Q} \bigl[D_{M}^{ (2 )} (\omega)\ind
_{A} (\omega) \bigr]
\\
&&\qquad = E_{\Psi} \bigl[\ind_{\operatorname{dist} (\omega_{1},\omega_{2} )>dM}
\cdot\ind_{\Omega
\times A} (\omega_{1},\omega_{2} ) \bigr]
\\
&&\qquad =  \Psi\bigl( \bigl\{ (\omega_{1},\omega_{2} ):
\operatorname{dist} (\omega_{1},\omega_{2} )>dM \bigr\} \cap\Omega
\times A \bigr)
\\
&&\qquad = \lim_{k\to\infty}\frac{1}{n_{k}}\sum
_{N=0}^{n_{k}-1}\Psi_{N} \bigl( \bigl\{ (
\omega_{1},\omega_{2} ): \operatorname{dist} (
\omega_{1},\omega_{2} )>dM \bigr\} \cap\Omega\times A \bigr)
\\
&&\qquad =  \lim_{k\to\infty}\frac{1}{n_{k}}\sum
_{N=0}^{n_{k}-1}E_{\Psi
_{N}} \bigl[
\ind_{\operatorname{dist} (\omega_{1},\omega_{2}
)>dM}\cdot\ind_{\Omega\times A} (\omega_{1},
\omega_{2} ) \bigr]
\\
&&\qquad =\lim_{k\to\infty}\frac{1}{n_{k}}\sum
_{N=0}^{n_{k}-1}E_{\Psi_{N}}
\bigl[D_{M,N}^{ (2 )} (\omega_{2} )\cdot
\ind_{\Omega\times A} (\omega_{1},\omega_{2} ) \bigr]
\\
&&\qquad =  \lim_{k\to\infty}\frac{1}{n_{k}}\sum
_{N=0}^{n_{k}-1}E_{Q_{N}} \bigl[D_{M,N}^{ (2 )}
(\omega_{2} )\cdot\ind_{A} (\omega_{2} )
\bigr]
\\
&&\qquad =\lim_{k\to\infty}\frac{1}{n_{k}}\sum
_{N=0}^{n_{k}-1}E_{Q} \biggl[ \biggl(
\frac{dQ}{dP} (\omega_{2} ) \biggr)^{-1}\cdot
\frac
{dQ_{N}}{dP} (\omega_{2} )\cdot D_{M,N}^{ (2
)}
(\omega_{2} )\cdot\ind_{A} (\omega_{2} )
\biggr]
\\
&&\qquad =  \lim_{k\to\infty}E_{Q} \Biggl[ \biggl(
\frac{dQ}{dP} (\omega_{2} ) \biggr)^{-1}\cdot
\frac{1}{n_{k}}\sum_{N=0}^{n_{k}-1}
\frac{dQ_{N}}{dP} (\omega_{2} )\cdot D_{M,N}^{ (2 )}
(\omega_{2} )\cdot\ind_{A} (\omega_{2} )
\Biggr].
\end{eqnarray*}
This proves the second quality in (\ref{eqConvergenceoftheDM}),
$Q$ (and thus $P$)-almost surely for an appropriate sub-sequence.
\end{pf}

\begin{pf*}{Proof of Lemma~\ref{lemRadonNikodymderivativeestimate}}
The proof deals separately with the events $B_{\varepsilon}^{-}\hspace*{-1pt}=\{
\omega\in\Omega: \frac{1}{\llvert \Delta_{0}\rrvert }\sum_{x\in\Delta
_{0}}\frac{dQ}{dP} (\sigma_{x}\omega
)<1-\varepsilon\} $
and $B_{\epsilon}^{+}= \{ \omega\in\Omega: \frac{1}{\llvert \Delta
_{0}\rrvert }\sum_{x\in\Delta_{0}}\frac{dQ}{dP}
(\sigma_{x}\omega)>1+\varepsilon\} $.
We start with the event $B_{\varepsilon}^{-}$. The idea is to separate
the event $B_{\varepsilon}^{-}$ into two events the first with probability
$M^{-\xi(1 )}$ and the second, denoted $S_{\varepsilon}^{-}$,
which will turn out to be with $P$ probability zero measure. To
this end, assume without loss of generality that $\Delta_{0}$ is centered
at the zero, denote $M_{\varepsilon}=\frac{\epsilon}{6d^{2}}M$, define
$\Delta_{0}^{-}= \{ x\in\mathbb{Z}^{d}: \llVert x\rrVert _{\infty
}<M-dM_{\varepsilon} \} $
and let
\[
S_{\varepsilon}^{-}= \bigl\{ \omega\in B_{\varepsilon}^{-}: \sigma
_{x}\omega\in\mathbf{F}_{M_{\varepsilon}}, \forall x\in
\Delta_{0} \bigr\},
\]
where $\mathbf{F}_{M_{\varepsilon}}$ is the event from Lemma~\ref
{lemPerliminarylemmatoRadonNikodymderivativeestimate}.
Due to property (1) of $\mathbf{F}_{M_{\varepsilon}}$ from Lemma~\ref{lemPerliminarylemmatoRadonNikodymderivativeestimate}
\begin{eqnarray*}
P \bigl(S_{\varepsilon}^{-} \bigr)&\geq& P \bigl(B_{\varepsilon
}^{-}
\bigr)-\llvert\Delta_{0}\rrvert P \bigl(\mathbf{F}_{M_{\varepsilon}}^{c}
\bigr)=P \bigl(B_{\varepsilon}^{-} \bigr)-M^{d}\cdot
(M_{\varepsilon} )^{-\xi(1
)}
\\
&=& P \bigl(B_{\varepsilon}^{-}
\bigr)-M^{-\xi(1 )},
\end{eqnarray*}
and, therefore, it is enough to show that $P (S_{\varepsilon
}^{-} )=0$.
We claim that there exists an event $K\subset S_{\varepsilon}^{-}$
such that $ (1 )$ $P (K )\geq P
(S_{\varepsilon}^{-} )\cdot( (4d )^{d}\llvert \Delta_{0}\rrvert )^{-1}$
and $ (2 )$ if $\omega,\omega'\in K$ and $\omega\neq
\omega'$,
then $\operatorname{dist} (\omega,\omega' )>4dM$. Indeed, for every
$x\in\mathbb{Z}^{d}$ let $U_{x}$ be an independent (of\vspace*{1pt} everything
defined so far) random variable uniformly distributed on $
[0,1 ]$,
and define\footnote{The event $K$ is not measurable in the $\sigma
$-algebra of $\Omega$.
However, using Fubini's theorem we can find a section in $\Omega$
which is measurable and have the desired properties.}
\[
K= \bigl\{ \omega\in S_{\varepsilon}^{-}: \forall x\in4d\Delta
_{0} \mbox{ if }\sigma_{x}\omega\in B_{\varepsilon}^{-}
\mbox{ then }U_{x}<U_{0} \bigr\}.
\]
Informally, from each family of environments whose distance is smaller
than $4dM$ we choose one uniformly. This immediately implies that
for two distinct points in $K$ property $(2)$ holds. Property $(1)$
on the other hand holds due to translation invariance of $P$.

Now, let
\[
H=\bigcup_{x\in\Delta_{0}}\sigma_{x}K\quad\mbox{and}
\quad H^{-}=\bigcup_{x\in\Delta_{0}^{-}}
\sigma_{x}K.
\]
By property (2) of $K$, in both cases this is a disjoint union and,
therefore, recalling once more the translation invariance of the measure
$P$, we have
%
\begin{eqnarray}\label{eqPprobestimationforHandH^-}
P (H ) &=& \llvert\Delta_{0}\rrvert P (K )\quad\mbox{and}
\nonumber\\[-8pt]\\[-8pt]\nonumber
P \bigl(H^{-} \bigr) &=&\bigl\llvert\Delta_{0}^{-}
\bigr\rrvert P (K )=\llvert\Delta_{0}\rrvert\biggl(1-
\frac{\varepsilon}{6d^{2}} \biggr)^{d}P (K )> \biggl(1-\frac{\varepsilon}{6}
\biggr)P (H ).
\end{eqnarray}
Going back to the definition of the event $B_{\epsilon}^{-}$ and
recalling that $K\subset S_{\varepsilon}^{-}\subset B_{\varepsilon}^{-}$
we get
\begin{eqnarray*}
Q (H ) & =&\int_{H}\frac{dQ}{dP} (\omega)\,dP (\omega)=
\sum_{x\in\Delta_{0}}\int_{\sigma
_{x}K}
\frac{dQ}{dP} (\omega)\,dP (\omega)
\\
& =&\int_{K}\sum_{x\in\Delta_{0}}
\frac{dQ}{dP} (\sigma_{x}\omega)\,dP (\omega)
\\
& \leq&\int_{K} (1-\epsilon)\llvert\Delta_{0}
\rrvert dP (\omega)= (1-\varepsilon)\llvert\Delta_{0}\rrvert P (K )
\\
& =& (1-\varepsilon)P (H ).
\end{eqnarray*}
Combining with (\ref{eqPprobestimationforHandH^-}), for small
enough $\varepsilon>0$ this yields
%
\begin{eqnarray}
Q (H )&\leq& (1-\varepsilon)P (H )=\frac{1-\varepsilon}{1-\sfrac
{\varepsilon}{6}} \biggl(1-\frac
{\varepsilon}{6}
\biggr)P (H )<\frac{1-\varepsilon }{1-\sfrac{\varepsilon}{6}}P \bigl(H^{-} \bigr)
\nonumber\\[-8pt]\\[-8pt]\nonumber
&<&  \biggl(1-
\frac
{\varepsilon}{3} \biggr)P \bigl(H^{-} \bigr).\label
{eqQprobestimationforHandH^-}
\end{eqnarray}
Let $A= \{ (\omega,\omega' ): \omega\in H^{-}, \omega'\notin H \} $.
Then by (\ref{eqPprobestimationforHandH^-}) and (\ref
{eqQprobestimationforHandH^-})
%
\begin{eqnarray}\label{eqPsiAestimation}
\Psi(A ) & \geq& P \bigl(H^{-} \bigr)-Q (H )\geq P (H )- \biggl(1-
\frac{\varepsilon}{3} \biggr)P \bigl(H^{-} \bigr)
\nonumber\\[-8pt]\\[-8pt]\nonumber
&\geq& \frac{\varepsilon}{3}P
\bigl(H^{-} \bigr)>\frac
{\varepsilon}{3} \biggl(1-\frac{\varepsilon}{6}
\biggr)P (H )>\frac{\varepsilon}{4}P (H ).
\end{eqnarray}
By the construction of $K$, for every $ (\omega,\omega'
)\in A$
we have $\operatorname{dist} (\omega,\omega' )>dM_{\varepsilon}$
and, therefore,
%
\begin{eqnarray}\label{eqfirstintDMHestimation}
&& \int_{H}D_{M_{\varepsilon}}^{ (1 )} (\omega)\,dP (
\omega)\nonumber
\\
&&\qquad   = \int_{H\times\Omega}D_{M_{\varepsilon
}}^{ (1 )} (\omega
)\,d\Psi\bigl(\omega,\omega' \bigr)\geq\int_{H^{-}\times\Omega
}D_{M_{\varepsilon}}^{
(1 )}
(\omega)\,d\Psi\bigl(\omega,\omega' \bigr)
\nonumber
\\
&&\qquad  = \int_{\Omega\times\Omega}E_{\Psi} [\ind_{\operatorname{dist} (\omega,\omega' )>dM_{\varepsilon}}\mid
\mathfrak{F}_{\omega} ] (\omega)\ind_{H^{-}\times\Omega
} \bigl(\omega,
\omega' \bigr)\,d\Psi\bigl(\omega,\omega' \bigr)
\nonumber
\\
&&\qquad  = \int_{\Omega\times\Omega}E_{\Psi} \bigl[\ind_{\operatorname{dist} (\omega,\omega' )>dM_{\varepsilon}}
\ind_{H^{-}\times\Omega} \bigl(\omega,\omega' \bigr)\mid\mathfrak
{F}_{\omega} \bigr] (\omega)\,d\Psi\bigl(\omega,\omega'
\bigr)
\nonumber\\[-8pt]\\[-8pt]\nonumber
&&\qquad  =\int_{\Omega\times\Omega}\ind_{\operatorname{dist} (\omega,\omega'
)>dM_{\varepsilon}}\ind_{H^{-}\times\Omega} \bigl(
\omega,\omega' \bigr)\,d\Psi\bigl(\omega,\omega' \bigr)
\nonumber
\\
&&\qquad  \geq \int_{\Omega\times\Omega}\ind_{\operatorname{dist} (\omega,\omega'
)>dM_{\varepsilon}}\ind_{A}
\bigl(\omega,\omega' \bigr)\,d\Psi\bigl(\omega,\omega'
\bigr)
\nonumber
\\
&&\qquad  =\int_{\Omega\times\Omega}\ind_{A} \bigl(\omega,
\omega' \bigr)\,d\Psi\bigl(\omega,\omega' \bigr)
\nonumber
\\
&&\qquad  =\Psi(A )>\frac{\varepsilon}{4}P (H ),\nonumber
\end{eqnarray}
where for the last inequality we used (\ref{eqPsiAestimation}).
However, recalling that $H\subset\mathbf{F}_{M_{\varepsilon}}$ by
definition, and using Lemma~\ref
{lemPerliminarylemmatoRadonNikodymderivativeestimate}
with $M_{\varepsilon}$ and $\frac{\varepsilon}{5}$ instead of $M$
and $\varepsilon$ we get
%
\begin{eqnarray}\label{eqsecondintDMHestimation}
\int_{H}D_{M_{\varepsilon}}^{ (1 )} (\omega)\,dP (
\omega) & \leq&\int_{H}\frac{\varepsilon}{5}\ind
_{\mathbf{F}_{M_{\varepsilon}}} (\omega)+\ind_{\mathbf
{F}_{M_{\varepsilon}}^{c}} (\omega)\,dP (\omega)
\nonumber
\nonumber\\[-8pt]\\[-8pt]\nonumber
& =&\int_{H}\frac{\varepsilon}{5}\,dP (\omega)=
\frac
{\varepsilon}{5}P (H ).
\end{eqnarray}
Combining (\ref{eqfirstintDMHestimation}) and (\ref{eqsecondintDMHestimation})
we must conclude that $P (H )=0$ and, therefore $P
(K )=0$.
This\vspace*{1pt} however, by property (1) of $K$, implies that $P
(S_{\varepsilon}^{-} )=0$
and, therefore, finally that $P (B_{\varepsilon}^{-}
)=M^{-\xi(1 )}$.

Next, we turn to deal with the event $B_{\varepsilon}^{+}$. As in
the proof for $B_{\varepsilon}^{-}$ for $\varepsilon>0$ we denote
$M_{\varepsilon}=\frac{\varepsilon}{6d^{2}}M$. Also\vspace*{1pt} assume without
loss of generality that $\Delta_{0}$ is centered in zero, define
$\Delta_{0}^{+}= \{ x\in\mathbb{Z}^{d}: \llVert x\rrVert _{\infty
}<M+dM_{\varepsilon} \} $
and let
\[
S_{\varepsilon}^{+}= \bigl\{ \omega\in B_{\varepsilon}^{+}: \sigma
_{x}\omega\in\mathbf{F}_{M_{\varepsilon}}, \forall x\in
\Delta_{0}^{+} \bigr\},
\]
where $\mathbf{F}_{M_{\varepsilon}}$ is the event from Lemma~\ref
{lemPerliminarylemmatoRadonNikodymderivativeestimate}.
Due to property (1) of $\mathbf{F}_{M_{\varepsilon}}$ from Lemma~\ref{lemPerliminarylemmatoRadonNikodymderivativeestimate}
\begin{eqnarray*}
P \bigl(S_{\varepsilon}^{+} \bigr)&\geq& P \bigl(B_{\varepsilon
}^{+}
\bigr)-\bigl\llvert\Delta_{0}^{+}\bigr\rrvert P \bigl(
\mathbf{F}_{M_{\varepsilon}}^{c} \bigr)=P \bigl(B_{\varepsilon}^{+}
\bigr)- \biggl(1+\frac{\varepsilon}{6d} \biggr)^{d}M^{d}\cdot
(M_{\varepsilon} )^{-\xi(1 )}
\\
&=&P \bigl(B_{\varepsilon}^{+}
\bigr)-M^{-\xi(1 )}
\end{eqnarray*}
and, therefore, it is enough to show that $P (S_{\varepsilon
}^{+} )=0$.
As for $S_{\varepsilon}^{-}$ we claim that there exists an event
$K\subset S_{\varepsilon}^{+}$ such that (1) $P (K )\geq
P (S_{\varepsilon}^{+} )/ ( (4d )^{d}\llvert \Delta_{0}^{+}\rrvert )^{-1}$
and (2) if $\omega,\omega'\in K$ and $\omega\neq\omega'$, then
$\operatorname{dist} (\omega,\omega' )>4d (M+M_{\varepsilon
} )$.

Now, let
\[
H=\bigcup_{x\in\Delta_{0}}\sigma_{x}K\quad\mbox{and}
\quad H^{+}=\bigcup_{x\in\Delta_{0}^{+}}
\sigma_{x}K.
\]
By property (2) of $K$, in both cases this is a disjoint union and,
therefore, recalling once more the translation invariance of the measure
$P$, we have for small enough $\varepsilon>0$
%
\begin{eqnarray}\label{eqappendixofBep^+1}
P (H ) &=&\llvert\Delta_{0}\rrvert P (K )\quad\mbox{and}
\nonumber\\[-8pt]\\[-8pt]\nonumber
P \bigl(H^{+} \bigr)&=&\bigl\llvert\Delta_{0}^{+}
\bigr\rrvert P (K )= \biggl(1+\frac{\varepsilon}{6d^{2}} \biggr
)^{d}\llvert
\Delta_{0}\rrvert P (K )< \biggl(1+\frac{\varepsilon}{6} \biggr)P (H
).
\end{eqnarray}
Going back to the definition of the event $B_{\varepsilon}^{+}$ and
recalling that $K\subset S_{\varepsilon}^{+}\subset B_{\varepsilon}^{+}$
we get
%
\begin{eqnarray}\label{eqappendixofBep^+15}
Q (H ) & =&\int_{H}\frac{dQ}{dP} (\omega)\,dP (\omega)=
\sum_{x\in\Delta_{0}}\int_{\sigma
_{x}K}
\frac{dQ}{dP} (\omega)\,dP (\omega)
\nonumber
\\
& =&\int_{K}\sum_{x\in\Delta_{0}}
\frac{dQ}{dP} (\sigma_{x}\omega)\,dP (\omega)
\nonumber\\[-8pt]\\[-8pt]\nonumber
& >& \int_{K} (1+\varepsilon)\llvert\Delta_{0}
\rrvert \,dP (\omega)= (1+\varepsilon)\llvert\Delta_{0}\rrvert P (K )
\nonumber
\\
& =& (1+\varepsilon)P (H ),\nonumber
\end{eqnarray}
and, therefore, combining with (\ref{eqappendixofBep^+1}), for
small enough $\varepsilon>0$ this yields
%
\begin{eqnarray}\label{eqappendixofBep^+2}
Q (H ) &>& (1+\varepsilon)P (H )=\frac
{1+\varepsilon}{1+\sfrac{\varepsilon}{6}} \biggl(1+\frac{\varepsilon
}{6}
\biggr)P (H )>\frac{1+\varepsilon}{1+\sfrac
{\varepsilon}{6}}P \bigl(H^{+} \bigr)
\nonumber\\[-8pt]\\[-8pt]\nonumber
&>&  \biggl(1+
\frac{\varepsilon
}{3} \biggr)P \bigl(H^{+} \bigr).
\end{eqnarray}
Let $A= \{ (\omega,\omega' ): \omega\notin
H^{+}, \omega'\in H \} $.
Then by (\ref{eqappendixofBep^+2})
%
\begin{eqnarray}\label{eqappendixofBep^+3}
\Psi(A )&\geq& Q (H )-P \bigl(H^{+} \bigr)>Q (H )-\frac{1}{1+\sfrac
{\varepsilon}{3}}Q
(H )=\frac{\sfrac{\varepsilon}{3}}{1+\sfrac{\varepsilon}{3}}Q (H )
\nonumber\\[-8pt]\\[-8pt]\nonumber
&\geq&
\frac{\varepsilon}{4}Q (H ).
\end{eqnarray}
By the construction of $K$, for every $ (\omega,\omega'
)\in A$
we have $\operatorname{dist} (\omega,\omega' )>dM_{\varepsilon}$
and, therefore,
%
\begin{eqnarray}\label{eqappendixofBep^+4}
&& \int_{H}D_{M_{\varepsilon}}^{ (2 )} (\omega)\,dQ (
\omega) \nonumber
\\[-1pt]
&&\qquad  = \int_{\Omega\times H}D_{M_{\varepsilon
}}^{ (2 )} \bigl(
\omega' \bigr)\,d\Psi\bigl(\omega,\omega' \bigr)
\nonumber
\\[-1pt]
&&\qquad  =\int_{\Omega\times\Omega}E_{\Psi} [\ind_{\operatorname{dist} (\omega,\omega' )>dM_{\varepsilon}}\mid
\mathfrak{F}_{\omega'} ] \bigl(\omega' \bigr)
\ind_{\Omega
\times H} \bigl(\omega,\omega' \bigr)\,d\Psi\bigl(\omega,
\omega' \bigr)
\nonumber
\\[-1pt]
&&\qquad  =\int_{\Omega\times\Omega}E_{\Psi} \bigl[\ind_{\operatorname{dist} (\omega,\omega' )>dM_{\varepsilon}}
\ind_{\Omega
\times H} \bigl(\omega,\omega' \bigr)\mid
\mathfrak{F}_{\omega
'} \bigr] \bigl(\omega' \bigr)\,d\Psi\bigl(
\omega,\omega' \bigr)
\nonumber\\[-9pt]\\[-9pt]\nonumber
&&\qquad  = \int_{\Omega\times\Omega}\ind_{\operatorname{dist} (\omega,\omega'
)>dM_{\varepsilon}}\ind_{\Omega\times H} \bigl(
\omega,\omega' \bigr)\,d\Psi\bigl(\omega,\omega' \bigr)
\nonumber
\\[-1pt]
&&\qquad  \geq \int_{\Omega\times\Omega}\ind_{\operatorname{dist} (\omega,\omega'
)>dM_{\varepsilon}}\ind_{A}
\bigl(\omega,\omega' \bigr)\,d\Psi\bigl(\omega,\omega'
\bigr)
\nonumber
\\[-1pt]
&&\qquad  = \int_{\Omega\times\Omega}\ind_{A} \bigl(\omega,
\omega' \bigr)\,d\Psi\bigl(\omega,\omega' \bigr)
\nonumber
\\[-1pt]
&&\qquad  = \Psi(A )\geq\frac{\varepsilon}{4}Q (H ),\nonumber
\end{eqnarray}
where for the last inequality we used (\ref{eqappendixofBep^+3}).
However, recalling that $H\subset\mathbf{F}_{M_{\varepsilon}}$ by
definition, that $P (H )\leq Q (H )$ by (\ref
{eqappendixofBep^+15})
and using Lemma~\ref{lemPerliminarylemmatoRadonNikodymderivativeestimate}
with $M_{\varepsilon}$ and $\frac{\varepsilon}{5}$ instead of $M$
and $\varepsilon$ we get
%
\begin{eqnarray}\label{eqappendixofBep^+5}
\int_{H}D_{M_{\varepsilon}}^{ (2 )}\,dQ (\omega) & \leq&\int_{H}\frac{1}{f (\omega)} \biggl[\frac
{\varepsilon}{5}
\ind_{\mathbf{F}_{M_{\varepsilon}}}+\ind_{\mathbf
{F}_{M_{\varepsilon}}^{c}} \biggr]\,dQ (\omega)\nonumber
\\
&=& \int
_{H} \biggl[\frac{\varepsilon}{5}\ind_{\mathbf{F}_{M_{\varepsilon
}}}+
\ind_{\mathbf{F}_{M_{\varepsilon}}^{c}} \biggr]\,dP (\omega)
\\
& =&\int_{H}\frac{\varepsilon}{5}\,dP (\omega)=
\frac
{\varepsilon}{5}P (H )\leq\frac{\varepsilon}{5}Q (H ).\nonumber
\end{eqnarray}
Combining (\ref{eqappendixofBep^+4}) and (\ref{eqappendixofBep^+5}),
we must conclude that $Q (H )=0$. Therefore, by (\ref
{eqappendixofBep^+15})
we have $P (H )=0$, and thus that $P (K )=0$. This\vspace*{1pt}
however, by property (1) of $K$, implies that $P (S_{\varepsilon
}^{+} )=0$
and, therefore, finally that $P (B_{\varepsilon}^{+}
)=M^{-\xi(1 )}$.
\end{pf*}

\begin{cor}
\label{corExistenceofmoments}Let $d\geq4$ and assume $P$ is
uniformly elliptic, i.i.d. and satisfies $ (\mathscr{P} )$.
Then $E [ (\frac{dQ}{dP} )^{k} ]<\infty$ for every
$k\in\mathbb{N}$.
\end{cor}

\begin{pf}
For every $M\in\mathbb{N}$ large enough, Lemma~\ref
{lemRadonNikodymderivativeestimate}
implies
\begin{eqnarray*}
P \biggl(\frac{dQ}{dP} (\omega)\geq 2 (2M+1 )^{d} \biggr)&\leq& P
\biggl(\frac{1}{ (2M+1 )^{d}}\sum_{x\in
[-M.M ]^{d}\cap\mathbb{Z}^{d}}\frac{dQ}{dP}
(\sigma_{x}\omega)\geq2 \biggr)
\\[-2pt]
&=& M^{-\xi(1 )}.
\end{eqnarray*}
Thus, $\frac{dQ}{dP}$ has super polynomial decay and the result follows.\vadjust{\goodbreak}
\end{pf}

\begin{pf*}{Proof of Theorem~\ref{teoAbsolutelycontinuousinvariantmeasure}}
The proof is the content of Lemmas~\ref{Lemmathereexistsnonsingular},
\ref{LemEitherSingularorabsolutelycontinuous} and Corollary
\ref{corExistenceofmoments}.
\end{pf*}

\section{Proof of Theorem \texorpdfstring{\protect\ref{teoPrefactorthm}}{1.11}}\label{sec7}

In this section, we prove our second main result, the prefactor local
limit theorem. The uniqueness of the prefactor follows from its definition
quite easily and most of the work is concentrated into the proof of
existence. Our candidate for the prefactor is the Radon--Nikodym derivative
of the equivalent measure $Q$ constructed in the previous section.
The\vspace*{1pt} proof proceeds as follows: instead of directly comparing the quenched
measure $P_{\omega}^{0} (X_{n}=\cdot)$ and the annealed
times the prefactor measure $\mathbb{P}^{0} (X_{n}=\cdot
)f (\sigma_{\cdot}\omega)$,
we introduce two new measures (denoted temporarily by $\rho_{1}$
and $\rho_{2}$) and show that the total variation of the pairs $
(\mathbb{P}^{0} (X_{n}=\cdot)f (\sigma_{\cdot
}\omega),\rho_{1} )$,
$ (\rho_{1},\rho_{2} )$, $ (\rho_{2},P_{\omega
}^{0} (X_{n}=\cdot) )$
goes to zero as $n$ goes to infinity for $P$-almost every environment.
Both measures $\rho_{1}$ and $\rho_{2}$ are constructed in a way
that allows us to exploit the previous results on the connection between
the quenched and annealed measures in the total variations estimations.
More formally, we fix two parameters $0<\delta<\varepsilon<\frac{1}{2}$
and define the measures $\rho_{1}$ and $\rho_{2}$ as follows: for
$\rho_{1}$, we choose a point by first choosing a point according
to the annealed law at time $n-n^{\varepsilon}$ times the prefactor
and then letting it ``evolve'' according to the quenched law for
$n^{\varepsilon}$ additional steps. For $\rho_{2}$ we fix some partition
of $\mathbb{Z}^{d}$ to boxes of side length $n^{\delta}$, choose
a box according to the quenched measure at time $n-n^{\varepsilon}$,
choose a point inside the box proportionally to its prefactor and
then let it ``evolve'' into a new point according to the quenched
law for $n^{\varepsilon}$ additional steps. For a more precise definition
of the measures, see Definition~\ref{DefLotsofmeasures}.

\subsection{Uniqueness}\label{sec7.1}

We start with a proof that the prefactor, if exists, is unique. Assume
both $f$ and $g$ satisfy (\ref{eqPrefactorthm}) and denote $h=f-g$.
By the triangle inequality for $P$-almost every $\omega\in\Omega$
%
\begin{equation}
\lim_{n\to\infty}\sum_{x\in\mathbb{Z}^{d}}
\mathbb{P}^{0} (X_{n}=x )\bigl\llvert h (\sigma_{x}
\omega)\bigr\rrvert=0,\label{eqprefactoruniqueness}
\end{equation}
that is, $\lim_{n\to\infty}\mathbb{E}^{0} [\llvert h
(\sigma_{X_{n}}\omega)\rrvert ]=0$,
$P$-a.s. If $h\neq0$, then there exists a measurable subset $A$
of $\Omega$ such that $P (A )>0$ and $\llvert h\rrvert >c>0$
on $A$. Thus, for every $n\in\mathbb{N}$
%
\begin{eqnarray}\label{equniqunessoftheprefactor}
E \bigl[\mathbb{E}^{0} \bigl[\bigl\llvert h (\sigma_{X_{n}}
\omega)\bigr\rrvert\bigr] \bigr] & \geq& E \bigl[\mathbb{E}^{0} \bigl[
\bigl\llvert h (\sigma_{X_{n}}\omega)\bigr\rrvert\ind_{\sigma
_{X_{n}}\omega\in A}
\bigr] \bigr]\geq cE \bigl[\mathbb{E}^{0} [\ind_{\sigma_{X_{n}}\omega
\in A} ] \bigr]
\nonumber
\\
& =&cE \bigl[\mathbb{P}^{0} (\sigma_{X_{n}}\omega\in A )
\bigr]=c\sum_{y\in\mathbb{Z}^{d}}\mathbb{P}^{0}
(X_{n}=y )E [\ind_{\sigma_{y}\omega\in A} ]
\nonumber
\\
& =&c\sum_{y\in\mathbb{Z}^{d}}\mathbb{P}^{0}
(X_{n}=y )P (\sigma_{y}\omega\in A )
\\
&=& cP (A )\sum
_{y\in
\mathbb{Z}^{d}}\mathbb{P}^{0} (X_{n}=y )
\nonumber
\\
& =&cP (A )>0.\nonumber
\end{eqnarray}
Since
\begin{eqnarray*}
E \bigl[\mathbb{E}^{0} \bigl[\bigl\llvert h (\sigma_{X_{n}}
\omega)\bigr\rrvert\bigr] \bigr] & =&E \biggl[\sum_{y\in\mathbb
{Z}^{d}}
\bigl\llvert h (\sigma_{y}\omega)\bigr\rrvert\mathbb
{P}^{0} (X_{n}=y ) \biggr]
\\
& =&\sum_{y\in\mathbb{Z}^{d}}\mathbb{P}^{0}
(X_{n}=y )\cdot E \bigl[\bigl\llvert h (\sigma_{y}\omega)
\bigr\rrvert\bigr]
\\
& =&\sum_{y\in\mathbb{Z}^{d}}\mathbb{P}^{0}
(X_{n}=y )\cdot E \bigl[\bigl\llvert h (\omega)\bigr\rrvert\bigr]=E
\bigl[\llvert h\rrvert\bigr],
\end{eqnarray*}
the sequence $ \{ \llvert h (\sigma_{X_{n}}\omega
)\rrvert \} _{n\in\mathbb{N}}$
is tight and, therefore, by (\ref{equniqunessoftheprefactor})
it follows that $\lim_{n\to\infty}\mathbb{E}^{0} [\llvert h (\sigma
_{X_{n}}\omega)\rrvert ]>0$
a contradiction to (\ref{eqprefactoruniqueness}).

\subsection{Existence}\label{sec7.2}

Let $f (\omega)$ be the Radon--Nikodym derivative of $Q$
defined in Theorem~\ref{teoAbsolutelycontinuousinvariantmeasure}.
We will show that $f$ satisfies Theorem~\ref{teoPrefactorthm}
starting with the following simple proposition.

\begin{prop}
\label{propInvarianceoftheprefactor}For $P$-almost every $\omega$
every $n\in\mathbb{N}$ and every $x\in\mathbb{Z}^{d}$
\[
f (\sigma_{x}\omega)=\sum_{y\in\mathbb
{Z}^{d}}P_{\omega}^{y}
(X_{n}=x )f (\sigma_{y}\omega).
\]
\end{prop}
\begin{pf}
For $n=1$, this follows from the definition of $f=\frac{dQ}{dP}$
as the Radon--Nikodym derivative of the measure $Q$ which is invariant
with respect to the point of view of the particle. Indeed, using (\ref
{eqTransitionkernel})
and (\ref{eqinvariantprobabilitymeasure}) and the translation
invariance of $P$ for every bounded measurable function $g:\Omega\to
\mathbb{R}$
we have
\begin{eqnarray*}
\int_{\Omega}g (\omega)f (\omega)\,dP (\omega) & =&\int
_{\Omega}g (\omega)\,dQ (\omega)=\int_{\Omega}
\mathfrak{R}g (\omega)\,dQ (\omega)
\\
& =&\int_{\Omega} \bigl(\mathfrak{R}g (\omega) \bigr)f (\omega
)\,dP (\omega)
\\
&=& \int_{\Omega}\sum_{e\in\mathcal{E}_{d}}
\omega(0,e )g (\sigma_{e}\omega)f (\omega)\,dP (\omega)
\\
& =&\int_{\Omega}\sum_{e\in\mathcal{E}_{d}}\omega
(-e,0 )g (\omega)f (\sigma_{-e}\omega)\,dP (\omega)
\\
& =&\int_{\Omega}g (\omega)\sum_{e\in\mathcal
{E}_{d}}
\omega(e,0 )f (\sigma_{e}\omega)\,dP (\omega)
\end{eqnarray*}
and, therefore,
\[
f (\omega)=\sum_{e\in\mathcal{E}_{d}}\omega(e,0 )f (
\sigma_{e}\omega)=\sum_{e\in\mathcal
{E}_{d}}P_{\omega}^{e}
(X_{1}=0 )f (\sigma_{e}\omega).
\]
Applying the last equality for $\sigma_{x}\omega$ gives the result
in the case $n=1$.

For $n>1$, the proof follows by induction. Indeed,
\begin{eqnarray*}
\sum_{y\in\mathbb{Z}^{d}}P_{\omega}^{y}
(X_{n}=x )f (\sigma_{y}\omega) & =&\sum
_{y\in\mathbb{Z}^{d}}\sum_{z\in
\mathbb{Z}^{d}}P_{\omega}^{y}
(X_{n-1}=z )P_{\omega
}^{z} (X_{1}=x )f (
\sigma_{y}\omega)
\\
& =&\sum_{z\in\mathbb{Z}^{d}}P_{\omega}^{z}
(X_{1}=x )\sum_{y\in\mathbb{Z}^{d}}P_{\omega}^{y}
(X_{n-1}=z )f (\sigma_{y}\omega)
\\
& \overset{ (1 )} {=}&\sum_{z\in\mathbb{Z}^{d}}P_{\omega
}^{z}
(X_{1}=x )f (\sigma_{z}\omega)\overset{ (2 )} {=}f (
\sigma_{x}\omega),
\end{eqnarray*}
where for $ (1 )$ we used the induction assumption and in
$ (2 )$ we used the case $n=1$.
\end{pf}

As stated at the beginning of the section, the proof of Theorem~\ref
{teoPrefactorthm}
uses comparison with two additional probability measures which we
now define.

\begin{defn}
\label{DefLotsofmeasures}For $n\in\mathbb{N}$ and $\omega\in\Omega$,
define the following probability measures on $\mathbb{Z}^{d}$:
\begin{longlist}[(2)]
\item[(1)]$\nu_{\omega}^{\mathrm{ann}\times\mathrm{pre},n}$---the annealed
at time
$n$ times the prefactor (normalized)
\[
\nu_{\omega}^{\mathrm{ann}\times\mathrm{pre},n} (x )=\frac
{1}{Z_{\omega,n}}\mathbb{P}^{0}
(X_{n}=x )f (\sigma_{x}\omega),
\]
where $Z_{\omega,n}=\sum_{x\in\mathbb{Z}^{d}}\mathbb{P}^{0}
(X_{n}=x )f (\sigma_{x}\omega)$
is a normalizing constant. In Lem\-ma~\ref
{lemEvaluationofthenormalizingconstant},
we show that $\lim_{n\to\infty}Z_{\omega,n}=1$, $P$-almost surely.
\item[(2)]$\nu_{\omega}^{\mathrm{que},n}$---the quenched measure at time $n$
\[
\nu_{\omega}^{\mathrm{que},n} (x )=P_{\omega}^{0}
(X_{n}=x ).
\]

\item[(3)]$\nu_{\omega}^{\mathrm{box}\mbox{-}\mathrm{que} \times\mathrm{pre},n}=\nu
_{\omega,\Pi}^{\mathrm{box}\mbox{-}\mathrm{que} \times\mathrm{pre},n}$---the quenched measure on boxes with a choice of a point in the box
proportional to the prefactor. Given a partition $\Pi$ of $\mathbb{Z}^{d}$
into boxes of side length $l$, we choose a box according to the quenched
measure at time $n$ and then choose a point inside of the box proportionally
to the value of the Radon--Nikodym derivative there.
\[
\nu_{\omega,\Pi}^{\mathrm{box}\mbox{-}\mathrm{que} \times\mathrm{pre},n} (x )= \cases{
\displaystyle P_{\omega}^{0}
(X_{n}\in\Delta_{x} )\frac{f (\sigma
_{x}\omega)}{
\sum_{y\in\Delta_{x},
y\leftrightarrow n} f (\sigma_{y}\omega)}, &\quad$x
\leftrightarrow n$,
\cr
0, &\quad\mbox{otherwise},}
\]
where $\Delta_{x}$ is the unique $d$-dimensional box that contains
$x$ in the partition $\Pi$.
\end{longlist}
\end{defn}

Before turning to the proof of Theorem~\ref{teoPrefactorthm}, we
wish to study the normalization constant $Z_{\omega,n}$ of the measure
$\nu_{\omega}^{\mathrm{ann}\times\mathrm{pre},n}$.

%
\begin{lem}
\label{lemEvaluationofthenormalizingconstant}With the notation
as in Definition~\ref{DefLotsofmeasures} for $P$-almost every
$\omega$, we have $\lim_{n\to\infty}Z_{\omega,n}=1$.
\end{lem}
\begin{pf}
Fix $\varepsilon>0$, $0<\delta<\frac{1}{6d}$ and let $\Pi$ be a
partition of $\mathbb{Z}^{d}$ into boxes of side length $n^{\delta}$.
If $x,y\in\Delta$ for some $\Delta\in\Pi$, then the annealed derivative
estimation (see Lemma~\ref{lemgeneralannealedestimations}) gives
%
\begin{eqnarray}\label{eqannealedestimationfornormalizationconstant}
\bigl\llvert\mathbb{P}^{0} (X_{n}=x )-
\mathbb{P}^{0} (X_{n}=y )\bigr\rrvert&\leq& C\llVert x-y
\rrVert_{1}n^{-\vfrac{d+1}{2}}
\nonumber\\[-8pt]\\[-8pt]\nonumber
&\leq& Cn^{-\vfrac{d+1}{2}+\delta}.
\end{eqnarray}
Denoting $\Pi_{n}= \{ \Delta\in\Pi: \Delta\cap
[-n,n ]^{d}\neq\varnothing\} $
we have
\begin{eqnarray*}
\llvert Z_{\omega,n}-1\rrvert &=& \biggl\llvert\sum
_{x\in\mathbb
{Z}^{d}}\mathbb{P} (X_{n}=x ) \bigl[f (\sigma
_{x}\omega)-1 \bigr]\biggr\rrvert
\\
&=&\biggl\llvert\sum
_{\Delta\in\Pi
_{n}}\mathop{\sum_{x\in\Delta}}_{x\leftrightarrow n}
\mathbb{P}^{0} (X_{n}=x ) \bigl[f (\sigma_{x}
\omega)-1 \bigr]\biggr\rrvert.
\end{eqnarray*}
By Lemma~\ref{lemgoodestimationforthelocation}, there exists
$C_{\varepsilon}>0$ such that $\mathbb{P}^{0} (\llVert
X_{n}-\mathbb{E}^{0} [X_{n} ]\rrVert
_{1}>C_{\varepsilon}\sqrt{n} )<\varepsilon$.
Separating the sum into boxes in $\widehat{\Pi}_{n}= \{ \Delta
\in\Pi_{n}: \Delta\cap\{ x\in\mathbb{Z}^{d}: \llVert x-\mathbb{E}^{0}
[X_{n} ]\rrVert \leq
C_{\varepsilon}\sqrt{n} \} \neq\varnothing\} $
and in $\Pi_{n}\setminus\widehat{\Pi}_{n}$ we can bound the last
term by
%
\begin{eqnarray}
\qquad & \leq&\biggl\llvert\sum_{\Delta\in\Pi_{n}\setminus\widehat{\Pi
}_{n}}\mathop{\sum
_{x\in\Delta}}_{x\leftrightarrow n}\mathbb{P}^{0}
(X_{n}=x ) \bigl[f (\sigma_{x}\omega)-1 \bigr]\biggr\rrvert
\label{eqnormalizationconstantlemma0}
\\
&&{} +\biggl\llvert\sum_{\Delta\in\widehat{\Pi}_{n}}\mathop{\sum
_{x\in\Delta}}_{x\leftrightarrow n} \biggl(\frac{1}{\llvert \Delta
\rrvert } \mathop{\sum
_{y\in\Delta}}_{y\leftrightarrow n} \bigl[\mathbb{P}^{0}
(X_{n}=y )-\mathbb{P}^{0} (X_{n}=x ) \bigr]
\biggr) \bigl[f (\sigma_{x}\omega)-1 \bigr]\biggr\rrvert\label
{eqnormalizationconstantlemma1}
\\
&&{} +\biggl\llvert\sum_{\Delta\in\widehat{\Pi}_{n}}\mathop{\sum
_{x\in\Delta}}_{x\leftrightarrow n}\frac{1}{\llvert \Delta\rrvert
}\mathop{\sum
_{y\in\Delta}}_{y\leftrightarrow n} \mathbb{P}^{0}
(X_{n}=y ) \bigl[f (\sigma_{x}\omega)-1 \bigr]\biggr\rrvert.
\label{eqnormalizationconstnatlemma2}
\end{eqnarray}
We start by evaluating the term (\ref{eqnormalizationconstantlemma0}).
By Lemma~\ref{lemRadonNikodymderivativeestimate}, there exists
some constant~$C$, such that with $P$ probability $\geq1-n^{-\xi
(1 )}$
for every $\Delta\in\Pi_{n}$ (and in particular in $\Pi
_{n}\setminus\widehat{\Pi}_{n}$)
we have $\mathop{\sum_{y\in\Delta}}_{y\leftrightarrow n} [f (\sigma
_{y}\omega)+1 ]\leq C\llvert \Delta\rrvert $. Therefore, under the
above event, we can bound (\ref
{eqnormalizationconstantlemma0})
by
\begin{eqnarray*}
&& \sum_{\Delta\in\Pi_{n}\setminus\widehat{\Pi}_{n}}\mathop{\sum
_{x\in\Delta}}_{x\leftrightarrow n}\mathbb{P}^{0}
(X_{n}=x ) \bigl[f (\sigma_{x}\omega)+1 \bigr]
\\
&&\qquad  \leq \sum
_{\Delta\in\Pi_{n}\setminus\widehat
{\Pi}_{n}}\mathop{\max_{
x\in\Delta}}_{
x\leftrightarrow n}
\mathbb{P}^{0} (X_{n}=x )\mathop{\sum
_{y\in\Delta}}_{x\leftrightarrow n} \bigl[f (\sigma_{y}\omega
)+1 \bigr]
\\
&&\qquad  \leq C\sum_{\Delta\in\Pi_{n}\setminus\widehat{\Pi}_{n}}\llvert\Delta
\rrvert\cdot
\mathop{\max_{
x\in\Delta}}_{x\leftrightarrow n}\mathbb{P}^{0}
(X_{n}=x ).
\end{eqnarray*}
Using Lemma~\ref{lemgeneralannealedestimations}, (\ref
{eqgeneralannealedestimation4})
and the definition of $\widehat{\Pi}_{n}$, we thus have
\begin{eqnarray*}
\eqref{eqnormalizationconstantlemma0} & \leq& C\sum_{\Delta\in\Pi
_{n}\setminus\widehat{\Pi}_{n}}
\llvert\Delta\rrvert\cdot\mathop{\max_{
x\in\Delta}}_{
x\leftrightarrow n}
\mathbb{P}^{0} (X_{n}=x )
\\[-3pt]
& \leq& C\sum_{\Delta\in\Pi_{n}\setminus\widehat{\Pi}_{n}}\mathop{\max_{
y\in\Delta}}_{
y\leftrightarrow n}
\Bigl[\mathop{\max_{
x\in\Delta}}_{
x\leftrightarrow n}
\mathbb{P}^{0} (X_{n}=x )-\mathbb{P}^{0}
(X_{n}=y ) \Bigr]
\\[-3pt]
&&{} +C\sum_{\Delta\in\Pi_{n}\setminus\widehat{\Pi}_{n}}\mathop{\max_{
y\in\Delta}}_{
y\leftrightarrow n}
\mathbb{P}^{0} (X_{n}=y )
\\
& \leq&\frac{C}{n^{\sfrac{1}{2}-3d\delta}}+\varepsilon.
\end{eqnarray*}
Recalling that $\delta<\frac{1}{6d}$ and taking $n\to\infty$ this
gives (by an application of the Borel--Cantelli lemma)
\[
\limsup_{n\to\infty}\eqref{eqnormalizationconstantlemma0}\leq
\limsup_{n\to\infty}\sum_{\Delta\in\Pi_{n}\setminus\widehat
{\Pi}_{n}}\mathop{
\sum_{x\in\Delta}}_{x\leftrightarrow n}\mathbb{P}^{0}
(X_{n}=x ) \bigl[f (\sigma_{x}\omega)+1 \bigr]\leq
\varepsilon,\qquad P\mbox{-a.s.}
\]

Next, we deal with the term (\ref{eqnormalizationconstantlemma1}).
Due to (\ref{eqannealedestimationfornormalizationconstant}),
this is bounded by
\begin{eqnarray*}
&& \sum_{\Delta\in\widehat{\Pi}_{n}}\mathop{\sum
_{x\in\Delta}}_{x\leftrightarrow n}\biggl\llvert\frac{1}{\llvert
\Delta\rrvert }
\mathop{\sum_{y\in\Delta}}_{y\leftrightarrow n} \bigl[\mathbb{P}
(X_{n}=y )-\mathbb{P} (X_{n}=x) \bigr]\biggr\rrvert\bigl[f
(\sigma_{x}\omega)+1 \bigr]
\\[-3pt]
&&\qquad \leq \sum_{\Delta\in\widehat{\Pi}_{n}}\mathop{\sum
_{x\in\Delta}}_{x\leftrightarrow n}\frac{C}{n^{\sklvfrac{d+1}{2}-\delta}}
\bigl[f (
\sigma_{x}\omega)+1 \bigr]
\\[-3pt]
&&\qquad \leq\frac{C}{n^{\sfrac{1}{2}-\delta}}  \cdot\biggl(\frac
{1}{n^{\sfrac{d}{2}}}\sum
_{\llVert x-\mathbb{E}^{0}
[X_{n} ]\rrVert \leq2C_{\varepsilon}\sqrt{n}}f (\sigma_{x}\omega)
\biggr)+
\frac{C}{n^{\sfrac{1}{2}-\delta}}.
\end{eqnarray*}
By Lemma~\ref{lemRadonNikodymderivativeestimate} and an application
of Borel--Cantelli for $P$-almost every $\omega$ once $n$ is large
enough, we have
$\frac{1}{n^{\sfrac{d}{2}}}\sum_{\llVert x-\mathbb{E}^{0}
[X_{n} ]\rrVert \leq2C_{\varepsilon}\sqrt{n}}f
(\sigma_{x}\omega)\leq(8C_{\varepsilon} )^{d}$,
and thus the last term tends to zero as $n$ goes to infinity $P$-almost
surely.

Finally, for (\ref{eqnormalizationconstnatlemma2}), we recall
that Lemma~\ref{lemgeneralannealedestimations} also ensures $\mathbb
{P} (X_{n}=x )\leq Cn^{-\sfrac{d}{2}}$
for every $x\in\mathbb{Z}^{d}$ and, therefore,
\begin{eqnarray*}
\eqref{eqnormalizationconstnatlemma2} & =&\biggl\llvert\sum
_{\Delta\in
\widehat{\Pi}_{n}}\sum_{x\in\Delta}
\frac{1}{\llvert \Delta\rrvert }\sum_{y\in\Delta}\mathbb{P}
(X_{n}=y ) \bigl[f (\sigma_{x}\omega)-1 \bigr]\biggr\rrvert
\\[-3pt]
& \leq&\sum_{\Delta\in\widehat{\Pi}_{n}}\frac{1}{\llvert \Delta
\rrvert }\sum
_{y\in\Delta}\mathbb{P} (X_{n}=y )\biggl\llvert\sum
_{x\in\Delta} \bigl[f (\sigma_{x}\omega)-1 \bigr]
\biggr\rrvert
\\[-3pt]
& \leq&\frac{C}{n^{\sfrac{d}{2}}}\sum_{\Delta\in\widehat{\Pi
}_{n}}\biggl\llvert
\sum_{x\in\Delta} \bigl[f (\sigma_{x}\omega)-1
\bigr]\biggr\rrvert
\\[-3pt]
& =&\frac{C}{n^{d (\sfrac{1}{2}-\delta)}}\sum_{\Delta\in
\widehat{\Pi}_{n}}\biggl\llvert
\frac{1}{\llvert \Delta\rrvert }\sum_{x\in
\Delta} \bigl[f (
\sigma_{x}\omega)-1 \bigr]\biggr\rrvert.
\end{eqnarray*}
Lemma~\ref{lemRadonNikodymderivativeestimate} now implies that
\begin{eqnarray*}
&& P \biggl(\frac{C}{n^{d (\sfrac{1}{2}-\delta)}}\sum_{\Delta\in\widehat
{\Pi}_{n}}\biggl
\llvert\frac{1}{\llvert \Delta\rrvert }\sum_{x\in\Delta} \bigl[f (
\sigma_{x}\omega)-1 \bigr]\biggr\rrvert>\varepsilon\biggr)
\\
&&\qquad \leq P \biggl(\exists\Delta\in\widehat{\Pi}_{n}: \biggl\llvert
\frac
{1}{\llvert \Delta\rrvert }\sum_{x\in\Delta} \bigl[f (\sigma
_{x}\omega)-1 \bigr]\biggr\rrvert>\frac{\varepsilon}{C\cdot
C_{\varepsilon}^{d}} \biggr)
\\
&&\qquad \leq n^{d (\sfrac{1}{2}-\delta)}P \biggl(\biggl\llvert\frac
{1}{\llvert \Delta_{0}\rrvert }\sum
_{x\in\Delta_{0}} \bigl[f (\sigma_{x}\omega)-1 \bigr]\biggr
\rrvert>\frac{\varepsilon
}{C\cdot C_{\varepsilon}^{d}} \biggr)
\\
&&\qquad =n^{d (\sfrac{1}{2}-\delta
)}\cdot n^{-\xi(1 )}=n^{-\xi(1 )},
\end{eqnarray*}
where $\Delta_{0}$ is any choice for $\Delta_{0}\in\widehat{\Pi}_{n}$.
Therefore, by Borel--Cantelli, we have
\[
\limsup_{n\to\infty}\eqref{eqnormalizationconstnatlemma2}\leq
\limsup_{n\to\infty}\frac{C}{n^{d (\sfrac{1}{2}-\delta
)}}\sum
_{\Delta\in\widehat{\Pi}_{n}}\biggl\llvert\frac{1}{\llvert \Delta
\rrvert }\sum
_{x\in\Delta} \bigl[f (\sigma_{x}\omega)-1 \bigr]\biggr
\rrvert\leq\varepsilon.
\]
Combining all of the above, we see that $P$-almost surely
\[
\limsup_{n\to\infty}\llvert Z_{\omega,n}-1\rrvert\leq2
\varepsilon.
\]
Since $\varepsilon>0$ was arbitrary, the result follows.
\end{pf}
Before turning to the main lemma in the proof of Theorem~\ref{teoPrefactorthm},
we give two additional preliminary definitions needed in order to
construct the intermediate measures:

\begin{defn}
\label{DefsomenotationsformeasuresonZ^d} Let $\nu_{\omega}^{1}$
and $\nu_{\omega}^{2}$ be two probability measures on $\mathbb{Z}^{d}$,
which may depend on $\omega\in\Omega$.
\begin{longlist}[(2)]
\item[(1)] The\vspace*{2pt} $L^{1}$ distance of $\nu_{\omega}^{1}$ and $\nu_{\omega}^{2}$
is given by $\llVert \nu_{\omega}^{1}-\nu_{\omega}^{2}\rrVert
_{1}=\sum_{x\in\mathbb{Z}^{d}}\llvert \nu_{\omega}^{1}
(x )-\nu_{\omega}^{2} (x )\rrvert $
(note that this equals twice the total variation between $\nu_{\omega}^{1}$
and $\nu_{\omega}^{2}$).
\item[(2)] The environment-convolution of $\nu_{\omega}^{1}$ and $\nu
_{\omega}^{2}$
is a new probability measure on $\mathbb{Z}^{d}$, denoted $ (\nu
^{1}*\nu^{2} )_{\omega}$,
given by
\[
\bigl(\nu^{1}*\nu^{2} \bigr)_{\omega} (x )=\sum
_{y\in
\mathbb{Z}^{d}}\nu_{\omega}^{1} (y )
\nu_{\sigma
_{y}\omega}^{2} (x-y ).
\]
\end{longlist}
\end{defn}

We can now state the main lemma in the proof of Theorem~\ref{teoPrefactorthm}.
As already stated above, instead of comparing directly the $L^{1}$
distance of $\nu_{\omega}^{\mathrm{ann}\times\mathrm{pre},n}$ and $\nu
_{\omega}^{\mathrm{que},n}$,
that is, the annealed times the prefactor and the quenched probability
measures, appearing in Theorem~\ref{teoPrefactorthm}, we take a
more indirect approach and use two other measures as intermediaries.
This allows us to use previous results on the Radon--Nikodym derivative
and other relations between the quenched and annealed measures in
the evaluation of the $L^{1}$ distances.

\begin{lem}
\label{lemPrefactormainlemma} Fix $0<\delta<\varepsilon<\frac{1}{4}$,
and for $n\in\mathbb{N}$ abbreviate $k= \lceil n^{\epsilon
} \rceil$
and $l= \lceil n^{\delta} \rceil$. Fix a partition $\Pi$
of $\mathbb{Z}^{d}$ into boxes of side length $l$. With the notation
as in Definitions~\ref{DefLotsofmeasures}~and~\ref
{DefsomenotationsformeasuresonZ^d},
we have for $P$-almost every $\omega\in\Omega$:
\begin{longlist}[(2)]
\item[(1)]$\lim_{n\to\infty}\llVert \nu_{\omega}^{\mathrm{ann}\times\mathrm{pre},n}- (\nu^{\mathrm{ann}\times\mathrm{pre},n-k}*\nu^{\mathrm{que},k} )_{\omega}\rrVert _{1}=0$.
\item[(2)]$\lim_{n\to\infty}\llVert (\nu^{\mathrm{ann}\times\mathrm{pre},n-k}*\nu^{\mathrm{que},k} )_{\omega
}- (\nu_{\Pi}^{\mathrm{box}\mbox{-}\mathrm{que}\times\mathrm{pre},n-k}*\nu
^{\mathrm{que},k} )_{\omega}\rrVert _{1}=0$.
\item[(3)]$\lim_{n\to\infty}\llVert (\nu_{\Pi
}^{\mathrm{box}\mbox{-}\mathrm{que}\times\mathrm{pre},n-k}*\nu^{\mathrm{que},k} )_{\omega}- (\nu^{\mathrm{que},n-k}*\nu^{\mathrm{que},k} )_{\omega}\rrVert _{1}=0$.
\end{longlist}
\end{lem}

\begin{rem}
(1)~In the temporary notation from the beginning of this section, we have
$\rho_{1}= (\nu^{\mathrm{ann}\times\mathrm{pre},n-k}*\nu^{\mathrm{que},k} )_{\omega}$
and $\rho_{2}= (\nu^{\mathrm{box}\mbox{-}\mathrm{que}\times\mathrm{pre},n-k}*\nu^{\mathrm{que},k} )_{\omega}$.

(2) Note that by the Markov property of the quenched walk $\nu
_{\omega}^{\mathrm{que},\cdot}$
we have $ (\nu^{\mathrm{que},n-k}*\nu^{\mathrm{que},k}
)_{\omega}=\nu_{\omega}^{\mathrm{que},n}$.
\end{rem}

\begin{pf*}{Proof of Lemma~\ref{lemPrefactormainlemma} part (1)}
We need to show that
\begin{eqnarray*}
&& \lim_{n\to\infty}\sum_{x\in\mathbb{Z}^{d}}\biggl
\llvert\frac
{1}{Z_{n,\omega}}\mathbb{P}^{0} (X_{n}=x )f (\sigma
_{x}\omega)
\\
&&\qquad{} -\frac{1}{Z_{n-k,\omega}}\sum_{y\in\mathbb
{Z}^{d}}
\mathbb{P}^{0} (X_{n-k}=y )f (\sigma_{y}\omega
)P_{\omega}^{y} (X_{k}=x )\biggr\rrvert=0,
\end{eqnarray*}
which by Lemma~\ref{lemEvaluationofthenormalizingconstant}
(and the fact that we can restrict attention to $x\in\mathbb{Z}^{d}$
such that $\llVert x\rrVert _{1}\leq n$) is equivalent to
showing
\begin{eqnarray*}
&& \lim_{n\to\infty}\sum_{x\in[-n,n ]^{d}\cap\mathbb
{Z}^{d}}\biggl
\llvert\mathbb{P}^{0} (X_{n}=x )f (\sigma_{x}
\omega)
\\
&&\qquad{} -\sum_{y\in[-n,n ]^{d}\cap\mathbb
{Z}^{d}}\mathbb{P}^{0}
(X_{n-k}=y )f (\sigma_{y}\omega)P_{\omega}^{y}
(X_{k}=x )\biggr\rrvert=0.
\end{eqnarray*}
Denote $B_{n}= \{ x\in[-n,n ]^{d}\cap\mathbb
{Z}^{d}: \llVert x-\mathbb{E}^{0} [X_{n} ]\rrVert _{1}\leq R_{5} (n
)\sqrt{n} \} $.
By the triangle inequality,
%
\begin{eqnarray}
&& \sum_{x\in[-n,n ]^{d}\cap\mathbb{Z}^{d}}\biggl\llvert\mathbb{P}^{0}
(X_{n}=x )f (\sigma_{x}\omega)\nonumber
\\
&&\quad{} -\sum
_{y\in[-n,n ]^{d}\cap\mathbb{Z}^{d}}\mathbb{P}^{0} (X_{n-k}=y )f (
\sigma_{y}\omega)P_{\omega}^{y}
(X_{k}=x )\biggr\rrvert
\nonumber
\\
&&\qquad \leq \sum_{x\in B_{n}}\biggl\llvert\sum
_{y\in[-n,n ]^{d}\cap
\mathbb{Z}^{d}} \bigl[\mathbb{P}^{0} (X_{n}=x )-
\mathbb{P}^{0} (X_{n-k}=y ) \bigr]f (\sigma_{y}
\omega)P_{\omega}^{y} (X_{k}=x )\biggr\rrvert
\label{eqProofofmainLemmapart11}
\\
&&\quad\qquad{}+  \sum_{x\in B_{n}}\mathbb{P}^{0}
(X_{n}=x )\biggl\llvert f (\sigma_{x}\omega)-\sum
_{y\in[-n,n
]^{d}\cap\mathbb{Z}^{d}}f (\sigma_{y}\omega)P_{\omega
}^{y}
(X_{k}=x )\biggr\rrvert\label{eqProofofmainLemmapart12}
\\
&&\quad\qquad{}+  \sum_{x\in[-n,n ]^{d}\cap\mathbb{Z}^{d}\setminus
B_{n}}\biggl\llvert\mathbb{P}^{0}
(X_{n}=x )f (\sigma_{x}\omega)
\nonumber\\[-8pt]\label{eqProofofmainLemmapart10} \\[-8pt]\nonumber
&&\quad\qquad{} -\sum
_{y\in[-n,n ]^{d}\cap\mathbb
{Z}^{d}}\mathbb{P}^{0} (X_{n-k}=y )f (\sigma
_{y}\omega)P_{\omega}^{y} (X_{k}=x )
\biggr\rrvert.
\end{eqnarray}
Dealing with each of the terms separately [starting with (\ref
{eqProofofmainLemmapart11})],
by the annealed derivative estimation from Lemma~\ref
{lemgeneralannealedestimations}
\begin{eqnarray*}
\eqref{eqProofofmainLemmapart11} & \leq&\sum_{x\in B_{n}}
\frac
{Ck}{n^{\vfrac{d+1}{2}}} \biggl(\sum_{\operatorname{dist} (y,B_{n}
)\leq k}f (
\sigma_{y}\omega)P_{\omega}^{y}
(X_{k}=x ) \biggr)
\\
& \leq&\frac{Ck}{n^{\vfrac{d+1}{2}}}\sum_{\operatorname{dist}
(y,B_{n} )\leq k}f (
\sigma_{y}\omega)
\\
&=&\frac
{Ck}{n^{\sfrac {1}{2}}}\cdot\frac{1}{n^{\sfrac{d}{2}}}\sum
_{\operatorname{dist} (y,B_{n} )\leq k}f (\sigma_{y}\omega).
\end{eqnarray*}
By Lemma~\ref{lemRadonNikodymderivativeestimate} for $P$-almost
every $\omega$ and large enough $n$, we have
\[
\frac{1}{n^{\sfrac{d}{2}}}\sum_{\operatorname{dist} (y,B_{n} )\leq k}f (\sigma
_{y}\omega)\leq2R_{6} (n ).
\]
Thus, using the fact that $k=n^{\varepsilon}\ll n^{\sfrac{1}{4}}$,
it follows that the last term tends to zero $P$-almost surely as
$n$ tends to $\infty$.

Turning to deal with (\ref{eqProofofmainLemmapart12}), we
recall that by Proposition~\ref{propInvarianceoftheprefactor}
we have $f (\sigma_{x}\omega)-\sum_{y\in
[-n,n ]^{d}\cap\mathbb{Z}^{d}}\allowbreak f (\sigma
_{y}\omega)P_{\omega}^{y} (X_{k}=x )=0$
for every $x\in\mathbb{Z}^{d}$ such that $x+ [-k,k
]^{d}\cap\mathbb{Z}^{d}\subset[-n,n ]^{d}\cap\mathbb{Z}^{d}$.
In particular, denoting
$\tilde{B}_{n}=B_{n}\setminus\{ x\in\mathbb{Z}^{d}:
x+ [-k,k ]^{d}\cap\mathbb{Z}^{d}\subset[-n,n
]^{d}\cap\mathbb{Z}^{d} \} $
and using the annealed estimations from Lemma~\ref
{lemgeneralannealedestimations}
\begin{eqnarray*}
\eqref{eqProofofmainLemmapart12} & =&\sum_{x\in\tilde{B}_{n}}
\mathbb{P}^{0} (X_{n}=x )\biggl\llvert f (
\sigma_{x}\omega)-\sum_{y\in[-n,n ]^{d}\cap\mathbb{Z}^{d}}f (
\sigma_{y}\omega)P_{\omega}^{y}
(X_{k}=x )\biggr\rrvert
\\
& \leq&\sum_{x\in\tilde{B}_{n}}\mathbb{P}^{0}
(X_{n}=x )f (\sigma_{x}\omega)
\\
&&{} +\sum
_{x\in\tilde{B}_{n}}\mathbb{P}^{0} (X_{n}=x )\sum
_{y\in[-n,n ]^{d}\cap
\mathbb{Z}^{d}}f (\sigma_{y}\omega
)P_{\omega}^{y} (X_{k}=x )
\\
& \leq&\frac{C\llvert \tilde{B}_{n}\rrvert }{n^{\sfrac{d}{2}}}\cdot
\frac{1}{\llvert \tilde{B}_{n}\rrvert }\sum
_{x\in\tilde
{B}_{n}}f (\sigma_{x}\omega)+\frac{C\llvert \tilde
{B}_{n}\rrvert }{n^{\sfrac{d}{2}}}\cdot
\frac{1}{\llvert \tilde
{B}_{n}\rrvert }\sum_{x\in\tilde{B}_{n}}\sum
_{\llVert y-x\rrVert _{1}\leq k}f (\sigma_{y}\omega),
\end{eqnarray*}
where $\llvert \tilde{B}_{n}\rrvert $ is the size of $\tilde{B}_{n}$.
From the definition of $\tilde{B}_{n}$, it follows that $\llvert \tilde
{B}_{n}\rrvert \leq Ck (\sqrt{n}R_{5} (n ) )^{d-1}$
and, therefore,
\begin{eqnarray*}
\eqref{eqProofofmainLemmapart12} & =&\frac{C\llvert \tilde
{B}_{n}\rrvert }{n^{\sfrac{d}{2}}}\cdot\frac{1}{\llvert \tilde
{B}_{n}\rrvert }
\sum_{x\in\tilde{B}_{n}}f (\sigma_{x}\omega)+
\frac{C\llvert \tilde{B}_{n}\rrvert }{n^{\sfrac{d}{2}}}\cdot\frac
{1}{\llvert \tilde{B}_{n}\rrvert }\sum_{x\in\tilde{B}_{n}}
\sum_{\llVert y-x\rrVert _{1}\leq k}f (\sigma_{y}\omega)
\\
& \leq&\frac{CkR_{5}^{d-1} (n )}{n^{\sfrac {1}{2}}}\cdot\frac{1}{\llvert
\tilde{B}_{n}\rrvert }\sum
_{x\in\tilde
{B}_{n}}f (\sigma_{x}\omega)
\\
&&{}+\frac{CkR_{5}^{d-1}
(n )}{n^{\sfrac {1}{2}}}\cdot
\frac{1}{\llvert \tilde
{B}_{n}\rrvert }\sum_{x\in\tilde{B}_{n}}\sum
_{\llVert y-x\rrVert _{1}\leq k}f (\sigma_{y}\omega).
\end{eqnarray*}
Using again Lemma~\ref{lemRadonNikodymderivativeestimate}, and
the choice $k=n^{\varepsilon}=o (n^{\sfrac{1}{4}} )$, it
follows that both sums tends to zero as $n$ goes to infinity, $P$-almost
surely.

Finally, we turn to deal with (\ref{eqProofofmainLemmapart10}).
Using Lemma~\ref{lemgoodestimationforthelocation}, we have $\mathbb
{P} (X_{n}\notin B_{n} )=n^{-\xi(1 )}$.
Recalling also that $k= \lceil n^{\varepsilon} \rceil
=o (n )$,
we note that if $P_{\omega}^{y} (X_{k}=x )>0$ then $\llVert
x-y\rrVert _{1}\leq k$,
and thus for $x\in[-n,n ]^{d}\cap\mathbb
{Z}^{d}\setminus B_{n}$
and large enough $n$
\begin{eqnarray*}
\bigl\llVert y-\mathbb{E}^{0} [X_{n-k} ]\bigr\rrVert
_{1} & \geq&\bigl\llVert x-\mathbb{E}^{0} [X_{n} ]
\bigr\rrVert_{1}-\bigl\llVert\mathbb{E}^{0}
[X_{n} ]-\mathbb{E}^{0} [X_{n-k} ]\bigr\rrVert
_{1}-\llVert x-y\rrVert_{1}
\\
& \geq&\sqrt{n}R_{5} (n )-2k\geq\tfrac{1}{2}\sqrt
{n}R_{5} (n ).
\end{eqnarray*}
This, however, due to Lemma~\ref{lemgoodestimationforthelocation},
yields $\mathbb{P}^{0} (X_{n-k}=y )=n^{-\xi(1 )}$
and, therefore,
\begin{eqnarray*}
\eqref{eqProofofmainLemmapart10} & \leq&\sum_{x\in[-n,n
]^{d}\cap\mathbb{Z}^{d}\setminus B_{n}}
\mathbb{P}^{0} (X_{n}=x )f (\sigma_{x}\omega)
\\
&&{} +\sum
_{x\in
[-n,n ]^{d}\cap\mathbb{Z}^{d}\setminus B_{n}}\sum_{y\in
[-n,n ]^{d}\cap\mathbb{Z}^{d}}
\mathbb{P}^{0} (X_{n-k}=y )f (\sigma_{y}\omega
)P_{\omega
}^{y} (X_{k}=x )
\\
& \leq& n^{-\xi(1 )}\sum_{x\in[-n,n ]^{d}\cap
\mathbb{Z}^{d}\setminus B_{n}}f (
\sigma_{x}\omega)
\\
&&{} +n^{-\xi(1 )}\sum_{x\in[-n,n ]^{d}\cap
\mathbb{Z}^{d}\setminus B_{n}}
\sum_{y\in[-n,n ]^{d}\cap
\mathbb{Z}^{d}}f (\sigma_{y}\omega
)P_{\omega}^{y} (X_{k}=x )
\\
& \leq&2\cdot n^{-\xi(1 )}\sum_{x\in[-n,n
]^{d}\cap\mathbb{Z}^{d}}f (
\sigma_{x}\omega).
\end{eqnarray*}
By Lemma~\ref{lemRadonNikodymderivativeestimate}, we have $P
(\sum_{x\in[-n,n ]^{d}\cap\mathbb{Z}^{d}}f (\sigma
_{x}\omega)\leq2n^{d} )>1-n^{-\xi(1 )}$
and, therefore, by the Borel--Cantelli lemma for large enough $n$
\[
\eqref{eqProofofmainLemmapart10}\leq4n^{-\xi(1 )}\cdot n^{d}=n^{-\xi
(1 )}
\mathop{\longrightarrow}_{n\to\infty}0.
\]\upqed
\end{pf*}

\begin{pf*}{Proof of Lemma~\ref{lemPrefactormainlemma} part (2)}
Since
\begin{eqnarray*}
&& \bigl\llVert\bigl(\nu^{\mathrm{ann}\times\mathrm{pre},n-k}*\nu^{\mathrm{que},k} \bigr)_{\omega}-
\bigl(\nu^{\mathrm{box}\mbox{-}\mathrm{que}\times\mathrm{pre},n-k}*\nu^{\mathrm{que},k} \bigr)_{\omega}\bigr\rrVert
_{1}
\\
&&\qquad \leq\bigl\llVert\nu_{\omega}^{\mathrm{ann}\times\mathrm{pre},n-k}-
\nu_{\omega}^{\mathrm{box}\mbox{-}\mathrm{que} \times \mathrm{pre},n-k}\bigr\rrVert_{1},
\end{eqnarray*}
it is enough to deal with the right-hand side and show that
\begin{eqnarray*}
&& \lim_{n\to\infty}\mathop{\sum_{
x\in\mathbb{Z}^{d}}}_{
x\leftrightarrow n-k}
\biggl\llvert\frac{1}{Z_{n-k,\omega}}\mathbb{P}^{0} (X_{n-k}=x )f
(\sigma_{x}\omega)
\\
&&\qquad{} -\frac{f
(\sigma_{x}\omega)}{\sum_
{
y\in\Delta_{x},
y\leftrightarrow n-k
}f (\sigma_{y}\omega)} P_{\omega}^{0}
(X_{n-k}\in\Delta_{x} )\biggr\rrvert=0,\qquad P\mbox{-a.s.}
\end{eqnarray*}
Using Lemma~\ref{lemEvaluationofthenormalizingconstant} once
more, this is equivalent to showing
%
\begin{eqnarray}\label{eqLemmameasurespart21}
&& \lim_{n\to\infty}\mathop{\sum_{
x\in\mathbb{Z}^{d}}}_{
x\leftrightarrow n-k}
f (\sigma_{x}\omega)\biggl\llvert\mathbb{P}^{0}
(X_{n-k}=x )
\nonumber\\[-8pt]\\[-8pt]\nonumber
&&\qquad{} -P_{\omega}^{0} (X_{n-k}\in
\Delta_{x} )\frac{1}{\sum_{
y\in\Delta_{x},
y\leftrightarrow n-k}
f (\sigma_{y}\omega)}\biggr\rrvert=0,\qquad P\mbox{-a.s.}
\end{eqnarray}
Denoting $B_{n}= \{ x\in[-n,n ]^{d}\cap\mathbb
{Z}^{d}: \llVert x-\mathbb{E}^{0} [X_{n} ]\rrVert _{1}\leq
C_{\varepsilon}\sqrt{n} \} $
[with $C_{\varepsilon}$ such that $\mathbb{P}^{0} (\llVert
X_{n}-\mathbb{E}^{0} [X_{n} ]\rrVert _{1}>\frac
{C_{\varepsilon}}{2}\sqrt{n} )<\varepsilon$
for large enough $n$] and using the triangle inequality the last
sum is bounded by
%
\begin{eqnarray}
&& \mathop{\sum_{
x\in[-n,n ]\cap\mathbb{Z}^{d}\setminus B_{n}}}_{
x\leftrightarrow n-k} f (
\sigma_{x}\omega)\biggl\llvert\mathbb{P}^{0}
(X_{n-k}=x )
\nonumber\\[-8pt]\label{eqLemmameasurespart20} \\[-8pt]\nonumber
&&\qquad{} -\frac{1}{\sum_
{y\in\Delta_{x},
y\leftrightarrow n-k}
f (\sigma_{y}\omega)}P_{\omega}^{0}
(X_{n-k}\in\Delta_{x} )\biggr\rrvert
\\
&&\qquad{}+  \mathop{\sum_{x\in B_{n}}}_{
x\leftrightarrow n-k} f (
\sigma_{x}\omega)\biggl\llvert\mathbb{P}^{0}
(X_{n-k}=x )-\frac{2}{\llvert \Delta_{x}\rrvert }\mathbb{P}^{0}
(X_{n-k}\in\Delta_{x} )\biggr\rrvert\label
{eqLemmameasurespart22}
\\
&&\qquad{}+  \mathop{\sum_{x\in B_{n}}}_{
x\leftrightarrow n-k} f (
\sigma_{x}\omega)\biggl\llvert\frac{2}{\llvert \Delta
_{x}\rrvert }\mathbb{P}^{0}
(X_{n-k}\in\Delta_{x} )
\nonumber\\[-8pt]\label{eqLemmameasurespart23} \\[-8pt]\nonumber
&&\qquad{} -\frac
{1}{\sum_{
y\in\Delta_{x},
y\leftrightarrow n-k}
f (\sigma_{y}\omega)}
\mathbb{P}^{0} (X_{n-k}\in\Delta_{x} )\biggr\rrvert
\\
&&\qquad+  \mathop{\sum_{
x\in B_{n}}}_{
x\leftrightarrow n-k} f (
\sigma_{x}\omega)\biggl\llvert\frac{1}{\sum_{
y\in\Delta_{x},
y\leftrightarrow n-k}
f (\sigma_{y}\omega)}\mathbb{P}^{0}
(X_{n-k}\in\Delta_{x} )
\nonumber\\[-8pt]\label{eqLemmameasurespart24} \\[-8pt]\nonumber
&&\qquad{} -\frac{1}{\sum_{y\in\Delta_{x},
y\leftrightarrow n-k}
f (\sigma_{y}\omega)}P_{\omega}^{0}
(X_{n-k}\in\Delta_{x} )\biggr\rrvert.
\end{eqnarray}
Dealing with each of the terms separately, and starting with (\ref
{eqLemmameasurespart20}),
we have the following estimate:
\begin{eqnarray*}
\eqref{eqLemmameasurespart20}&\leq&\sum_{x\in[-n,n
]^{d}\cap\mathbb{Z}^{d}\setminus B_{n}}
\mathbb{P}^{0} (X_{n-k}=x )f (\sigma_{x}\omega)
\\
&&{} +P_{\omega
}^{0} \bigl(\bigl\llVert X_{n-k}-
\mathbb{E}^{0} [X_{n} ]\bigr\rrVert_{1}>C_{\varepsilon}
\sqrt{n} \bigr).
\end{eqnarray*}
The term $\sum_{x\in[-n,n ]^{d}\cap\mathbb
{Z}^{d}\setminus B_{n}}\mathbb{P}^{0} (X_{n-k}=x )f
(\sigma_{x}\omega)$
goes to zero as $n$ goes to infinity by the same argument used to
bound (\ref{eqnormalizationconstantlemma0}) in Lemma~\ref
{lemEvaluationofthenormalizingconstant}.
For the second term, Claim~\ref{Clmsimpleclaimannealedtoquenched}
implies that for a set of environments, with $P$ probability $>1-\sqrt
{\varepsilon}$,
for large enough $n$
\[
P_{\omega}^{0} \bigl(\bigl\llVert X_{n-k}-
\mathbb{E}^{0} [X_{n} ]\bigr\rrVert_{1}>C_{\varepsilon}
\sqrt{n} \bigr)\leq P_{\omega}^{0} \biggl(\bigl\llVert
X_{n}-\mathbb{E}^{0} [X_{n} ]\bigr\rrVert
_{1}>\frac{C_{\varepsilon}}{2}\sqrt{n} \biggr)\leq\sqrt{\varepsilon}.
\]
Since $\varepsilon>0$ was arbitrary, this proves that the first term
goes to zero as $n$ goes to infinity.

Turning to (\ref{eqLemmameasurespart22}), the annealed derivative
estimations from Lemma~\ref{lemgeneralannealedestimations} yields
\begin{eqnarray*}
\eqref{eqLemmameasurespart22} & \leq& C\cdot\mathop{\sum
_{x\in B_{n}}}_{
x\leftrightarrow n-k} f (\sigma_{x}\omega)
\frac{2}{\llvert \Delta_{x}\rrvert }\mathop{\sum_{
y\in\Delta_{x}}}_{
y\leftrightarrow n-k}
\bigl\llvert\mathbb{P} (X_{n-k}=x )-\mathbb{P} (X_{n-k}=y )
\bigr\rrvert
\\
& \leq& C\cdot\mathop{\sum_{
x\in B_{n}}}_{
x\leftrightarrow n-k} f
(\sigma_{x}\omega)\frac{2}{\llvert \Delta_{x}\rrvert }\mathop{\sum
_{y\in\Delta_{x}}}_{
y\leftrightarrow n-k} \frac{C}{ (n-k )^{\vfrac{d+1}{2}}}\llVert
x-y\rrVert
_{1}
\\
& \overset{ (1 )} {\leq}& C\cdot\sum_{x\in B_{n}}f (
\sigma_{x}\omega)\frac{1}{\llvert \Delta_{x}\rrvert }\sum_{y\in\Delta_{x}}
\frac{C}{ (n-k )^{\vfrac{d+1}{2}}}\cdot dn^{\delta}
\\
& =&\frac{C'\cdot n^{\sfrac{d}{2}+\delta}}{ (n-k )^{\vfrac
{d+1}{2}}}\cdot\biggl(\frac{1}{n^{\sfrac{d}{2}}}\sum
_{x\in
B_{n}}f (\sigma_{x}\omega) \biggr)
\stackrel{n\to\infty}{\longrightarrow}0, \qquad P\mbox{-a.s.},
\end{eqnarray*}
where for $ (1 )$ we used the fact that the side length
of every cube $\Delta$ is $n^{\delta}$ and for the limit we used
Lemma~\ref{lemRadonNikodymderivativeestimate}, the fact that
$k=n^{\epsilon}=o (n^{\sfrac{1}{4}} )$ and also that
$\delta<\epsilon<\frac{1}{2}$.

Next, we deal with (\ref{eqLemmameasurespart23}). Denoting
$\widehat{\Pi}_{n}= \{ \Delta\in\Pi: \Delta\cap B_{n}\neq
\varnothing\} $
and using the annealed derivative estimations from Lemma~\ref
{lemRadonNikodymderivativeestimate}
we conclude that
\begin{eqnarray*}
\eqref{eqLemmameasurespart23} & =&\mathop{\sum_{
x\in B_{n}}}_{
x\leftrightarrow n-k}
f (\sigma_{x}\omega)\frac{2}{\llvert \Delta_{x}\rrvert }\mathbb{P} (X_{n-k}
\in\Delta_{x} )\biggl\llvert1-\frac
{1}{\sfrac{2}{\llvert \Delta_{x}\rrvert }\sum_{
y\in\Delta_{x},
y\leftrightarrow n-k}f (\sigma_{y}\omega)}\biggr\rrvert
\\
& \leq&\frac{C}{ (n-k )^{\sfrac{d}{2}}}\mathop{\sum_{
x\in B_{n}}}_{
x\leftrightarrow n-k}
f (\sigma_{x}\omega)\biggl\llvert1-\frac{1}{\sfrac{2}{\llvert \Delta
_{x}\rrvert }\sum_{
y\in\Delta_{x},
y\leftrightarrow n-k}
f (\sigma_{y}\omega)}\biggr\rrvert
\\
& \leq& C \biggl(1-\frac{k}{n} \biggr)^{-\sfrac{d}{2}}\frac{1}{n^{\sfrac{d}{2}}}
\\
&&{}\times
\sum_{\Delta\in\widehat{\Pi}_{n}}\sum_{x\in\Delta
}f (
\sigma_{x}\omega)\biggl\llvert1-\frac{1}{\sfrac{2}{\llvert \Delta
\rrvert }\sum_{
y\in\Delta_{x},
y\leftrightarrow n-k}f (\sigma_{y}\omega)}\biggr\rrvert
\\
& =&C \biggl(1-\frac{k}{n} \biggr)^{-\sfrac{d}{2}}\frac{1}{n^{\sklfrac
{d}{2} (1-2\delta)}}\sum
_{\Delta\in\widehat{\Pi
}_{n}}\biggl\llvert\frac{1}{\llvert \Delta\rrvert }\mathop{\sum
_{
x\in\Delta}}_{
x\leftrightarrow n-k} f (\sigma_{x}\omega
)-1\biggr\rrvert.
\end{eqnarray*}
Using the same argument that was used to bound~(\ref{eqnormalizationconstnatlemma2}),
we get that the last term goes to zero as $n$ goes to infinity $P$-a.s.
Finally, we estimate (\ref{eqLemmameasurespart24}).
\begin{eqnarray*}
\eqref{eqLemmameasurespart24} & \leq&\mathop{\sum_{
x\in B_{n}}}_{
x\leftrightarrow n-k}
\frac{f (\sigma_{x}\omega)}{\sum_{y\in\Delta_{x},
y\leftrightarrow n-k}
f (\sigma_{y}\omega)}\bigl\llvert\mathbb{P}^{0} (X_{n-k}\in
\Delta_{x} )-P_{\omega}^{0} (X_{n-k}\in\Delta
_{x} )\bigr\rrvert
\\
& =&\sum_{\Delta\in\widehat{\Pi}_{n}}\bigl\llvert\mathbb{P}^{0}
(X_{n-k}\in\Delta)-P_{\omega}^{0} (X_{n-k}
\in\Delta)\bigr\rrvert.
\end{eqnarray*}
The last term, however, is bounded by $Cn^{-\sklfrac{1}{3}\delta}$ by
Proposition~\ref{PropLorentztransformation-1} for $P$-almost every
$\omega$ and large enough $n$, and thus goes to zero as $n$ goes
to infinity.
\end{pf*}

Part (3) of Lemma~\ref{lemPrefactormainlemma} will follow from
the following more general lemma.

\begin{lem}
\label{lemLemmaforpart3} Let $x,y\in\mathbb{Z}^{d}$ satisfy
$\llVert x-y\rrVert _{1}\leq n^{\theta}$ for some $\theta
<\frac{1}{2}$.
Then the set of environments for which
\[
\bigl\llvert P_{\omega}^{x} (X_{n}=z
)-P_{\omega}^{y} (X_{n}=z )\bigr\rrvert
=n^{-\xi(1 )}\qquad\forall z\in\mathbb{Z}^{d}
\]
has $P$ probability $\geq1-n^{-\xi(1 )}$.
\end{lem}

\begin{pf}
Fix $\theta<\theta'<1$ such that $\theta'<\frac{d+1}{2}\theta$,
$M\in\mathbb{N}$ and a partition $\Pi$ of $\mathbb{Z}^{d}$ into
boxes of side length $M$. By Theorem~\ref{teoTotalvariationonsmallboxes},
if $M$ is large enough, then the event
\begin{eqnarray*}
G (n,M ) &=& \biggl\{ \omega\in\Omega: \sum_{\Delta\in
\Pi}
\bigl\llvert P_{\omega}^{w} (X_{ \lceil n^{\theta}
\rceil}\in\Delta)-
\mathbb{P}^{w} (X_{ \lceil
n^{\theta} \rceil}\in\Delta)\bigr\rrvert<
\frac{1}{8},
\\
&&{} \forall w \mbox{ s.t. }\llVert w-x\rrVert_{1}\leq
n^{2} \biggr\}
\end{eqnarray*}
satisfies $P (G (n,M ) )=1-n^{-\xi(1 )}$.
In particular, using Lemma~\ref{lemgeneralannealedestimations},
whenever $\llVert y-x\rrVert _{1}\leq n^{\theta'}$, for large
enough $n$ we have
\begin{eqnarray*}
&& \bigl\llvert P_{\omega}^{x} (X_{ \lceil n^{\theta} \rceil
}\in\Delta
)-P_{\omega}^{y} (X_{ \lceil n^{\theta
} \rceil}\in\Delta)\bigr\rrvert
\\
&&\qquad \leq \bigl\llvert P_{\omega
}^{x} (X_{ \lceil n^{\theta} \rceil}\in\Delta
)-\mathbb{P}^{x} (X_{ \lceil n^{\theta} \rceil}\in\Delta)\bigr\rrvert
+\bigl
\llvert\mathbb{P}^{x} (X_{ \lceil
n^{\theta} \rceil}\in\Delta)-
\mathbb{P}^{y} (X_{ \lceil n^{\theta} \rceil}\in\Delta)\bigr\rrvert
\\
&&\quad\qquad{}+ \bigl\llvert P_{\omega}^{y} (X_{ \lceil n^{\theta}
\rceil}\in\Delta)-
\mathbb{P}^{y} (X_{ \lceil
n^{\theta} \rceil}\in\Delta)\bigr\rrvert
\\
&&\qquad \leq\frac{1}{4}+\frac{Cn^{\theta'}}{n^{\sklvfrac{d+1}{2}\theta
}}=\frac{1}{4}+Cn^{\theta'-\sklvfrac{d+1}{2}\theta}<
\frac{1}{2}.
\end{eqnarray*}
Consequently,\vspace*{1pt} there exists a coupling of $P_{\omega}^{x}
(X_{ \lceil n^{\theta} \rceil}\in\cdot)$
and $P_{\omega}^{y} (X_{ \lceil n^{\theta} \rceil
}\in\cdot)$
on $\Pi\times\Pi$ denoted $\tilde{\Xi}^{x,y}=\tilde{\Xi
}_{n,\theta,\omega}^{x,y}$
such that $\tilde{\Xi}^{x,y} ( \{ (\Delta,\Delta
): \Delta\in Q \} )>\frac{1}{2}$.
Using the uniform ellipticity, and the last coupling we can construct
a new coupling $\Xi_{1}^{x,y}$ of $P_{\omega}^{x} (X_{
\lceil n^{\theta} \rceil}=\cdot)$
and $P_{\omega}^{y} (X_{ \lceil n^{\theta} \rceil
}=\cdot)$
on $\mathbb{Z}^{d}\times\mathbb{Z}^{d}$ such that $\Xi
_{1}^{x,y} (\Lambda)\geq\frac{1}{2}\eta^{2dM}$,
where $\Lambda= \{ (z,z ): z\in\mathbb
{Z}^{d} \} $
(for a more detailed explanation on the construction, see the proof
of Lemma~\ref{Lemmathereexistsnonsingular}). Next, for $k\geq2$
we construct inductively a new coupling of $P_{\omega}^{x}
(X_{k \lceil n^{\theta} \rceil}=\cdot)$
and $P_{\omega}^{y} (X_{k \lceil n^{\theta} \rceil
}=\cdot)$
on $\mathbb{Z}^{d}\times\mathbb{Z}^{d}$ such that $\Xi_{
\lceil\log^{2}n \rceil\cdot\lceil n^{\theta'}
\rceil}^{x,y} (\Delta)=1-n^{-\xi(1 )}$.
The construction goes as follows: first, note that if $a,b\in\mathbb{Z}^{d}$
are any pair of points such that $\llVert a-x\rrVert
_{1},\llVert b-x\rrVert _{1}\leq n^{2}$,
then by the same reasoning, we can construct a coupling of $P_{\omega
}^{a} (X_{ \lceil n^{\theta} \rceil}=\cdot)$
and $P_{\omega}^{b} (X_{ \lceil n^{\theta} \rceil
}=\cdot)$
on $\mathbb{Z}^{d}\times\mathbb{Z}^{d}$, denoted $\Xi_{1}^{a,b}$,
such that $\Xi_{1}^{a,b} (\Lambda)>\frac{1}{2}\eta^{2dM}$.
Next, assuming the coupling $\Xi_{k-1}^{x,y}$ was constructed we
define $\Xi_{k}^{x,y}$ via the following procedure: choose a pair
of points $ (a,b )$ according to the previous coupling $\Xi
_{k-1}^{x,y}$.
If $a=b$, couple the random walks together (to go along the same
path) for additional $ \lceil n^{\theta} \rceil$ steps.
If $a\neq b$ and $\llVert a-b\rrVert \leq n^{\theta'}$, couple
the random walks using the coupling $\Xi_{1}^{a,b}$. Finally, if
$\llVert a-b\rrVert >n^{\theta'}$ we let the random walks
evolve independently. Formally, this is given by
\begin{eqnarray*}
\Xi_{k}^{x,y} (w_{1},w_{2} ) & =&\sum
_{a,b\in\mathbb
{Z}^{d}}\Xi_{k-1}^{x,y} (a,b )
\bigl[\ind_{a=b}\ind_{w_{1}=w_{2}}P_{\omega}^{a}
(X_{ \lceil n^{\theta}
\rceil}=w_{1} )
\\[-3pt]
&&{} +\ind_{0<\llVert a-b\rrVert _{1}\leq
n^{\theta'}}\cdot
\Xi_{1}^{a,b} (w_{1},w_{2} )
\\[-3pt]
&&{} +\ind_{\llVert a-b\rrVert _{1}>n^{\theta'}}\cdot P_{\omega
}^{a}
(X_{n^{\theta}}=w_{1} )P_{\omega}^{b}
(X_{n^{\theta}}=w_{2} ) \bigr].
\end{eqnarray*}
It is not hard to verify that this indeed defines a coupling of
$P_{\omega}^{x} (X_{k \lceil n^{\theta} \rceil
}=\cdot)$
and $P_{\omega}^{y} (X_{k \lceil n^{\theta} \rceil
}=\cdot)$,
and that in fact by the definition of $\theta'$ and the assumption
$\llVert x-y\rrVert _{1}\leq n^{\theta}$ that $\Xi
_{k}^{x,y} (a,b )=0$
whenever $\llVert a-b\rrVert _{1}\geq n^{\theta'}$ and $n$
is large enough. Therefore,
\begin{eqnarray*}
\Xi_{k}^{x,y} (\Lambda) & =&\sum
_{w\in\mathbb
{Z}^{d}}\sum_{a,b\in\mathbb{Z}^{d}}
\Xi_{k-1}^{x,y} (a,b ) \bigl[\ind_{a=b}
\ind_{w=w}P_{\omega}^{a} (X_{ \lceil
n^{\theta} \rceil}=w )
\\[-3pt]
&&{} +
\ind_{0<\llVert a-b\rrVert _{1}\leq n^{\theta'}}\cdot\Xi_{1}^{a,b}
(w,w )
\\[-3pt]
&&{} +\ind_{\llVert a-b\rrVert >n^{\theta'}}\cdot P_{\omega
}^{a}
(X_{ \lceil n^{\theta} \rceil}=w )P_{\omega}^{b} (X_{ \lceil n^{\theta}
\rceil
}=w )
\bigr]
\\[-3pt]
& \geq&\sum_{w\in\mathbb{Z}^{d}}\sum_{a,b\in\mathbb{Z}^{d}}
\Xi_{k-1}^{x,y} (a,b ) \bigl[\ind_{a=b}P_{\omega}^{a}
(X_{ \lceil n^{\theta} \rceil}=w )
\\[-3pt]
&&{}+\ind_{0<\llVert a-b\rrVert
_{1}\leq n^{\theta'}}\cdot\Xi_{1}^{a,b}
(w,w ) \bigr]
\\[-3pt]
& =&\Xi_{k-1}^{x,y} (\Lambda)+\sum
_{a,b\in\mathbb
{Z}^{d}}\Xi_{k-1}^{x,y} (a,b )
\ind_{0<\llVert
a-b\rrVert _{1}\leq n^{\theta'}}\cdot\Xi_{1}^{a,b} (\Delta)
\\[-3pt]
& \geq&\Xi_{k-1}^{x,y} (\Lambda)+\sum
_{a,b\in\mathbb
{Z}^{d}}\Xi_{k-1}^{x,y} (a,b )
\ind_{0<\llVert
a-b\rrVert _{1}\leq n^{\theta'}}\cdot\frac{1}{2}\eta^{2dM}
\\[-3pt]
& =&\Xi_{k-1}^{x,y} (\Lambda)+\frac{1}{2}
\eta^{2dM}\Xi_{k-1}^{x,y} \bigl( \bigl\{ (a,b ):
0<\llVert a-b\rrVert_{1}\leq n^{\theta'} \bigr\} \bigr)
\\[-3pt]
& =&\Xi_{k-1}^{x,y} (\Lambda)+\frac{1}{2}
\eta^{2dM} \bigl(1-\Xi_{k-1}^{x,y} (\Lambda)-
\Xi_{k-1}^{x,y} \bigl( \bigl\{ (a,b ): \llVert a-b\rrVert
_{1}>n^{\theta
'} \bigr\} \bigr) \bigr)
\\[-3pt]
& \geq&\Xi_{k-1}^{x,y} (\Lambda)+\frac{1}{2}\eta
^{2dM} \bigl(1-\Xi_{k-1}^{x,y} (\Lambda) \bigr).
\end{eqnarray*}
Fixing $r>0$, as long as $\Xi_{j}^{x,y} (\Lambda)<1-n^{-r}$
for $j\leq k$ the last inequality gives
\[
\frac{\Xi_{k}^{x,y} (\Lambda)}{\Xi_{k-1}^{x,y}
(\Lambda)}\geq1+\frac{1}{2}\eta^{2dM}\cdot
\frac{n^{-r}}{1-n^{-r}},
\]
which implies that $\Xi_{k}^{x,y} (\Lambda)$ grows exponentially
in this regime. Hence, for some $C=C (r )<\infty$ we have
$\Xi_{ \lceil C\log n \rceil}^{x,y} (\Lambda
)>1-n^{-r}$
and in particular since $ \{ \Xi_{k}^{x,y} (\Lambda
) \} _{k\geq1}$
is nondecreasing $\Xi_{ \lceil\log^{2}n \rceil
}^{x,y} (\Lambda)>\Xi_{ \lceil C\log n \rceil
}^{x,y} (\Lambda)>1-n^{-r}$
for every $r\in\mathbb{N}$ and large enough $n$, that is, $\Xi
_{ \lceil\log^{2}n \rceil}^{x,y} (\Lambda
)=1-n^{-\xi(1 )}$.
We can\vspace*{1pt} now construct a coupling of $P_{\omega}^{x} (X_{n}=\cdot
)$
and $P_{\omega}^{y} (X_{n}=\cdot)$ on $\mathbb{Z}^{d}$,
by using the coupling $\Xi_{ \lceil\log^{2}n \rceil}^{x,y}$
until time $ \lceil\log^{2}n \rceil\cdot\lceil
n^{\theta} \rceil$.
Formally, if they coincided until time $ \lceil\log^{2}n
\rceil\cdot\lceil n^{\theta} \rceil$
we couple them together (to go along the same path) until time $n$,
or if not to move independently until time $n$. This yields a coupling
such that $\sum_{z\in\mathbb{Z}}\llvert P_{\omega}^{x}
(X_{n}=z )-P_{\omega}^{y} (X_{n}=z )\rrvert <n^{-\xi
(1 )}$
as required.
\end{pf}

\begin{pf*}{Proof of Lemma~\ref{lemPrefactormainlemma} part (3)}
Written explicitly (after some manipulations)
%
\begin{eqnarray}\label{eqLemmameasurespart31}
&& \bigl\llVert\bigl(\nu^{\mathrm{box}\mbox{-}\mathrm{que}\times \mathrm{pre},n-k}*\nu^{\mathrm{que},k}
\bigr)_{\omega}- \bigl(\nu^{\mathrm{que},n-k}*\nu^{\mathrm{que},k}
\bigr)_{\omega}\bigr\rrVert
\nonumber
\\
&&\qquad =  \mathop{\sum_{
x\in\mathbb{Z}^{d}}}_{
x\leftrightarrow n} \biggl
\llvert\sum_{\Delta\in\Pi}\mathop{\sum
_{y\in\Delta}}_{
y\leftrightarrow n-k} P_{\omega}^{0}
(X_{n-k}\in\Delta)P_{\omega}^{y}
(X_{k}=x ) \nonumber
\\
&&\quad\qquad{}\times \biggl[\frac{f (\sigma_{y}\omega
)}{\sum_{z\in\Delta} f (\sigma
_{z}\omega)}-P_{\omega}^{0}
(X_{n-k}=y | X_{n-k}\in\Delta) \biggr]\biggr\rrvert
\\
&&\qquad\leq \mathop{\sum_{
x\in\mathbb{Z}^{d}}}_{
x\leftrightarrow n} \sum
_{\Delta\in\Pi}P_{\omega}^{0}
(X_{n-k}\in\Delta)\biggl\llvert\mathop{\sum
_{y\in\Delta}}_{
y\leftrightarrow n-k} P_{\omega}^{y}
(X_{k}=x ) \biggl[\frac{f (\sigma
_{y}\omega)}{\sum_{z\in\Delta}f (\sigma_{z}\omega)}\nonumber
\\
&&\quad\qquad{} -P_{\omega}^{0}
(X_{n-k}=y|  X_{n-k}\in\Delta) \biggr]\biggr\rrvert.\nonumber
\end{eqnarray}
Noticing that for every $\Delta\in\Pi$ and $x\in\mathbb{Z}^{d}$
we have
\[
\mathop{\sum_{
y\in\Delta}}_{
y\leftrightarrow n-k} \biggl(
\frac{1}{\llvert \Delta\rrvert }\sum_{w\in\Delta}P_{\omega
}^{w}
(X_{k}=x ) \biggr) \biggl[\frac{f (\sigma
_{y}\omega)}{\sum_{z\in\Delta}f (\sigma_{z}\omega)}-P_{\omega}^{0}
(X_{n-k}=y| X_{n-k}\in\Delta) \biggr]=0,
\]
it follows that (\ref{eqLemmameasurespart31}) equals
\begin{eqnarray*}
&& \mathop{\sum_{x\in\mathbb{Z}^{d}}}_{x\leftrightarrow n}\sum
_{\Delta\in\Pi}P_{\omega}^{0} (X_{n-k}\in
\Delta)\biggl\llvert\mathop{\sum_{y\in\Delta}}_{y\leftrightarrow n-k}
\biggl[P_{\omega}^{y} (X_{k}=x )- \biggl(
\frac
{2}{\llvert \Delta\rrvert } \mathop{\sum_{w\in\Delta
}}_{w\leftrightarrow n}
P_{\omega}^{w} (X_{k}=x ) \biggr) \biggr]
\\
&&\quad{}\times \biggl[\frac{f (\sigma
_{y}\omega)}{\sum_{
z\in\Delta,
z\leftrightarrow n-k}
f (\sigma_{z}\omega)}-P_{\omega}^{0} (X_{n-k}=y
|  X_{n-k}\in\Delta) \biggr]\biggr\rrvert
\\
&&\qquad =  \mathop{\sum_{x\in\mathbb{Z}^{d}}}_{x\leftrightarrow n}\sum
_{\Delta\in\Pi}P_{\omega}^{0}
(X_{n-k}\in\Delta)\biggl\llvert\mathop{\sum
_{y\in\Delta}}_{y\leftrightarrow n-k} \biggl[\frac{2}{\llvert \Delta
\rrvert } \mathop{
\sum_{w\in\Delta}}_{w\leftrightarrow n} \bigl(P_{\omega}^{y}
(X_{k}=x )-P_{\omega
}^{w} (X_{k}=x )
\bigr) \biggr]
\\
&&\quad\qquad{}\times \biggl[\frac{f (\sigma
_{y}\omega)}{\sum_{
z\in\Delta,
z\leftrightarrow n-k}f (\sigma_{z}\omega)}-P_{\omega}^{0}
(X_{n-k}=y| X_{n-k}\in\Delta) \biggr]\biggr\rrvert
\\
&&\qquad\leq \mathop{\sum_{x\in\mathbb{Z}^{d}}}_{x\leftrightarrow n}\sum
_{\Delta\in\Pi}P_{\omega}^{0}
(X_{n-k}\in\Delta) \mathop{\sum_{y\in\Delta}}_{y\leftrightarrow n-k}
\biggl\llvert\frac{2}{\llvert \Delta\rrvert } \mathop{\sum_{w\in
\Delta}}_{w\leftrightarrow n}
\bigl(P_{\omega}^{y} (X_{k}=x )-P_{\omega
}^{w}
(X_{k}=x ) \bigr)\biggr\rrvert
\\
&&\quad\qquad{}\times \biggl\llvert\frac{f (\sigma
_{y}\omega)}{\sum_{
z\in\Delta,
z\leftrightarrow n-k}
f (\sigma_{z}\omega)}-P_{\omega}^{0}
(X_{n-k}=y| X_{n-k}\in\Delta)\biggr\rrvert
\\
&&\qquad=  \sum_{\Delta\in\Pi}\mathop{\mathop{\sum
_{x\in\mathbb{Z}^{d}}}_{\operatorname{dist} (x,\Delta)\leq k}}
_{x\leftrightarrow n}P_{\omega}^{0}
(X_{n-k}\in\Delta) \mathop{\sum_{y\in\Delta}}_
{y\leftrightarrow n-k}
\biggl\llvert\frac{2}{\llvert \Delta\rrvert }
\mathop{\sum_{w\in
\Delta}}_{w\leftrightarrow n}
\bigl(P_{\omega}^{y} (X_{k}=x )
\\
&&\quad\qquad{}-P_{\omega
}^{w}
(X_{k}=x ) \bigr)\biggr\rrvert\biggl\llvert\frac{f (\sigma
_{y}\omega)}{\sum_{
z\in\Delta,
z\leftrightarrow n-k}
f (\sigma_{z}\omega)}-P_{\omega}^{0}
(X_{n-k}=y| X_{n-k}\in\Delta)\biggr\rrvert.
\end{eqnarray*}
By Lemma~\ref{lemLemmaforpart3} applied with $k$ as $n$
and $\theta=\frac{\delta}{\epsilon}$, this is bounded by
\begin{eqnarray*}
&& \sum_{\Delta\in\Pi} \mathop{\mathop{\sum
_{x\in\mathbb{Z}^{d}}}_{\operatorname{dist} (x,\Delta)\leq k}}_{
x\leftrightarrow n}P_{\omega}^{0}
(X_{n-k}\in\Delta)
\\
&&\quad{}\times\mathop{\sum_{
y\in\Delta}}_{
y\leftrightarrow n-k}
k^{-\xi(1 )}\biggl\llvert \biggl[\frac{f (\sigma
_{y}\omega)}{\sum_{
z\in\Delta,
z\leftrightarrow n-k}
f (\sigma_{z}\omega)}-P_{\omega}^{0}
(X_{n-k}=y| X_{n-k}\in\Delta) \biggr]\biggr\rrvert
\\
&&\qquad\leq \sum_{\Delta\in\Pi}\mathop{\sum
_{x\in\mathbb{Z}^{d}}}_{\operatorname{dist} (x,\Delta)\leq k}P_{\omega}^{0}
(X_{n-k}\in\Delta)2k^{-\xi
(1 )}
\\
&&\qquad\leq 2k^{-\xi(1 )}\cdot(2k+l )^{d}\mathop{\longrightarrow}_{n\to\infty}0.
\end{eqnarray*}\upqed
\end{pf*}

\begin{pf*}{Proof of Theorem~\ref{teoPrefactorthm}}
Combining all the claims of Lemma~\ref{lemPrefactormainlemma}
and using the triangle inequality, we conclude that
\[
\lim_{n\to\infty}\sum_{x\in\mathbb{Z}^{d}}\biggl
\llvert\frac
{1}{Z}_{\omega,n}\mathbb{P}^{0}
(X_{n}=x )f (\sigma_{x}\omega)-\sum
_{y\in\mathbb{Z}^{d}}P_{\omega}^{0} (X_{n-k}=y
)P_{\sigma_{y}\omega}^{0} (X_{k}=x-y )\biggr\rrvert=0.
\]
Using the Markov property of the quenched law and Lemma~\ref
{lemEvaluationofthenormalizingconstant},
this implies
\[
\lim_{n\to\infty}\sum_{x\in\mathbb{Z}^{d}}\bigl\llvert
\mathbb{P}^{0} (X_{n}=x )f (\sigma_{x}\omega
)-P_{\omega}^{0} (X_{n}=x )\bigr\rrvert=0,
\]
and completes the proof.
\end{pf*}\vspace*{-12pt}

\setcounter{teo}{0}
\setcounter{equation}{0}

\begin{appendix}\label{appendi}
\section*{Appendix}\label{sec8}

\subsection{Annealed local CLT}\label{sec8.1}\label{appendixsubAnnealedlocalCLT}

\begin{prop}\label{propA.1}
Let $d\geq2$ and assume $P$ is uniformly elliptic, i.i.d. and satisfies
$ (\mathscr{P} )$. Then
\begin{eqnarray*}
&& \lim_{n\to\infty}\mathop{\sum_{
x\in\mathbb{Z}^{d}}}_{x\leftrightarrow n}
\biggl\llvert\mathbb{P}^{0} (X_{n}=x )
\\
&&\qquad{} -\frac{2}{ (2\pi
n )^{\sfrac{d}{2}}\sqrt{\det\Sigma}}
\exp\biggl(-\frac
{1}{2n} \bigl(x-\mathbb{E}^{0}
[X_{n} ] \bigr)^{T}\Sigma^{-1} \bigl(x-
\mathbb{E}^{0} [X_{n} ] \bigr) \biggr)\biggr\rrvert=0,
\end{eqnarray*}
where $\Sigma$ is the covariance matrix of $X_{\tau_{2}-\tau_{1}}$.
\end{prop}

The crucial ingredient in the proof is the annealed CLT proved by
Sznitman in \cite{Sz01} for uniformly elliptic i.i.d. random walks
in random environments satisfying condition $ (\mathscr{P} )$.
Then standard annealed derivative estimations to approximate the
value of the annealed in a given point by its average on a box of
side length $\varepsilon n^{\sfrac {1}{2}}$ gives the required result.

\begin{pf*}{Proof of Proposition \ref{propA.1}}
Fix $\varepsilon,\delta>0$ and let $\Pi_{n}^{ (\varepsilon
)}$
be a partition of $\mathbb{Z}^{d}$ into boxes of side length $
\lceil\varepsilon n^{\sfrac {1}{2}} \rceil$.
Let $C_{\delta}>0$ be a constant such that (due to Lemma~\ref
{lemgoodestimationforthelocation})
$\mathbb{P}^{0} (\llVert X_{n}-\mathbb{E}^{0}
[X_{n} ]\rrVert _{\infty}>C_{\delta}\sqrt{n}
)<\delta$
and denote by $\widehat{\Pi}_{n}^{ (\varepsilon,\delta)}$
the family of boxes in $\Pi_{n}^{ (\varepsilon)}$ intersecting
$ \{ x\in\mathbb{Z}^{d}: \llVert x-\mathbb{E}^{0}
[X_{n} ]\rrVert _{\infty}\leq C_{\delta}\sqrt{n} \} $.
Then
%
\begin{eqnarray}
&& \mathop{\sum_{x\in\mathbb{Z}^{d}}}_{x\leftrightarrow n}\biggl\llvert
\mathbb{P}^{0} (X_{n}=x )\nonumber
\\
&&\quad{}-\frac{2}{ (2\pi
n )^{\sfrac{d}{2}}\sqrt{\det\Sigma}}\exp\biggl(-
\frac
{1}{2n} \bigl(x-\mathbb{E}^{0} [X_{n} ]
\bigr)^{T}\Sigma^{-1} \bigl(x-\mathbb{E}^{0}
[X_{n} ] \bigr) \biggr)\biggr\rrvert
\nonumber
\\
&&\qquad=  \sum_{\Delta\in\Pi_{n}^{ (\varepsilon)}\setminus
\widehat{\Pi}_{n}^{ (\varepsilon,\delta)}}\mathop{\sum
_{x\in\mathbb{Z}^{d}}}_{x\leftrightarrow n}\biggl\llvert\mathbb{P}^{0}
(X_{n}=x )
\nonumber\\[-8pt]\label{eqannealedlclt1} \\[-8pt]\nonumber
&&\quad\qquad{} -\frac{2}{ (2\pi
n )^{\sfrac{d}{2}}\sqrt{\det\Sigma}}\exp\biggl(-\frac
{1}{2n}
\bigl(x-\mathbb{E}^{0} [X_{n} ] \bigr)^{T}\Sigma
^{-1} \bigl(x-\mathbb{E}^{0} [X_{n} ] \bigr)
\biggr)\biggr\rrvert
\\
&&\quad\qquad{}+  \sum_{\Delta\in\widehat{\Pi}_{n}^{ (\varepsilon,\delta
)}}\mathop{\sum
_{x\in\mathbb{Z}^{d}}}_{x\leftrightarrow n} \biggl\llvert\mathbb{P}^{0}
(X_{n}=x )
\nonumber\\[-8pt]\label{eqannealedlclt2}\\[-8pt]\nonumber
&&\quad\qquad{} -\frac{2}{ (2\pi
n )^{\sfrac{d}{2}}\sqrt{\det\Sigma}}\exp\biggl(-\frac
{1}{2n}
\bigl(x-\mathbb{E}^{0} [X_{n} ] \bigr)^{T}\Sigma
^{-1} \bigl(x-\mathbb{E}^{0} [X_{n} ] \bigr)
\biggr)\biggr\rrvert.
\end{eqnarray}
We estimate each of the term separately starting with (\ref{eqannealedlclt1}).
Due to the choice of~$C_{\delta}$
\begin{eqnarray*}
\eqref{eqannealedlclt1} & \leq&\sum_{x: \llVert x-\mathbb
{E}^{0} [X_{n} ]\rrVert _{\infty}>C_{\delta}\sqrt
{n}}
\mathbb{P}^{0} (X_{n}=x )
\\
&&{}+\frac{2}{ (2\pi n
)^{\sfrac{d}{2}}\sqrt{\det\Sigma}}\exp\biggl(-
\frac{1}{2n} \bigl(x-\mathbb{E}^{0} [X_{n} ]
\bigr)^{T}\Sigma^{-1} \bigl(x-\mathbb{E}^{0}
[X_{n} ] \bigr) \biggr)
\\
&\leq& \delta+C\cdot\exp\biggl(-\frac{c}{2}C_{\delta}^{2}
\biggr)\cdot
\end{eqnarray*}
Thus, by increasing $C_{\delta}$ we can ensure that (\ref{eqannealedlclt1})
is bounded by $2\delta$.

Turning to deal with (\ref{eqannealedlclt2}), we estimate each
of the terms inside the absolute value by an average on the appropriate
box containing it. Due to the annealed derivative estimations from
Lemma~\ref{lemgeneralannealedestimations}, we have
\[
\biggl\llvert\mathbb{P}^{0} (X_{n}=x )-\frac{2}{ (
\lceil\varepsilon n^{\sfrac {1}{2}} \rceil)^{d}}
\mathbb{P}^{0} (X_{n}\in\Delta)\biggr\rrvert\leq
\frac{C\varepsilon
n^{\sfrac {1}{2}}}{n^{\vfrac{d+1}{2}}}=\frac{C\varepsilon}{n^{\sfrac{d}{2}}},
\]
for every $\Delta\in\Pi_{n}^{ (\varepsilon)}$ and every
$x\in\Delta$ such that $x\leftrightarrow n$. In addition, for every
$\Delta\in\Pi_{n}^{ (\varepsilon)}$ and every $x\in
\Delta$
\begin{eqnarray*}
&& \biggl\llvert\exp\biggl(-\frac{1}{2n} \bigl(x-\mathbb{E}^{0}
[X_{n} ] \bigr)^{T}\Sigma^{-1} \bigl(x-
\mathbb{E}^{0} [X_{n} ] \bigr) \biggr)
\\[-2pt]
&&\quad{}-\frac{1}{ ( \lceil
\varepsilon n^{\sfrac {1}{2}} \rceil)^{d}}
\int_{\Delta
}\exp\biggl(-\frac{1}{2n} \bigl(y-
\mathbb{E}^{0} [X_{n} ] \bigr)^{T}
\Sigma^{-1} \bigl(y-\mathbb{E}^{0} [X_{n} ] \bigr)
\biggr)\,dy\biggr\rrvert
\\[-2pt]
&&\qquad \leq \exp\biggl(-\frac{1}{2n} \bigl(x-\mathbb{E}^{0}
[X_{n} ] \bigr)^{T}\Sigma^{-1} \bigl(x-
\mathbb{E}^{0} [X_{n} ] \bigr) \biggr)
\\[-2pt]
&&\quad\qquad{}\times
\frac{1}{ ( \lceil
\varepsilon n^{\sfrac {1}{2}} \rceil)^{d}}\biggl\llvert\int_{\Delta
}1-\exp\biggl(-
\frac{1}{2n} (y-x )^{T}\Sigma^{-1} (y-x ) \biggr)\,dy
\biggr\rrvert
\\[-2pt]
&&\qquad\leq \exp\biggl(-\frac{1}{2n} \bigl(x-\mathbb{E}^{0}
[X_{n} ] \bigr)^{T}\Sigma^{-1} \bigl(x-
\mathbb{E}^{0} [X_{n} ] \bigr) \biggr)\cdot C
\varepsilon^{2}\leq C\varepsilon^{2}.
\end{eqnarray*}
Combining the last two estimation gives
\begin{eqnarray*}
\eqref{eqannealedlclt1}&\leq&\sum_{\Delta\in\widehat{\Pi
}_{n}^{ (\varepsilon,\delta)}} \mathop{
\sum_{x\in\mathbb{Z}^{d}}}_{x\leftrightarrow n} \biggl\llvert
\mathbb{P}^{0} (X_{n}=x )-\frac{2}{ (
\lceil\varepsilon n^{\sfrac {1}{2}} \rceil)^{d}}\mathbb
{P}^{0} (X_{n}\in\Delta)\biggr\rrvert
\\[-2pt]
&&{} +\sum_{\Delta\in\widehat{\Pi}_{n}^{ (\varepsilon,\delta
)}} \mathop{\sum
_{x\in\mathbb{Z}^{d}}}_{x\leftrightarrow n} \biggl\llvert\frac{2}{ (
\lceil\varepsilon n^{\sfrac {1}{2}} \rceil)^{d}}
\mathbb{P}^{0} (X_{n}\in\Delta)-\frac{2}{ (2\pi n )^{\sfrac{d}{2}}\sqrt
{\det
\Sigma}}
\\[-2pt]
&&{}\times
\frac{1}{ ( \lceil\varepsilon n^{\sfrac {1}{2}} \rceil)^{d}}\int_{\Delta}\exp\biggl(-\frac
{1}{2n}
\bigl(y-\mathbb{E}^{0} [X_{n} ] \bigr)^{T}\Sigma
^{-1} \bigl(y-\mathbb{E}^{0} [X_{n} ] \bigr)
\biggr)\,dy\biggr\rrvert
\\[-2pt]
&&{} +\sum_{\Delta\in\widehat{\Pi}_{n}^{ (\varepsilon,\delta
)}} \mathop{\sum
_{x\in\mathbb{Z}^{d}}}_{x\leftrightarrow n} \frac{2}{ (2\pi n )^{\sfrac{d}{2}}\sqrt{\det
\Sigma}}
\\[-2pt]
&&{}\times \biggl\llvert
\frac{1}{ ( \lceil\varepsilon n^{\sfrac {1}{2}} \rceil)^{d}}\int_{\Delta}\exp\biggl(-\frac
{1}{2n}
\bigl(y-\mathbb{E}^{0} [X_{n} ] \bigr)^{T}\Sigma
^{-1} \bigl(y-\mathbb{E}^{0} [X_{n} ] \bigr)
\biggr)\,dy
\\[-2pt]
&&{}-\exp\biggl(-\frac{1}{2n} \bigl(x-\mathbb{E}^{0}
[X_{n} ] \bigr)^{T}\Sigma^{-1} \bigl(x-
\mathbb{E}^{0} [X_{n} ] \bigr) \biggr)\biggr\rrvert
\\[-2pt]
& \leq&\sum_{\Delta\in\widehat{\Pi}_{n}^{ (\varepsilon,\delta)}} \mathop
{\sum
_{x\in\mathbb{Z}^{d}}}_{x\leftrightarrow n} \frac{C\varepsilon}{n^{\sfrac{d}{2}}}+\sum
_{\Delta\in\widehat
{\Pi}_{n}^{ (\varepsilon,\delta)}}\mathop{\sum_{
x\in\Delta}}_{
x\leftrightarrow n}
\frac{C\varepsilon^{2}}{ (2\pi n )^{\sfrac{d}{2}}\sqrt
{\det\Sigma}}
\\[-2pt]
&&{} +\sum_{\Delta\in\widehat{\Pi}_{n}^{ (\varepsilon,\delta
)}} \mathop{\sum
_{
x\in\mathbb{Z}^{d}}}_{x\leftrightarrow n} \biggl\llvert\frac{2}{ (
\lceil\varepsilon n^{\sfrac {1}{2}} \rceil)^{d}}
\mathbb{P}^{0} (X_{n}\in\Delta)-\frac{1}{ (2\pi n )^{\sfrac{d}{2}}\sqrt
{\det
\Sigma}}
\\[-2pt]
&&{}\times \frac{2}{ ( \lceil\varepsilon n^{\sfrac {1}{2}} \rceil)^{d}}\int_{\Delta}\exp\biggl(-\frac
{1}{2n}
\bigl(y-\mathbb{E}^{0} [X_{n} ] \bigr)^{T}\Sigma
^{-1} \bigl(y-\mathbb{E}^{0} [X_{n} ] \bigr)
\biggr)\,dy\biggr\rrvert.
\end{eqnarray*}
The total number of vertices in the boxes in $\widehat{\Pi
}_{n}^{ (\varepsilon,\delta)}$
is $ (C_{\delta}n^{\sfrac {1}{2}} )^{d}$ and, therefore, the
first two sums are bounded by $C_{\delta}^{d}\cdot C\varepsilon$.
As for the last term, we have
\begin{eqnarray*}
&& \sum_{\Delta\in\widehat{\Pi}_{n}^{ (\varepsilon,\delta
)}}\mathop{\sum
_{
x\in\mathbb{Z}^{d}}}_{
x\leftrightarrow n} \biggl\llvert\frac{2}{ ( \lceil\varepsilon n^{\sfrac {1}{2}} \rceil)^{d}}
\mathbb{P}^{0} (X_{n}\in\Delta)-\frac{1}{ (2\pi n )^{\sfrac{d}{2}}\sqrt
{\det
\Sigma}}\cdot
\frac{2}{ ( \lceil\varepsilon n^{\sfrac {1}{2}} \rceil)^{d}}
\\
&&\quad{}\times \int_{\Delta}\exp\biggl(-\frac
{1}{2n}
\bigl(y-\mathbb{E}^{0} [X_{n} ] \bigr)^{T}\Sigma
^{-1} \bigl(y-\mathbb{E}^{0} [X_{n} ] \bigr)
\biggr)\,dy\biggr\rrvert
\\
&&\qquad=  \sum_{\Delta\in\widehat{\Pi}_{n}^{ (\varepsilon,\delta
)}}\biggl\llvert\mathbb{P}^{0}
(X_{n}\in\Delta)-\frac
{2}{ (2\pi n )^{\sfrac{d}{2}}\sqrt{\det\Sigma}}
\\
&&\quad\qquad{}\times \int_{\Delta}
\exp\biggl(-\frac{1}{2n} \bigl(y-\mathbb{E}^{0}
[X_{n} ] \bigr)^{T}\Sigma^{-1} \bigl(y-
\mathbb{E}^{0} [X_{n} ] \bigr) \biggr)\,dy\biggr\rrvert.
\end{eqnarray*}
Apply the functional CLT proved by Sznitman in \cite{Sz01} and
nothing that for a fixed $\varepsilon$ and $\delta$ the sum is finite
gives that the last term goes to zero as $n$ goes to infinity.

Combining all of the above, we conclude that
\begin{eqnarray*}
&& \limsup_{n\to\infty}\mathop{\sum_{
x\in\mathbb{Z}^{d}}}_{x\leftrightarrow n}
\biggl\llvert\mathbb{P}^{0} (X_{n}=x )-\frac{2}{ (2\pi
n )^{\sfrac{d}{2}}\sqrt{\det\Sigma}}
\\
&&\quad{}\times
\exp\biggl(-\frac
{1}{2n} \bigl(x-\mathbb{E}^{0}
[X_{n} ] \bigr)^{T}\Sigma^{-1} \bigl(x-
\mathbb{E}^{0} [X_{n} ] \bigr) \biggr)\biggr\rrvert
\\
&&\qquad\leq C_{\delta}^{d}\cdot C\varepsilon+2\delta.
\end{eqnarray*}
By first taking $\delta>0$ arbitrary small and then choosing
$\varepsilon>0$
even smaller so that $C_{\delta}^{d}\cdot C\varepsilon<\delta$ the
result follows.
\end{pf*}

\subsection{Annealed derivative
estimations}\label{sec8.2}
\label{appendixsubAnnealed-derivative-estimations}

In this part of the \hyperref[appendi]{Appendix}, we prove Lemma~\ref
{lemAnnealedderivativeestimationsd-1+time}
and Lemma~\ref{lemgeneralannealedestimations} regarding annealed
derivative estimations.

\subsubsection{General estimations}\label{sec8.2.1}

We start with the following claim which is a general result for i.i.d.
random variables on a lattice:

\begin{claim}
\label{clmFourierclaim}Let $ \{ Y_{i} \} _{i=1}^{\infty}$
and $ \{ Z_{i} \} _{i=1}^{\infty}$ be a sequence of $d$-dimensional
random variables and a sequence of $1$-dimensional nonnegative
integer valued random variables, respectively, such that $ \{
(Y_{i},Z_{i} ) \} _{i=1}^{\infty}$
are independent and have joint distribution $\mathsf{P}$. Assume
in addition that $ \{ (Y_{i},Z_{i} ) \}
_{i=2}^{\infty}$
are i.i.d.\vspace*{1pt} and there exists $v\in\mathbb{Z}^{d}$, $k\in\mathbb{N}$
such that $\mathsf{P} ( (Y_{2},Z_{2} )=
(v,k ) )>0$
and $\mathsf{P} ( (Y_{2},Z_{2} )=
(v+e_{i},k+1 ) )>0$
for every $1\leq i\leq d$. Let $S_{n}=\sum_{i=1}^{n}Y_{i}$ and
$T_{n}=\sum_{i=1}^{n}Z_{i}$.
Then there exists $C<\infty$ which is determined by distribution
$\mathsf{P}$ such that for every $n\in\mathbb{N}$, $m\in\mathbb{N}$,
every $x,y\in\mathbb{Z}^{d}$ such that $\llVert x-y\rrVert _{1}=1$
and every $1\leq k\leq d$
%
\begin{eqnarray}
\mathsf{P} \bigl( (S_{n},T_{n} )= (x,m )
\bigr)&<&Cn^{-\vfrac{d+1}{2}},\label{eqFourieranalysis1}
\\
\bigl\llvert\mathsf{P} \bigl( (S_{n},T_{n} )= (x,m )
\bigr)-\mathsf{P} \bigl( (S_{n},T_{n} )= (y,m+1 ) \bigr)
\bigr\rrvert&<&Cn^{-\vfrac{d+2}{2}}\label{eqFourieranalysis2}
\end{eqnarray}
and
%
\begin{eqnarray}\label{eqFourieranalysis25}
&& \bigl\llvert\mathsf{P} \bigl( (S_{n},T_{n} )= (x,m )
\bigr)+\mathsf{P} \bigl( (S_{n},T_{n} )=
(y+e_{k},m ) \bigr)-\mathsf{P} \bigl( (S_{n},T_{n})
\nonumber\\[-8pt]\\[-8pt]\nonumber
&&\qquad = (x+e_{k},m+1 ) \bigr)-\mathsf{P} \bigl( (S_{n},T_{n}
)= (y,m+1 ) \bigr)\bigr\rrvert<Cn^{-\vfrac{d+3}{2}}.
\end{eqnarray}
In addition, if $ \{ Y_{i} \} _{i=1}^{\infty}$ and $
\{ Z_{i} \} _{i=1}^{\infty}$
have finite moments. Then for every $\varepsilon>0$, every $m\in
\mathbb{N}$
and every partition $\Pi_{n}$ of $\mathbb{Z}^{d}$ into boxes of
side length $n^{\varepsilon}$
%
\begin{eqnarray}\label{eqFourieranalysis3}
&& \sum_{\Delta\in\Pi_{n}}\mathop{\sum_{
x\in\Delta}}_{
x\leftrightarrow m}
\Bigl[\max_{y\in\Delta}\mathsf{P} \bigl( (S_{n},T_{n}
)= (y,m ) \bigr)-\mathsf{P} \bigl( (S_{n},T_{n} )= (x,m )
\bigr) \Bigr]
\nonumber\\[-8pt]\\[-8pt]\nonumber
&&\qquad \leq Cn^{-1+3d\varepsilon}.
\end{eqnarray}
\end{claim}
\begin{pf}
Let $\chi$ be the characteristic function of $ (Y_{2},Z_{2} )$.
Since $ (Y_{2},Z_{2} )$ is concentrated on a lattice $\chi$
is periodic. The existence of $v,k$ as above implies that the period
is $2\pi$ in every coordinate. The existence of $v$ and $k$ also
implies that there are $D>0$ and $\delta>0$ such that:
\begin{longlist}[(2)]
\item[(1)]$\llvert \chi(\xi,s )\rrvert <1-D$ for every $
(\xi,s )\in[-\pi,\pi]^{d+1}$
such that $\llVert (\xi,s )\rrVert _{1}\geq
\delta$,
\item[(2)]$\llvert \chi(\xi,s )\rrvert <1-D\llVert
(\xi,s )\rrVert _{1}^{2}$
for every $ (\xi,s )\in[-\pi,\pi]^{d+1}$ such
that $\llVert (\xi,s )\rrVert _{1}<\delta$.
\end{longlist}

The last two facts implies [(\ref{eqFourieranalysis1})--(\ref
{eqFourieranalysis25})].
Indeed,
\begin{eqnarray*}
&& \mathsf{P} \Biggl(\sum_{i=2}^{n}
(Y_{i},Z_{i} )= (x,m ) \Biggr)
\\[-2pt]
&&\qquad  = \frac{1}{ (2\pi)^{d+1}}\int
_{ [-\pi,\pi]^{d+1}}e^{-i \langle\xi,x
\rangle-i \langle s,m \rangle}\chi^{n-1} (\xi,s )\,d\xi \,ds
\\[-2pt]
&&\qquad  \leq \int_{ [-\pi,\pi]^{d+1}}\bigl\llvert\chi^{n-1} (\xi,s )
\bigr\rrvert \,d\xi \,ds
\\[-2pt]
&&\qquad  =\int_{\llVert (\xi,s )\rrVert _{1}>\delta
}\bigl\llvert\chi^{n-1} (\xi,s )\bigr
\rrvert \,d\xi \,ds+\int_{\llVert (\xi,s )\rrVert _{1}\leq\delta}\bigl
\llvert\chi^{n-1} (
\xi,s )\bigr\rrvert \,d\xi \,ds
\\[-2pt]
&&\qquad  \leq (2\pi)^{d+1} (1-D )^{n-1}+\int_{\llVert (\xi,s )\rrVert
_{1}\leq\delta}
\bigl(1-D\bigl\llVert(\xi,s )\bigr\rrVert_{1}^{2}
\bigr)^{n-1}\,d\xi \,ds
\\[-2pt]
&&\qquad <Cn^{-\vfrac{d+1}{2}}
\end{eqnarray*}
and convolution with the distribution of $ (Y_{1},Z_{1} )$
only decreases the supremum.

To see (\ref{eqFourieranalysis2}), note that $y=x\pm e_{j}$ for
some $1\leq j\leq d$ and, therefore,
\begin{eqnarray*}
&& \Biggl\llvert\mathsf{P} \Biggl(\sum_{i=2}^{n}
(Y_{i},Z_{i} )= (x,m ) \Biggr)-\mathsf{P} \Biggl(\sum
_{i=2}^{n} (Y_{i},Z_{i}
)= (x+e_{j},m+1 ) \Biggr)\Biggr\rrvert
\\[-2pt]
&&\qquad=  \frac{1}{ (2\pi)^{d+1}}
\\[-2pt]
&&\quad\qquad{}\times \biggl\llvert\int_{ [-\pi,\pi
]^{d+1}}
\bigl(e^{-i \langle\xi,x \rangle-i
\langle s,m \rangle}-e^{-i \langle\xi,x\pm e_{j}
\rangle-i \langle s,m+1 \rangle} \bigr)\chi(\xi,s )^{n-1}\,d\xi
\,ds\biggr\rrvert
\\[-2pt]
&&\qquad\leq \frac{1}{ (2\pi)^{d+1}}\int_{ [-\pi,\pi
]^{d+1}}\bigl\llvert
e^{-i \langle\xi,x \rangle-i
\langle s,m \rangle}-e^{-i \langle\xi,x\pm e_{j}
\rangle-i \langle s,m+1 \rangle}\bigr\rrvert
\\[-2pt]
&&\quad\qquad{}\times \bigl\llvert\chi(\xi
,s )
\bigr\rrvert^{n-1}\,d\xi \,ds.
\end{eqnarray*}
Recalling that $\llvert e^{-i \langle\xi,x \rangle-i
\langle s,m \rangle}-e^{-i \langle\xi,x\pm e_{j}
\rangle-i \langle s,m+1 \rangle}\rrvert \leq\llvert
\langle(\xi,s ), \langle\pm e_{j},1 \rangle
\rangle\rrvert $,
we can bound the last term by
\begin{eqnarray*}
&& \frac{1}{ (2\pi)^{d+1}}\int_{ [-\pi,\pi
]^{d+1}}\bigl\llvert s\pm\langle
\xi,e_{j} \rangle\bigr\rrvert\cdot\bigl\llvert\chi(\xi,s )\bigr
\rrvert \,d\xi \,ds
\\[-1pt]
&&\qquad \leq (1-D )^{n+1}+C\int_{\llVert (\xi,s
)\rrVert _{1}\leq\delta}\bigl\llvert s\pm
\langle\xi,e_{j} \rangle\bigr\rrvert\bigl(1-D\bigl\llVert(\xi,s )
\bigr\rrVert_{1}^{2} \bigr)^{n-1}\,d\xi \,ds
\\[-1pt]
&&\qquad \leq (1-D )^{n+1}+C\int_{\llVert (\xi,s
)\rrVert _{1}\leq\delta}\bigl\llvert s\pm
\langle\xi,e_{j} \rangle\bigr\rrvert e^{-Dn\llVert (\xi,s
)\rrVert _{1}^{2}}\,d\xi \,ds.
\end{eqnarray*}
Substituting $\zeta=\xi\sqrt{n}$ and $t=s\sqrt{n}$ the last integral
equals
\begin{eqnarray*}
&& \frac{C}{\sqrt{n}^{d+2}}\int_{\llVert (\zeta,t
)\rrVert _{1}\leq\delta\sqrt{n}}\bigl\llvert t\pm\langle
\zeta,e_{j} \rangle\bigr\rrvert e^{-D\llVert (\zeta,t
)\rrVert _{1}^{2}}\,d\zeta \,dt
\\[-1pt]
&&\qquad\leq \frac{C}{n^{\vfrac{d+2}{2}}}\int_{\mathbb{R}^{d+1}}\bigl\llvert
t\pm\langle
\zeta,e_{j} \rangle\bigr\rrvert e^{-D\llVert
(\zeta,t )\rrVert _{1}^{2}}\,d\zeta \,dt=O
\bigl(n^{-\vfrac{d+2}{2}} \bigr).
\end{eqnarray*}
For (\ref{eqFourieranalysis25}), note that
\begin{eqnarray*}
&& \Biggl\llvert\mathsf{P} \Biggl(\sum_{i=2}^{n}
(Y_{i},Z_{i} )= (x,m ) \Biggr) +\mathsf{P} \Biggl(\sum
_{i=2}^{n} (Y_{i},Z_{i}
)= (y+e_{k},m ) \Biggr)
\\[-1pt]
&&\quad{}-\mathsf{P} \Biggl(\sum
_{i=2}^{n} (Y_{i},Z_{i} )=
(y,m+1 ) \Biggr) -\mathsf{P} \Biggl(\sum_{i=2}^{n}
(Y_{i},Z_{i} )= (x+e_{k},m+1 ) \Biggr)\Biggr
\rrvert
\\[-1pt]
&&\qquad=  \Biggl\llvert\mathsf{P} \Biggl(\sum_{i=2}^{n}
(Y_{i},Z_{i} )= (x,m ) \Biggr) +\mathsf{P} \Biggl(\sum
_{i=2}^{n} (Y_{i},Z_{i}
)= (x+e_{j}+e_{k},m ) \Biggr)
\\[-1pt]
&&\quad\qquad{} -\mathsf{P} \Biggl(\sum
_{i=2}^{n} (Y_{i},Z_{i}
)= (x+e_{j},m+1 ) \Biggr)
\\[-1pt]
&&\quad\qquad{} -\mathsf{P} \Biggl(\sum
_{i=2}^{n} (Y_{i},Z_{i} )=
(x+e_{k},m+1 ) \Biggr)\Biggr\rrvert
\\[-1pt]
&&\qquad=  \frac{1}{ (2\pi)^{d+1}}\biggl\llvert\int_{ [-\pi,\pi
]^{d+1}}
\bigl(e^{-i \langle\xi,x \rangle-i
\langle s,m \rangle}+e^{-i \langle\xi,x+e_{j}+e_{k}
\rangle-i \langle s,m \rangle}
\\[-1pt]
&&\quad\qquad{}-e^{-i \langle\xi,x+e_{j} \rangle-i
\langle s,m+1 \rangle} -e^{-i \langle\xi,x+e_{k} \rangle-i \langle
s,m+1 \rangle}
\bigr)\chi(\xi,s )^{n-1}\,d\xi \,ds\biggr\rrvert
\\[-1pt]
&&\qquad\leq \biggl\llvert\int_{ [-\pi,\pi]^{d+1}}\bigl\llvert
1+e^{-i \langle\xi,e_{j}+e_{k} \rangle}-e^{-i
\langle\xi,e_{j} \rangle-i \langle s,1 \rangle
}-e^{-i \langle\xi,e_{k} \rangle-i \langle s,1
\rangle}\bigr\rrvert
\\[-1pt]
&&\quad\qquad{}\times \bigl\llvert
\chi(\xi,s )\bigr\rrvert^{n-1}\,d\xi \,ds\biggr\rrvert
\\[-1pt]
&&\qquad\leq \biggl\llvert\int_{ [-\pi,\pi]^{d+1}}\bigl\llvert\bigl\langle(
\xi,s ), (e_{j},1 ) \bigr\rangle\bigr\rrvert\bigl\llvert\bigl
\langle(\xi,s ), (e_{k},1 ) \bigr\rangle\bigr\rrvert\bigl\llvert\chi
(\xi,s )\bigr\rrvert^{n-1}\,d\xi \,ds\biggr\rrvert
\end{eqnarray*}
and the proof continuous now as before except that we gained an
additional factor of $n^{-\sfrac{1}{2}}$.

Finally, we turn to the proof of (\ref{eqFourieranalysis3}). For
every $\Delta\in\Pi$ denote by $x_{\Delta}$ a point in $\Delta$
such that $\mathsf{P} (S_{n}=x_{\Delta} )=\max_{y\in
\Delta}\mathsf{P} (S_{n}=y )$.
As a first step, we show that
%
\begin{equation}
\sum_{\Delta\in\Pi_{n}}\sum_{x\in\Delta}
\bigl[\mathsf{P} (S_{n}=x_{\Delta} )-\mathsf{P}
(S_{n}=x ) \bigr]\leq Cn^{-\sfrac{1}{2}+3d\varepsilon}.\label
{eqFouriermaxvssumwithouttime}
\end{equation}
By \cite{berger2008slowdown}, Claim 4.3, for every $x,y\in\mathbb{Z}^{d}$
such that $\llVert x-y\rrVert _{1}=1$, we have $\llvert \mathsf
{P} (S_{n}=y )-\mathsf{P} (S_{n}=x )\rrvert \leq
Cn^{-\vfrac{d+1}{2}}$
and, therefore, for every $\Delta\in\Pi$
%
\begin{eqnarray}\label{eqFouriermaxvssumwithouttime2}
&& \sum_{x\in\Delta} \bigl[\mathsf{P} (S_{n}=x_{\Delta}
)-\mathsf{P} (S_{n}=x ) \bigr]\nonumber
\\
&&\qquad  \leq \sum
_{x\in\Delta
}\llVert x_{\Delta}-x\rrVert_{\infty}Cn^{-\vfrac{d+1}{2}}
\\
&&\qquad \leq C\sum_{x\in\Delta}n^{d\varepsilon}\cdot
n^{-\vfrac{d+1}{2}}
= Cn^{-\vfrac{d+1}{2}+2d\varepsilon}.\nonumber
\end{eqnarray}
Splitting the sum over the boxes to those boxes whose distance from
$\mathsf{E} [S_{n} ]$ is greater or smaller than $n^{\sfrac{1}{2}+\varepsilon}$,
we can the first with by an Azuma inequality which yields $\mathsf
{P} (\llVert S_{n}-\mathsf{E} [S_{n} ]\rrVert
_{1}>n^{\sfrac{1}{2}+\varepsilon} )=n^{-\xi(1 )}$
and the second with (\ref{eqFouriermaxvssumwithouttime2})
we get
\begin{eqnarray*}
&& \sum_{\Delta\in\Pi_{n}}\sum_{x\in\Delta}
\bigl[\mathsf{P} (S_{n}=x_{\Delta} )-\mathsf{P}
(S_{n}=x ) \bigr]
\\
&&\qquad  \leq \mathop{\sum_{
\Delta\in\Pi_{n}}}_{
\operatorname{dist} (\Delta,\mathsf{E} [S_{n} ] )\leq
n^{\sfrac{1}{2}+\varepsilon}
}Cn^{-\vfrac{d+1}{2}+2d\varepsilon}+n^{-\xi(1 )}
\leq Cn^{-\sfrac{1}{2}+3d\varepsilon}
\end{eqnarray*}

Finally, we turn to prove (\ref{eqFourieranalysis3}). Denote
\[
g (m )=\sum_{\Delta\in\Pi_{n}}
\mathop{\sum_{
x\in\Delta}}_{
x\leftrightarrow m}
\Bigl[\max_{y\in\Delta}\mathsf{P} \bigl( (S_{n},T_{n}
)= (y,m ) \bigr)-\mathsf{P} \bigl( (S_{n},T_{n} )= (x,m )
\bigr) \Bigr].
\]
By Azuma's inequality $\mathsf{P} (\llvert T_{n}-\mathsf{E}
[T_{n} ]\rrvert >n^{\sfrac{1}{2}+\varepsilon} )=n^{-\xi
(1 )}$
and, therefore, it is enough to deal with $m\in\mathbb{N}$ such that
$\llvert m-\mathsf{E} [T_{n} ]\rrvert \leq n^{\sfrac{1}{2}+\varepsilon}$.
By the same estimation and (\ref{eqFouriermaxvssumwithouttime}),
we also know that $n^{\sfrac{1}{2}+\varepsilon}\min_{m: \llvert
m-\mathsf{E} [T_{n} ]\rrvert \leq n^{\sfrac{1}{2}+\varepsilon}}g (m )\le Cn^{-\sfrac{1}{2} +3d\varepsilon}$
and, therefore,
%
\begin{equation}
\min_{m: \llvert m-\mathsf{E} [T_{n} ]\rrvert \leq
n^{\sfrac{1}{2}+\varepsilon}}g (m )\le Cn^{-1+
(3d-1 )\varepsilon}.\label{eqminisfine}
\end{equation}

However, using (\ref{eqFourieranalysis25}), for every $m\in\mathbb{N}$
and $x,z\in\Delta$ such that $x\leftrightarrow m$, $z\leftrightarrow m+1$
we have
\begin{eqnarray*}
&& \max_{y\in\Delta}\mathsf{P} \bigl( (S_{n},T_{n}
)= (y,m ) \bigr)-\mathsf{P} \bigl( (S_{n},T_{n} )= (x,m )
\bigr)
\\
&&\quad{} -\max_{w\in\Delta}\mathsf{P} \bigl( (S_{n},T_{n}
)= (w,m+1 ) \bigr)-\mathsf{P} \bigl( (S_{n},T_{n} )= (z,m
) \bigr)
\\
&&\qquad \leq \sum_{k=1}^{d} \mathop{\mathop{\sum
_{
x,y\in\Delta}}_{
x\leftrightarrow m}}_{
\llVert y-x\rrVert _{1}=1} \bigl\llvert
\mathsf{P} \bigl( (S_{n},T_{n} )= (x,m ) \bigr) +\mathsf{P}
\bigl( (S_{n},T_{n} )= (y+e_{k},m ) \bigr)
\\
&&\quad\qquad{}-
\mathsf{P} \bigl( (S_{n},T_{n} )= (x+e_{k},m+1 )
\bigr) -\mathsf{P} \bigl( (S_{n},T_{n} )= (y,m+1 ) \bigr)
\bigr\rrvert
\\
&&\qquad\leq Cn^{-\vfrac{d+3}{2}+d\varepsilon}
\end{eqnarray*}
and, therefore, by separating the sum into boxes which are at distance
$\leq n^{\sfrac{1}{2}+\varepsilon}$ and those who at distance
$>n^{\sfrac{1}{2}+\varepsilon}$
we get
%
\begin{eqnarray}\label{eqsecondderivativeisgood}
\bigl\llvert g (m )-g (m+1 )\bigr\rrvert &\leq& Cn^{(\sfrac{1}{2}+\varepsilon)d}\cdot
n^{-\vfrac
{d+3}{2}+d\varepsilon}+n^{-\xi(1 )}
\nonumber\\[-8pt]\\[-8pt]\nonumber
&=& Cn^{-\sfrac
{3}{2}+2d\varepsilon}.
\end{eqnarray}
Using (\ref{eqminisfine}) and (\ref{eqsecondderivativeisgood})
gives
\[
g (m )\leq Cn^{-1+ (3d-1 )\varepsilon
}+Cn^{-\sfrac{3}{2}+2d\varepsilon}\cdot n^{\sfrac{1}{2}+\varepsilon
}\leq
Cn^{-1+3d\varepsilon}
\]
for every $m\in\mathbb{N}$ such that $\llvert m-\mathsf{E}
[T_{n} ]\rrvert \leq n^{\sfrac{1}{2}+\varepsilon}$,
and thus completes the proof.
\end{pf}

\subsubsection{Proof of Lemma \protect\ref
{lemAnnealedderivativeestimationsd-1+time}}\label{sec8.2.2}

Before turning to the proof of the lemmas, we give the following estimations
on hitting point of an hyperplane conditioned to contain a regeneration
time. More formally, we have the following.

\begin{lem}
\label{lemmiddlestepannealedestimations}Let $d\geq4$ and assume
$P$ is uniformly elliptic, i.i.d. and satisfies $ (\mathscr
{P} )$.
Fix $z\in\BZ^{d}$, $N\in\mathbb{N}$ and let $z_{1}\in\tilde{\CP
} (0,N )$.
Let $ \{ X_{n} \} $ be an RWRE starting at $z_{1}$. For
$k,l\in\mathbb{N}$ let $B (l,k )$ be the event that $
\langle X_{\tau_{k}},e_{1} \rangle=l$,
$B (l )=\bigcup_{k}B (l,k )$ and
\[
\hat{B} \bigl(l,N^{2} \bigr)\equiv\hat{B} (l )=B (l )\cap\bigcap
_{j=l+1}^{N^{2}}B^{c} (j ).
\]
Then (for a given $l\leq N^{2}$):
\begin{longlist}[(2)]
\item[(1)] For every $n\in\mathbb{N}$ and $w\in H_{l}$
%
\begin{equation}
\mathbb{P}^{z_{1}} \bigl( (X_{T_{l}},T_{l} )= (w,n )
\mid\hat{B} (l ) \bigr)\leq Cl^{-\sfrac{d}{2}}.\label{eqAnnealedestmiddle1-1}
\end{equation}

\item[(2)] For every $n\in\mathbb{N}$, and every $w,z\in H_{l}$ such that
$\llVert w-z\rrVert _{1}=1$
%
\begin{eqnarray}\label{eqAnnealedestmiddle2-1}
&& \bigl\llvert\mathbb{P}^{z_{1}} \bigl( (X_{T_{l}},T_{l}
)= (w,n )\mid \hat{B} (l ) \bigr)-\mathbb{P}^{z_{1}} \bigl(
(X_{T_{l}},T_{l} )= (z,n+1 )\mid\hat{B} (l ) \bigr)
\bigr\rrvert
\nonumber\\[-8pt]\\[-8pt]\nonumber
&&\qquad \leq Cl^{-\vfrac{d+1}{2}}.
\end{eqnarray}

\item[(3)] For every $n\in\mathbb{N}$ and every $w\in H_{l}$
%
\begin{eqnarray}\label{eqAnnealedestmiddle3-1}
\qquad && \bigl\llvert\mathbb{P}^{z_{1}} \bigl( (X_{T_{l}},T_{l}
)= (w,n )\mid\hat{B} (l ) \bigr)-\mathbb{P}^{z_{1}+e_{1}}
\bigl(
(X_{T_{l}},T_{l} )= (w,n+1 )\mid \hat{B} (l ) \bigr)
\bigr\rrvert
\nonumber\\[-8pt]\\[-8pt]\nonumber
&&\qquad \leq Cl^{-\vfrac{d+1}{2}}.
\end{eqnarray}
\end{longlist}
\end{lem}

\begin{pf}
Due to the independence of $ (X_{T_{l}},T_{l} )$ from
$\bigcap_{j=l+1}^{N^{2}}B^{c} (j )$,
we get that for every $M\in\mathbb{N}$
\begin{eqnarray*}
&& \mathbb{P}^{z_{1}} \bigl( (X_{T_{l}},T_{l} )= (w,n )
\mid\hat{B} (l ) \bigr)
\\
&&\qquad   = \mathbb{P}^{z_{1}} \bigl( (X_{T_{l}},T_{l}
)= (w,n )\mid B (l ) \bigr)
\\
&&\qquad  = \frac{1}{\mathbb{P}^{z_{1}} (B (l ) )}\sum_{k=1}^{\infty}
\mathbb{P}^{z_{1}} \bigl( (X_{\tau_{k}},\tau_{k} )= (w,n )
\bigr)
\\
&&\qquad  =\frac{1}{\mathbb{P}^{z_{1}} (B (l ) )}\sum_{k=1}^{M}
\mathbb{P}^{z_{1}} \biggl( (X_{\tau_{k}},\tau_{k} )= (w,n
), \langle X_{\tau_{ \lceil\sfrac{k}{2} \rceil}},e_{1} \rangle\geq\frac
{l}{2}
\biggr)
\\
&&\quad\qquad{} +\frac{1}{\mathbb{P}^{z_{1}} (B (l ) )}\sum_{k=1}^{M}
\mathbb{P}^{z_{1}} \biggl( (X_{\tau_{k}},\tau_{k} )= (w,n
), \langle X_{\tau_{k}}-X_{\tau _{ \lceil\sfrac{k}{2} \rceil}},e_{1} \rangle\geq
\frac{l}{2} \biggr)
\\
&&\quad\qquad{} +\frac{1}{\mathbb{P}^{z_{1}} (B (l ) )}\sum_{k=M+1}^{\infty}
\mathbb{P}^{z_{1}} \biggl( (X_{\tau_{k}},\tau_{k} )= (w,n
), \langle X_{\tau_{ \lceil\sfrac{k}{2} \rceil}},e_{1} \rangle\leq\frac
{l}{2}
\biggr)
\\
&&\quad\qquad{} +\frac{1}{\mathbb{P}^{z_{1}} (B (l ) )}\sum_{k=M+1}^{\infty}
\mathbb{P}^{z_{1}} \biggl( (X_{\tau_{k}},\tau_{k} )= (w,n
), \langle X_{\tau_{k}}-X_{\tau
_{ \lceil\sfrac{k}{2} \rceil}},e_{1} \rangle\leq
\frac{l}{2} \biggr).
\end{eqnarray*}
Using Claim~\ref{clmFourierclaim} gives
%
\begin{eqnarray}\label{eqmiddlestepannealedderivativeestimationmarker}
&& \mathbb{P}^{z_{1}} \biggl( (X_{\tau_{k}},\tau_{k} )=
(w,n ), \langle X_{\tau_{ \lceil\sfrac{k}{2} \rceil}},e_{1} \rangle\leq
\frac{l}{2} \biggr)
\nonumber
\\
&&\qquad=  \sum_{x: \langle x,e_{1} \rangle\leq\sfrac{l}{2}}\sum_{s\in\mathbb{N}}
\mathbb{P}^{z_{1}} \bigl( (X_{\tau_{ \lceil\sfrac{k}{2} \rceil}},\tau_{ \lceil\sfrac
{k}{2} \rceil} )= (x,s )
\bigr)\nonumber
\\
&&\quad\qquad{}\times \mathbb{P}^{z_{1}} \bigl( (X_{\tau_{k}},\tau_{k} )=
(w,n )\mid(X_{\tau_{ \lceil\sfrac{k}{2}
\rceil}},\tau_{ \lceil\sfrac{k}{2} \rceil} )= (x,s ) \bigr)
\\
&&\qquad\leq Ck^{-\vfrac{d+1}{2}}\sum_{x: \langle x,e_{1} \rangle
\leq\sfrac{l}{2}}\sum
_{s\in\mathbb{N}}\mathbb{P}^{z_{1}} \bigl( (X_{\tau_{ \lceil\sfrac{k}{2}
\rceil}},\tau
_{ \lceil\sfrac{k}{2} \rceil} )= (x,s ) \bigr)\nonumber
\\
&&\qquad=  Ck^{-\vfrac{d+1}{2}}\mathbb{P}^{z_{1}} \biggl( \langle
X_{\tau_{ \lceil\sfrac{k}{2} \rceil}},e_{1} \rangle\leq\frac{l}{2} \biggr),
\nonumber
\end{eqnarray}
and in a similar manner
\begin{eqnarray*}
&& \mathbb{P}^{z_{1}} \biggl( (X_{\tau_{k}},\tau_{k} )= (w,n
), \langle X_{\tau_{k}}-X_{\tau_{ \lceil
\sfrac{k}{2} \rceil}},e_{1} \rangle\leq \frac {l}{2} \biggr)
\\
&&\qquad  \leq  Ck^{-\vfrac{d+1}{2}}\mathbb{P}^{z_{1}}
\biggl( \langle X_{\tau_{k}}-X_{\tau_{ \lceil\sfrac{k}{2}
\rceil}},e_{1} \rangle\leq
\frac{l}{2} \biggr)
\\
&&\qquad  \leq  Ck^{-\vfrac{d+1}{2}}\mathbb{P}^{z_{1}} \biggl( \langle
X_{\tau_{ \lceil\sfrac{k}{2} \rceil}},e_{1} \rangle\leq\frac{l}{2}
\biggr).
\end{eqnarray*}
Repeating the same calculations while separating the sum according
to the events $ \langle X_{\tau_{ \lfloor\sfrac{k}{2}
\rfloor}},e_{1} \rangle\geq\frac{l}{2}$
and
$ \langle X_{\tau_{k}}-X_{\tau_{ \lfloor\sfrac{k}{2}
\rfloor}},e_{1} \rangle\geq\frac{l}{2}$
we get that
\begin{eqnarray*}
&& \mathbb{P}^{z_{1}} \biggl( (X_{\tau_{k}},\tau_{k} )= (w,n
), \langle X_{\tau_{ \lceil\sfrac
{k}{2} \rceil}},e_{1} \rangle\geq\frac{l}{2}
\biggr)
\\
&&\qquad \leq Ck^{-\vfrac{d+1}{2}}\mathbb{P}^{z_{1}} \biggl( \langle
X_{\tau_{ \lceil\sfrac{k}{2} \rceil}},e_{1} \rangle\geq\frac{l}{2}
\biggr)
\end{eqnarray*}
and
\begin{eqnarray*}
&& \mathbb{P}^{z_{1}} \biggl( (X_{\tau_{k}},\tau_{k} )= (w,n
), \langle X_{\tau_{k}}-X_{\tau_{ \lceil \sfrac{k}{2} \rceil}},e_{1} \rangle\geq\frac
{l}{2} \biggr)
\\
&&\qquad \leq Ck^{-\vfrac{d+1}{2}}\mathbb{P}^{z_{1}} \biggl(
\langle X_{\tau_{ \lceil\sfrac{k}{2} \rceil}},e_{1} \rangle\geq\frac{l}{2}
\biggr),
\end{eqnarray*}
combining all of the above yields
\begin{eqnarray*}
\mathbb{P}^{z_{1}} \bigl( (X_{T_{l}},T_{l} )= (w,n )
\mid\hat{B} (l ) \bigr) & \leq& C\sum_{k=1}^{M}k^{-\vfrac{d+1}{2}}
\mathbb{P}^{z_{1}} \biggl( \langle X_{\tau_{ \lceil\sfrac{k}{2} \rceil}},e_{1}
\rangle\geq\frac{l}{2} \biggr)
\\
&&{} +C\sum_{k=M+1}^{\infty}k^{-\vfrac{d+1}{2}}
\mathbb{P}^{z_{1}} \biggl( \langle X_{\tau_{ \lceil\sfrac{k}{2} \rceil
}},e_{1}
\rangle\leq\frac{l}{2} \biggr).
\end{eqnarray*}
The argument continues now as in the proof of \cite
{berger2008slowdown}, Lemma~4.2.
Choosing
\[
M=\frac{l}{\mathbb{E}^{z_{1}} [ \langle X_{\tau
_{2}}-X_{\tau_{1}},e_{1} \rangle]}=\Theta(l )
\]
and using Theorem~\ref{teoregenerationtimesthm} (see also Remark
\ref{RemTheeventAN}) we get that $\tau_{k}-\tau_{k-1}$ has finite
$2d$ moments, and from standard estimates for the sum of i.i.d. variable
[that the $2d$ moment for the sum of $k$ i.i.d. mean zero random
variables grows like $O (k^{d} )$], thus
\[
\mathbb{P}^{z_{1}} \biggl( \langle X_{\tau_{ \lceil\sfrac
{k}{2} \rceil}},e_{1}
\rangle\geq\frac{l}{2} \biggr),\mathbb{P}^{z_{1}} \biggl( \langle
X_{\tau_{ \lceil\sfrac
{k}{2} \rceil}},e_{1} \rangle\leq\frac{l}{2} \biggr)\leq\min
\biggl[1,\frac{Ck^{d}}{ (M-k )^{2d}} \biggr]
\]
and, therefore,
\[
\mathbb{P}^{z_{1}} \bigl(X_{l}=x\mid\hat{B} (l ) \bigr)\leq
C\sum_{k=1}^{\infty}k^{-\vfrac{d+1}{2}}\min
\biggl[1,\frac
{Ck^{d}}{ (M-k )^{2d}} \biggr]=O \bigl(l^{-\sfrac{d}{2}} \bigr).
\]
To see the last equality, we have to separate the sum into four parts
as in \cite{berger2008slowdown}, Lemma 4.2, the first part is a sum
over $k\in[1,\frac{M}{2} ]$ which contains roughly $l$
summands, each of them is bounded by $Cl^{-d}$ and therefore the
whole sum is bounded by $Cl^{-d+1}$. The second part is a sum over
$k\in[\frac{M}{2},M-\sqrt{M} ]$. In this case, the sum is
bounded up to a constant by
\begin{eqnarray*}
&& \int_{\sfrac{M}{2}}^{M-\sqrt{M}}x^{\vfrac{d-1}{2}} (M-x
)^{-2d}\,dx
\\
&&\qquad   = \int_{\sqrt{M}}^{\sfrac{M}{2}} (M-y
)^{\vfrac{d-1}{2}}y^{-2d}\,dy
\\
&&\qquad  \leq \biggl(\frac{M}{2} \biggr)^{\vfrac{d-1}{2}}\int
_{\sqrt
{M}}^{\sfrac{M}{2}}y^{-2d}\,dy
\\
&&\qquad \leq
CM^{\vfrac{d-1}{2}}\sqrt{M}^{-2d+1}=O \bigl(M^{-\sfrac{d}{2}} \bigr)=O
\bigl(l^{-\sfrac{d}{2}} \bigr).
\end{eqnarray*}
The third sum is over $k\in[M-\sqrt{M},M+\sqrt{M} ]$. This
part contains roughly $\sqrt{M}$ summands, each of them is bounded
by $M^{-\vfrac{d+1}{2}}$ so the sum is $O (l^{-\sfrac{d}{2}} )$.
Finally, the last sum is over $k\geq M+\sqrt{M}$. This case is similar
to the second sum and is bounded up to a constant by
\begin{eqnarray*}
&& \int_{M+\sqrt{M}}^{\infty}x^{\vfrac{d-1}{2}} (x-M
)^{-2d}\,dx
\\
&&\qquad   = \int_{M+\sqrt{M}}^{2M}x^{\vfrac{d-1}{2}}
(x-M )^{-2d}\,dx+\int_{2M}^{\infty}x^{\vfrac{d-1}{2}}
(x-M )^{-2d}\,dx
\\
&&\qquad  = \int_{\sqrt{M}}^{M} (y+M )^{\vfrac {d-1}{2}}y^{-2d}\,dy+
\int_{M}^{\infty} (y+M )^{\vfrac {d-1}{2}}y^{-2d}\,dy
\\
&&\qquad  \leq  CM^{\vfrac{d-1}{2}}\int_{\sqrt{M}}^{M}y^{-2d}\,dy+C
\int_{M}^{\infty}y^{\vfrac{d-1}{2}}y^{-2d}\,dy=O
\bigl(l^{-\sfrac{d}{2}} \bigr).
\end{eqnarray*}
Thus,
\[
\mathbb{P}^{z_{1}} \bigl( (X_{T_{l}},T_{l} )= (w,n )
\mid\hat{B} (l ) \bigr)\leq Cl^{-\sfrac{d}{2}}.
\]
The arguments for the other two inequalities are very similar and,
therefore, we only discuss the proof for~(\ref{eqAnnealedestmiddle2-1}).
Assuming without loss of generality that $y=x+e_{j}$ for some $1\leq
j\leq d$
we have
\begin{eqnarray*}
&& \bigl\llvert\mathbb{P}^{z_{1}} \bigl( (X_{T_{l}},T_{l}
)= (w,n )\mid \hat{B} (l ) \bigr)-\mathbb{P}^{z_{1}} \bigl(
(X_{T_{l}},T_{l} )= (z,n+1 )\mid\hat{B} (l ) \bigr)
\bigr\rrvert
\\
&&\qquad=  \bigl\llvert\mathbb{P}^{z_{1}} \bigl( (X_{T_{l}},T_{l}
)= (w,n )\bigl\llvert B (l ) \bigr)-\mathbb{P}^{z_{1}} \bigl(
(X_{T_{l}},T_{l} )= (z,n+1 )\bigr\rrvert B (l ) \bigr)\bigr
\rrvert
\\
&&\qquad\leq \frac{1}{\mathbb{P}^{z_{1}} (B (l )
)}\sum_{k=1}^{\infty}
\bigl\llvert\mathbb{P}^{z_{1}} \bigl( (X_{\tau
_{k}},
\tau_{k} )= (w,n ) \bigr)-\mathbb{P}^{z_{1}} \bigl(
(X_{\tau_{k}},\tau_{k} )= (z,n+1 ) \bigr)\bigr\rrvert.
\end{eqnarray*}
We\vspace*{1pt} can now continue as in the previous case by separating the sum
for $k\leq M$ and $k>M$ and also adding either the assumption $
\langle X_{\tau_{k}}-X_{\tau_{ \lceil\sfrac{k}{2} \rceil
}},e_{1} \rangle\geq\frac{l}{2}$
or $ \langle X_{\tau_{ \lceil\sfrac{k}{2} \rceil
}},e_{1} \rangle\geq\frac{l}{2}$.
Now we can estimate each term in the same way except that in (\ref
{eqmiddlestepannealedderivativeestimationmarker})
we have
\begin{eqnarray*}
&& \biggl\llvert\mathbb{P}^{z_{1}} \biggl( (X_{\tau_{k}},
\tau_{k} )= (w,n ), \langle X_{\tau_{ \lceil\sfrac
{k}{2} \rceil},e_{1}} \rangle\leq
\frac{l}{2} \biggr)
\\[-2pt]
&&\quad{}-\mathbb{P}^{z_{1}} \biggl( (X_{\tau_{k}},
\tau_{k} )= (z,n+1 ), \langle X_{\tau_{ \lceil\sfrac
{k}{2} \rceil},e_{1}} \rangle\leq
\frac{l}{2} \biggr)\biggr\rrvert
\\[-2pt]
&&\qquad\leq \sum_{x: \langle x,e_{1} \rangle\leq\sfrac
{l}{2}}\sum_{s\in\mathbb{N}}
\mathbb{P}^{z_{1}} \bigl( (X_{\tau
_{ \lceil\sfrac{k}{2} \rceil}},\tau_{ \lceil\sfrac
{k}{2} \rceil} )= (x,s )
\bigr)
\\[-2pt]
&&\quad\qquad{}\times \bigl\llvert\mathbb{P}^{z_{1}} \bigl( (X_{\tau_{k}},\tau
_{k} )= (w,n )\mid (X_{\tau_{ \lceil \sfrac{k}{2} \rceil}},\tau_{ \lceil\sfrac{k}{2}
\rceil} )=
(x,s ) \bigr)
\\[-2pt]
&&\quad\qquad{}-\mathbb{P}^{z_{1}} \bigl( (X_{\tau_{k}},
\tau_{k} )= (z,n+1 )\mid (X_{\tau_{ \lceil\sfrac{k}{2} \rceil
}},\tau
_{ \lceil\sfrac{k}{2} \rceil} )= (x,s ) \bigr)\bigr\rrvert
\\[-2pt]
&&\qquad \leq \sum_{x: \langle x,e_{1} \rangle\leq\sfrac
{l}{2}}\sum_{s\in\mathbb{N}}Ck^{-\vfrac{d+2}{2}}
\mathbb{P}^{z_{1}} \bigl( (X_{\tau_{ \lceil\sfrac{k}{2}
\rceil}},\tau_{ \lceil\sfrac{k}{2} \rceil} )= (x,s
) \bigr)
\\[-2pt]
&&\qquad =Ck^{-\vfrac{d+2}{2}}\mathbb{P}^{z_{1}} \biggl( \langle
X_{\tau_{ \lceil\sfrac{k}{2} \rceil
}},e_{1} \rangle\leq\frac{l}{2} \biggr),
\end{eqnarray*}
where for the last inequality we used (\ref{eqFourieranalysis2})
instead of (\ref{eqFourieranalysis1}). Apart from that difference,
the proof continues via the same lines.
\end{pf}
Next, we turn to the proof of the annealed estimations. We follow the
same ideas as in the proof of \cite{berger2008slowdown}, Lemma 4.2.

\begin{pf*}{Proof of Lemma~\ref{lemAnnealedderivativeestimationsd-1+time}}
(1) Denote $u:= (X_{{T_{\partial\mathcal
{P}(0,N)}}},{T_{\partial\mathcal{P}(0,N)}} )$. Then
\begin{eqnarray*}
&& \mathbb{P}^{z} \bigl(u= (x,m ) \bigr)
\\[-2pt]
&&\qquad  =\sum_{l\leq N^{2}}\mathbb{P}^{z} \bigl(
\hat{B} (l ) \bigr) \mathop{\sum_{
w\in H_{l}, n\in\mathbb{N}}}_{
w-z_{1}\leftrightarrow n}\mathbb{P}^{z} \bigl( (X_{T_{l}},T_{l} )= (w,n )
\mid\hat{B} (l ) \bigr)
\\[-2pt]
&&\quad\qquad{}\times
\BP^{z} \bigl(u= (x,m )\mid\hat{B} (l ),
(X_{T_{l}},T_{l} )= (w,n ) \bigr)
\\[-2pt]
&&\qquad \overset{ (1 )} {\leq}\sum_{l\leq N^{2}}\mathbb
{P}^{z} \bigl(\hat{B} (l ) \bigr) \mathop{\max_{
w\in H_{l}, n\in\mathbb{N}}}_{
w-z_{1}\leftrightarrow n}
Cl^{-\sfrac{d}{2}}\BP^{z} \bigl(u= (x,m )\mid\hat{B} (l ),
(X_{T_{l}},T_{l} )= (w,n ) \bigr)
\\[-2pt]
&&\qquad \overset{ (2 )} {=}\sum
_{l\leq N^{2}}\mathbb{P}^{z} \bigl(\hat{B} (l )
\bigr)Cl^{-\sfrac{d}{2}}
\\[-2pt]
&&\qquad  \overset{ (3 )} {\leq}\sum_{l\leq N^{2}}Ce^{-c
(\vfrac{N^{2}-l}{2} )^{\gamma}}Cl^{-\sfrac{d}{2}}
\\[-2pt]
&&\qquad
=\sum_{l\leq\sfrac{N^{2}}{2}}Ce^{-c (\vfrac{N^{2}-l}{2}
)^{\gamma}}Cl^{-\sfrac{d}{2}}+\sum
_{\sfrac{N^{2}}{2}\leq l\leq
N^{2}}Ce^{-c (\vfrac{N^{2}-l}{2} )^{\gamma}}Cl^{-\sfrac{d}{2}}
\\[-2pt]
&&\qquad  \leq C
\frac{N^{2}}{2}e^{-c (\sfrac{N^{2}}{2} )^{\gamma
}}+\frac{C}{N^{d}}\sum
_{\sfrac{N^{2}}{2}\leq l\leq N^{2}}Ce^{-c (\vfrac{N^{2}-l}{2} )^{\gamma}}\leq\frac{C}{N^{d}},
\end{eqnarray*}
where for $ (1 )$ we used Lemma~\ref
{lemAnnealedderivativeestimationsd-1+time},
(\ref{eqannealedderivativeestimation1}), for $ (2 )$
we used the shift invariance of the annealed walk and for the first
sum in $ (3 )$ we used Corollary~\ref{corlengthofregenerations}
(see also Remark~\ref{RemTheeventAN}).

(2) For $y\in\mathbb{Z}^{d}$ such that $\llVert y-x\rrVert _{1}=1$,
we can find
$2\leq j\leq d$
such that $y=x\pm e_{j}$ (without loss of generality assume the sign
is $+$) and, therefore,
\begin{eqnarray*}
&& \mathbb{P}^{z} \bigl(u= (y,m+1 ) \bigr)
\\[-1pt]
&&\qquad=  \sum_{l\leq N^{2}}\mathbb{P}^{z} \bigl(
\hat{B} (l ) \bigr) \mathop{\sum_{
w\in H_{l}, n\in\mathbb{N}}}_{
w-z_{1}\leftrightarrow n}\mathbb{P}^{z} \bigl( (X_{T_{l}},T_{l} )=
(w+e_{j},n+1 )\mid\hat{B} (l ) \bigr)
\\[-1pt]
&&\quad\qquad{}\times
\BP^{z} \bigl(u=
(y,m+1 )\mid\hat{B} (l ), (X_{T_{l}},T_{l} )=
(w+e_{j},n+1 ) \bigr)
\\[-1pt]
&&\qquad=  \sum_{l\leq N^{2}}\mathbb{P}^{z} \bigl(
\hat{B} (l ) \bigr) \mathop{\sum_{
w\in H_{l}, n\in\mathbb{N}}}_{
w-z_{1}\leftrightarrow n}
\mathbb{P}^{z} \bigl( (X_{T_{l}},T_{l} )=
(w+e_{j},n+1 )\mid\hat{B} (l ) \bigr)
\\[-1pt]
&&\quad\qquad{}\times \BP^{z} \bigl(u=
(x,m )\mid\hat{B} (l ), (X_{T_{l}},T_{l} )= (w,n ) \bigr).
\end{eqnarray*}
Subtracting the formula for $ $ $\mathbb{P}^{z} (u=
(y,m+1 ) )$
from the one for $\mathbb{P}^{z_{1}} (u= (x,m ) )$,
we thus get
\begin{eqnarray*}
&& \bigl\llvert\mathbb{P}^{z} \bigl(u= (x,m ) \bigr)-\mathbb
{P}^{z} \bigl(u= (y,m+1 ) \bigr)\bigr\rrvert
\\[-1pt]
&&\qquad\leq \sum_{l\leq N^{2}}\mathbb{P}^{z} \bigl(
\hat{B} (l ) \bigr) \mathop{\sum_{
w\in H_{l}, n\in\mathbb{N}}}_{
w-z_{1}\leftrightarrow n}
\bigl[\bigl\llvert\mathbb{P}^{z} \bigl( (X_{T_{l}},T_{l}
)= (w,n )\mid \hat{B} (l ) \bigr)
\\[-1pt]
&&\quad\qquad{}-\mathbb{P}^{z} \bigl(
(X_{T_{l}},T_{l} )= (w+e_{j},n+1 )\mid
\hat{B} (l ) \bigr)\bigl\llvert
\\[-1pt]
&&\quad\qquad{}\times \BP^{z} \bigl(u= (x,m )\mid\hat{B} (l
), (X_{T_{l}},T_{l} )= (w,n ) \bigr) \bigr]
\\[-1pt]
&&\qquad\leq \sum_{l\leq N^{2}}\mathbb{P}^{z} \bigl(
\hat{B} (l ) \bigr)\max_{w\in H_{l}, 2\leq j\leq d, n\in\mathbb{N}}\bigl
\llvert\mathbb{P}^{z}
\bigl( (X_{T_{l}},T_{l} )= (w,n )\mid \hat{B} (l )
\bigr)
\\[-1pt]
&&\quad\qquad{} -\mathbb{P}^{z} \bigl( (X_{T_{l}},T_{l} )=
(w+e_{j},n+1 )\mid\hat{B} (l ) \bigr)\bigr\rrvert,
\end{eqnarray*}
where as before we used the shift invariance. Using (\ref
{eqAnnealedestmiddle2-1}),
we get
\[
\bigl\llvert\mathbb{P}^{z} \bigl(u= (x,m ) \bigr)-\mathbb
{P}^{z} \bigl(u= (y,m+1 ) \bigr)\bigr\rrvert\leq\sum
_{l\leq
N^{2}}\mathbb{P}^{z} \bigl(\hat{B} (l )
\bigr)Cl^{-\vfrac{d+1}{2}}
\]
which by the same argument as before is bounded by $CN^{-d-1}$.

(3) We start with the case where $w=z+e_{j}$ for some $2\leq j\leq d$.
Due to the (2), we have
\begin{eqnarray*}
&& \bigl\llvert\mathbb{P}^{z} \bigl(u= (x,m ) \bigr)-\mathbb
{P}^{z+e_{j}} \bigl(u= (x,m+1 ) \bigr)\bigr\rrvert
\\
&&\qquad  \leq \bigl\llvert
\mathbb{P}^{z} \bigl(u= (x,m ) \bigr)-\mathbb{P}^{z+e_{j}}
\bigl(u= (x+e_{j},m ) \bigr)\bigr\rrvert
\\
&&\quad\qquad{} +\bigl\llvert\mathbb{P}^{z+e_{j}} \bigl(u= (x+e_{j},m )
\bigr)-\mathbb{P}^{z+e_{j}} \bigl(u= (x,m+1 ) \bigr)\bigr\rrvert
\\
&&\qquad \leq\bigl\llvert\mathbb{P}^{z} \bigl(u= (x,m ) \bigr)-
\mathbb{P}^{z+e_{j}} \bigl(u= (x+e_{j},m ) \bigr)\bigr\rrvert
+CN^{-d-1},
\end{eqnarray*}
and, therefore, it is enough to compare $\mathbb{P}^{z} (u=
(x,m ) )$
with $\mathbb{P}^{z+e_{j}} (u= (x+e_{j},m ) )$.
In this case, we have
\begin{eqnarray*}
&& \mathbb{P}^{z+e_{j}} \bigl(u= (x+e_{j},m ) \bigr)
\\
&&\qquad=  \sum_{l\leq N^{2}}\mathbb{P}^{z+e_{j}} \bigl(
\hat{B} (l ) \bigr) \mathop{\sum_{
w\in H_{l}, n\in\mathbb{N}}}_{
w-z_{1}\leftrightarrow n}
\mathbb{P}^{z+e_{j}} \bigl( (X_{T_{l}},T_{l} )=
(w+e_{j},n )\mid \hat{B} (l ) \bigr)
\\
&&\quad\qquad{}\times \BP^{z+e_{j}}
\bigl(u= (x+e_{j},m )\mid \hat{B} (l ), (X_{T_{l}},T_{l}
)= (w+e_{j},n ) \bigr)
\\
&&\qquad=  \sum_{l\leq N^{2}}\mathbb{P}^{z+e_{j}} \bigl(
\hat{B} (l ) \bigr) \mathop{\sum_{
w\in H_{l}, n\in\mathbb{N}}}_{
w-z_{1}\leftrightarrow n}
\mathbb{P}^{z+e_{j}} \bigl( (X_{T_{l}},T_{l} )=
(w+e_{j},n )\mid \hat{B} (l ) \bigr)
\\
&&\quad\qquad{}\times \BP^{z}
\bigl(u= (x,m )\mid \hat{B} (l ), (X_{T_{l}},T_{l} )=
(w,n ) \bigr)
\\
&&\qquad=  \sum_{l\leq N^{2}}\mathbb{P}^{z} \bigl(
\hat{B} (l ) \bigr) \mathop{\sum_{
w\in H_{l}, n\in\mathbb{N}}}_{
w-z_{1}\leftrightarrow n}
\mathbb{P}^{z+e_{j}} \bigl( (X_{T_{l}},T_{l} )=
(w+e_{j},n )\mid \hat{B} (l ) \bigr)
\\
&&\quad\qquad{}\times \BP^{z}
\bigl(u= (x,m )\mid \hat{B} (l ), (X_{T_{l}},T_{l} )=
(w,n ) \bigr)
\end{eqnarray*}
and, therefore,
\begin{eqnarray*}
&& \bigl\llvert\mathbb{P}^{z} \bigl(u= (x,m ) \bigr)-\mathbb
{P}^{z+e_{j}} \bigl(u= (x+e_{j},m ) \bigr)\bigr\rrvert
\\
&&\qquad\leq \sum_{l\leq N^{2}}\mathbb{P}^{z} \bigl(
\hat{B} (l ) \bigr) \mathop{\sum_{
w\in H_{l}, n\in\mathbb{N}}}_{
w-z_{1}\leftrightarrow n}
\bigl[\bigl\llvert\mathbb{P}^{z+e_{j}} \bigl( (X_{T_{l}},T_{l}
)= (w+e_{j},n )\mid \hat{B} (l ) \bigr)
\\
&&\quad\qquad{} -\mathbb{P}^{z}
\bigl( (X_{T_{l}},T_{l} )= (w,n )\mid \hat{B} (l )
\bigr)\bigr\rrvert\BP^{z} \bigl(u= (x,m )\mid \hat{B}
(l ), (X_{T_{l}},T_{l} )= (w,n ) \bigr) \bigr]
\\
&&\qquad\leq \sum_{l\leq N^{2}}\mathbb{P}^{z} \bigl(
\hat{B} (l ) \bigr)\max_{w\in H_{l}, 2\leq j\leq d, n\in\mathbb{N}}\bigl\llvert \mathbb{P}^{z+e_{j}}
\bigl( (X_{T_{l}},T_{l} )= (w+e_{j},n )\mid \hat{B} (l ) \bigr)
\\
&&\quad\qquad{} -\mathbb{P}^{z} \bigl( (X_{T_{l}},T_{l}
)= (w,n )\mid \hat{B} (l ) \bigr)\bigr\rrvert.
\end{eqnarray*}
Using (\ref{eqAnnealedestmiddle2-1}) and (\ref{eqAnnealedestmiddle3-1})
gives
\[
\bigl\llvert\mathbb{P}^{z} \bigl(u= (x,m ) \bigr)-\mathbb
{P}^{z+e_{1}} \bigl(u= (x+e_{1},m ) \bigr)\bigr\rrvert\leq\sum
_{l\leq N^{2}}\mathbb{P}^{z} \bigl(\hat{B} (l )
\bigr)Cl^{-\vfrac{d+1}{2}}
\]
and the proof is completed in the same way as in (2). Finally, we
turn to deal with the case $w=z+e_{1}$. One can rewrite the term
for $\mathbb{P}^{z+e_{1}} (u= (x=m ) )$ as
\begin{eqnarray*}
&& \mathbb{P}^{z+e_{1}} \bigl(u= (x,m ) \bigr)
\\
&&\qquad=  \sum_{l\leq N^{2}-1}\mathbb{P}^{z+e_{1}} \bigl(
\hat{B} (l+1 ) \bigr)
\\
&&\quad\qquad{}\times \mathop{\sum_{
w\in H_{l}, n\in\mathbb{N}}}_{
w-z\leftrightarrow n}
\mathbb{P}^{z+e_{1}} \bigl( (X_{T_{l+1}},T_{l+1} )=
(w+e_{1},n )\mid\hat{B} (l+1 ) \bigr)
\\
&&\quad\qquad{}\times  \BP^{z+e_{1}} \bigl(u= (x,m )\mid\hat{B} (l+1 ),
(X_{T_{l+1}},T_{l+1} )= (w+e_{1},n ) \bigr)
\\
&&\qquad =  \sum_{l\leq N^{2}-1}\mathbb{P}^{z} \bigl(
\hat{B} (l ) \bigr) \mathop{\sum_{
w\in H_{l}, n\in\mathbb{N}}}_{
w-z\leftrightarrow n}\mathbb{P}^{z} \bigl( (X_{T_{l}},T_{l} )= (w,n )
\mid\hat{B} (l ) \bigr)
\\
&&\quad\qquad{}\times
\BP^{z+e_{1}} \bigl(u= (x,m )\mid\hat{B} (l+1 ),
(X_{T_{l+1}},T_{l+1} )= (w+e_{1},n ) \bigr)
\end{eqnarray*}
and, therefore, using Lemma~\ref{lemmiddlestepannealedestimations}
\begin{eqnarray*}
&& \bigl\llvert\mathbb{P}^{z} \bigl(u= (x,m ) \bigr)-\mathbb
{P}^{z+e_{1}} \bigl(u= (x,m ) \bigr)\bigr\rrvert
\\
&&\qquad\leq \sum_{l\leq N^{2}-1}\mathbb{P}^{z} \bigl(
\hat{B} (l ) \bigr) \mathop{\sum_{
w\in H_{l}, n\in\mathbb{N}}}_{
w-z\leftrightarrow n}
\mathbb{P}^{z} \bigl( (X_{T_{l}},T_{l} )= (w,n )
\mid\hat{B} (l ) \bigr)
\\
&&\quad\qquad{}\times \bigl\llvert\BP^{z} \bigl(u= (x,m )\mid \hat{B} (l ),
(X_{T_{l}},T_{l} )= (w,n ) \bigr)
\\
&&\quad\qquad{} -\BP^{z+e_{1}}
\bigl(u= (x,m )\mid\hat{B} (l+1 ), (X_{T_{l+1}},T_{l+1}
)= (w+e_{1},n ) \bigr)\bigr\rrvert
\\
&&\quad\qquad{}+  \mathbb{P}^{z} \bigl(\hat{B} \bigl(N^{2} \bigr) \bigr)
\mathop{\max_{
w\in H_{N^{2}}, n\in\mathbb{N}}}_{
w-z\leftrightarrow n} \mathbb{P}^{z}
\bigl( (X_{T_{N^{2}}},T_{N^{2}} )= (w,n )\mid\hat{B}
\bigl(N^{2} \bigr) \bigr)\\
&&\quad\qquad{}\times \BP^{z} \bigl(u= (x,m )\mid\hat
{B} \bigl(N^{2} \bigr), (X_{T_{N^{2}}},T_{N^{2}} )= (w,n
) \bigr)
\\
&&\qquad\leq \sum_{l\leq N^{2}-1} \mathbb{P}^{z} \bigl(
\hat{B} (l ) \bigr)
\\
&&\quad\qquad{}\times  \mathop{\sum_{
w\in H_{l}, n\in\mathbb{N}}}_{
w-z\leftrightarrow n}
Cl^{-\vfrac{d+1}{2}}\big| \BP^{z} \bigl(u= (x,m )\mid\hat{B} (l ),
(X_{T_{l}},T_{l} )= (w,n ) \bigr)
\\
&&\quad\qquad{}-\BP^{z+e_{1}} \bigl(u= (x,m )\mid\hat{B} (l+1 ),
(X_{T_{l+1}},T_{l+1} )= (w+e_{1},n ) \bigr)\big|
\\
&&\quad\qquad{}+  \mathbb{P}^{z} \bigl(\hat{B} \bigl(N^{2} \bigr) \bigr)
\cdot CN^{-\vfrac{d+1}{2}}
\\
&&\qquad\leq 2C\sum_{l\leq N^{2}-1}\mathbb{P}^{z} \bigl(
\hat{B} (l ) \bigr)l^{-\vfrac{d+1}{2}}+\mathbb{P}^{z} \bigl(\hat{B}
\bigl(N^{2} \bigr) \bigr)\cdot CN^{-\vfrac{d+1}{2}}
\\
&&\qquad \leq CN^{-\vfrac{d+1}{2}},
\end{eqnarray*}
as required.
\end{pf*}

\subsubsection{Proof of Lemma \protect\ref{lemgeneralannealedestimations}}\label{sec8.2.3}

The proof of Lemma~\ref{lemgeneralannealedestimations} follows
very similar lines to the one of Lemma~\ref
{lemAnnealedderivativeestimationsd-1+time}
and is based on very similar estimations to the one obtained in Lemma
\ref{lemmiddlestepannealedestimations}. Here, we need a slightly
different version of it in which we replace $\hat{B (l )}$
the event that the last regeneration time is in the hyperplane $l$
with the event $\hat{Z} (l )$ in which the last regeneration
time is at time $l$.

\begin{lem}[(Middle step in Lemma~\ref
{lemAnnealedderivativeestimationsd-1+time})]
\label{lemMiddlestepinLemmageneralestimations} Let $d\geq4$
and assume $P$ is uniformly elliptic, i.i.d. and satisfies $
(\mathscr{P} )$.
Fix $z\in\BZ^{d}$ and $n\in\mathbb{N}$. For $k,l\in\mathbb{N}$
let $Z (l,k )$ be the event that $\tau_{k}=l$, $Z
(l )=\bigcup_{k}Z (l,k )$
and
\[
\hat{Z}_{n} (l )\equiv\hat{Z} (l )=Z (l )\cap\bigcap
_{j=l+1}^{n}Z^{c} (j ).
\]
Then for every $z\in\mathbb{Z}^{d}$:
\begin{longlist}[(2)]
\item[(1)] For every $l\leq n$ and $x\in\mathbb{Z}^{d}$
%
\begin{equation}
\mathbb{P}^{z} \bigl(X_{l}=x\mid\hat{Z} (l ) \bigr)\leq
Cl^{-\sfrac{d}{2}}.\label{eqAnnealedestmiddle1}
\end{equation}

\item[(2)] For every $l\in\mathbb{N}$, and every $x,y\in\mathbb{Z}^{d}$ such
that $\llVert x-y\rrVert _{1}=1$
%
\begin{equation}
\bigl\llvert\mathbb{P}^{z} \bigl(X_{l}=x\mid\hat{Z} (l
) \bigr)-\mathbb{P}^{z} \bigl(X_{l+1}=y\mid\hat{Z} (l+1 )
\bigr)\bigr\rrvert\leq Cl^{-\vfrac{d+1}{2}}.\label
{eqAnnealedestmiddle2}
\end{equation}

\item[(3)] For every $l\in\mathbb{N}$ every $x\in\mathbb{Z}^{d}$ and every
$1\leq j\leq d$
%
\begin{equation}
\bigl\llvert\mathbb{P}^{z} \bigl(X_{l}=x\mid\hat{Z} (l
) \bigr)-\mathbb{P}^{z+e_{j}} \bigl(X_{l+1}=x\mid\hat{Z} (l+1 )
\bigr)\bigr\rrvert\leq Cl^{-\vfrac{d+1}{2}}.\label
{eqAnnealedestmiddle3}
\end{equation}

\item[(4)] For every $\varepsilon>0$, every partition $\Pi_{n}$ of
$\mathbb{Z}^{d}$
into boxes of side length $n^{\varepsilon}$ and any $l\in\mathbb{N}$
%
\begin{equation}
\qquad \sum_{\Delta\in\Pi_{n}}\mathop{\sum
_{
x\in\Delta}}_{
x-z\leftrightarrow l} \max_{y\in\Delta}
\mathbb{P}^{z} \bigl(X_{l}=y\mid\hat{Z} (l ) \bigr)-
\mathbb{P}^{z} \bigl(X_{l}=x\mid\hat{Z} (l ) \bigr)\leq
Cl^{-\sfrac{1}{2}+3d\varepsilon}.\label{eqeqAnnealedestmiddle4}
\end{equation}
\end{longlist}
\end{lem}

\begin{pf}
Due to the independence of $X_{l}$ from $\bigcap_{j=l+1}^{n}Z^{c}
(j )$
conditioned on $Z (l )$ we get that for every $M\in\mathbb{N}$
\begin{eqnarray*}
\mathbb{P}^{z}\bigl(X_{l}=x\mid \hat{Z} (l ) \bigr) & =&\mathbb{P}^{z} \bigl(X_{l}=x\mid Z
(l ) \bigr)
\\
& =&\frac{1}{\mathbb{P}^{z} (Z (l ) )}\sum_{k=1}^{\infty}
\mathbb{P}^{z} \bigl( (X_{\tau_{k}},\tau_{k} )= (x,l )
\bigr)
\\
& =&\frac{1}{\mathbb{P}^{z} (Z (l ) )}\sum_{k=1}^{M}
\mathbb{P}^{z_{1}} \biggl( (X_{\tau_{k}},\tau_{k} )= (x,l
), \tau_{ \lceil\sfrac
{k}{2} \rceil}\geq\frac{l}{2} \biggr)
\\
&&{} +\frac{1}{\mathbb{P}^{z} (Z (l ) )}\sum_{k=1}^{M}
\mathbb{P}^{z_{1}} \biggl( (X_{\tau_{k}},\tau_{k} )= (x,l
), \tau_{k}-\tau_{ \lceil\sfrac
{k}{2} \rceil}\geq\frac{l}{2} \biggr)
\\
&&{} +\frac{1}{\mathbb{P}^{z} (Z (l ) )}\sum_{k=M+1}^{\infty}
\mathbb{P}^{z_{1}} \biggl( (X_{\tau_{k}},\tau_{k} )= (x,l
), \tau_{ \lceil\sfrac
{k}{2} \rceil}\leq\frac{l}{2} \biggr)
\\
&&{} +\frac{1}{\mathbb{P}^{z} (Z (l ) )}\sum_{k=M+1}^{\infty}
\mathbb{P}^{z_{1}} \biggl( (X_{\tau_{k}},\tau_{k} )= (x,l
), \tau_{k}-\tau_{ \lceil\sfrac
{k}{2} \rceil}\leq\frac{l}{2} \biggr).
\end{eqnarray*}
Claim~\ref{clmFourierclaim} gives
\begin{eqnarray*}
&& \mathbb{P}^{z} \biggl( (X_{\tau_{k}},\tau_{k} )= (x,l
), \tau_{ \lceil\sfrac{k}{2} \rceil}\leq\frac{l}{2} \biggr)
\\
&&\qquad   =\sum
_{s=0}^{\sfrac{l}{2}} \mathop{\sum
_{w\in\mathbb{Z}^{d}}}_{
w-z\leftrightarrow s} \mathbb{P}^{z} \bigl(
(X_{\tau_{ \lceil\sfrac{k}{2} \rceil}},\tau_{ \lceil\sfrac
{k}{2} \rceil} )= (w,s ) \bigr)
\\
&&\quad\qquad{}\times \mathbb
{P}^{z} \bigl( (X_{\tau_{k}},\tau_{k} )= (x,l )\mid
(X_{\tau_{ \lceil\sfrac{k}{2} \rceil}},\tau_{ \lceil\sfrac{k}{2} \rceil}
)= (w,s ) \bigr)
\\
&&\qquad  \leq Ck^{-\vfrac{d+1}{2}}\sum_{s=0}^{\sfrac{l}{2}}
\mathop{\sum_{
w\in\mathbb{Z}^{d}}}_{
w-z\leftrightarrow s}
\mathbb{P}^{z} \bigl( (X_{\tau_{ \lceil\sfrac{k}{2} \rceil}},\tau_{ \lceil\sfrac
{k}{2} \rceil} )= (w,s )
\bigr)
\\
&&\qquad  = Ck^{-\vfrac{d+1}{2}}\mathbb{P}^{z_{1}} \biggl(\tau_{ \lceil \sfrac{k}{2} \rceil}\leq
\frac{l}{2} \biggr),
\end{eqnarray*}
and similarly we have
\begin{eqnarray*}
\mathbb{P}^{z} \biggl( (X_{\tau_{k}},\tau_{k} )= (x,l
), \tau_{k}-\tau_{ \lceil\sfrac{k}{2}
\rceil}\leq\frac{l}{2} \biggr) & \leq& Ck^{-\vfrac{d+1}{2}}\mathbb{P}^{z_{1}} \biggl(\tau_{k}-
\tau_{ \lceil\sfrac{k}{2}
\rceil}\leq\frac{l}{2} \biggr)
\\
& \leq& Ck^{-\vfrac{d+1}{2}}\mathbb{P}^{z_{1}} \biggl(\tau_{ \lceil\sfrac{k}{2} \rceil}
\leq\frac{l}{2} \biggr).
\end{eqnarray*}
Repeating\vspace*{1pt} the same calculations while separating the sum according
to the events $\tau_{ \lfloor\sfrac{k}{2} \rfloor}\geq
\frac{l}{2}$
and $\tau_{k}-\tau_{ \lfloor\sfrac{k}{2} \rfloor}\geq
\frac{l}{2}$,
we get that
\[
\mathbb{P}^{z} \biggl( (X_{\tau_{k}},\tau_{k} )= (x,l
), \tau_{ \lceil\sfrac{k}{2} \rceil}\geq\frac{l}{2} \biggr)\leq
Ck^{-\vfrac{d+1}{2}}
\mathbb{P}^{z} \biggl(\tau_{ \lceil\sfrac{k}{2} \rceil}\geq\frac{l}{2}
\biggr)
\]
and
\[
\mathbb{P}^{z} \biggl( (X_{\tau_{k}},\tau_{k} )= (x,l
), \tau_{k}-\tau_{ \lceil\sfrac{k}{2}
\rceil}\geq\frac{l}{2} \biggr)\leq
Ck^{-\vfrac{d+1}{2}}\mathbb{P}^{z} \biggl(\tau_{ \lceil\sfrac{k}{2}
\rceil}\geq
\frac
{l}{2} \biggr),
\]
combining all of the above yields
\begin{eqnarray*}
\mathbb{P}^{z} \bigl(X_{l}=x\mid\hat{Z} (l ) \bigr) & \leq& C\sum_{k=1}^{M}k^{-\vfrac{d+1}{2}}
\mathbb{P}^{z} \biggl(\tau_{ \lceil\sfrac{k}{2} \rceil}\geq\frac{l}{2}
\biggr)
\\
&&{} +C\sum_{k=M+1}^{\infty}k^{-\vfrac{d+1}{2}}
\mathbb{P}^{z} \biggl(\tau_{ \lceil\sfrac{k}{2} \rceil}\leq\frac{l}{2}
\biggr).
\end{eqnarray*}
Choosing
\[
M=\frac{l}{\mathbb{E}^{z} [ \langle\tau_{2}-\tau
_{1},e_{1} \rangle]}=O (l )
\]
and using Theorem~\ref{teoRegenerationtimestimeestimate} (see
also Corollary~\ref{corlengthofregenerations}) we get that $\tau
_{k}-\tau_{k-1}$
has finite $2d$ moments, and from standard estimates for the sum
of i.i.d. variable [that the $2d$ moment for the sum of $k$ i.i.d.
mean zero random variables grows like $O (k^{d} )$], thus
\[
\mathbb{P}^{z} \biggl(\tau_{ \lceil\sfrac{k}{2} \rceil
}\geq\frac{l}{2}
\biggr),\mathbb{P}^{z} \biggl(\tau_{ \lceil \sfrac{k}{2} \rceil}\leq
\frac{l}{2} \biggr)\leq\min\biggl[1,\frac{Ck^{d}}{ (M-k )^{2d}} \biggr]
\]
and, therefore,
\[
\mathbb{P}^{z_{1}} \bigl(X_{l}=x\mid\hat{Z} (l ) \bigr)\leq
C\sum_{k=1}^{\infty}k^{-\vfrac{d+1}{2}}\min
\biggl[1,\frac
{Ck^{d}}{ (M-k )^{2d}} \biggr]=O \bigl(l^{-\sfrac{d}{2}} \bigr).
\]
To see the last equality, we have to separate the sum into four parts
as in \cite{berger2008slowdown}, Lemma 4.2, the first part is a sum
over $k\in[1,\frac{M}{2} ]$ which contains roughly $l$
summands, each of them is bounded by $Cl^{-d}$ and therefore the
whole sum is bounded by $Cl^{-d+1}$. The second part is a sum over
$k\in[\frac{M}{2},M-\sqrt{M} ]$. In this case, the sum is
bounded up to a constant by
\begin{eqnarray*}
\int_{\sfrac{M}{2}}^{M-\sqrt{M}}x^{\vfrac{d-1}{2}} (M-x
)^{-2d}\,dx & =&\int_{\sqrt{M}}^{\sfrac{M}{2}} (M-y
)^{\vfrac {d-1}{2}}y^{-2d}\,dy
\\
& \leq&\biggl(\frac{M}{2} \biggr)^{\vfrac{d-1}{2}}\int
_{\sqrt
{M}}^{\sfrac{M}{2}}y^{-2d}\,dy
\\
&\leq&
CM^{\vfrac{d-1}{2}}\sqrt{M}^{-2d+1}
\\
&=&O \bigl(M^{-\sfrac{d}{2}} \bigr)=O
\bigl(l^{-\sfrac{d}{2}} \bigr).
\end{eqnarray*}
The third sum is over $k\in[M-\sqrt{M},M+\sqrt{M} ]$. This
part contains roughly $\sqrt{M}$ summands, each of them is bounded
by $M^{-\vfrac{d+1}{2}}$ so the sum is $O (l^{-\sfrac{d}{2}} )$.
Finally, the last sum is over $k\geq M+\sqrt{M}$. This case is similar
to the second sum and is bounded up to a constant by
\begin{eqnarray*}
&& \int_{M+\sqrt{M}}^{\infty}x^{\vfrac{d-1}{2}} (x-M
)^{-2d}\,dx
\\
&&\qquad   =\int_{M+\sqrt{M}}^{2M}x^{\vfrac{d-1}{2}}
(x-M )^{-2d}\,dx+\int_{2M}^{\infty}x^{\vfrac{d-1}{2}}
(x-M )^{-2d}\,dx
\\
&&\qquad  =\int_{\sqrt{M}}^{M} (y+M )^{\vfrac {d-1}{2}}y^{-2d}\,dy+
\int_{M}^{\infty} (y+M )^{\vfrac {d-1}{2}}y^{-2d}\,dy
\\
&&\qquad  \leq CM^{\vfrac{d-1}{2}}\int_{\sqrt{M}}^{M}y^{-2d}\,dy+C
\int_{M}^{\infty}y^{\vfrac{d-1}{2}}y^{-2d}\,dy=O
\bigl(l^{-\sfrac{d}{2}} \bigr).
\end{eqnarray*}
Thus,
\[
\mathbb{P}^{z_{1}} \bigl(X_{l}=x\mid\hat{Z} (l ) \bigr)\leq
Cl^{-\sfrac{d}{2}}.
\]
The arguments for the other two inequalities are very similar and,
therefore, we only discuss the proof for~\ref{eqAnnealedestmiddle2-1}.
Similarly, to the first case, we have
%
\begin{eqnarray}
&& \bigl\llvert\mathbb{P}^{z_{1}} \bigl(X_{l}=x\mid
\hat{Z}_{n} (l ) \bigr)-\mathbb{P}^{z_{1}} \bigl(X_{l+1}=y
\mid\hat{Z}_{n+1} (l+1 ) \bigr)\bigr\rrvert
\nonumber
\\
&&\qquad =  \bigl\llvert\mathbb{P}^{z_{1}} \bigl(X_{l}=x\mid Z (l )
\bigr)-\mathbb{P}^{z_{1}} \bigl(X_{l+1}=y\mid Z (l+1 ) \bigr)
\bigr\rrvert
\nonumber
\\
&&\qquad=  \Biggl\llvert\frac{1}{\mathbb{P}^{z_{1}} (Z (l )
)}\sum_{k=1}^{\infty}
\mathbb{P}^{z_{1}} \bigl( (X_{\tau
_{k}},\tau_{k} )= (x,l )
\bigr)
\nonumber
\\
&&\quad\qquad{} -\frac{1}{\mathbb
{P}^{z_{1}} (Z (l+1 ) )}\sum_{k=1}^{\infty
}\mathbb{P}^{z_{1}} \bigl( (X_{\tau_{k}},\tau_{k} )= (y,l+1
) \bigr)\Biggr\rrvert
\nonumber
\\
&&\qquad \leq \frac{1}{\mathbb{P}^{z_{1}} (Z (l )
)}\sum_{k=1}^{\infty}
\bigl\llvert\mathbb{P}^{z_{1}} \bigl( (X_{\tau
_{k}},
\tau_{k} )= (x,l ) \bigr)-\mathbb{P}^{z_{1}} \bigl(
(X_{\tau_{k}},\tau_{k} )= (y,l+1 ) \bigr)\bigr\rrvert
\label{eqZllemma1}
\\
&&\quad\qquad{}+  \biggl\llvert\frac{1}{\mathbb{P}^{z_{1}} (Z (l )
)}-\frac{1}{\mathbb{P}^{z_{1}} (Z (l+1 )
)}\biggr\rrvert\sum
_{k=1}^{\infty}\mathbb{P}^{z_{1}} \bigl(
(X_{\tau
_{k}},\tau_{k} )= (y,l+1 ) \bigr).\label{eqZllemma2}
\end{eqnarray}

For the first term, we can now continue as in the previous case, by
first separating the sum for $k\leq M$ and $k>M$ and also adding
either the assumption $ \langle X_{\tau_{k}}-X_{\tau_{ \lceil\sfrac{k}{2} \rceil}},e_{1} \rangle\geq\frac{l}{2}$
or $ \langle X_{\tau_{ \lceil\sfrac{k}{2} \rceil
}},e_{1} \rangle\geq\frac{l}{2}$.
Now we can estimate each term in the same way except that in (\ref
{eqmiddlestepannealedderivativeestimationmarker})
we have
\begin{eqnarray*}
&& \biggl\llvert\mathbb{P}^{z_{1}} \biggl( (X_{\tau_{k}},
\tau_{k} )= (x,l ), \langle X_{\tau_{ \lceil\sfrac
{k}{2} \rceil},e_{1}} \rangle\leq
\frac{l}{2} \biggr)
\\
&&\quad{} -\mathbb{P}^{z_{1}} \biggl( (X_{\tau_{k}},
\tau_{k} )= (y,l+1 ), \langle X_{\tau_{ \lceil\sfrac
{k}{2} \rceil},e_{1}} \rangle\leq
\frac{l}{2} \biggr)\biggr\rrvert
\\
&&\qquad \leq \sum_{w: \langle w,e_{1} \rangle\leq\sfrac
{l}{2}}\sum_{s\in\mathbb{N}}
\mathbb{P}^{z_{1}} \bigl( (X_{\tau
_{ \lceil\sfrac{k}{2} \rceil}},\tau_{ \lceil\sfrac
{k}{2} \rceil} )= (w,s )
\bigr)
\\
&&\quad\qquad{}\times \bigl\llvert\mathbb{P}^{z_{1}} \bigl( (X_{\tau_{k}},\tau
_{k} )= (x,l )\mid(X_{\tau_{ \lceil \sfrac{k}{2} \rceil}},\tau_{ \lceil\sfrac{k}{2}
\rceil} )= (w,s )
\bigr)
\\
&&\quad\qquad{}-\mathbb{P}^{z_{1}} \bigl( (X_{\tau_{k}},\tau_{k} )=
(y,l+1 )\mid(X_{\tau_{ \lceil\sfrac{k}{2} \rceil}},\tau_{ \lceil\sfrac
{k}{2} \rceil} )= (w,s ) \bigr)\bigr
\rrvert
\\
&&\qquad \leq \sum_{w: \langle w,e_{1} \rangle\leq\sfrac
{l}{2}}\sum_{s\in\mathbb{N}}Ck^{-\vfrac{d+2}{2}}
\mathbb{P}^{z_{1}} \bigl( (X_{\tau_{ \lceil\sfrac{k}{2}
\rceil}},\tau_{ \lceil\sfrac{k}{2} \rceil} )= (w,s
) \bigr)
\\
&&\qquad =Ck^{-\vfrac{d+2}{2}}\mathbb{P}^{z_{1}} \biggl( \langle
X_{\tau_{ \lceil\sfrac{k}{2} \rceil
}},e_{1} \rangle\leq\frac{l}{2} \biggr),
\end{eqnarray*}
where for the last inequality we used (\ref{eqFourieranalysis2})
instead of (\ref{eqFourieranalysis1}). Apart from that difference
the proof for the first term continues via the same lines. Regarding
the second term, since $ \{ \tau_{k}-\tau_{k-1} \}
_{k=2}^{\infty}$
are i.i.d. under $\mathbb{P}$ using the Fourier analysis of Claim~\ref{clmFourierclaim}, one can verify that $\llvert \mathbb
{P}^{z_{1}} (Z (l ) )-\mathbb{P}^{z_{1}}
(Z (l+1 ) )\rrvert \leq\frac{C}{l}$
and, therefore,
\begin{eqnarray*}
\eqref{eqZllemma2} & \leq&\frac{C}{l\cdot\mathbb{P}^{z_{1}}
(l )}\cdot\mathbb{P}^{z_{1}}
\bigl(X_{l+1}=y\mid Z (l+1 ) \bigr)
\\
& \leq&\frac{C}{l}\cdot\mathbb{P}^{z_{1}} \bigl(X_{l+1}=y
\mid Z (l+1 ) \bigr)\leq C\cdot l^{\vfrac{d+2}{2}}=o \bigl(l^{\vfrac{d+1}{2}}
\bigr),
\end{eqnarray*}
where for the last inequality we used the first part of the lemma
[see (\ref{eqAnnealedestmiddle2})].

Finally, we turn to discuss the last term. The proof is very similar
to the previous ones. For every $\Delta\in\Pi_{n}$ denote by
$x_{\Delta}$
a point in $\Delta$ such that $\mathbb{P}^{z} (X_{l}=x_{\Delta
}\mid\hat{Z} (l ) )=\max_{y\in\Delta}\mathbb
{P}^{z} (X_{l}=y\mid\hat{Z} (l ) )$
we get
\begin{eqnarray*}
&& \sum_{\Delta\in\Pi_{n}}\mathop{\sum_{
x\in\Delta}}_{
x-z\leftrightarrow l}
\max_{y\in\Delta}\mathbb{P}^{z} \bigl(X_{l}=y
\mid\hat{Z} (l ) \bigr)-\mathbb{P}^{z} \bigl(X_{l}=x\mid
\hat{Z} (l ) \bigr)
\\
&&\qquad=  \frac{1}{\mathbb{P}^{z} (Z (l ) )}
\\
&&\quad\qquad{}\times \sum_{k=1}^{\infty}
\sum_{\Delta\in\Pi_{n}}
\mathop{\sum_{
x\in\Delta}}_{
x-z\leftrightarrow l}
\mathbb{P}^{z} \bigl( (X_{\tau_{k}},\tau_{k} )=
(x_{\Delta},l ) \bigr)-\mathbb{P}^{z} \bigl(
(X_{\tau
_{k}},\tau_{k} )= (x,l ) \bigr).
\end{eqnarray*}
Separating the sum as in the previous cases and using (\ref
{eqFourieranalysis3})
in the appropriate inequality, this completes the proof.
\end{pf}

\begin{pf*}{Proof of Lemma~\ref{lemgeneralannealedestimations}}
(1) We have
\begin{eqnarray*}
&& \mathbb{P}^{z} (X_{n}=x )
\\
&&\qquad  =\sum_{l\leq n}\mathbb{P}^{z} \bigl(
\hat{Z} (l ) \bigr)\sum_{w\in\mathbb{Z}^{d}}\mathbb{P}^{z}
\bigl(X_{l}=w\mid\hat{Z} (l ) \bigr)\mathbb{P}^{z}
\bigl(X_{n}=x\mid\hat{Z} (l ), X_{l}=w \bigr)
\\
&&\qquad  \overset{ (1 )} {\leq}\sum_{l\leq n}
\mathbb{P}^{z} \bigl(\hat{Z} (l ) \bigr)\sum
_{w\in\mathbb{Z}^{d}}Cl^{-\sfrac{d}{2}}\mathbb{P}^{z}
\bigl(X_{n}=x\mid\hat{Z} (l ), X_{l}=w \bigr)
\\
&&\qquad  \overset{ (2 )} {=}\sum_{l\leq n}
\mathbb{P}^{z} \bigl(\hat{Z} (l ) \bigr)Cl^{-\sfrac{d}{2}}
\\
&&\qquad  \overset{ (3 )} {\leq}\sum_{l\leq n}e^{- (\log
(n-l ) )^{2}}Cl^{-\sfrac{d}{2}}
\\
&&\qquad  =\sum_{l\leq\sfrac{n}{2}}e^{- (\log(n-l )
)^{2}}Cl^{-\sfrac{d}{2}}+
\sum_{\sfrac{n}{2}\leq l\leq n}e^{-
(\log(n-l ) )^{2}}Cl^{-\sfrac{d}{2}}
\\
&&\qquad  \leq Ce^{-c (\log n )^{2}}+Cn^{-\sfrac{d}{2}}\sum_{\sfrac
{n}{2}\leq l\le n}e^{- (\log(n-l ) )^{2}}
\leq Cn^{-\sfrac{d}{2}},
\end{eqnarray*}
where for $ (1 )$ we used Lemma~\ref
{lemMiddlestepinLemmageneralestimations},
(\ref{eqannealedderivativeestimation1}), for $ (2 )$
we used the shift invariance of the annealed walk, and for the first
sum in $ (3 )$ we used Corollary~\ref{corlengthofregenerations}.

(2) For $y\in\mathbb{Z}^{d}$ such that $\llVert y-x\rrVert _{1}=1$,
we can find
$1\leq j\leq d$
such that $y=x+e_{j}$ and, therefore,
\begin{eqnarray*}
&& \mathbb{P}^{z} (X_{n+1}=y )
\\
&&\qquad=  \sum_{l\leq n+1}\mathbb{P}^{z} \bigl(
\hat{Z}_{n+1} (l ) \bigr) \mathop{\sum_{
w\in\mathbb{Z}^{d}}}_{
w-z\leftrightarrow l}\mathbb{P}^{z} \bigl(X_{l}=w\mid\hat{Z}_{n+1}
(l ) \bigr)
\\
&&\quad\qquad{}\times
\BP^{z} \bigl(X_{n+1}=y\mid\hat
{Z}_{n+1} (l ), X_{l}=w \bigr)
\\
&&\qquad=  \sum_{l\leq n}\mathbb{P}^{z} \bigl(
\hat{Z}_{n+1} (l+1 ) \bigr) \mathop{\sum_{
w\in\mathbb{Z}^{d}}}_{
w-z\leftrightarrow l}\mathbb{P}^{z} \bigl(X_{l+1}=w+e_{j}\mid
\hat{Z}_{n+1} (l ) \bigr)
\\
&&\quad\qquad{}\times
\BP^{z} \bigl(X_{n+1}=y
\mid\hat{Z}_{n+1} (l ), X_{l+1}=w+e_{j} \bigr)
\\
&&\qquad =  \sum_{l\leq n}\mathbb{P}^{z} \bigl(
\hat{Z}_{n+1} (l+1 ) \bigr) \mathop{\sum_{
w\in\mathbb{Z}^{d}}}_{
w-z\leftrightarrow l}\mathbb{P}^{z} \bigl(X_{l+1}=w+e_{j}\mid
\hat{Z}_{n+1} (l+1 ) \bigr)
\\
&&\quad\qquad{}\times
\BP^{z} \bigl(X_{n}=x
\mid\hat{Z}_{n} (l ), X_{l}=w \bigr).
\end{eqnarray*}
Subtracting the formula for $\mathbb{P}^{z} (X_{n+1}=y )$
from the one for $\mathbb{P}^{z} (X_{n}=x )$, we thus get
\begin{eqnarray*}
&& \bigl\llvert\mathbb{P}^{z} (X_{n}=x )-
\mathbb{P}^{z} (X_{n+1}=y )\bigr\rrvert
\\
&&\qquad\leq \sum_{l\leq n-1}\bigl\llvert\mathbb{P}^{z}
\bigl(\hat{Z}_{n+1} (l+1 ) \bigr)-\mathbb{P}^{z} \bigl(\hat
{Z}_{n} (l ) \bigr)\bigr\rrvert
\\
&&\quad\qquad{}\times \sum_{w\in\mathbb
{Z}^{d}}
\mathbb{P}^{z} \bigl(X_{l+1}=w+e_{j}\mid
\hat{Z}_{n+1} (l+1 ) \bigr)\BP^{z} \bigl(X_{n}=x
\mid\hat{Z}_{n} (l ), X_{l}=w \bigr)
\\
&&\quad\qquad{}+  \sum_{l\leq n}\mathbb{P}^{z} \bigl(
\hat{Z}_{n} (l ) \bigr)\sum_{w\in\mathbb{Z}^{d}}\bigl
\llvert\mathbb{P}^{z} \bigl(X_{l}=w\mid
\hat{Z}_{n} (l ) \bigr)-\mathbb{P}^{z}
\bigl(X_{l}=w+e_{j}\mid\hat{Z}_{n} (l ) \bigr)
\bigr\rrvert
\\
&&\quad\qquad{}\times  \BP^{z} \bigl(X_{n}=x\mid
\hat{Z}_{n} (l ), X_{l}=w \bigr)
\\
&&\qquad\leq \sum_{l\leq n-1}\bigl\llvert\mathbb{P}^{z}
\bigl(\hat{Z}_{n+1} (l+1 ) \bigr)-\mathbb{P}^{z} \bigl(\hat
{Z}_{n} (l ) \bigr)\bigr\rrvert\BP^{z}
\bigl(X_{n+1}=x+e_{j}\mid\hat{Z}_{n+1} (l+1 )
\bigr)
\\
&&\qquad\quad{}+  \sum_{l\leq n}\mathbb{P}^{z} \bigl(
\hat{Z}_{n+1} (l ) \bigr)\max_{w\in\mathbb{Z}^{d}, 2\leq j\leq d}\bigl
\llvert
\mathbb{P}^{z} \bigl(X_{l}=w\mid\hat{Z}_{n} (l
) \bigr)
\\
&&\quad\qquad{}-\mathbb{P}^{z} \bigl(X_{l+1}=w+e_{j}
\mid\hat{Z}_{n+1} (l+1 ) \bigr)\bigr\rrvert,
\end{eqnarray*}
where as before we used the shift invariance. Using (\ref
{eqAnnealedestmiddle2-1}),
the second term is bounded by $\sum_{l\leq n}\mathbb{P}^{z}
(\hat{Z}_{n} (l ) )Cl^{-\vfrac{d+1}{2}}$,
which by the same argument as before is bounded by $Cn^{-\vfrac{d+1}{2}}$.
Regarding the first term, using the first part of the lemma, see (\ref
{eqgeneralannealedestimation1})
and Theorem~\ref{teoRegenerationtimestimeestimate} gives
\begin{eqnarray*}
&& \sum_{l\leq n-1}\bigl\llvert\mathbb{P}^{z}
\bigl(\hat{Z}_{n+1} (l+1 ) \bigr)-\mathbb{P}_{n+1}^{z}
\bigl(\hat{Z} (l ) \bigr)\bigr\rrvert\BP^{z} \bigl(X_{n+1}=x+e_{j}
\mid\hat{Z} (l+1 ) \bigr)
\\[-2pt]
&&\qquad\leq \sum_{l\leq n-1}\frac{\llvert \mathbb{P}^{z} (\hat
{Z}_{n+1} (l+1 ) )-\mathbb{P}^{z} (\hat
{Z}_{n} (l ) )\rrvert }{\mathbb{P}^{z} (\hat
{Z}_{n+1} (l+1 ) )}
\BP^{z} (X_{n+1}=x+e_{j} )
\\[-2pt]
&&\qquad\leq \sum_{l\leq n-1}\frac{\llvert \mathbb{P}^{z} (\hat
{Z}_{n+1} (l+1 ) )-\mathbb{P}^{z} (\hat
{Z}_{n} (l ) )\rrvert }{\mathbb{P}^{z} (\hat
{Z}_{n+1} (l+1 ) )}Cn^{-\sfrac{d}{2}}
\\[-2pt]
&&\qquad\leq \sum_{l\leq n-n^{\sfrac{1}{4}}}e^{- (\log(n-l
) )^{2}}Cn^{-\sfrac{d}{2}}+n^{-\sfrac{d}{2}}
\\[-2pt]
&&\qquad\quad{}\times
\sum_{n-n^{\sfrac
{1}{4}}\leq l\leq n-1}\bigl\llvert\mathbb{P}^{z}
\bigl(\hat{Z}_{n+1} (l+1 ) \bigr)-\mathbb{P}^{z} \bigl(
\hat{Z}_{n} (l ) \bigr)\bigr\rrvert
\\[-2pt]
&&\qquad \overset{ (1 )} {=}  o \bigl(n^{-\vfrac{d+1}{2}} \bigr)+n^{-\sfrac{d}{2}}\sum
_{n-n^{\sfrac{1}{4}}\leq l\leq n-1}\frac
{C}{l}\leq Cn^{-\sfrac{d}{2}}\cdot
n^{\sfrac{1}{4}}\cdot n^{-1}
\\[-2pt]
&&\qquad =o \bigl(n^{-\vfrac{d+1}{2}} \bigr),
\end{eqnarray*}
where for $ (1 )$ we used the fact that for $l\geq
n-n^{\sfrac{1}{4}}\geq\frac{n}{2}$
\begin{eqnarray*}
&& \bigl\llvert\mathbb{P}^{z} \bigl(\hat{Z}_{n+1} (l+1 )
\bigr)-\mathbb{P}^{z} \bigl(\hat{Z}_{n} (l ) \bigr)\bigr
\rrvert
\\[-2pt]
&&\qquad=  \Biggl\llvert\mathbb{P}^{z} \Biggl(Z (l+1 )\cap\bigcap
_{j=l+2}^{n+1}Z (j )^{c} \Biggr)-
\mathbb{P}^{z} \Biggl(Z (l )\cap\bigcap_{j=l+1}^{n}Z
(j )^{c} \Biggr)\Biggr\rrvert
\\[-2pt]
&&\qquad=  \Biggl\llvert\mathbb{P}^{z} \bigl(Z (l+1 ) \bigr)\mathbb
{P}^{z} \Biggl(\bigcap_{j=l+2}^{n+1}Z
(j )^{c}\Big| Z (l+1 ) \Biggr)
\\[-2pt]
&&\quad\qquad{} -
\mathbb{P}^{z} \bigl(Z (l ) \bigr)\mathbb{P}^{z} \Biggl(
\bigcap_{j=l+1}^{n}Z (j )^{c}
\Big| Z (l ) \Biggr)\Biggr\rrvert
\\[-2pt]
&&\qquad=  \Biggl\llvert\mathbb{P}^{z} \bigl(Z (l+1 ) \bigr)\mathbb
{P}^{z} \Biggl(\bigcap_{j=l+1}^{n}Z
(j )^{c}\Big| Z (l ) \Biggr)
\\[-2pt]
&&\quad\qquad{} -
\mathbb{P}^{z} \bigl(Z (l ) \bigr)\mathbb{P}^{z} \Biggl(
\bigcap_{j=l+1}^{n}Z (j )^{c}
\Big| Z (l ) \Biggr)\Biggr\rrvert
\\[-2pt]
&&\qquad=  \mathbb{P}^{z} \Biggl(\bigcap_{j=l+1}^{n}Z
(j )^{c}\Big|  Z (l ) \Biggr)\cdot\bigl\llvert\mathbb{P}^{z}
\bigl(Z (l+1 ) \bigr)-\mathbb{P}^{z} \bigl(Z (l ) \bigr)\bigr\rrvert
\\[-2pt]
&&\qquad\leq \bigl\llvert\mathbb{P}^{z} \bigl(Z (l+1 ) \bigr)-\mathbb
{P}^{z} \bigl(Z (l ) \bigr)\bigr\rrvert\leq\frac{C}{l}\leq
Cn^{-1}.
\end{eqnarray*}

(3) This follows exactly the same lines as the argument for the previous
inequality.

(4) A similar calculation gives
\begin{eqnarray*}
&& \sum_{\Delta\in\Pi_{n}^{ (\varepsilon)}}\mathop{\sum
_{
x\in\Delta}}_{
x\leftrightarrow n} \Bigl[\max_{y\in\Delta}
\mathbb{P}^{0} (X_{n}=y )-\mathbb{P}^{0}
(X_{n}=x ) \Bigr]
\\
&&\qquad=  \sum_{\Delta\in\Pi_{n}^{ (\varepsilon)}}\mathop{\max_{
x\in\Delta}}_{
x\leftrightarrow n}
\sum_{l\leq n}\mathbb{P}^{z} \bigl(\hat{Z} (l
) \bigr) \mathop{\sum_{
w\in\mathbb{Z}^{d}}}_{
w\leftrightarrow l}
\mathbb{P}^{z} \bigl(X_{l}=w\mid\hat{Z} (l ) \bigr)
\\
&&\quad\qquad{}\times \bigl[
\mathbb{P}^{z} \bigl(X_{n}=x_{\Delta}\mid\hat{Z} (l
), X_{l}=w \bigr)-\mathbb{P}^{z} \bigl(X_{n}=x
\mid\hat{Z} (l ), X_{l}=w \bigr) \bigr]
\\
&&\qquad =  \sum_{l\leq n}\mathbb{P}^{z} \bigl(
\hat{Z} (l ) \bigr) \mathop{\sum_{
w\in\mathbb{Z}^{d}}}_{
w\leftrightarrow l}
\sum_{\Delta\in\Pi_{n}^{ (\varepsilon)}}\mathop{\sum_{
x\in\Delta}}_{
x\leftrightarrow n}
\bigl[\mathbb{P}^{z} \bigl(X_{l}=w\mid\hat{Z} (l )
\bigr)
\\
&&\quad\qquad{}-\mathbb{P}^{z} \bigl(X_{l}=w+x-x_{0}\mid
\hat{Z} (l ) \bigr) \bigr]\mathbb{P}^{z} \bigl(X_{n}=x_{\Delta}
\mid\hat{Z} (l ), X_{l}=w \bigr).
\end{eqnarray*}
Using (\ref{eqeqAnnealedestmiddle4}) and the shift invariance
of the annealed measure this is bounded by
\[
\sum_{l\leq n}\mathbb{P}^{z} \bigl(\hat{Z}
(l ) \bigr)Cl^{-\sfrac{1}{2}+3d\varepsilon}\leq\sum_{l\leq n}e^{- (\log
(n-l ) )^{2}}
\cdot Cl^{-\sfrac{1}{2}+3d\varepsilon
}\leq Cn^{-\sfrac{1}{2}+3d\varepsilon}.
\]\upqed
\end{pf*}

\subsubsection{Proof of Lemma \protect\ref{lemgoodestimationforthelocation}}\label{sec8.2.4}
%
Recalling Corollary~\ref{corlengthofregenerations}, we have
%
\begin{eqnarray}\label{eqgoodestimationforthelocation1}
&& \mathbb{P}^{z} \bigl(\bigl\llVert X_{n}-
\mathbb{E}^{z} [X_{n} ]\bigr\rrVert_{\infty}>
\sqrt{n}R_{5} (n ) \bigr)\nonumber
\\
&&\qquad   \leq \mathbb{P}^{z} \bigl(\bigl
\llVert X_{n}-\mathbb{E}^{z} [X_{n} ]\bigr
\rrVert_{\infty}>\sqrt{n}R_{5} (n )\mid B_{n}
\bigr)+P \bigl(B_{n}^{c} \bigr)
\\
&&\qquad \leq \mathbb{P}^{z} \bigl(\exists k\leq n: \bigl\llVert
X_{\tau
_{k}}-\mathbb{E}^{z} [X_{\tau_{k}} ]\bigr\rrVert
_{\infty
}>\tfrac{1}{3}\sqrt{n}R_{5} (n )\mid
B_{n} \bigr)+n^{-\xi(1 )}.\nonumber
\end{eqnarray}
Note that conditioned on $B_{n}$ the regenerations are still independent
and all of them are bounded by $R (n )$. If we could show
that $\llVert \mathbb{E}^{z} [X_{\tau_{k}} ]-\mathbb
{E}^{z} [X_{\tau_{k}}\mid\break  B_{n} ]\rrVert _{\infty
}=n^{-\xi(1 )}$
then
\[
\eqref{eqgoodestimationforthelocation1}\leq\mathbb{P}^{z} \bigl(
\exists k\leq n:, \bigl\llVert X_{\tau_{k}}-\mathbb{E}^{z}
[X_{\tau_{k}}\mid B_{n} ]\bigr\rrVert_{\infty
}>
\tfrac{1}{4}\sqrt{n}R_{5} (n )\mid B_{n}
\bigr)+n^{-\xi(1 )}
\]
which by Azuma's inequality is no more than
\begin{eqnarray*}
&& \sum_{k=1}^{n}\mathbb{P}^{z}
\biggl(\bigl\llVert X_{\tau
_{k}}-\mathbb{E}^{z}
[X_{\tau_{k}}\mid B_{n} ]\bigr\rrVert_{\infty}>
\frac{1}{4}\sqrt{n}R_{5} (n ) \mid B_{n}
\biggr)+n^{-\xi(1 )}
\\
&&\qquad\leq d\sum_{k=1}^{n}\exp\biggl(-
\frac{nR_{5}^{2} (n
)}{16kR^{2} (n )} \biggr)\leq de^{-R_{5} (n
)}=n^{-\xi(1 )}.
\end{eqnarray*}
Thus, it is left to show that $\llVert \mathbb{E}^{z}
[X_{\tau_{k}} ]-\mathbb{E}^{z} [X_{\tau_{k}}\mid
B_{n} ]\rrVert _{\infty}=n^{-\xi(1 )}$.
Since $k\leq n$, by the triangle inequality it is enough to show that
$\llVert \mathbb{E}^{z} [X_{\tau_{k}}-X_{\tau_{k-1}}
]-\mathbb{E}^{z} [X_{\tau_{k}}-X_{\tau_{k-1}}\mid B_{n}
]\rrVert _{\infty}=n^{-\xi(1 )}$.
However, for every $k\leq n$
\begin{eqnarray*}
&& \bigl\llVert\mathbb{E}^{z} [X_{\tau_{k}}-X_{\tau_{k-1}} ]-
\mathbb{E}^{z} [X_{\tau_{k}}-X_{\tau_{k-1}}\mid
B_{n} ]\bigr\rrVert_{\infty}
\\
&&\qquad\leq \bigl\llVert\mathbb{E}^{z} [X_{\tau_{k}}-X_{\tau
_{k-1}}
]-\mathbb{E}^{z} \bigl[ (X_{\tau_{k}}-X_{\tau
_{k-1}} )
\ind_{B_{n}} \bigr]\bigr\rrVert_{\infty}
\\
&&\quad\qquad{}+  \bigl\llVert\mathbb{E}^{z} \bigl[ (X_{\tau_{k}}-X_{\tau
_{k-1}}
)\ind_{B_{n}} \bigr]-\mathbb{E}^{z} [X_{\tau
_{k}}-X_{\tau_{k-1}}
\mid B_{n} ]\bigr\rrVert_{\infty}
\\
&&\qquad=  \bigl\llVert\mathbb{E}^{z} \bigl[ (X_{\tau_{k}}-X_{\tau
_{k-1}}
)\ind_{B_{n}^{c}} \bigr]\bigr\rrVert_{\infty}+\mathbb
{P}^{z} \bigl(B_{n}^{c} \bigr)\bigl\llVert
\mathbb{E}^{z} [X_{\tau_{k}}-X_{\tau_{k-1}}\mid
B_{n} ]\bigr\rrVert_{\infty
}
\\
&&\qquad\leq \bigl\llVert\mathbb{E}^{z} \bigl[ (X_{\tau_{k}}-X_{\tau
_{k-1}}
)\ind_{B_{n}^{c}} \bigr]\bigr\rrVert_{\infty}+R (n )
\mathbb{P}^{z} \bigl(B_{n}^{c} \bigr)
\\
&&\qquad\leq \bigl\llVert\mathbb{E}^{z} \bigl[ (X_{\tau_{k}}-X_{\tau
_{k-1}}
)\ind_{\exists j\neq k, \tau_{j}-\tau_{j-1}>R
(n )} \bigr]\bigr\rrVert_{\infty}
\\
&&\quad\qquad{}+\bigl\llVert\mathbb
{E}^{z} \bigl[ (X_{\tau_{k}}-X_{\tau_{k-1}} )
\ind_{\tau
_{k}-\tau_{k-1}>R (n )} \bigr]\bigr\rrVert_{\infty
}+R (n )
\mathbb{P}^{z} \bigl(B_{n}^{c} \bigr)
\\
&&\qquad\leq \bigl\llVert\mathbb{E}^{z} \bigl[ (X_{\tau_{k}}-X_{\tau
_{k-1}}
) \bigr]\bigr\rrVert_{\infty}\mathbb{P}^{z}
\bigl(B_{n}^{c} \bigr)+\sum_{l=R (n )}^{\infty}le^{-cl^{\gamma
}}+R
(n )\mathbb{P}^{z} \bigl(B_{n}^{c}
\bigr)
\\
&&\qquad =n^{-\xi
(1 )},
\end{eqnarray*}
where for the last inequality we used the assumption $ (\mathscr
{P} )$
which implies $T_{\gamma}$ for any $0<\gamma<1$, and for the last
equality we used Corollary~\ref{corlengthofregenerations}.

The quenched estimation follows from the first inequality together
with Claim~\ref{Clmsimpleclaimannealedtoquenched}, while the
second annealed estimation follows by the exact same proof with
$R_{5} (n )$
replaced with a large constant $C$.

\subsection{More annealed estimations}\label{sec8.3}
\label{AppendixsubMore-annealed-estimations}

\subsubsection{Proof of (\protect\ref{eqFirstdifferenceestimationannealedestimation1})--(\protect\ref
{eqFirstdifferenceestimationannealedestimation4})}\label{sec8.3.1}

We start with the proof of (\ref{eqFirstdifferenceestimationannealedestimation1}).
\begin{eqnarray*}
&& \mathbb{P}^{z} \bigl(X_{T_{M+V}}\in\Delta^{ (1
)},T_{M+V}
\in I^{ (1 )} \bigr)
\\
&&\qquad=  \mathbb{P}^{z} \bigl(X_{T_{M+V}}\in\Delta^{ (1
)},T_{M+V}
\in I^{ (1 )},X_{T_{M}}\in\Delta,T_{M}\in I \bigr)
\\
&&\quad\qquad{}+  \mathbb{P}^{z} \bigl(X_{T_{M+V}}\in\Delta^{ (1
)},T_{M+V}
\in I^{ (1 )}, (X_{T_{M}}\in\Delta,T_{M}\in I
)^{c} \bigr)
\\
&&\qquad\leq \mathbb{P}^{z} (X_{T_{M}}\in\Delta,T_{M}\in
I )+\mathbb{P}^{z} \bigl(X_{T_{M+V}}\in\Delta^{ (1
)},X_{T_{M}}
\notin\Delta\bigr)
\\
&&\quad\qquad{}+\mathbb{P}^{z} \bigl(T_{M+V}\in
I^{ (1 )},T_{M}\notin I \bigr).
\end{eqnarray*}
Thus, it is enough to show that $\mathbb{P}^{z} (X_{T_{M+V}}\in
\Delta^{ (1 )},X_{T_{M}}\notin\Delta)=N^{-\xi
(1 )}$
and $\mathbb{P}^{z} (T_{M+V}\in I^{ (1 )},T_{M}\notin
I )=N^{-\xi(1 )}$.
Since by Lemma~\ref{lemDistanceoftimefromexpectation} $\mathbb
{P}^{z} (X_{T_{M}}\notin\mathcal{P} (0,N )
)=\mathbb{P}^{z} (X_{T_{M}}\notin\mathcal{P} (0,N
),T_{\partial\mathcal{P} (0,N )}\neq T_{\partial
^{+}\mathcal{P} (0,N )} )+N^{-\xi(1
)}=N^{-\xi(1 )}$,
we have
\begin{eqnarray*}
&& \mathbb{P}^{z} \bigl(X_{T_{M+V}}\in\Delta^{ (1
)},X_{T_{M}}
\notin\Delta\bigr)
\\
&&\qquad =\sum_{y\in\mathcal{P}
(0,N )\cap H_{M}\setminus\Delta}\mathbb{P}^{z}
\bigl(X_{T_{M}}=y, X_{T_{M+V}}\in\Delta^{ (1 )}
\bigr)+N^{-\xi(1 )}.
\end{eqnarray*}
However, using Lemma~\ref{lemDistanceoftimefromexpectation}
once more, for every $y\in\mathcal{P} (0,N )\cap
H_{M}\setminus\Delta$
and every $w\in\tilde{\mathcal{P}} (y,\sqrt{V} )$ we have\vspace*{1pt}
$\mathbb{P}^{w} (T_{\partial\mathcal{P} (y,\sqrt{V}
)}=T_{\partial^{+}\mathcal{P} (y,\sqrt{V} )}
)=1- (\sqrt{V} )^{-\xi(1 )}=1-N^{-\xi
(1 )}$.
Since in addition $\operatorname{dist} (\partial^{+}\mathcal{P}
(y,\sqrt{V} ),\Delta)>\frac{1}{2}N^{\theta}-\frac
{1}{2}\cdot\frac{9}{10}N^{\theta}-\break \sqrt{V}R_{5} (V )\geq
\frac{1}{20}N^{\theta}-cN^{\theta'}$
it follows that $\mathbb{P}^{w} (X_{T_{M+V}}\in\Delta^{
(1 )} )=N^{-\xi(1 )}$.
To complete the argument, we note that we note that (by Corollary~\ref
{corlengthofregenerations})
\begin{eqnarray*}
\mathbb{P}^{z} \bigl(X_{T_{M}}=y, X_{T_{M+V}}\in
\Delta^{
(1 )} \bigr) & =&\mathbb{P}^{z} \bigl(X_{T_{M}}=y, X_{T_{M+V}}\in\Delta
^{ (1 )},B_{N}
\bigr)+N^{-\xi
(1 )}
\end{eqnarray*}
and under $B_{N}$ there is a regeneration time at distance at most
$R (N )$ from $y$. This\vspace*{1pt} gives a new point $w\in\mathbb{Z}^{d}$
[such that $\llVert w-y\rrVert _{\infty}\leq R (N )$
and in particular $w\in\tilde{\mathcal{P}} (0,N )$] from
which the probability to hit $\Delta^{ (1 )}$ when hitting
the hyperplane $H_{M+V}$ (conditioned to start in a regeneration
time). Since the last conditioning has a positive probability this
is bounded by $C\mathbb{P}^{w} (X_{T_{M+V}}\in\Delta^{
(1 )} )=N^{-\xi(1 )}$.
Thus,
\begin{eqnarray*}
&& \mathbb{P}^{z} \bigl(X_{T_{M+V}}\in\Delta^{ (1
)},X_{T_{M}}
\notin\Delta\bigr)
\\
&&\qquad  =\sum_{y\in\mathcal{P}
(0,N )\cap H_{M}\setminus\Delta}
\mathbb{P}^{z} \bigl(X_{T_{M}}=y, X_{T_{M+V}}\in
\Delta^{ (1 )} \bigr)+N^{-\xi(1 )}
\\
&&\qquad  \leq \bigl\llvert\mathcal{P} (0,N )\bigr\rrvert\cdot N^{-\xi
(1 )}+N^{-\xi(1 )}=N^{-\xi(1 )}.
\end{eqnarray*}
A similar argument shows that $\mathbb{P}^{z} (T_{M+V}\in
I^{ (1 )},T_{M}\notin I )=N^{-\xi(1 )}$.
Indeed, by Lemma~\ref{lemDistanceoftimefromexpectation} up to
an event of probability $M^{-\xi(1 )}=N^{-\xi
(1 )}$
the first hitting time to the hyperplane $H_{M}$ is the same as the
exit time of the box $\mathcal{P} (0,M )$. By the same lemma,
we also know that up to an event of probability $N^{-\xi(1 )}$
this time is at distance at most $NR_{2} (N )$ from the
expectation of $\mathbb{E}^{z} [T_{M} ]$. Therefore,
\begin{eqnarray*}
&& \mathbb{P}^{z} \bigl(T_{M+V}\in I^{ (1 )},T_{M}
\notin I \bigr)
\\
&&\qquad =\mathop{\sum_{
t: \llvert t-\mathbb{E}^{z} [T_{M} ]\rrvert <NR_{2}
(N )}}_{
t\notin I}
\mathbb{P}^{z}
\bigl(T_{m}=t, T_{M+V}\in I^{ (1
)}
\bigr)+N^{-\xi(1 )}.
\end{eqnarray*}
In the case $t<c (I )-N^{\theta}$ if $T_{M+V}\in I^{
(1 )}$,
then the random walk crossed the distance from $H_{M}$ to $H_{M+V}$
in more then $V\frac{1}{ \langle\bbbv,e_{1} \rangle
}-\frac{1}{2}\cdot\frac{9}{10}N^{\theta}+N^{\theta}=V\frac
{1}{ \langle\bbbv,e_{1} \rangle}+\frac
{1}{20}N^{\theta}$
which happens with probability $N^{-\xi(1 )}$ by Lemma
\ref{lemDistanceoftimefromexpectation}. Similarly, if $t>C
(I )+N^{\theta}$
and $T_{M+V}\in I^{ (1 )}$ then the random walk crossed
the distance from $H_{M}$ to $H_{M+V}$ in less than $V\frac{1}{
\langle\bbbv,e_{1} \rangle}+\frac{1}{2}\cdot\frac
{9}{10}N^{\theta}-\frac{1}{2}N^{\theta}=V\frac{1}{ \langle
\bbbv,e_{1} \rangle}-\frac{1}{20}N^{\theta}$
which also has probability $N^{-\xi(1 )}$ by Lemma~\ref
{lemDistanceoftimefromexpectation}.

Thus,
\begin{eqnarray*}
&& \mathbb{P}^{z} \bigl(T_{M+V}\in I^{ (1 )},T_{M}
\notin I \bigr)
\\
&&\qquad  =
\mathop{\sum_{
t: \llvert t-\mathbb{E}^{z} [T_{M} ]\rrvert <NR_{2}
(N )}}_{
t\notin I}
\mathbb{P}^{z}
\bigl(T_{m}=t, T_{M+V}\in I^{ (1
)}
\bigr)+N^{-\xi(1 )}
\\
&&\qquad  \leq  CNR_{2} (N )\cdot N^{-\xi(1 )}+N^{-\xi
(1 )}=N^{-\xi(1 )}.
\end{eqnarray*}

Turning to (\ref{eqFirstdifferenceestimationannealedestimation3}),
we have
\begin{eqnarray*}
&& E \bigl[P_{\omega}^{z} \bigl(X_{T_{M+V}}\in
\Delta^{ (1
)},T_{M+V}\in I^{ (1 )} \bigr)\mid\mathcal{G}
\bigr]
\\
&&\qquad= E \bigl[P_{\omega}^{z} \bigl(X_{T_{M+V}}\in
\Delta^{ (1
)},T_{M+V}\in I^{ (1 )},X_{T_{M}}\in
\Delta,T_{M}\in I \bigr)\mid\mathcal{G} \bigr]
\\
&&\quad\qquad{}+ E \bigl[P_{\omega}^{z} \bigl(X_{T_{M+V}}\in
\Delta^{ (1
)},T_{M+V}\in I^{ (1 )}, (X_{T_{M}}\in
\Delta,T_{M}\in I )^{c} \bigr)\mid\mathcal{G} \bigr]
\\
&&\qquad\leq E \bigl[P_{\omega}^{z} (X_{T_{M}}\in
\Delta,T_{M}\in I )\mid\mathcal{G} \bigr]+E \bigl[P_{\omega}^{z}
\bigl(X_{T_{M+V}}\in\Delta^{ (1 )},X_{T_{M}}\notin\Delta
\bigr)\mid\mathcal{G} \bigr]
\\
&&\quad\qquad{}+ E \bigl[P_{\omega}^{z} \bigl(T_{M+V}\in
I^{ (1
)},T_{M}\notin I \bigr)\mid\mathcal{G} \bigr]
\\
&&\qquad= P_{\omega}^{z} (X_{T_{M}}\in\Delta,T_{M}
\in I )+E \bigl[P_{\omega}^{z} \bigl(X_{T_{M+V}}\in
\Delta^{ (1
)},X_{T_{M}}\notin\Delta\bigr)\mid\mathcal{G} \bigr]
\\
&&\quad\qquad{}+ E \bigl[P_{\omega}^{z} \bigl(T_{M+V}\in
I^{ (1
)},T_{M}\notin I \bigr)\mid\mathcal{G} \bigr].
\end{eqnarray*}
Separating to the case when $B_{N}$ holds and when $B_{N}^{c}$ (which
has probability $N^{-\xi(1 )}$) we can control the terms
$E [P_{\omega}^{z} (X_{T_{M+V}}\in\Delta^{ (1
)},X_{T_{M}}\notin\Delta)\mid\mathcal{G} ]$
and $E [P_{\omega}^{z} (T_{M+V}\in I^{ (1
)},T_{M}\notin I )\mid\mathcal{G} ]$
by the annealed probability of the events $ \{ X_{T_{M+V}}\in
\Delta^{ (1 )},X_{T_{M}}=w \} $
and $ \{ T_{M+V}\in I^{ (1 )},T_{m}=t \} $ with
$w$ and $t$ the place and time of the first regeneration time after
hitting the hyerplane $H_{M}$ (outside of $\Delta$). Since by the
first argument those events have probability $N^{-\xi(1 )}$
the proof is complete.

The proof of (\ref{eqFirstdifferenceestimationannealedestimation2})
and (\ref{eqFirstdifferenceestimationannealedestimation4})
is very similar and, therefore, is left to the reader.

\subsubsection{Proof of (\protect\ref
{eqTimed-1boxdifferenceannealedestimation1})--(\protect\ref
{eqTimed-1boxdifferenceannealedestimation4})}\label{sec8.3.2}

The proof of (\ref{eqTimed-1boxdifferenceannealedestimation1})--(\ref
{eqTimed-1boxdifferenceannealedestimation4})
follows the same lines as the proof of (\ref
{eqFirstdifferenceestimationannealedestimation1})--(\ref
{eqFirstdifferenceestimationannealedestimation4}).
The\vspace*{1pt} only difference is that in (\ref
{eqFirstdifferenceestimationannealedestimation1})--(\ref
{eqFirstdifferenceestimationannealedestimation4})
we took boxes of side length $\frac{9}{10}N^{\theta}$ and $\frac
{11}{10}N^{\theta}$
leaving a difference of wide $\frac{1}{10}N^{\theta}$ from the original
box whose side length is $N^{\theta}$. This together with the fact
that the distance between the hyperplanes was $V=N^{2\theta'}$ for
some $\theta'<\theta$ made it impossible to hit one box without hitting
the other. Similarly in (\ref
{eqTimed-1boxdifferenceannealedestimation1})--(\ref
{eqTimed-1boxdifferenceannealedestimation4}),
we take boxes of side length $N^{\theta}\pm R_{3} (N )\sqrt{V}$.
As in the previous case, we have $R_{3} (N )\sqrt{V}\gg
\sqrt{V}$
and, therefore, the probability to hit one box without hitting the other
is still of magnitude $N^{-\xi(1 )}$.

\subsubsection{Proof of (\protect\ref{eqLorentztrans3})--(\protect\ref{eqLorentztrans4})}\label{sec8.3.3}

We start with the proof of (\ref{eqLorentztrans3}). Denoting by
$A_{t,s,w}$ the event that the first regeneration time after time
$t$ is at time $s$ and $X_{s}=w$ we have
\begin{eqnarray*}
&& \mathbb{P}^{z} \bigl(X_{T_{\partial\mathcal{P} (0,\sqrt
{L} )}}\in\Delta^{ (1 )},
T_{\partial\mathcal
{P} (0,\sqrt{L} )}\in I^{ (1 )}, X_{N}\notin\Delta\bigr)
\\
&&\qquad =  \mathop{\sum_{
y\in\Delta^{ (1 )}}}_{
t\in I^{ (1 )}}
\mathbb{P}^{z} (X_{T_{\partial\mathcal{P} (0,\sqrt
{L} )}}=y, T_{\partial\mathcal{P} (0,\sqrt{L}
)}=t,
X_{N}\notin\Delta, B_{N} )+N^{-\xi(1
)}
\\
&&\qquad=  \mathop{\sum_{
y\in\Delta^{ (1 )}}}_{
t\in I^{ (1 )}} \mathop{
\sum_{
w: \llVert w-y\rrVert _{\infty}\leq R (N )}}_{
s: \llvert t-s\rrvert \leq R (N )} \mathbb{P}^{z}
(X_{T_{\partial\mathcal{P} (0,\sqrt
{L} )}}=y,
\\
&&\quad\qquad{} T_{\partial\mathcal{P} (0,\sqrt{L}
)}=t, B_{N}, A_{t,s,w}, X_{N}\notin\Delta)+N^{-\xi
(1 )}
\\
&&\qquad\leq \mathop{\sum_{
y\in\Delta^{ (1 )}}}_{
t\in I^{ (1 )}} \mathop{
\sum_{
w: \llVert w-y\rrVert _{\infty}\leq R (N )}}_{
s: \llvert t-s\rrvert \leq R (N )} \mathbb{P}^{z}
(X_{T_{\partial\mathcal{P} (0,\sqrt
{L} )}}=y, T_{\partial\mathcal{P} (0,\sqrt{L}
)}=t, B_{N}, A_{t,s,w}
)
\\
&&\quad\qquad{}\times
 \mathbb{P}^{z} (X_{N}\notin\Delta\mid
X_{T_{\partial
\mathcal{P} (0,\sqrt{L} )}}=y, T_{\partial\mathcal
{P} (0,\sqrt{L} )}=t, B_{N}, A_{t,s,w}
)+N^{-\xi(1 )}
\\
&&\qquad \leq \mathop{\sum_{
y\in\Delta^{ (1 )}}}_{
t\in I^{ (1 )}} \mathop{
\sum_{
w: \llVert w-y\rrVert _{\infty}\leq R (N )}}_{
s: \llvert t-s\rrvert \leq R (N )} \mathbb{P}^{z}
(X_{T_{\partial\mathcal
{P} (0,\sqrt{L} )}}=y, T_{\partial\mathcal{P}
(0,\sqrt{L} )}=t, B_{N}, A_{t,s,w} )
\\
&&\quad\qquad{}\times
\mathbb{P}^{w}(X_{N-s}\notin\Delta\mid
\mbox{0 is a regeneration time}
)+N^{-\xi(1 )}
\\
&&\qquad \leq C\cdot\mathop{\sum_{
y\in\Delta^{ (1 )}}}_{
t\in I^{ (1 )}}
\mathop{\sum_{
w: \llVert w-y\rrVert _{\infty}\leq R (N )}}_{
s: \llvert t-s\rrvert \leq R (N )}
\mathbb{P}^{z} (X_{T_{\partial\mathcal{P} (0,\sqrt
{L} )}}=y, T_{\partial\mathcal{P} (0,\sqrt{L}
)}=t,
B_{N}, A_{t,s,w} )
\\
&&\quad\qquad{}\times \mathbb{P}^{w}
(X_{N-s}\notin\Delta)+N^{-\xi(1 )}
\\
&&\qquad\leq CR (N )^{d} \mathop{\sum_{
w: \operatorname{dist} (w,\Delta^{ (1 )} )\leq
R (N )}}_{
s: \operatorname{dist} (s,I^{ (1 )} )\leq R
(N )}
\mathbb{P}^{w} (X_{N-s}\notin\Delta)+N^{-\xi
(1 )}.
\end{eqnarray*}
Since the number of pairs $ (w,s )$ satisfying the above
inequalities is at most $ (N^{\theta}R (N ) )^{d}$,
it is enough to show that for every $w\in\mathbb{Z}^{d}$ such that
$\operatorname{dist} (w,\break \Delta^{ (1 )} )\leq R
(N )$
and every $s\in\mathbb{N}$ such that $\operatorname{dist} (s,I^{
(1 )} )\leq R (N )$
we have $\mathbb{P}^{w} (X_{N-s}\notin\Delta)=N^{-\xi
(1 )}$.
To this end, fix $w$ and $s$ as above, and note that
\begin{eqnarray*}
&& \mathbb{P}^{w} (X_{N-s}\notin\Delta)
\\
&&\qquad=  \mathbb{P}^{w} (X_{N-s}\notin\Delta,
T_{\partial
\mathcal{P} (w,\sqrt{N^{\theta}} )}=T_{\partial
^{+}\mathcal{P} (w,\sqrt{N^{\theta}} )},
\\
&&\quad\qquad{} T_{\partial
\mathcal{P} (w,\sqrt{\sklfrac{1}{2}N^{\theta}}
)}=T_{\partial^{+}\mathcal{P} (w,\sqrt{\sklfrac{1}{2}N^{\theta
}} )}
)+N^{-\xi(1 )}
\\
&&\qquad=  \mathbb{P}^{w} (X_{N-s}\notin\Delta,
T_{\partial
\mathcal{P} (w,\sqrt{\sklfrac{3}{2}N^{\theta}}
)}=T_{\partial^{+}\mathcal{P} (w,\sqrt{\sklfrac{3}{2}N^{\theta
}} )},
\\
&&\quad\qquad{}
T_{\partial\mathcal{P} (w,\sqrt{\sklfrac
{1}{2}N^{\theta}} )}=T_{\partial^{+}\mathcal{P} (w,\sqrt
{\sklfrac{1}{2}N^{\theta}} )},
\\
&&\quad\qquad{} T_{\partial\mathcal{P}
(w,\sqrt{\sklfrac{1}{2}N^{\theta}} )}\leq N-s\leq T_{\partial
\mathcal{P} (w,\sqrt{\sklfrac{3}{2}N^{\theta}} )} )
\\
&&\quad\qquad{}+  \mathbb{P}^{w} (X_{N-s}\notin\Delta,
T_{\partial
\mathcal{P} (w,\sqrt{\sklfrac{3}{2}N^{\theta}}
)}=T_{\partial^{+}\mathcal{P} (w,\sqrt{\sklfrac{3}{2}N^{\theta
}} )},
\\
&&\quad\qquad{}  N-s>T_{\partial\mathcal{P} (w,\sqrt{\sklfrac
{3}{2}N^{\theta}} )} )
\\
&&\quad\qquad{}+  \mathbb{P}^{w} (X_{N-s}\notin\Delta,
T_{\partial
\mathcal{P} (w,\sqrt{\sklfrac{1}{2}N^{\theta}}
)}=T_{\partial^{+}\mathcal{P} (w,\sqrt{\sklfrac{1}{2}N^{\theta
}} )},
\\
&&\quad\qquad{} N-s<T_{\partial\mathcal{P} (w,\sqrt{\sklfrac
{1}{2}N^{\theta}} )} )+N^{-\xi(1 )}.
\end{eqnarray*}
Note, however, that if $T_{\partial\mathcal{P} (w,\sqrt{\sklfrac
{3}{2}N^{\theta}} )}=T_{\partial^{+}\mathcal{P} (w,\sqrt
{\sklfrac{3}{2}N^{\theta}} )}$,
$T_{\partial\mathcal{P} (w,\sqrt{\sklfrac{1}{2}N^{\theta}}
)}=T_{\partial^{+}\mathcal{P} (w,\sqrt{\sklfrac{1}{2}N^{\theta
}} )}$
and $T_{\partial\mathcal{P} (w,\sqrt{\sklfrac{1}{2}N^{\theta
}} )}\leq N-s\leq T_{\partial\mathcal{P} (w,\sqrt{\sklfrac
{3}{2}N^{\theta}} )}$
then $X_{N-s}\in\Delta$ and, therefore,
\begin{eqnarray*}
&& \mathbb{P}^{w} (X_{N-s}\notin\Delta)
\\
&&\qquad   \leq \mathbb
{P}^{w} (T_{\partial\mathcal{P} (w,\sqrt{\sklfrac
{3}{2}N^{\theta}} )}=T_{\partial^{+}\mathcal{P} (w,\sqrt
{\sklfrac{3}{2}N^{\theta}} )}, N-s>T_{\partial\mathcal
{P} (w,\sqrt{\sklfrac{3}{2}N^{\theta}} )} )
\\
&&\quad\qquad{} +\mathbb{P}^{w} (T_{\partial\mathcal{P} (w,\sqrt{\sklfrac
{1}{2}N^{\theta}} )}=T_{\partial^{+}\mathcal{P} (w,\sqrt
{\sklfrac{1}{2}N^{\theta}} )}, N-s\leq
T_{\partial\mathcal
{P} (w,\sqrt{\sklfrac{1}{2}N^{\theta}} )} )
\\
&&\quad\qquad{} +N^{-\xi
(1 )}.
\end{eqnarray*}
Since $N-s\in[\frac{N^{\theta}}{ \langle\bbbv,e_{1} \rangle}-\frac{1}{2} (N^{\theta}-R_{5}
(N )N^{\sfrac{\theta}{2}} ),\frac{N^{\theta}}{
\langle\bbbv,e_{1} \rangle}+\frac{1}{2} (N^{\theta
}-R_{5} (N )N^{\sfrac{\theta}{2}} ) ]$ it
follows that
\begin{eqnarray*}
&& \mathbb{P}^{w} (T_{\partial\mathcal{P} (w,\sqrt{\sklfrac
{3}{2}N^{\theta}} )}=T_{\partial^{+}\mathcal{P} (w,\sqrt
{\sklfrac{3}{2}N^{\theta}} )},
N-s>T_{\partial\mathcal
{P} (w,\sqrt{\sklfrac{3}{2}N^{\theta}} )} )
\\
&&\qquad=  \mathbb{P}^{0} (T_{\partial\mathcal{P} (0,\sqrt{\sklfrac
{3}{2}N^{\theta}} )}=T_{\partial^{+}\mathcal{P} (0,\sqrt
{\sklfrac{3}{2}N^{\theta}} )},
N-s>T_{\partial\mathcal
{P} (0,\sqrt{\sklfrac{3}{2}N^{\theta}} )} )
\\
&&\qquad\leq \mathbb{P}^{0} \biggl(T_{\sklfrac{3}{2}N^{\theta}}<\frac
{N^{\theta}}{ \langle\bbbv,e_{1} \rangle}+
\frac
{1}{2} \bigl(N^{\theta}-R_{5} (N )N^{\sfrac{\theta
}{2}}
\bigr) \biggr)
\\
&&\qquad\leq \mathbb{P}^{0} \biggl(T_{\sklfrac{3}{2}N^{\theta}}<\frac{\sklfrac
{3}{2}N^{\theta}}{ \langle\bbbv,e_{1} \rangle
}-
\frac{1}{2}R_{5} (N )N^{\sfrac{\theta}{2}} \biggr)=N^{-\xi(1 )}
\end{eqnarray*}
and similarly
\begin{eqnarray*}
&& \mathbb{P}^{w} (X_{N-s}\notin\Delta,
T_{\partial\mathcal
{P} (w,\sqrt{\sklfrac{1}{2}N^{\theta}} )}=T_{\partial
^{+}\mathcal{P} (w,\sqrt{\sklfrac{1}{2}N^{\theta}} )}, N-s\leq T_{\partial
\mathcal{P} (w,\sqrt{\sklfrac{1}{2}N^{\theta
}} )} )
\\
&&\qquad=  \mathbb{P}^{0} (T_{\partial\mathcal{P} (0,\sqrt{\sklfrac
{1}{2}N^{\theta}} )}=T_{\partial^{+}\mathcal{P} (0,\sqrt
{\sklfrac{1}{2}N^{\theta}} )},
N-s<T_{\partial\mathcal
{P} (0,\sqrt{\sklfrac{1}{2}N^{\theta}} )} )
\\
&&\qquad\leq \mathbb{P}^{0} \biggl(T_{\sklfrac{1}{2}N^{\theta}}>\frac
{N^{\theta}}{ \langle\bbbv,e_{1} \rangle}-
\frac
{1}{2} \bigl(N^{\theta}-R_{5} (N )N^{\sfrac{\theta
}{2}}
\bigr) \biggr)
\\
&&\qquad\leq \mathbb{P}^{0} \biggl(T_{\sklfrac{1}{2}N^{\theta}}>\frac{\sklfrac
{1}{2}N^{\theta}}{ \langle\bbbv,e_{1} \rangle
}+
\frac{1}{2}R_{5} (N )N^{\sfrac{\theta}{2}} \biggr)=N^{-\xi(1 )}.
\end{eqnarray*}
\end{appendix}

\section*{Acknowledgments}
The authors would like to thank Marek Biskup, Jean-Dominique Deuschel,
Tal Orenshtein, Jon Peterson, Pierre-Fran\c{c}ois Rodriguez, Atilla Yilmaz
and Ofer Zeitouni for useful discussions and to the referee for useful
comments.


%

\printaddresses

\begin{thebibliography}{29}
\bibitem{Al99}
%
\begin{barticle}[mr]
\bauthor{\bsnm{Alili},~\bfnm{S.}\binits{S.}}
(\byear{1999}).
\btitle{Asymptotic behaviour for random walks in random environments}.
\bjournal{J. Appl. Probab.}
\bvolume{36}
\bpages{334--349}.
\bid{issn={0021-9002}, mr={1724844}}
\end{barticle}
%

\bptok{imsref}%
\endbibitem

\bibitem{berger2008slowdown}
%
\begin{barticle}[mr]
\bauthor{\bsnm{Berger},~\bfnm{Noam}\binits{N.}}
(\byear{2012}).
\btitle{Slowdown estimates for ballistic random walk in random environment}.
\bjournal{J.~Eur. Math. Soc. (JEMS)}
\bvolume{14}
\bpages{127--174}.
\bid{doi={10.4171/JEMS/298}, issn={1435-9855}, mr={2862036}}
\end{barticle}
%

\bptok{imsref}%
\endbibitem

\bibitem{BB07}
%
\begin{barticle}[mr]
\bauthor{\bsnm{Berger},~\bfnm{Noam}\binits{N.}} \AND
\bauthor{\bsnm{Biskup},~\bfnm{Marek}\binits{M.}}
(\byear{2007}).
\btitle{Quenched invariance principle for simple random walk on
percolation clusters}.
\bjournal{Probab. Theory Related Fields}
\bvolume{137}
\bpages{83--120}.
\bid{doi={10.1007/s00440-006-0498-z}, issn={0178-8051}, mr={2278453}}
\end{barticle}
%

\bptok{imsref}%
\endbibitem

\bibitem{BD14}
%
\begin{barticle}[mr]
\bauthor{\bsnm{Berger},~\bfnm{Noam}\binits{N.}} \AND
\bauthor{\bsnm{Deuschel},~\bfnm{Jean-Dominique}\binits{J.-D.}}
(\byear{2014}).
\btitle{A quenched invariance principle for non-elliptic random walk
in i.i.d. balanced random environment}.
\bjournal{Probab. Theory Related Fields}
\bvolume{158}
\bpages{91--126}.
\bid{doi={10.1007/s00440-012-0478-4}, issn={0178-8051}, mr={3152781}}
\end{barticle}
%

\bptok{imsref}%
\endbibitem

\bibitem{BDR12}
%
\begin{barticle}[mr]
\bauthor{\bsnm{Berger},~\bfnm{Noam}\binits{N.}},
\bauthor{\bsnm{Drewitz},~\bfnm{Alexander}\binits{A.}} \AND
\bauthor{\bsnm{Ram{\'{\i}}rez},~\bfnm{Alejandro~F.}\binits{A.~F.}}
(\byear{2014}).
\btitle{Effective polynomial ballisticity conditions for random walk
in random environment}.
\bjournal{Comm. Pure Appl. Math.}
\bvolume{67}
\bpages{1947--1973}.
\bid{doi={10.1002/cpa.21500}, issn={0010-3640}, mr={3272364}}
\bptnote{check volume, check pages, check year}%
\end{barticle}
%

\bptok{imsref}%
\endbibitem

\bibitem{BZ08}
%
\begin{bincollection}[mr]
\bauthor{\bsnm{Berger},~\bfnm{Noam}\binits{N.}} \AND
\bauthor{\bsnm{Zeitouni},~\bfnm{Ofer}\binits{O.}}
(\byear{2008}).
\btitle{A quenched invariance principle for certain ballistic random
walks in i.i.d. environments}.
In \bbooktitle{In and Out of Equilibrium. 2}.
\bseries{Progress in Probability}
\bvolume{60}
\bpages{137--160}.
\bpublisher{Birkh\"auser},
\blocation{Basel}.
\bid{doi={10.1007/978-3-7643-8786-0_7}, mr={2477380}}
\end{bincollection}
%

\bptok{imsref}%
\endbibitem

\bibitem{bolthausen2002satic}
%
\begin{barticle}[mr]
\bauthor{\bsnm{Bolthausen},~\bfnm{Erwin}\binits{E.}} \AND
\bauthor{\bsnm{Sznitman},~\bfnm{Alain-Sol}\binits{A.-S.}}
(\byear{2002}).
\btitle{On the static and dynamic points of view for certain random
walks in random environment}.
\bjournal{Methods Appl. Anal.}
\bvolume{9}
\bpages{345--375}.
\bid{doi={10.4310/MAA.2002.v9.n3.a4}, issn={1073-2772}, mr={2023130}}
\bptnote{check pages}%
\end{barticle}
%

\bptok{imsref}%
\endbibitem

\bibitem{bolthausen2002ten}
%
\begin{bbook}[mr]
\bauthor{\bsnm{Bolthausen},~\bfnm{Erwin}\binits{E.}} \AND
\bauthor{\bsnm{Sznitman},~\bfnm{Alain-Sol}\binits{A.-S.}}
(\byear{2002}).
\btitle{Ten Lectures on Random Media}.
\bseries{DMV Seminar}
\bvolume{32}.
\bpublisher{Birkh\"auser},
\blocation{Basel}.
\bid{doi={10.1007/978-3-0348-8159-3}, mr={1890289}}
\end{bbook}
%

\bptok{imsref}%
\endbibitem

\bibitem{CR13}
%
\begin{barticle}[auto:parserefs-M02]
\bauthor{\bsnm{Campos},~\bfnm{David}\binits{D.}} \AND
\bauthor{\bsnm{Ram{\'{\i}}rez},~\bfnm{Alejandro~F.}\binits{A.~F.}}
(\byear{2013}).
\btitle{Ellipticity criteria for ballistic behavior of random walks in
random environment}.
\bjournal{Probab. Theory Related Fields}
\bvolume{160}
\bpages{189--251}.
\end{barticle}
%

\bptok{imsref}%
\endbibitem

\bibitem{DR11}
%
\begin{barticle}[mr]
\bauthor{\bsnm{Drewitz},~\bfnm{A.}\binits{A.}} \AND
\bauthor{\bsnm{Ram{\'{\i}}rez},~\bfnm{A.~F.}\binits{A.~F.}}
(\byear{2011}).
\btitle{Ballisticity conditions for random walk in random environment}.
\bjournal{Probab. Theory Related Fields}
\bvolume{150}
\bpages{61--75}.
\bid{doi={10.1007/s00440-010-0268-9}, issn={0178-8051}, mr={2800904}}
\end{barticle}
%

\bptok{imsref}%
\endbibitem

\bibitem{DR12}
%
\begin{barticle}[mr]
\bauthor{\bsnm{Drewitz},~\bfnm{Alexander}\binits{A.}} \AND
\bauthor{\bsnm{Ram{\'{\i}}rez},~\bfnm{Alejandro~F.}\binits{A.~F.}}
(\byear{2012}).
\btitle{Quenched exit estimates and ballisticity conditions for
higher-dimensional random walk in random environment}.
\bjournal{Ann. Probab.}
\bvolume{40}
\bpages{459--534}.
\bid{doi={10.1214/10-AOP637}, issn={0091-1798}, mr={2952083}}
\end{barticle}
%

\bptok{imsref}%
\endbibitem

\bibitem{DR13}
%
\begin{barticle}[mr]
\bauthor{\bsnm{Drewitz},~\bfnm{Alexander}\binits{A.}} \AND
\bauthor{\bsnm{Ram\'{\i}rez},~\bfnm{Alejandro~F.}\binits{A.~F.}}
(\byear{2014}).
\btitle{Selected topics in random walk in random environment}.
\bjournal{Topics in Percolative and Disordered Systems. Springer Proc. Math. Stat.}
\bvolume{69}
\bpages{23--83}.
\bid{mr={3229286}}
\end{barticle}
%

\bptok{imsref}%
\endbibitem

\bibitem{GZ12}
%
\begin{barticle}[mr]
\bauthor{\bsnm{Guo},~\bfnm{Xiaoqin}\binits{X.}} \AND
\bauthor{\bsnm{Zeitouni},~\bfnm{Ofer}\binits{O.}}
(\byear{2012}).
\btitle{Quenched invariance principle for random walks in balanced
random environment}.
\bjournal{Probab. Theory Related Fields}
\bvolume{152}
\bpages{207--230}.
\bid{doi={10.1007/s00440-010-0320-9}, issn={0178-8051}, mr={2875757}}
\end{barticle}
%

\bptok{imsref}%
\endbibitem

\bibitem{KV86}
%
\begin{barticle}[mr]
\bauthor{\bsnm{Kipnis},~\bfnm{C.}\binits{C.}} \AND
\bauthor{\bsnm{Varadhan},~\bfnm{S.~R.~S.}\binits{S.~R.~S.}}
(\byear{1986}).
\btitle{Central limit theorem for additive functionals of reversible
{M}arkov processes and applications to simple exclusions}.
\bjournal{Comm. Math. Phys.}
\bvolume{104}
\bpages{1--19}.
\bid{issn={0010-3616}, mr={0834478}}
\end{barticle}
%

\bptok{imsref}%
\endbibitem

\bibitem{Ko85}
%
\begin{barticle}[mr]
\bauthor{\bsnm{Kozlov},~\bfnm{S.~M.}\binits{S.~M.}}
(\byear{1985}).
\btitle{The averaging method and walks in inhomogeneous environments}.
\bjournal{Uspekhi Mat. Nauk}
\bvolume{40}
\bpages{61--120}.
\bid{issn={0042-1316}, mr={0786087}}
\end{barticle}
%

\bptok{imsref}%
\endbibitem

\bibitem{La87}
%
\begin{barticle}[mr]
\bauthor{\bsnm{Lawler},~\bfnm{Gregory~F.}\binits{G.~F.}}
(\byear{1982/1983}).
\btitle{Weak convergence of a random walk in a random environment}.
\bjournal{Comm. Math. Phys.}
\bvolume{87}
\bpages{81--87}.
\bid{issn={0010-3616}, mr={0680649}}
\end{barticle}
%

\bptok{imsref}%
\endbibitem

\bibitem{MP07}
%
\begin{barticle}[mr]
\bauthor{\bsnm{Mathieu},~\bfnm{P.}\binits{P.}} \AND
\bauthor{\bsnm{Piatnitski},~\bfnm{A.}\binits{A.}}
(\byear{2007}).
\btitle{Quenched invariance principles for random walks on percolation
clusters}.
\bjournal{Proc. R. Soc. Lond. Ser. A Math. Phys. Eng. Sci.}
\bvolume{463}
\bpages{2287--2307}.
\bid{doi={10.1098/rspa.2007.1876}, issn={1364-5021}, mr={2345229}}
\end{barticle}
%

\bptok{imsref}%
\endbibitem

\bibitem{McD98}
%
\begin{bincollection}[mr]
\bauthor{\bsnm{McDiarmid},~\bfnm{Colin}\binits{C.}}
(\byear{1998}).
\btitle{Concentration}.
In \bbooktitle{Probabilistic Methods for Algorithmic Discrete Mathematics}.
\bseries{Algorithms Combin.}
\bvolume{16}
\bpages{195--248}.
\bpublisher{Springer},
\blocation{Berlin}.
\bid{doi={10.1007/978-3-662-12788-9_6}, mr={1678578}}
\end{bincollection}
%

\bptok{imsref}%
\endbibitem

\bibitem{RA03}
%
\begin{barticle}[mr]
\bauthor{\bsnm{Rassoul-Agha},~\bfnm{Firas}\binits{F.}}
(\byear{2003}).
\btitle{The point of view of the particle on the law of large numbers
for random walks in a mixing random environment}.
\bjournal{Ann. Probab.}
\bvolume{31}
\bpages{1441--1463}.
\bid{doi={10.1214/aop/1055425786}, issn={0091-1798}, mr={1989439}}
\end{barticle}
%

\bptok{imsref}%
\endbibitem

\bibitem{RAS05}
%
\begin{barticle}[mr]
\bauthor{\bsnm{Rassoul-Agha},~\bfnm{Firas}\binits{F.}} \AND
\bauthor{\bsnm{Sepp{\"a}l{\"a}inen},~\bfnm{Timo}\binits{T.}}
(\byear{2005}).
\btitle{An almost sure invariance principle for random walks in a
space--time random environment}.
\bjournal{Probab. Theory Related Fields}
\bvolume{133}
\bpages{299--314}.
\bid{doi={10.1007/s00440-004-0424-1}, issn={0178-8051}, mr={2198014}}
\end{barticle}
%

\bptok{imsref}%
\endbibitem

\bibitem{RAS07}
%
\begin{barticle}[mr]
\bauthor{\bsnm{Rassoul-Agha},~\bfnm{Firas}\binits{F.}} \AND
\bauthor{\bsnm{Sepp{\"a}l{\"a}inen},~\bfnm{Timo}\binits{T.}}
(\byear{2007}).
\btitle{Quenched invariance principle for multidimensional ballistic
random walk in a random environment with a forbidden direction}.
\bjournal{Ann. Probab.}
\bvolume{35}
\bpages{1--31}.
\bid{doi={10.1214/009117906000000610}, issn={0091-1798}, mr={2303942}}
\end{barticle}
%

\bptok{imsref}%
\endbibitem

\bibitem{RAS09}
%
\begin{barticle}[mr]
\bauthor{\bsnm{Rassoul-Agha},~\bfnm{Firas}\binits{F.}} \AND
\bauthor{\bsnm{Sepp{\"a}l{\"a}inen},~\bfnm{Timo}\binits{T.}}
(\byear{2009}).
\btitle{Almost sure functional central limit theorem for ballistic
random walk in random environment}.
\bjournal{Ann. Inst. Henri Poincar\'e Probab. Stat.}
\bvolume{45}
\bpages{373--420}.
\bid{doi={10.1214/08-AIHP167}, issn={0246-0203}, mr={2521407}}
\end{barticle}
%

\bptok{imsref}%
\endbibitem

\bibitem{Sa13}
%
\begin{barticle}[mr]
\bauthor{\bsnm{Sabot},~\bfnm{Christophe}\binits{C.}}
(\byear{2013}).
\btitle{Random {D}irichlet environment viewed from the particle in
dimension {$d\geq3$}}.
\bjournal{Ann. Probab.}
\bvolume{41}
\bpages{722--743}.
\bid{doi={10.1214/11-AOP699}, issn={0091-1798}, mr={3077524}}
\end{barticle}
%

\bptok{imsref}%
\endbibitem

\bibitem{SS04}
%
\begin{barticle}[mr]
\bauthor{\bsnm{Sidoravicius},~\bfnm{Vladas}\binits{V.}} \AND
\bauthor{\bsnm{Sznitman},~\bfnm{Alain-Sol}\binits{A.-S.}}
(\byear{2004}).
\btitle{Quenched invariance principles for walks on clusters of
percolation or among random conductances}.
\bjournal{Probab. Theory Related Fields}
\bvolume{129}
\bpages{219--244}.
\bid{doi={10.1007/s00440-004-0336-0}, issn={0178-8051}, mr={2063376}}
\end{barticle}
%

\bptok{imsref}%
\endbibitem

\bibitem{Sz01}
%
\begin{barticle}[mr]
\bauthor{\bsnm{Sznitman},~\bfnm{Alain-Sol}\binits{A.-S.}}
(\byear{2001}).
\btitle{On a class of transient random walks in random environment}.
\bjournal{Ann. Probab.}
\bvolume{29}
\bpages{724--765}.
\bid{doi={10.1214/aop/1008956691}, issn={0091-1798}, mr={1849176}}
\end{barticle}
%

\bptok{imsref}%
\endbibitem

\bibitem{Sz02}
%
\begin{barticle}[mr]
\bauthor{\bsnm{Sznitman},~\bfnm{Alain-Sol}\binits{A.-S.}}
(\byear{2002}).
\btitle{An effective criterion for ballistic behavior of random walks
in random environment}.
\bjournal{Probab. Theory Related Fields}
\bvolume{122}
\bpages{509--544}.
\bid{doi={10.1007/s004400100177}, issn={0178-8051}, mr={1902189}}
\end{barticle}
%

\bptok{imsref}%
\endbibitem

\bibitem{SZ99}
%
\begin{barticle}[mr]
\bauthor{\bsnm{Sznitman},~\bfnm{Alain-Sol}\binits{A.-S.}} \AND
\bauthor{\bsnm{Zerner},~\bfnm{Martin}\binits{M.}}
(\byear{1999}).
\btitle{A law of large numbers for random walks in random environment}.
\bjournal{Ann. Probab.}
\bvolume{27}
\bpages{1851--1869}.
\bid{doi={10.1214/aop/1022874818}, issn={0091-1798}, mr={1742891}}
\end{barticle}
%

\bptok{imsref}%
\endbibitem

\bibitem{Ze04}
%
\begin{bincollection}[mr]
\bauthor{\bsnm{Zeitouni},~\bfnm{Ofer}\binits{O.}}
(\byear{2004}).
\btitle{Random walks in random environment}.
In \bbooktitle{Lectures on Probability Theory and Statistics}.
\bseries{Lecture Notes in Math.}
\bvolume{1837}
\bpages{189--312}.
\bpublisher{Springer},
\blocation{Berlin}.
\bid{doi={10.1007/978-3-540-39874-5_2}, mr={2071631}}
\end{bincollection}
%

\bptok{imsref}%
\endbibitem

\bibitem{Ze02}
%
\begin{barticle}[mr]
\bauthor{\bsnm{Zerner},~\bfnm{Martin~P.~W.}\binits{M.~P.~W.}}
(\byear{2002}).
\btitle{A non-ballistic law of large numbers for random walks in
i.i.d. random environment}.
\bjournal{Electron. Commun. Probab.}
\bvolume{7}
\bpages{191--197}.
\bid{doi={10.1214/ECP.v7-1060}, issn={1083-589X}, mr={1937904}}
\end{barticle}
%

\bptok{imsref}%
\endbibitem
\end{thebibliography}
\end{document}